\documentclass{memo-l}

\usepackage[german,english]{babel}
\usepackage[T1]{fontenc}
\usepackage[utf8]{inputenc}
\usepackage[textsize=footnotesize,color=yellow, bordercolor=white]{todonotes}

\renewcommand{\todo}[1]{}

\newtheorem{thm}{Theorem}[chapter]
\newtheorem{lem}[thm]{Lemma}
\newtheorem{claim}[thm]{Claim}
\newtheorem{fct}[thm]{Fact}
\newtheorem{question}[thm]{Question}
\newtheorem{cor}[thm]{Corollary}
\newtheorem{conjecture}[thm]{Conjecture}

\theoremstyle{definition}
\newtheorem{dfn}[thm]{Definition}
\newtheorem{con}[thm]{Convention}
\newtheorem{exm}[thm]{Example}
\newtheorem{exc}[thm]{Exercise}

\theoremstyle{remark}
\newtheorem{rem}[thm]{Remark}
\newtheorem{ass}[thm]{Assumption}

\numberwithin{section}{chapter}
\numberwithin{equation}{chapter}

\usepackage[original]{imakeidx}
\usepackage{amsrefs}

\newcommand{\myplacecite}[2]{\cite{#2}*{#1}}
\newcommand{\mycites}{\cites}

\usepackage{amssymb}
\usepackage[all]{xy}

\usepackage[shortlabels]{enumitem}
\setlist[enumerate,1]{label=\arabic*., ref=\arabic*}

\usepackage{accents}

\usepackage[pdfauthor={Sy Friedman, David Schrittesser},pdftitle={Projective measure without projective Baire},pdfproducer={Latex with hyperref}]{hyperref}

\newenvironment{eqpar}{\begin{equation}\begin{minipage}{0.8\columnwidth}}%
{\end{minipage}\end{equation}}

\newenvironment{eqpar*}{\begin{equation*}\begin{minipage}{0.8\columnwidth}}%
{\end{minipage}\end{equation*}}

\newcommand{\g}{g}
\newcommand{\gT}{T_{\g}}
\newcommand{\gH}{H_{\g}}
\newcommand{\gr}{\rho_{\g}}
\newcommand{\gl}{\lambda_{\g}}
\newcommand{\gq}{q_{\g}}
\DeclareMathOperator{\width}{width}
\providecommand{\qcdefF}{\Pi^T_1}
\providecommand{\qcdefD}{\Pi^T_1}
\providecommand{\qcdefSeq}{\Pi^T_1}
\providecommand{\qcdefG}{\Sigma^T_1}

\newcommand{\param}{c}
\providecommand{\pcsl}{\ensuremath{(\D, \param,\leqlol)_{\lambda\in I}}}
\providecommand{\pcs}{\ensuremath{\mathbf{s}}}

\providecommand{\pssl}{\ensuremath{(\D, \leqlol,\lequpl,\Cl)_{\lambda\in I}}}
\providecommand{\pss}{\ensuremath{\mathbf{S}}}

\providecommand{\is}{\lhd}

\DeclareMathOperator{\Th}{\mathbf{Th}}

\DeclareMathOperator{\pset}{\mathcal{P}}

\DeclareMathOperator{\tcl}{tcl}

\providecommand{\forces}{\Vdash}

\providecommand{\On}{\mathbf{On}}

\newcommand{\Card}{\mathbf{Card}}
\newcommand{\Succ}{\mathbf{Succ}}
\newcommand{\Inacc}{\mathbf{Inacc}}
\newcommand{\Reg}{\mathbf{Reg}}
\newcommand{\Sing}{\mathbf{Sing}}

\providecommand{\ccc}{\textup{ccc}}
\providecommand{\card}[1]{\lVert#1\rVert}

\providecommand{\ro}{\mathrm{r.o.}}

\providecommand{\supp}{\mathrm{ supp }}
\providecommand{\suppl}{{\supp^\lambda}}
\providecommand{\res}{\mathbin{\upharpoonright} }
\providecommand{\conc}{ \mathbin{{}^\frown}}
\providecommand{\Int}{\mathbb{Z}}
\providecommand{\reals}{\mathbb{R}}
\providecommand{\nat}{\mathbb{N}}

\providecommand{\CH}{\ensuremath{\textup{CH}}}
\providecommand{\GCH}{\ensuremath{\textup{GCH}}}
\providecommand{\ZFC}{\textup{ZFC}}
\providecommand{\ZF}{\textup{ZF}}

\DeclareMathOperator{\HSize}{\mathbf{H}}

\DeclareMathOperator{\hgt}{ht}

\DeclareMathOperator{\cof}{cf}

\DeclareMathOperator{\otp}{otp}

\providecommand{\setdef}{\;\vert\;}

\providecommand{\Hhier}{\mathbf{H}}

\providecommand{\Coll}{\textup{Coll}}
\providecommand{\Borel}{\mathbf{Borel}}
\providecommand{\Borelplus}{\Borel^+}
\providecommand{\Add}{\textup{Add}}
\providecommand{\bv}[1]{\lVert #1 \rVert}

\providecommand{\Dam}{\mathcal{D}}
\providecommand{\am}{\ensuremath{{\mathbf{Am}_1}}}
\providecommand{\simpleram}{\mathbf{Am}_2}
\providecommand{\blowup}[1]{\widehat{ #1 }}

\providecommand{\restam}[2]{ #1^{ #2 }_f  }
\providecommand{\loweram}[2]{ #1^{ (-\infty, #2 ] }_f}
\providecommand{\upperam}[2]{ #1^{ [ #2 ,-\infty) }_f}

\providecommand{\id}{\textup{id}}
\providecommand{\bg}[1]{B(#1)}
\providecommand{\quot}[2]{#1 : #2}
\providecommand{\ecn}[1]{[ #1 ]_{\mathbf{n}}}
\providecommand{\gen}[1]{\langle #1 \rangle}

\DeclareMathOperator{\dom}{dom}
\DeclareMathOperator{\lh}{lh}

\DeclareMathOperator{\ran}{ran}

\DeclareMathOperator{\crit}{crit}
\DeclareMathOperator{\Ult}{Ult}

\providecommand{\leqlo}{\preccurlyeq}
\providecommand{\lequp}{\curlyeqprec}
\providecommand{\Clink}{\mathbf{C}}

\providecommand{\leqlol}{\leqlo^\lambda}
\providecommand{\bleqlol}{ \mathrel{{\bar \leqlo}^\lambda}}
\providecommand{\bleqlo}{ \mathrel{\bar \leqlo} }
\providecommand{\dleqlol}{ {\mathrel{\dot \leqlo}^\lambda} }
\providecommand{\dleqlo}{ \mathrel{\dot \leqlo} }
\providecommand{\lequpl}{\lequp^\lambda}
\providecommand{\blequpl}{\mathrel{\bar \lequp^\lambda}}
\providecommand{\dlequpl}{{\mathrel{\dot \lequp}^\lambda}}
\providecommand{\Cl}{ \Clink^\lambda }
\providecommand{\bCl}{ {\bar \Clink}^\lambda }

\providecommand{\D}{\mathbf{D}}

\DeclareMathOperator{\bij}{\dot{\mathcal B}}

\providecommand{\alo}{\mathrel{\bar{\leqlo}}}
\providecommand{\aup}{\mathrel{\bar{\lequp}}}
\providecommand{\aD}{\bar{\D}}
\providecommand{\aC}{\bar{\Clink}}

\providecommand{\alol}{\mathrel{\alo^\lambda}}
\providecommand{\aupl}{\mathrel{\aup^\lambda}}
\providecommand{\aCl}{\aC^\lambda}
\providecommand{\stm}{\wedge}

\providecommand{\qstm}{\stm_Q}
\providecommand{\astm}{\stm_P}

\newcommand{\baseset}{\star}

\providecommand{\pcsIsRef}{\ensuremath{\is_\text{c}}\arabic*}
\providecommand{\pssIsRe}{\ensuremath{\is_\text{s}}}
\providecommand{\pssIsRef}{\pssIsRe\arabic*}
\providecommand{\qcExtRef}{\ensuremath{\text{E}_\text{c}}\Roman*}
\providecommand{\sExtRef}{\ensuremath{\text{E}_\text{s}}\Roman*}

\providecommand{\pmaao}[3][q]{\ensuremath{P^{#2,#1}_{#3}}}
\providecommand{\pmao}[2][q]{\pmaao[#1]{\nu}{#2}}
\providecommand{\pnu}{\pmao{\xi}}

\providecommand{\emaao}[3][q]{\ensuremath{e^{#2,#1}_{#3}}}
\providecommand{\emao}[2][q]{\emaao[#1]{\nu}{#2}}
\providecommand{\emb}{\emao{\xi}}

\providecommand{\tmaao}[3][q]{\ensuremath{t^{#2,#1}_{#3}}}
\providecommand{\tmao}[2][q]{\tmaao[#1]{\nu}{#2}}
\providecommand{\tm}{\tmao{\xi}}

\makeindex
\makeindex[name=notation, title=Index of Notation]

\begin{document}

\frontmatter

\title{Projective Measure Without Projective Baire}

\author{Sy David Friedman}
\address{G\"odel Research Center, University of Vienna}
\curraddr{}
\email{sdf@logic.univie.ac.at}
\thanks{The first author would like to thank the FWF for its support through research project P25748.}

\author{David Schrittesser}
\address{G\"odel Research Center, University of Vienna}
\curraddr{University of Toronto, ON, Canada}
\email{david@logic.univie.ac.at}
\thanks{The second author would like to express great gratitude to his PhD-adviser Sy David Friedman for his enduring support and his guidance. Furthermore, he acknowledges generous support through SFB 878 for which he wants to thank Ralf Schindler and partly through FWF projects P24725-N25 and P23875-N13 as well as through the DNRF Niels Bohr Professorship of Lars Hesselholt during the final stage of preparation of this memoir.}

\date{}



\begin{abstract}
We prove that it is consistent (relative to a Mahlo cardinal) that all projective sets of reals are Lebesgue measurable, but there is a $\Delta^1_3$ set without the Baire property.
The complexity of the set which provides a counterexample to the Baire property is optimal.
\end{abstract}

\maketitle

\tableofcontents

\mainmatter

\chapter{Introduction}

The purpose of this memoir is to answer a long standing question regarding \emph{Lebesgue measurability} and the \emph{Baire property} and how their behavior can differ with respect to sets in the \emph{projective hierarchy}.
Namely, we prove the following theorem (see also Theorem~\ref{p.t.main2} below):
\begin{thm}\label{p.t.main}
Assume that $\ZFC$ together with `there exists a Mahlo cardinal' is consistent.
Then so is $\ZFC$ together with the conjunction of the following two statements:
\begin{itemize}
\item Every projective set is Lebesgue measurable;
\item There is a projective  set without the Baire property.
\end{itemize}
\end{thm}

With a view to future applications, certain elements of the proof are treated in a quite general setting.
One such is our notion of \emph{stratified forcing} (Chapter~\ref{sec:stratified forcing}), which essentially provides a theory of iterated Jensen coding (an Easton supported version of which is developed in Chapter~\ref{sec:coding}) and similar forcings. 
\emph{Stratified extension} (Chapter~\ref{sec:ext}) gives this iteration theory a form amenable to our variant of 
\emph{amalgamation}, a method to build `sufficiently homogeneous' iterations which is described in Chapter~\ref{sec:amalgamation}.
As many questions similar to the one which inspired this work are still open (see Section~\ref{sec:where} below) we expect that our very general treatment of these notions will remain useful beyond the present context.

In Chapter~\ref{s.overview.proof} we give an overview of the proof, as it is expounded in detail in the remainder of this memoir (in particular, in Chapter~\ref{sec:main}).
We hope that this will give a reader with some knowledge of set theory a better sense of orientation in this long and sometimes very technical proof.

\medskip

With the nature as well as the history of the problem in view it is perhaps to be expected that the proof of Theorem~\ref{p.t.main} takes so much work.
In the remainder of this chapter we give a short historical account helping to contextualize this theorem (many of the historical remarks below are taken from \cite{bagaria:woodin}).
We try to do this in a manner at least somewhat accessible to a hypothetical reader with little knowledge of (descriptive) set theory.

\section{A Short History of the Problem}\label{s.history}

Descriptive set theory can be said to be the study of properties of \emph{definable subsets of $\reals$}, i.e., subsets which arise in some concrete manner.
A question of particular interest has always been if such sets must be `well-behaved' or `regular' in the sense of, e.g., being Lebesgue measurable~\cite{lebesgue1902}.\index{Lebesgue measurability}

\medskip

A property which is related to Lebesgue measurability is the \index{Baire property|textbf}\emph{Baire property}~\cite{baire1899}: recall that a subset of $\reals$ is said to have the property of Baire if and only if it agrees with an open (equivalently, closed) set on a dense $G_\delta$ set (a $G_\delta$ set is just a countable intersection of open sets).
Recall further that supersets of dense $G_\delta$ sets are also called sets of \emph{second category} and their complements are called \emph{meager sets} or \emph{sets of first category}. 
The meager sets form a $\sigma$-ideal (resp., the co-meager sets or sets of second category form a $\sigma$-complete filter), as do the Lebesgue null sets (resp., co-null sets).
In this sense, sets with the Baire property correspond to the measurable sets.
More on the Baire property can be found in \cite{kechris} (see also \cite{oxtoby}).

We think of a set with the  Baire property as likewise being ``well-behaved'' or regular in some sense. 
The Baire property can often play a parallel or complementary role to that of measurability in applications and has the advantage of being useful in some topological spaces where no measure is available.%

\medskip

Of the many parallels between measure and category (i.e., the Baire property), most interesting in the present context is the following: while the Axiom of Choice allows us to construct non-measurable sets and sets without the Baire property (e.g., \emph{Vitali sets} \cite{vitali} or \emph{Luzin sets} \cite{lusin1914}), subsets of $\reals$ which arise in a very simple manner are always Lebesgue measurable and have the property of Baire: this holds, e.g., for the Borel sets; %
or more generally for the \emph{analytic sets} \mycites{lusin1917,lusin-sierpinski-b}, the continuous images of Borel sets (see \cite{kechris}).\footnote{For the early history of descriptive set theory see \mycites{kanamori-emergence,kanovei-developement}; some bibliographical details can also be found in \myplacecite{p.~153}{jech-set}.} %

Nonetheless the null sets and the meager sets are in many aspects very different. %
One would not expect that regularity with respect to Lebesgue measure should in some general manner correlate with regularity with respect to the Baire property---and indeed it does not.

\medskip

Beyond the analytic sets, descriptive set theory studies the \emph{projective sets}:\index{projective (set, hierarchy)} These are the sets which can be obtained from open sets by finitely many iterations of taking complements and images under continuous functions $f\colon \reals \to \reals$.
The projective sets are naturally stratified into the hierarchy of of $\mathbf{\Sigma}^1_n$ and $\mathbf{\Pi}^1_n$ sets, where $n\in \nat$.
An even finer stratification into a hierarchy is obtained by allowing only images under \emph{recursively continuous} functions;
this is the so-called \emph{effective} or \emph{light-face hierarchy} of $\Sigma^1_n$ and $\Pi^1_n$ sets, where $n\in \nat$. For details on the projective hierarchy see \cite{kechris}; good sources for its light-face companion are \mycites{moschovakis,mansfield-recursive}.

\medskip

$\ZFC$ leaves open many questions about relatively simply definable sets (i.e., sets low in the projective hierarchy);
in particular it leaves open whether they have the properties of regularity discussed above, i.e., Lebesgue measurability and the Baire property.

In the 1930s 
G\"odel \mycites{godel-consistency,godel-monograph} constructed his famous \emph{universe $L$ of constructible sets}  and showed that in this model, there is $\Delta^1_2$ well-order of $\reals$ (i.e., a well-order which as a subset of $\reals\times\reals$ is both $\Sigma^1_2$ and $\Pi^1_2$).
Consequently it is consistent with $\ZFC$ that there are $\Delta^1_2$ sets without Lebesgue measure and without the Baire property. %

Around 1960 Paul Cohen \mycites{Cohen1963,Cohen1964} developed forcing and famously used it to show that the continuum hypothesis is independent of $\ZFC$.
Not too long after that Solovay\index{Solovay's model} used forcing to construct a model of $\ZFC$ in which every projective set is measurable and has the Baire property \cite{solovay}. %
In fact, in a slightly smaller submodel of said model, only a weak version of the Axiom of Choice (namely, Dependent Choice) holds and \emph{all sets} are Lebesgue measurable and have the Baire property.%

Thus, G\"odels $L$ and the model that came to be known as Solovay's model represent two extreme cases of behavior for measurability and the Baire property in the projective hierarchy.

\medskip

Before we continue with the main story, two points have to be noted:
Firstly, we should mention that already at an early stage also other \emph{notions of regularity}\index{notions of regularity} had come into focus, such as the \emph{perfect set property} and the property of being \emph{completely Ramsey}. Since this memoir deals with Lebesgue measure and the Baire property, it suffices to concentrate on these two properties in our historical overview, but we will mention some other notions of regularity below in Section~\ref{sec:where} when we discuss possibilities to extend the present work.

\medskip

Secondly, we should mention that since the 1980s due to the work of Mycielski, Steinhaus, Moschovakis, Martin, Woodin, Steel, and others, large cardinal axioms have been found to settle almost all questions about the projective hierarchy (and in fact much more).
For example, if there are infinitely many Woodin cardinals\index{cardinal!Woodin}, every projective set of reals is Lebesgue measurable and has the Baire property (see \cite{kanamori} or \cite{kechris}).

Yet whether we should admit the existence of Woodin cardinals into our standard set of axioms for main-stream mathematics arguably remains to be determined.
Although these are exciting developments, in this memoir we are not concerned with the effects of potential new axioms or whether there are good reasons to adopt them. 
Rather we are interested in the mere \emph{consistency} of the set of statements\footnote{Of course, a set of sentences is consistent when no contradictions can be proved from it, or equivalently, if it has a model.} described in Theorem~\ref{p.t.main} 
and want to prove this consistency from an assumption which is as weak as possible.

\medskip

So far in this historical overview, we have treated Lebesgue measurability and the Baire property on par. 
In fact, for a long time no model was known where the ``simplest'' non-measurable set and the ``simplest'' set lacking the Baire property were of different complexity.
This unsatisfactory state of affairs was at least partially resolved by Shelah, in his article titled ``Can you take Solovay's inaccessible away?'' \cite{shelah:amalgamation}.

Solovay had constructed his eponymous model (where all sets are measurable and have the Baire property) starting from a ground model containing an inaccessible cardinal.%
\index{cardinal!inaccessible}
Shelah showed that it was necessary to start with an inaccessible in the case of Lebesgue measure:
\begin{thm}[\cite{shelah:amalgamation}]\label{p.t.lm}
Suppose $\omega_1$ is not inaccessible in $L$. Then one of the following holds:
\begin{enumerate}
\item\label{p.i.random-reals} For some $a \in \omega^\omega$, the set $N^*_a$ defined by
\[
N^*_a = \bigcup \{ A \subseteq \reals \setdef \text{$A$ is null and has a Borel code in $L[a]$}\}
\]
is not a null set;
\item For some $a \in \omega^\omega$, ${\omega_1}^{L[a]} = \omega_1$ and there is a $\Sigma^1_3(a)$ set which is not measurable.
\end{enumerate}
In particular, 
 if every $\mathbf{\Sigma}^1_3$ set is Lebesgue measurable $\omega_1$ is inaccessible in $L$. 
\end{thm}
The last statement holds as Item~\ref{p.i.random-reals} above implies that $N^*_a$ is a $\Sigma^1_2(a)$ set which is not measurable (compare Theorem~\ref{p.t.sigma-1-2} and also Theorem~\ref{p.t.delta-1-2} below).
Thus to construct a model where all projective (in fact, just all $\mathbf{\Sigma}^1_3$) sets are Lebesgue measurable, one \emph{must} start with an inaccessible. 
The same result was obtained by Raisonnier in \cite{raisonnier}.
Shelah also showed in the same paper \cite{shelah:amalgamation} that every model of $\ZFC$ \emph{does} have an extension where every $\mathbf{\Delta}^1_3$ set is Lebesgue measurable.

\medskip

Moreover, again in the same work \cite{shelah:amalgamation} Shelah showed that the situation for the Baire property is entirely different (see also \cite{jr:amalgamation} for an exposition of the following theorem).
\begin{thm}[\cite{shelah:amalgamation}]\label{p.t.bp}
Any model of $\ZFC$ has a forcing extension in which all projective sets have the Baire property. %
In an inner model of this extension, $\ZF$ together with the Axiom of Dependent Choice\index{Axiom!of Dependent Choice} holds and all sets have the Baire property.
\end{thm}

If one starts with a model without an inaccessible, by Theorem~\ref{p.t.lm} there must be a $\mathbf{\Sigma}^1_3$ set which is not Lebesgue measurable in both models from Theorem~\ref{p.t.bp}. 
Here, finally, the expected asymmetry between measure and category was beginning to be seen! 

The forcing mentioned in Theorem~\ref{p.t.bp} is \emph{sweet}\index{forcing!sweet}\index{sweetness} (see \cite{shelah:amalgamation}) and thus $\sigma$-centered and so by \cite{judah:repicky} does not add a random real.\index{random real!not adding a random real} Thus by later results (see Theorem~\ref{p.t.delta-1-2} below) in fact the following sharp version is true:
\begin{cor}\label{p.t.asym}
There is a model of $\ZFC$ where every projective set has the Baire property but there is a $\Delta^1_2$ set which is not Lebesgue measurable %
 (supposing $\ZFC$ is consistent).
\end{cor}

At the $\Sigma^1_2$ level, on the other hand there exists a surprising implication\index{implications between regularity properties}\index{regularity properties!implications between} discovered by Bartoszy\'nski \cite{bartoszynski:additivity}, and independently by Raisonnier and Stern \cite{raisonnier:stern}:
\begin{thm}[\mycites{bartoszynski:additivity,raisonnier:stern}]\label{t.barto}
If all $\mathbf{\Sigma}^1_2$ sets are Lebesgue measurable then all $\mathbf{\Sigma}^1_2$ sets have the property of Baire.
\end{thm}
The theorem has a counterpart in the effective  hierarchy.
Bartoszy\'nski and Raisonnier and Stern in fact showed even more, namely that the \emph{additivity} of measure is at least the additivity of Baire---see also, e.g., \mycites{settheoryoftherealline,handbook-invariants}.
Note that Theorem~\ref{t.barto} also follows from Theorem~\ref{p.t.delta-1-2} (which was shown later) below.
It follows from folklore results going back to Martin and Solovay \cite{martin:solovay} that the converse to Theorem~\ref{t.barto} does not hold.
Finally, we mention that by Theorem~\ref{t.barto} the complexity of the set lacking the Baire property  in our main result (namely $\Delta^1_3$; see Theorem~\ref{p.t.main2} below) is optimal.

\medskip

This called for further elucidation of  the asymmetric behavior of Lebesgue measure and category.
Are there any implications between statements of the form `all sets of reals in the class $\Gamma$ have property $P$', where $\Gamma$ is the class of $\mathbf{\Sigma}^1_n$, $\mathbf{\Delta}^1_n$, $\Sigma^1_n$, or $\Delta^1_n$, for some $n\in \nat$, or the class of all projective sets; and $P$ is  the property of being Lebesgue measurable or the Baire property? 

In particular, the following question, asking whether the situation opposite to that in Theorem~\ref{p.t.asym} is consistent, remained open (the question we settle in this memoir with Theorem~\ref{p.t.main}):
\begin{question}\label{p.t.q}
Is it consistent that all projective sets are Lebesgue measurable but there is a projective set without the Baire property? 
\end{question}

A first step towards an answer to Question~\ref{p.t.q} was taken soon after by Shelah~\cite{shelah:meas_wo_cat}:
\begin{thm}[\cite{shelah:meas_wo_cat}]\label{shelah:set:wo:bp}
If $\ZFC$ together with the existence of an inaccessible cardinal %
\index{cardinal!inaccessible} is consistent, so is $\ZF$ together with:
\begin{itemize}
\item The Axiom of Dependent Choice;\index{Axiom!of Dependent Choice}
\item All sets are Lebesgue measurable; %
\item There is a set without the Baire property.
\end{itemize}
\end{thm}
The set without the Baire property in this construction is emphatically \emph{not} definable, let alone projective.
For this result (just as for the earlier \cite{shelah:amalgamation}) see also the excellent exposition in \cite{jr:amalgamation}.

\medskip

Another step towards showing a projective version was taken by Judah and again, Shelah \cite{ihoda:shelah}:
\begin{thm}[\cite{ihoda:shelah}]\label{p.t.delta-1-2}
The following are equivalent for any $a \in \omega^\omega$:
\begin{itemize}
\item There is a real which is random generic (resp., a Cohen real) over $L[a]$;\index{random real}\index{Cohen real}
\item Every $\Delta^1_2(a)$ set is Lebesgue measurable (resp., has the Baire property).
\end{itemize}
\end{thm}
This can be viewed as a refinement of the following classical result implicit in \cite{solovay} (see also \myplacecite{26.20,~p.~522}{jech-set}).
\begin{thm}[\cite{solovay}]\label{p.t.sigma-1-2}\index{random real!measure one set of random reals}\index{Cohen real!co-meager set of Cohen reals}
The following are equivalent for any $a \in \omega^\omega$:
\begin{itemize}
\item The set  
 $\bigcup \{ B \subseteq \reals \setdef \text{$B$ is null and has a Borel code in $L[a]$}\}$
is a null set;
\item Every $\Sigma^1_2(a)$ set is Lebesgue measurable.
\end{itemize}
The same is true with ``is Lebesgue measurable'' replaced by ``has the Baire property'' and ``null'' replaced by ``meager''.
\end{thm}
To see the parallel, note that the first item in Theorem~\ref{p.t.sigma-1-2} can be equivalently  expressed by saying that \emph{almost all} reals are
random (resp.\ Cohen) generic over $L[a]$---i.e., all reals with the exception of a null (resp.\ meager) set.

\medskip

It follows from Theorem~\ref{p.t.delta-1-2} that one can construct models where all $\Delta^1_2$ sets are measurable (or not) and have the Baire property (or not) in any desired combination  starting just from the assumption that $\ZFC$ is consistent.

\medskip

Judah and Shelah lifted their result to the next level in the projective hierarchy with \cite{js:delta-1-3}, in which they showed that assuming the consistency of $\ZFC$ together with the existence of a measurable cardinal\index{measurable cardinal}\index{cardinal!measurable},
one can construct a model in which every $\mathbf{\Delta}^1_3$ is measurable but there is a $\mathbf{\Delta}^1_3$ set without the property of Baire.
Bagaria and Judah in unpublished work removed the measurable cardinal and showed the same result under just the assumption that $\ZFC$ is consistent.
This was achieved using ideas from \cite{harrington:shelah} and making a connection to Martin's Axiom
(interest in Martin's Axiom in connection to these questions goes back to Truss' early attempt to construct a model where all sets are Lebesgue measurable from just $\ZFC$---see \myplacecite{p.~15}{shelah:amalgamation}).

\medskip

Bagaria and Woodin \cite{bagaria:woodin} pushed asymmetry results to the next level, obtaining a model
where all $\mathbf{\Delta}^1_4$ sets are measurable but there is a $\mathbf{\Delta}^1_4$ set without the property of Baire,
starting from the assumption that $\ZFC$ plus $\mathbf{\Sigma}^1_1$ determinacy\index{Axiom!of Projective Determinacy} is consistent (see \cite{kanamori} for more on the Axiom of Determinacy and its projective fragments).
The assumptions they start with are thus astronomical in terms of consistency strength, compared to our assumption of a Mahlo cardinal. 
In the same work they claim their results would lift to higher levels of the projective hierarchy under the appropriate stronger determinacy axioms.

\medskip

This finally brings us to the result of the present memoir:
\begin{thm}\label{p.t.main2}
Assume that $\ZFC$ together with the existence of a Mahlo cardinal\index{Mahlo cardinal}\index{cardinal!Mahlo} is consistent.
Then so is $\ZFC$ together with:
\begin{itemize}
\item Every set of reals definable by a first order formula, allowing a sequence of ordinals of length $\omega$ as a parameter (in particular every projective set) is Lebesgue measurable.
\item There is a $\Delta^1_3$ set without the Baire property.
\end{itemize}
\end{thm}

Just as in Solovay's classical work \cite{solovay}, Theorem~\ref{p.t.main2} naturally implies the following `choiceless' version by relativizing to $L(\On^\omega)$ or $L(\reals)$: 
\begin{cor}
Assume that $\ZFC$ together with the existence of a Mahlo cardinal is consistent.
Then so is $\ZF$ together with:
\begin{itemize}
\item Dependent Choice;\index{Axiom!of Dependent Choice}
\item Every set is Lebesgue measurable;
\item There is a $\Delta^1_3$ set without the Baire property.
\end{itemize}
\end{cor}

Since the proof of Theorem~\ref{p.t.main2} but prior to this publication, further results have been obtained in \cite{laguzzi} and \cite{cichon}.
These results will be discussed in the next section.

\section{Where Do We Go from Here?}\label{sec:where}

The first and most obvious question left open in this memoir is the following:
\begin{question}
What is the exact consistency strength of the set of sentences from Theorem~\ref{p.t.main2}?\index{Mahlo cardinal}\index{cardinal!Mahlo}
\end{question}
By our main theorem together with Theorem~\ref{p.t.lm} the exact consistency strength must lie between an inaccessible and a Mahlo cardinal, which is already a very narrow gap in terms of consistency strength.

\medskip

So far, we have only discussed two properties of regularity, Lebesgue measurability and the property of Baire.\index{regularity properties}
Of the many others that deserve study we have already mentioned the perfect set property\index{perfect set property} (already considered by Solovay and closely related to \emph{Marczewski-measurability})\index{Marczewski-measurability} and being completely Ramsey\index{completely Ramsey} (brought to prominence in work of Galvin, Prikry \cite{galvin:prikry}, Mathias \mycites{mathias:ramsey,mathias.1967}, Silver \cite{silver:ramsey}, and Ellentuck \cite{ellentuck}).
So many notions have been considered to this date that we have no hope of reviewing their history here. 

\medskip

Instead we recall the following definitions from \cite{ikegami2010}:
\begin{dfn}\index{tree forcing}\index{forcing!tree forcing}
  We say $P$ is a \emph{tree forcing} if and only if the conditions in
  $P$ are perfect subtrees of ${}^{<\omega}\baseset$, where
  $\baseset \in \omega+1$, $P$ is ordered by reverse inclusion, and for all $T \in P$ and all $s \in T$,
  $\{ t \in T \setdef s \subseteq t \text{ or } t \subseteq s \} \in
 P$. 
\end{dfn}
In \cite{ikegami2010} where these definitions originate
the forcings we call tree forcings are called \emph{strongly arboreal}.\index{strongly arboreal forcing}\index{forcing!strongly arboreal}

\medskip

Each tree forcing has associated to it an ideal and a \emph{regularity property} called $P$-measurability.\index{P-measurable@$P$-measurable}
For the following  recall that for a subtree $T$ of ${}^{<\omega}\baseset$, $[T]$ denotes the set of its branches (see \cite{kechris}).
\begin{dfn}\index[notation]{NP, ideal of P-null set@$N_P$, ideal of $P$-null set}\index[notation]{IP, sigma-ideal@$I_P$, $\sigma$-ideal}\index{P-null set@$P$-null set} 
  Suppose that $P$ is a tree forcing and $A$ is a subset of
  ${}^{\omega}\baseset$.
\begin{enumerate} 
\item We say $A\in N_{P}$  if and only if for every $T\in P$, there is
  some $S\in P$ with $S\subseteq T$ and
  $[S]\cap A=\emptyset$. 
\item Denote by $I_{P}$ is the $\sigma$-ideal\index{sigma-ideal@$\sigma$-ideal} generated by
  $N_{P}$. A set $A$ in $I_{P}$ is also called
    $P$-null.

  \item We say $A$  is \emph{$P$-measurable} if and only if for
  every $T \in P$ there is some $S \in P$ with
  $S \subseteq T$ such that either $[S] \cap A \in I_{P}$ or
  $[S] \cap A^c \in I_{P}$, where $A^c$ denotes
  ${}^{\omega}\baseset \setminus A$.
\end{enumerate} 
\end{dfn} 
Working in Baire space $\omega^\omega$, the notion of measurability associated to Cohen forcing is precisely the Baire property, and (the tree version of) random forcing yields precisely Lebesgue measurability (with the appropriate measure on $\omega^\omega$; see the remarks on p.~\pageref{ss.reals} below).
The following table lists some well-known notions of regularity associated to tree forcings (we make no claim to completeness; one could include for instance \emph{eventually different forcing}, \emph{Heckler forcing} and many more).

\medskip
\begin{center}
 \begin{tabular}{ l  l  l }\index{regularity properties!associated to tree forcings}

Symbol & Tree forcing  & Common name (if any)\\
\hline
$\mathbf{B}$& Random\index{random real}\index{forcing!random}  &       measurable\index{Lebesgue measurability}        \\% & (perfect) $T$ s.t. $\mu([T])>0$ \\
$\mathbf{C}$& Cohen\index{Cohen real}\index{forcing!Cohen}  &        Baire property\index{Baire property}         \\%&full tree above stem \\
$\mathbf{S}$ & Sacks\index{Sacks forcing}\index{forcing!Sacks} &       Marczewski-measurable\index{Marczewski-measurability}          \\%& perfect trees on $2^{<\omega}$\\
$\mathbf{V}$ & Silver\index{Silver forcing}\index{forcing!Silver} &      donut property\index{donut property}           \\%&uniform perfect trees  on $2^{<\omega}$\\
$\mathbf{M}$& Miller\index{Miller forcing}\index{forcing!Miller} &                   \\%&superperfect tree on $\omega^{\omega}$ (infinite splitting)\\
$\mathbf{L}$& Laver\index{Laver forcing}\index{forcing!Laver} &                    \\%&every node  above stem has infinite splitting\\
$\mathbf{R}$& Mathias\index{Mathias forcing}\index{forcing!Mathias} &   completely Ramsey\index{completely Ramsey}               \\%&see \cite{} \\

\end{tabular}
\end{center}

\medskip

We introduce another convenient short-hand:
\begin{dfn}
If  $P$ is a tree forcing and $\Gamma$ is a class of subsets of ${}^\omega \baseset$ closed under continuous preimages (a \emph{pointclass})\index{pointclass}, ``$\Gamma( P)$'' denotes the statement ``all sets in $\Gamma$ are $P$-measurable.'' 
We denote by $\mathbf{\Sigma}^1_\omega$ the class of projective sets (and by $\mathbf{\Sigma}^1_n$, $\mathbf{\Pi}^1_n$, \dots{} the class of sets which are $\mathbf{\Sigma}^1_n$, $\mathbf{\Pi}^1_n$, \dots).
\end{dfn}
With this notation, we can give a concise table (called ``Cicho\'{n}'s diagram for regularity properties'' in \cite{cichon}) of relationships known to hold 
 in $\ZFC$.  %

\begin{figure}[h]
\[
\xymatrix@C=0.6cm@R=0.8cm{
   \Gamma(\mathbf{B}) \ar@{=>}[r]&  \Gamma(\mathbf{V})\ar@{=>}[rr] & &
\Gamma(\mathbf{S}) \\
   \Gamma(\mathbf{R}) \ar@{=>}[ru] \ar@{=>}[r]& \Gamma(\mathbf{L}) \ar@{=>}[r] &
\Gamma(\mathbf{M}) \ar@{=>}[ru] \\
    & & \Gamma(\mathbf{C}) \ar@{=>}[luu] \ar@{=>}[u]
 }
\]
\caption{Cicho\'{n}'s diagram for regularity properties\index{regularity properties!Cichon's diagram for@Cicho\'{n}'s diagram for}\index{implications between regularity properties|textbf}\index{regularity properties!implications between|textbf}\index{Cichon's diagram for regularity properties@Cicho\'{n}'s diagram for regularity properties}}\label{Cicho\'{n}}\end{figure}
In addition to these implications and the familiar $\mathbf{\Sigma}^1_2(\mathbf{B})\Rightarrow \mathbf{\Sigma}^1_2(\mathbf{C})$ (from Theorem~\ref{t.barto}), %
it also holds that $\mathbf{\Sigma}^1_2(\mathbf{V})\Rightarrow \mathbf{\Sigma}^1_2(\mathbf{M})$.
Moreover
for certain $P$
we have $\mathbf{\Delta}^1_2(P ) \iff \mathbf{\Sigma}^1_2( P)$;
for some others, it is open whether this equivalence holds.
Some of these implications have involved proofs,
so it is natural to ask 
whether more is provable in $\ZFC$.

Thus, we shall say that this \emph{diagram is solved for $\Gamma = \mathbf{\Sigma}^1_n$} ($n$ finite or $n=\omega$) if we can show that
the implications listed in the diagram and in the previous paragraph are the only ones provable in $\ZFC$ (alternatively, one should prove the remaining ones).
Almost every remaining question for $n\leq4$ and some for $n =5$ were answered by Fischer, Friedman and Khomskii; see \cite{cichon}.

Laguzzi in \cite{laguzzi} showed the consistency, relative to an inaccessible, of $\ZFC$ together with ``$\omega_1$ is inaccessible to reals, Silver measurability holds for all sets but Miller and Lebesgue measurability fail for some sets''.

\medskip

Many questions are open;
the most famous one since the late 1960s:
\begin{question} %
\index{cardinal!inaccessible}\index{completely Ramsey}
Is it necessary to start with an inaccessible in order to build a model of $\mathbf{\Sigma}^1_\omega(\mathbf{R})$?
\end{question}
The next question is also long standing:
\begin{question}\index{implications between regularity properties}\index{regularity properties!implications between}
Does $\mathbf{\Sigma}^1_2(\mathbf{L})\Rightarrow \mathbf{\Sigma}^1_2(\mathbf{V})$ hold?
\end{question}

The techniques developed in this memoir are particularly relevant to the following problem, a stepping stone towards a full understanding of regularity properties:
\begin{question}
Assume $\ZFC$ together with a Mahlo cardinal (or less) is consistent.
Let $P_0, \hdots, P_n$, $P'_0, \hdots, P'_m$ and $P''_0, \hdots, P''_l$ be tree forcings (say, from the above table).
Is the conjunction of the following consistent (with $\ZFC$):
\begin{itemize}
\item For each $i \leq n$, $\mathbf{\Sigma}^1_\omega(P_i)$;
\item For each $i \leq m$, $\neg\Delta^1_3(P'_i)$;
\item For each $i \leq l$, $\neg\Delta^1_2(P''_i)$.
\end{itemize}
\end{question}
Of course this question is only interesting if the combination is not already ruled out by known implications.

\medskip

A more general version of the above is of course obtained by replacing $\Delta^1_2$ and $\Delta^1_3$ by an arbitrary higher level.
\begin{conjecture}\label{p.q.general}
Let $P_0, \hdots, P_n$ be tree forcings (say, from the above table) and let $m_0, \hdots, m_n \in \omega$.
Assuming just the consistency of  $\ZFC$ with a Mahlo cardinal (or less), one can show the following to be consistent (with $\ZFC$):
\begin{itemize}
\item For each $i \leq n$, $\mathbf{\Sigma}^1_{m_i}(P_i)$ and $\neg\Delta^1_{m_i+1}(P_i)$,
\end{itemize}
unless the situation described is explicitly ruled out by one of the implications already known to hold in $\ZFC$.\footnote{We hope we have listed all the known implications in this section.}
\end{conjecture}
As mentioned before partial results in this direction have been obtained in \cite{cichon} . 

\medskip

The ultimate regularity property is determinacy,
and therefore one should  ask whether we can force to obtain a model analogous to the one constructed in this memoir while preserving 
a fragment of projective determinacy.
For instance consider the following question (this is known only when all $m_i=2$ for all $i\leq n$):
\begin{question}\label{p.q.determinacy}
Let $P_0, \hdots, P_n$ be tree forcings (say, from the above table) and let $2 \leq m_0, \hdots, m_n \in \omega+1$.
Assuming the consistency of  $\ZFC$ plus the universe is closed under sharps, can one show the following to be consistent (with $\ZFC$):
\begin{itemize}
\item For each $i \leq n$, $\mathbf{\Sigma}^1_{m_i}(P_i)$ and $\neg\Delta^1_{m_i+1}(P_i)$,
\item $\mathbf{\Pi}^1_1$ determinacy;
\end{itemize}
unless the situation described is explicitly ruled out by one of the implications listed above.
\end{question}

This brings us naturally to a related set of open questions:
We can ask whether techniques such as the ones discussed in this memoir allow us to tinker with $P$-regularity 
in the presence of large cardinals, or more accurately, how much of the large cardinal structure of the ground model can be preserved when building a model such as the one demanded in Question~\ref{p.q.general}.
This is especially interesting since, as we have mentioned in the case of Lebesgue measure and the Baire property, Woodin cardinals present a fundamental obstacle to changing the projective theory by forcing.

\begin{question}\label{p.q.woodins}\index{cardinal!Woodin}
Let $P_0, \hdots, P_n$ be tree forcings (say, from the above table), let $m_0, \hdots, m_n \in \omega+1$
and let $0<m\leq\min\{m_i\setdef i\leq n\}-2$.
Assuming the consistency of  $\ZFC$ with the existence of $m$ Woodin cardinals, can one show the following to be consistent (with $\ZFC$):
\begin{itemize}
\item For each $i \leq n$, $\mathbf{\Sigma}^1_{m_i}(P_i)$ and $\neg\Delta^1_{m_i+1}(P_i)$,
\item There are $m$ Woodin cardinals;
\end{itemize}
(of course, unless the situation described is explicitly ruled out by one of the implications listed above).
\end{question}

As the authors have shown in joint work with Schindler \cite{friedman:ea:coding}, the use of David's trick seems to restrict our possibilities severely when it comes to preservation of large cardinals on the order of \emph{strong cardinals}\index{cardinal!strong cardinal}.
Therefore, an answer to Question~\ref{p.q.woodins} currently seems far out of reach.

\chapter{Notation and Preliminaries}\label{sec:notation}

\section{General Set Theory}\label{ss.reals}

In this section, we fix some notation and collect a few key definitions.
We recommend \mycites{jech-set,kanamori,kechris,moschovakis,settheoryoftherealline} for any terms or notation used here without definition.

\medskip

If $f$ is a function, we use both $f[A]$ and $f''\!A$ for $\{ y \setdef \exists x \in A\; y=f(x)\}$. 
We use $B^A$ and in unclear cases also ${}^A B$ for  for the set of functions $f\colon A \to B$,
and ${}^{<\alpha} A$ denotes $\bigcup_{\xi < \alpha} {}^\xi A$, where $\alpha$ is an ordinal.
We use the notation $\langle s_\xi \setdef \xi \in \theta\rangle$ but also $(s_\xi)_{\xi \in \theta}$ for sequences.
That $s'$ is a proper initial segment of the sequence $s$ is written $s' \is s$\index[notation]{lessthan triangle@$\is$|textbf}\index[notation]{  lessthan triangle@$\is$|textbf}, the length of $s$ is denoted by $\lh(s)$,
and $s \conc t$\index[notation]{s^t@$s \conc t$} denotes the concatenation of $s$ with a sequence $t$. 

The cardinality of $A$ is denoted by $\card{A}$, $[A]^\alpha$ denotes $\{X\in\pset(A) \setdef \card{X}=\alpha\}$, 
and both $[A]^{<\alpha}$ and $\pset_{<\alpha}(A)$ mean $\{X\in\pset(A) \setdef \card{X}<\alpha\}$. We write $\On$, $\Card$, $\Sing$, $\Succ$, $\Reg$, and $\Inacc$ for the classes of all ordinals, all cardinals, and cardinals which are singular, infinite cardinal successors, infinite and regular, and inaccessible, respectively.
We shall sometimes write $\Card'$\index[notation]{Card'@$\Card'$} for $(\Card\setminus\omega)\cup\{0\}$; in such a context, we abuse notation by writing $0^+=\omega$\index[notation]{0 plus@$0^+$}.

By $\langle .\,, .\rangle$\index[notation]{ ., . @$\langle .\,, .\rangle$ (G\"odel pairing)} 
we denote the G\"odel pairing function.
By $\ZF^-$ we mean $\ZF$ without the Power Set Axiom, and with a suitable form of the Collection Axiom Scheme (compare \cite{GitmanHamkinsJohnstone2016:WhatIsTheTheoryZFC-Powerset?}).
The powerset of $A$ is $\pset(A)$. 
For a $\Sigma_1$-elementary embedding $j$ which is not the identity, $\crit(j)$ denotes the least ordinal moved by $j$.
 
\medskip

We make use of the following convention, just as is tacitly done in many set theory texts.
\begin{con}\label{d.localobject}
Suppose $X$ is definable and fix a formula $\phi$ so that 
\[
\forall x\; [ \phi(x) \iff x=X ].
\]
Then for an arbitrary (set or class) $\in$-model $M$
we say ``$M\vDash X^M$ exists'' to mean 
\begin{equation*}
M\vDash\text{``there exists precisely one $x$ such that $\phi(x)$''.}
\end{equation*}
Under this condition, we write 
$X^M$ for the unique $x \in M$ such that $M\vDash \phi(x)$.
We will not always be explicit about what the formula $\phi$ is, as long as it is reasonably clear that such a formula can be found,
and may omit the superscript $M$ when possible.
\end{con}

\subsection{What We Talk about When We Talk about the Reals}
So far our exposition has been in terms of subsets of $\reals$ and Lebesgue measure (on $\reals$).
Of course  the (lightface) projective hierarchy can be defined on any (effective) Polish space;
and indeed our main theorem could be phrased as a theorem about effective Polish spaces (or via relativization, on arbitrary Polish spaces; \cite{moschovakis}) carrying an appropriate measure (e.g., the completion of any $\sigma$-finite continuous Borel measure).

For simplicity we shall from now on work in Baire space, i.e., $\omega^\omega$;
when we say  `real' we mean an element of this space (although we shall also sometimes sloppily refer to elements of $2^\omega$ or $\pset(\omega)$ as reals).\index{real number}
We also fix a specific measure $\mu$ on this space:
give $\{n\}$ weight $2^{-n-1}$ and let $\mu$ be the resulting product measure. 
 For the remainder of this memoir we shall call a set (Lebesgue) \emph{measurable}\index{Lebesgue measurability|textbf}
if and only if it is in the measure algebra w.r.t.\ $\mu$ (i.e.\ equal to a Borel set modulo a $\mu$-null set).
The measure algebra of $\mu$ is isomorphic as a complete Boolean algebra to the random algebra\index{random forcing}
(see \myplacecite{17}{kechris}, in particular 17.41).

\medskip

\subsection{Constructible Higher Suslin Trees}\label{p.s.trees}\index{Suslin trees}

For a tree $T$ and $t\in T$ write $\hgt_{T}(t)$ for the height of $t$ in $T$ and $T_\xi$ for
$\{ t \in T \setdef \hgt(t) = \xi \}$, the $\xi$-th level of $T$. 
A \emph{branch through $T$} for us means a \emph{cofinal} branch, i.e., a set $B \subseteq T$ which is linearly ordered by $\leq_T$ and meets every level of $T$. 

\medskip

Our coding methods shall make use of what we call \emph{independent sequences} of higher Suslin trees:
\begin{dfn}\label{p.n.d.independent}\index{independent!Suslin trees|textbf}\index{Suslin trees!independent|textbf}
Suppose $\alpha$ is a cardinal and $\bar T= \langle T(\xi) \setdef \xi < \theta\rangle$ is a sequence of $\alpha$-Suslin trees.
We say $\bar T$ is \emph{independent} if and only if 
\begin{equation}\label{p.e.T}
\bigcup_{\nu < \alpha} \prod_{\xi < \theta} T(\xi)_\nu
\end{equation}
equipped with the product ordering
\[
\bar t \leq \bar t' \iff \forall \xi<\theta \; \bar t(\xi) \leq_{T(\xi)} \bar t'(\xi)
\]
is an $\alpha$-Suslin tree.
\end{dfn}
Given a partial order $\langle P, \leq_P\rangle$, denote by $\langle P, \leq_P\rangle^{\mathrm{op}}$, or sloppily, by $P^{\mathrm{op}}$ 
the order $\langle P, \geq_P\rangle$, called the \emph{reverse order of} $P$.
If $\bar T$ is independent and $\zeta < \theta$, forcing with 
$(\bigcup_{\nu < \alpha} \prod_{\xi \neq \zeta} T(\xi)_\nu)^{\mathrm{op}}$  does not add a branch through $T(\zeta)$.
Also note if each $T(\xi)$ is well-pruned (see \myplacecite{p.~71}{kunen}) and $\theta < \cof(\alpha)$,  the reverse order of \eqref{p.e.T}
is a dense subset of  $T=\prod_{\xi < \theta} T(\xi)^{\mathrm{op}}$, and $\bar T$ is independent if and only if $T$ has the $\alpha$-chain condition.
In the context of trees and forcing, we shall from now on trust the reader to replace partial orders by their reverse as necessary. 

\medskip

We shall make extensive use of a certain independent sequence of higher Suslin trees $\bar T$, built using a variation of a well-known method. 
We shall not describe this method but instead state, abstractly, the properties we require from $\bar T$.

\begin{dfn}\label{p.d.locally.definable}\index{locally semidecidable}\index{locally semidecidable}
Suppose $\Delta$ is a set of sentences and $L\vDash \Delta$.
We say a set $X \in L$ is \emph{locally semidecidable (in $L$ relative to $\Delta$\index{locally semidecidable})} if 
for some $\mathbf{\Sigma}_1$ formula $\phi(x,y)$ 
\[
\forall x\; \big(x \in X \iff \phi(x, \card{ \tcl(x) }^L)\big)
\] 
and if $x \in X$ and $L_\xi\vDash\Delta$, letting $\beta =\card{ \tcl(x) }^L$,
\[
\big( \{ x, \beta\}\subseteq L_\xi \wedge  \card{\tcl{(x)}}^{L_\xi} = \beta  \big)\Rightarrow L_\xi\vDash\phi(x,\beta).
\]
In this situation, we also say $\phi$ \emph{witnesses} that $X$ is locally semidecidable.
We may occasionally neglect to mention $\Delta$, in which case it is the readers task to find an appropriate set of sentences.
\end{dfn}
The usual construction of an independent sequence of higher Suslin trees in $L$ will produce a  locally semidecidable such sequence (For a somewhat sketchy proof of a special case, see \myplacecite{Section~1.1}{david:singleton}):
\begin{lem}\label{p.l.canonical}\index[notation]{alphaplusplus-Suslin trees@$\alpha^{++}$-Suslin trees}\index[notation]{T bar alpha, T alpha (xi)@$\bar T^\alpha$, $T^\alpha(\xi)$}
Suppose $V=L$. Then for every regular cardinal $\alpha$ there is an independent sequence of $\bar T^\alpha = \langle T^\alpha(\xi) \setdef \xi < \alpha \rangle$ of $\alpha^{++}$-Suslin trees such that the following holds:
For each $\xi < \alpha$, $T^\alpha(\xi)$ is well-pruned and ascending sequences of length $\alpha$ have a supremum in $T^\alpha(\xi)$, i.e.,  $T^\alpha(\xi)$ is $\alpha^{+}$-closed. Further, for each $\nu < \alpha^{++}$,
$T^\alpha(\xi)_\nu \subseteq {}^{\nu}2$, and ${}^{<\alpha^+}2 \subseteq T^\alpha(\xi)$. 
Moreover, there is a set of sentences $\Delta$ and a formula $\phi_{\bar T}(\alpha,x, y)$ such that for each $\alpha\in\Reg$ 
the set 
\begin{equation}\label{p.e.canonicaltrees}
\{ (\xi, t) \setdef t \in T^\alpha(\xi) \}
\end{equation}
is locally semidecidable relative to $\Delta$ as witnessed by $\phi_{\bar T}(x, y)$ (supressing $\alpha$, which is regarded as a parameter).
\end{lem}

For easy reference, we shall refer to the sequence from the above lemma as the ``canonical'' one, at $\alpha$:
\begin{dfn}\label{p.n.d.canonical}\index{canonical!sequence of trees}\index{Suslin trees!canonical sequence of|textbf}
We shall refer to the sequence defined by $\phi_{\bar T}(\alpha, y)$ as the \emph{canonical (independent) sequence of $\alpha^{++}$-Suslin trees} (for $\alpha \in \Reg$).
\end{dfn}

\medskip

Except for the sake of illustration in Section~\ref{p.s.exm}, we shall only be interested in the case where $\alpha$ is the unique Mahlo in $L$ (or some initial segment of $L$).
Then, letting $\Delta$ be $\ZF^-$ together with ``there is precisely one Mahlo cardinal, and its double successor exists'' the following holds:
\begin{lem}
For any $\xi \in \On$ such that $L_\xi\vDash\Delta$, $L_\xi \vDash$``if $\bar \kappa$ is the least Mahlo, $ (\bar T^{\bar\kappa^{++}})^{L_\xi}$ exists and is locally semidecidable as witnessed by $\phi_{\bar T}(\bar\kappa^{++},x,y)$''. 
\end{lem}

\section{Forcing}

We take forcing notions to be preorders (also called quasi-orders, i.e., transitive and reflexive binary relations) and
write $q\leq p$ when we mean that $q$ carries \emph{more} information than $p$. 
In some cases the ordering will also be antisymmetric, i.e., a partial order.

By $\alpha$-closed, where $\alpha$ is a cardinal, we mean closed under descending sequences of length less than $\alpha$; similarly for $\alpha$-distributive, and $\alpha$-cc (hence $\omega_1$-closed is the same as $\sigma$-closed, and $\omega_1$-cc is the same as $\ccc$).
In contrast, that a preorder $P$ is $\alpha$-linked means that $P$ can be partitioned into $\alpha$-many sets of pairwise compatible conditions.

\medskip

Next, we define \emph{strong projection} (from \cite{handbook_proper}), \emph{strong suborder} and \emph{independence}.
These are practical in understanding how amalgamation is a stratified extension. Then we fix our terminology for iterations and some ideals.

\subsection{Strong Suborders}

In the following, assume $P$ and $Q$ are separative partial orders.
Versions of the following are of course true for preorders; we leave it to the reader to make the routine adjustments.

\begin{dfn}
We say \emph{$Q$ is a strong suborder of $P$}\index{strong suborder} if and only if $Q$ is a complete suborder of $P$ (which entails that $\ro(Q)$ is a complete subalgebra of $\ro(P)$) and for every $p \in P$ and $q\in Q$ such that $q \leq \pi(p)$, we have $q\cdot p \in P$. 
\end{dfn}

Here, $q \cdot p$ denotes the meet of $p$ and $q$ in $\ro(P)$ and $\pi$ denotes the canonical projection\index{canonical! projection (forcing)}\index{projection (forcing)!canonical} from $\ro(P)$ to $\ro(Q)$, given for each $a \in \ro(P)$ by
\[
\pi(a)=\prod^{\ro(Q)}\{ b\in\ro(Q)\setdef a \leq_{\ro(P)} b \}.
\]
This is related to the notion of projection\index{projection (forcing)} and strong projection:
We say $\pi \colon P \rightarrow Q$ is a \emph{projection} if and only if for all $p,p' \in P$
\begin{enumerate}[1. ]
\item $p \leq p' \Rightarrow \pi(p)\leq \pi(p')$,
\item $\ran(\pi)=Q$,
\item \label{projection} if $q \in Q$ and $q \leq \pi(p)$, there is $\bar p \in P$ such that $\bar p \leq p$ and $\pi(\bar p) \leq q$.\footnote{\cite{handbook_proper} defines (ordinary) projection with 
 \ref{projection} replaced by the stronger: if $q \leq \pi(p)$, there is $\bar p\leq p$ such that $\pi(\bar p)=q$.}
\end{enumerate}
Observe this implies that $\pi(1_P)=1_Q$.
In \cite{handbook_proper}, $\pi$ is defined to be a \emph{strong projection}\index{strong projection}\index{projection (forcing)!strong} if and only if it satisfies the first two requirements above and the following strengthening of the third requirement:
\begin{enumerate}[1. ]
\item[$3'$.] If  $q \leq \pi(p)$, there is $\bar p\leq p$ such that 
\begin{enumerate}[a. ]
\item $\pi(\bar p)= q$,
\item for any $r \in P$, if $r \leq p$ and $\pi(r) \leq q$ then $r \leq \bar p$. \label{q.p} 
\end{enumerate}
\end{enumerate}
This uniquely determines $\bar p$, which will shall denote by $q \cdot p$---with good reason as will be clear in a moment.

\medskip

If $\pi\colon P \rightarrow Q$ is a projection, $\pi[G]$ generates a $Q$-generic Filter whenever $G$ is a $P$-generic Filter, moreover $Q$ (and hence $\ro(Q)$) completely embeds into $\ro(P)$.
If $\pi\colon P \rightarrow Q$ is a strong projection, the map $\iota$ sending $q \in Q$ to $\iota(q)=q \cdot 1_P$ is a complete embedding and we can assume that $Q$ is a subset of $P$. It follows from Item~$3'$b that $\forall p \in P \; p \leq \iota(\pi(p))$.
Under the assumption that $Q$ is a subset of $P$, the condition $\bar p$ from Item~$3'$b really coincides with $q\cdot p$, i.e., the meet in $\ro(P)$.

\medskip

For the next lemma, recall the following definition:
\begin{dfn}\index{reduction (forcing)}
When $Q$ is a complete suborder of $P$, we say $q \in Q$ is a \emph{reduction (to $Q$) of} $p \in P$ if and only if for all $q' \in Q$, if $q' \leq q$ then $q'$ and $p$ are compatible.
\end{dfn}
It is easy to see that when $Q$ is a complete suborder of $P$ and $p \in P$, $\pi(p)$ is the supremum (the maximum) in $\ro(Q)$ of all reductions of $p$.

The notions of reduction, strong projection and canonical projection are closely related as follows:
\begin{lem}\label{strong:vs:canonical:proj}
Let $Q$ be a complete suborder of $P$ and let $\pi$ be the canonical projection $\pi\colon \ro(P)\rightarrow \ro(Q)$. 
Say $p \in P$ and $q \in Q$ is a reduction of $p$ such that $q \geq p$; then $q = \pi(p)$.
If $\bar \pi\colon P \rightarrow Q$ is a strong projection, then $\bar \pi$ coincides with the canonical projection on $P$.
\end{lem}
Observe that if $\pi\colon \ro(P) \rightarrow \ro(Q)$ is the canonical projection, then $\pi \res P$ is a strong projection if and only if
for every $p \in P$ and $q \in Q$ such that $q \leq \pi(p)$ we have $p\cdot q \in P$. All of the above gives us:
\begin{lem}\label{stron:proj:equiv}
The following are equivalent for partial orders $Q$ and $P$ (if $Q \subseteq P$, the qualifications in parentheses can be omitted):
\begin{itemize}
\item $Q$ is (isomorphic to) a strong suborder of $P$.
\item There is a strong projection $\pi\colon P \rightarrow Q$.
\item $Q$ is (isomorphic to) a complete suborder of $P$ such that the restriction to $P$ of the canonical projection $\pi\colon\ro(P)\rightarrow \ro(Q)$ is the unique strong projection into (this copy of) $Q$ (and the isomorphism is $\iota$ from above).
\end{itemize}
\end{lem}
Also observe that when $Q$ is a strong suborder of $P$ which is a strong suborder of $R$ with $\pi_P\colon R\rightarrow P$ a strong projection, then $1_Q$ forces $\pi_P\res\quot{R}{Q}$ is a strong projection from $\quot{R}{Q}$ to $\quot{P}{Q}$.

\subsection{Independent Suborders}

Imagine an iteration $R=(Q_0 \times Q_1 )* \dot Q_2$. Then in an extension by $Q_0$, the preorder $Q_1$ is a complete suborder of the tail $\quot{R}{Q_0}=Q_1*\dot Q_2$.
This special situation is captured well by the following:
\begin{dfn}\label{indie}\index{independent!suborders (forcing)}
Let $Q$ and $C$ be suborders of $P$ with strong projections 
$\pi_Q\colon P \rightarrow Q$ and 
$\pi_C\colon P \rightarrow C$.
We say \emph{ $C$ is independent over $Q$ in $P$} if and only if for all  $c \in C$ and $p \in P$ such that $c \leq \pi_C(p)$, we have
$\pi_Q(p \cdot c) = \pi_Q(p)$.

For a $P$-name $\dot C$, we say  \emph{$\dot C$ is independent in $P$ over $Q$}\index{independent!name (forcing)} if and only if $\dot C$ is a name for a generic of an independent complete suborder of $P$; i.e. there is a complete suborder $R_C$ of $P$ (with a strong projection $\pi_C\colon P \rightarrow R_C$) such that $R_C$ is dense in $\gen{\dot C}^{\ro(P)}$ and $R_C$ is independent in $P$ over $Q$.
By $\gen{\dot C}^{\ro(P)}$\index[notation]{CroP@$\gen{\dot C}^{\ro(P)}$}\index[notation]{ CroP@$\gen{\dot C}^{\ro(P)}$} we mean the smallest Boolean subalgebra of $\ro(P)$ which contains
all the Boolean values occurring in the name $\dot C$. Thus $\gen{\dot C}^{\ro(P)}$ is a ground model object.
\end{dfn}
We use the following in \ref{index:sequ} (p.~\pageref{index:sequ}), via the notion of ``remoteness'' (see also Lemma \ref{remote:lemma:not:in}
and Section~\ref{sec:remote}).
\begin{lem}\label{indie:lemma:not:in}
If $C$ is independent over $Q$ in $P$, $1_Q$ forces that $C$ is a complete suborder of $\quot{P}{Q}$ and $\pi_C$ restricts to a strong projection.
If $\dot C$ is a $P$-name which is independent over $Q$, then $\dot C$ is not in $V^Q$.
\end{lem}

\medskip

\subsection{Iterations}

An \emph{iteration} is a sequence $\bar Q^\theta=(P_\iota,\dot Q_\iota)_{\iota<\theta}$\index[notation]{Qthetha@$\bar Q^\theta$}\index{iteration (forcing)} such that 
for each $\iota <\theta$, 
\begin{enumerate}
\item $P_\iota$ is a preorder and $\dot Q_\iota$ is a $P_\iota$-name such that $\forces_{P_\iota} \dot Q_\iota$ is a preorder
\item \label{def:it:thread} $P_\iota$ consists of sequences $p$ with $\dom(p) = \iota$ and for each $\nu <\iota$, $p(\nu)$ is a $P_\nu$-name and  
\begin{equation}\label{thread}
\forall \nu\in\dom(p)\; 1_{P_\nu}\forces p(\nu) \in \dot Q_\nu.
\end{equation}
\item \label{def:it:order} The ordering of $P_\iota$ is given by:
\begin{equation}\label{threads:order}
r \leq p \iff \forall \nu<\iota \quad r \res \nu \forces_{P_\nu} r(\nu)\leq_{\dot Q_\nu} p(\nu).
\end{equation}
\end{enumerate}
\noindent
We state this as some would not agree with \eqref{thread}.
\begin{dfn}\label{it:def2}
More generally we will also call sequence 
$(P_\iota)_{\iota < \theta}$ of preorders an iteration when it comes with strong projections $\pi^{\bar \iota}_\iota\colon P_{\bar \iota} \rightarrow P_\iota$ for each $\iota <\bar \iota < \theta$, and for limit $\iota$, $P_\iota$ consists of threads in the second sense, and is ordered by \eqref{threads:order}.
\end{dfn}
Any iteration in the sense of the second definition can be transcribed into one in the first sense by 
letting $\dot Q_\iota$ be a $P_\iota$-name such that $\forces_{P_\iota} \dot Q_\iota = P_{\iota+1} : P_\iota$.

If an iteration is presented as in Definition~\ref{it:def2} we shall always regard $P_\nu$ as being a subset of $P_\iota$ for $\nu < \iota < \theta$ via the identification discussed in the section about strong projections.

\medskip

Regardless of what support is used, if $\iota < \theta$ is a limit, $P_\iota$ will always be a subset of the set of \emph{threads} through the earlier stages:

\pagebreak[3]

\begin{dfn}~\label{it:terminology} 
\begin{enumerate}
\item Suppose $\bar Q^{\theta}$ is an iteration. We call a sequence $p$ with $\dom(p) = \theta$ a \emph{thread\index{thread (forcing)|textbf} through (or in) $\bar Q^\theta$} if and only if it satisfies \eqref{thread}. The set of threads through $\bar Q^\theta$ we shall sometimes denote by $\prod \bar Q^\theta$\index[notation]{PiQtheta@$\prod\bar Q^\theta$ (preorder of threads)}.\label{threads:def}
It is endowed with the ordering given by \eqref{threads:order} (for $r$, $p \in \prod\bar Q^\theta$).

\item \label{threads:var:def} We also use the term \emph{thread} in a second, related sense: if $\bar p\in \prod_{\eta<\iota <\theta} P_\iota$ for some $\eta<\theta$---i.e $\bar p=(p_\iota)_{\iota\in(\eta,\theta)}$ and for each $\iota\in \dom(\bar p)$ we have $p_\iota \in P_\iota$---we say $\bar p$ forms or defines or simply \emph{is a thread (through $\bar Q^\theta$)} if and only if
\begin{equation*}\label{threads:var}
\forall \iota, \bar\iota \in \dom(\bar p) \quad \iota\leq\bar\iota \Rightarrow \pi_\iota(p_{\bar\iota})=p_\iota.
\end{equation*}
\end{enumerate}
\end{dfn}
\noindent
The point is that a thread in the first sense yields one in the second sense and vice versa. 
Finally, threads naturally have a length:

\begin{dfn}
Given an iteration
$(P_\iota)_{\iota < \theta}$, and $p\in P_{\theta}$ we call $\lh(p)$\index[notation]{lh(p)@$\lh(p)$ (forcing)} the least $\iota < \theta$ such that $p \in P_\iota$, i.e., such that $\pi_\iota(p)=p$.
\end{dfn}

\subsection{Ideals}
If $c$ is a Borel code\index{Borel code}, we write $B_c$ for the Borel set coded by $c$. Of course given two models of set theory, both containing a Borel code $c$, it may be that $c$ codes a different set in each model. 
\begin{dfn}\label{generic:reals}
Say $\ro(Q)$ is a complete subalgebra of $\ro(P)$.
Let $\dot G$ be the $P$-name for the $P$-generic over $V$ and $\pi\colon \ro(P) \rightarrow \ro(Q)$ the canonical projection.
Let $\dot I$ be a $P$-name for an ideal\index{ideal} on $\omega^\omega$ in the extension via $P$.
For a $P$-name $\dot r$ and $p \in P$, we say is \emph{$p$ forces $\dot r$ is $\dot I$-generic over $V^Q$},
just if $p\forces_P \dot r \in \omega^\omega$ and for every $Q$-name for a Borel code $\dot c$, 
\[ p \forces_P B_{\dot c} \in\dot I \Rightarrow \dot r \not \in B_{\dot c}.\]

We say $p$ forces $\dot r$ is \emph{fully} $\dot I$-generic over $V^Q$  if and only if $p$ forces $\dot r$ is $\dot I$-generic over $V^Q$ and in addition, for every $Q$-name $\dot c$ such that $\pi(p)\forces_Q \dot c$ is a Borel code,
\[p \forces_P \dot r \not\in  B_{\dot c} \Rightarrow   p \forces_P B_{\dot c} \in\dot I.\]
In other words, $p$ does not force anything non-trivial about $\dot r$.
We say \emph{$\dot r$ is (fully) $\dot I$-generic} just if $1_P$ forces $\dot r$ is (fully) $\dot I$-generic.
Instead of ``$\dot I$-generic'',
\begin{itemize}
\item If $\dot I$ is a name for the ideal of sets with measure zero,  we say \emph{random over $V^Q$}.\index{random real!over V Q@over $V^Q$}
\item If $\dot I$ is a name for the ideal of meager sets, we say \emph{Cohen over $V^Q$}.\index{Cohen real!over V Q@over $V^Q$}
\item If $\dot I$ is a name for $\pset_{<\omega}(\omega^\omega)^{V[\pi(\dot G)]}$---or equivalently, for $\pset((\omega^\omega)^{V[\pi(\dot G)]})$---we say $\dot r \not\in V^Q$ or \emph{$\dot r$ is not in $V^Q$}.
\item If $\dot I$ is a name for the ideal of sets which are bounded by a real in $V[\pi(\dot G)]$ in the sense of eventual domination, we say \emph{unbounded over $V^Q$}.\index{unbounded (over V Q)@unbounded (over $V^Q$)|textbf}
\end{itemize}
The terms \emph{$p$ forces $\dot r$ is fully random over $V^Q$} and \emph{fully Cohen over $V^Q$} are to be understood analogously.\index{fully!random (over VQ)@random (over $V^Q$)}\index{fully!Cohen (over VQ)@Cohen (over $V^Q$)}
\end{dfn}

\begin{lem}\label{unbounded}
Let $P$ and $Q$ be arbitrary partial orders and let $\dot r$ be a $P$-name for a real.
If 
$\dot r$ is unbounded over $V$, viewing $\dot r$ as a $P\times Q$ name via the natural embedding, 
$\dot r$ is unbounded over $V^Q$.
\end{lem}
\noindent For a proof, see \myplacecite{Lemma 3.3, p.~392}{jr:amalgamation}.

\chapter{Overview of the Proof}\label{s.overview.proof}

\section{The Basic Idea}\label{p.s.gamma-0}

\subsection{Solovay's Way}

Since Solovay's work \cite{solovay} the following lemma has been known as a tool to construct models where all projective sets (in fact, all sets definable from a sequence of ordinals of length $\omega$) are Lebesgue measurable:

\begin{lem}[Solovay's Lemma]\label{l.solovay}\index{Solovay's Lemma}
Suppose $P$ is a preorder such that
\begin{enumerate}
\item $P$ forces that
$%
\bigcup \{ A \subseteq \reals \setdef \text{$A$ is null and has a Borel code in $V$}\}
$ %
is a null set;
\item For any $b \in \ro(P)$ and any pair of complete embeddings of complete Boolean algebras 
\[
e_0, e_1\colon B \to \ro(P)
\]
where $B$ is the random algebra, there is an automorphism 
\[
\Phi\colon \ro( P ) \to \ro( P)
\]
such that $\Phi \circ e_0 = e_1$ and $\Phi(b) \neq b$.
\end{enumerate}
Then $P$ forces that every set of reals which is definable by a formula with parameters in the ground model $V$ is Lebesgue measurable.
\end{lem}
We give a proof on p.~\pageref{s.every-projective-set}. 
A similar lemma holds for the Baire property and the other regularity properties discussed in the previous section (see \cite{jr:amalgamation}).

\medskip

Solovay's Lemma~\ref{l.solovay} and its relatives are often applied as follows to build models where all sets definable from sequences of ordinals posses a given regularity property. Again, for our purpose, it suffices to consider Lebesgue measurability.
\begin{cor}\label{p.c.premain}
Suppose $P = P_\kappa$ is a limit of an iteration $(P_\iota,\dot Q_\iota)_{\iota<\kappa}$ of preorders 
and let $G$ be $(V,P)$-generic.
Suppose further that the following conditions are satisfied, where for $\xi \leq \kappa$, $G_\xi$ denotes the generic filter on $P_\xi$ induced by $G$ and $\dot G_\xi$ denotes its name:
\begin{enumerate}
\item\label{p.c.premain.amoeba} For any $\xi < \kappa$, $P$ forces that the following set has measure zero:
\[
\bigcup \{ A \subseteq \reals \setdef \text{$A$ is null and has a Borel code in $V[\dot G_\xi]$}\}
\]
\item\label{p.c.main.i.sequence} Any sequence of ordinals in $V[G]$ already appears in $V[G_\xi]$ for some $\xi < \kappa$.
\item\label{p.c.premain.hom}\label{p.c.main.hom} For any $\xi < \kappa$ the following holds in $V[G_\xi]$: For any $b \in \ro(P)/{G_\xi}$ and any pair of complete embeddings of complete Boolean algebras 
\[
e_0, e_1\colon B \to \ro(P)/{G_\xi}
\]
where $B$ is the random algebra in $V[G_\xi]$, there is an automorphism 
\[
\Phi\colon \ro(P)/{G_\xi} \to \ro(P)/{G_\xi}
\]
such that $\Phi \circ e_0 = e_1$ and $\Phi(b) \neq b$.
\end{enumerate}
Then every projective set (in fact, every set definable with a parameter in $\On^\omega$) is Lebesgue measurable in $V[G]$.
\end{cor}
For convenience we introduce the following terminology:
\begin{dfn}\label{p.d.hom}
We say an iteration  
$(P_\iota)_{\iota\leq\kappa}$ is \emph{sufficiently homogeneous (for random subalgebras)}\index{sufficiently homogeneous iteration} if and only if Item~\ref{p.c.premain.hom} in Corollary~\ref{p.c.premain} above holds.
\end{dfn}
We will use \emph{amalgamation} to build an iteration which is sufficiently homogeneous (see Section~\ref{p.s.amalgamation} for an overview; cf.\ Chapter~\ref{sec:amalgamation}).

\medskip

By Theorem~\ref{p.t.lm}, with $P$ and $G$ as in Corollary~\ref{p.c.premain} it must be the case ${\omega_1}^{V[G]}$  is inaccessible in $L[r]$, for each $r \in \pset(\omega)^{V[G]}$. We are lead to expect that $\kappa$ should be inaccessible in $V$ and cofinally often in our iteration, a cardinal below $\kappa$  is collapsed (in fact, we shall assume $\kappa$ to be Mahlo---see Section~\ref{p.s.mahlo}).
Moreover, we expect that for each $\xi < \kappa$, $(2^\omega)^{V[G_\xi]} < \kappa$ (and as we shall see below, our use of Jensen coding requires this as well).
With all this in mind, we replace Item~\ref{p.c.premain.amoeba} of Corollary~\ref{p.c.premain} by the following:

\begin{cor}\label{p.c.main}
Corollary~\ref{p.c.premain} still holds if we replace the first item by:
\begin{enumerate}
\item\label{p.c.main.i.collapse} For any $\xi < \kappa$, $(2^\omega)^{V[G_\xi]} < \kappa$ and there is $\xi'$ such that $\xi < \xi' < \kappa$ and $P_{\xi'}$ collapses $(2^\omega)^{V[G_\xi]}$ to $\omega$.
\end{enumerate} 
\end{cor} 
For convenience, we can assume that for unboundedly many $\xi < \kappa$, $P_{\xi+1} = P_\xi * \Coll(\omega,{\omega_1}^{V[G_\xi]})$. 
We stress that Item~\ref{p.c.main.i.sequence} requires $P$ to not add any sequences of length $\omega$ in the last stage. 
We also expect that 
$\kappa$ remains a cardinal (in fact, Mahlo) in $V[G_\xi]$ for each $\xi < \kappa$ so $\kappa = {\omega_1}^{V[G]}$.

\subsection{The Irregular Set}
It is not difficult to arrange that $\langle P_\xi \setdef \xi \leq \kappa \rangle$ as above forces that there is a set without the Baire property:
For each $\xi < \kappa$ we may find a real $c_\xi$ which is Cohen over $V[G_\xi]$;
in fact it is easy to arrange (see Section~\ref{def:it:succ:stage}, Case $k=0$) that for any $I\subseteq \kappa$ which is unbounded in $\kappa$, the set $\{ c_\xi \setdef \xi \in I \}$ is dense in ${\omega^\omega}^{V[G]}$.
Then by the following well-known fact (a proof can be found, e.g., in \cite{jr:amalgamation}) $\{c_\xi \setdef \xi < \kappa$ is a limit ordinal$\}$ does not have the Baire property in $V[G]$:
\begin{lem}\label{p.l.gamma}
Suppose the situation from Corollary~\ref{p.c.main} obtains. 
Suppose further that $\Gamma \in V[G] \cap \pset(\omega^\omega)$ satisfies that for every $\xi < \kappa$, both $\Gamma$ and $\omega^\omega \setminus \Gamma$ contain a dense (in $\omega^\omega$) set of reals which are Cohen generic over $V[G_\xi]$. Then $\Gamma$ does not have the Baire property in $V[G]$. 
\end{lem}

A necessary condition for $\Gamma$ to be definable is that $\Gamma$ be closed under \emph{all} automorphisms of $P_\kappa$; e.g., those required by Item~\ref{p.c.main.hom} in Corollary \ref{p.c.premain}.
As it turns out, it will be enough to define a specific family of automorphisms $\mathcal F$ of $P_\kappa$---namely, $\mathcal F_\kappa$ in  \ref{m.d.F}---large enough to witness the fact that our iteration is sufficiently homogeneous;
we then let
\begin{equation*}\label{}
\Gamma^0 = \{ \Phi(c_\xi) \setdef \text{$\Phi \in \mathcal F$, $\xi < \kappa$, $\xi$  limit}\}
\end{equation*}
and $\Gamma^1 = \omega^\omega\setminus \Gamma^0$.
By the above lemma, $\Gamma^0$ (and $\Gamma^1$) will not have the Baire property provided we can show that $\Gamma^1$ still contains a dense set of Cohen reals over $V[G_\xi]$ for any $\xi<\kappa$.
For this we will show that no automorphism in $\mathcal F$ sends $c_\xi$ to $c_{\xi'}$ for two distinct $\xi, \xi' < \kappa$ (see Lemma~\ref{index:sequ}).

It follows from the definability of $\Gamma^0$ in $V[G]$ that $\Gamma^0$ and $\Gamma^1$ are indeed closed under \emph{all} automorphisms of $P_\kappa$.

\subsection{Making the Irregular Set Definable}\label{p.s.making.def}

We now want to understand how we will achieve that the set $\Gamma^0$, chosen as described above, is made definable in the course of our iteration.
Assume we have a real $r$ added by an initial segment of our iteration, say, before stage $\xi$.
We force at the next stage to make a formula $\Psi(r)$ true if $r\in \Gamma^0$.
We want $\Psi(x)$ to be upwards absolute
so that $\Psi(r)^{V[G]}$ holds, 
and we want $\Psi(r)$ to remain false for any $r \notin \Gamma^0$.
If we can find a solution to this problem \emph{always using the same formula $\Psi(x)$}, there is hope that in the course of our iteration, we can do the same for every real $r$ added and thus make $\Gamma^0$ projectively definable.

Thus, we arrive at the following basic idea for the proof of Theorem~\ref{p.t.main2}.
\begin{lem}\label{p.l.def}
Suppose $P_\xi$, $G_\xi$ for $\xi\leq \kappa$  and $G$ are as in Corollary~\ref{p.c.main}. For $\Gamma^0$ as above,
suppose further that for some upwards absolute formula $\Psi(x)$ it holds that: %
\begin{enumerate}[start=4]
\item For any real $r\in \Gamma^0$, there is a stage $\xi$ such that 
$\forces_{P_{\xi+1}} \Psi(r)$;\label{p.l.def.item1}
\item $V[G]\vDash \neg\Psi(r)$ for any real $r \in V[G]\setminus\Gamma^0$.\label{p.l.def.item2}
\end{enumerate}
Then  in $V[G]$, the formula $\Psi$ defines $\Gamma^0$, a set without the Baire property.
\end{lem}
We want $\Psi(x)$ to be parameter free and projective, i.e., involve only quantifiers over $\HSize(\omega_1)^{V[G]}$---corresponding to $\HSize(\kappa)$ in $V[G_\xi]$. 
For now, we ignore the requirement that $\Psi(x)$ be a projective formula and discuss the simpler problem of making $\Gamma^0$ definable (by \emph{any} formula $\Psi(x)$). 

\begin{exm}\label{p.exm.trees}\index{independent!Suslin trees|(}\index{Suslin trees!independent|(}
One (well-known) method to make some set reals $\Gamma = \{ r_\xi \setdef \xi < \omega_1 \}$ definable is the following.
Let $\langle T(\xi) \setdef \xi <\omega_1 \rangle$ be a constructible sequence of (well-pruned) $\omega_3$-Suslin trees with the following two properties:
\begin{itemize}
\item $\langle T(\xi) \setdef \xi \in \omega_1 \rangle$ is \emph{independent}, i.e., $\prod_{\xi<\omega_1} T(\xi)$ 
has no uncountable antichains;
\item $\langle T(\xi) \setdef \xi \in \omega_1 \rangle$ is definable in $L$.
\end{itemize}
Recall (from Definition~\ref{p.n.d.independent}) that by \emph{independent} we mean that 
\[
\bar T = \prod_{\xi<\omega_1} T(\xi) 
\]
has no antichains of size $\omega_3$. 

After forcing with 
$\prod_{\xi < \omega_1, r_\xi(n) =1} T(\omega\cdot \xi + n)$ it holds that
\begin{equation}\label{p.e.Suslintreesdef}
r \in \Gamma \iff \exists \xi < \omega_1 \, [ r(n) = 1 \iff T(\omega\cdot \xi + n) \text{ has a branch.}]
\end{equation}
Since $\langle T(\xi) \setdef \xi < \omega_1 \rangle$ is definable, we have made $\Gamma^0$ definable.
The direction $\Leftarrow$ in \eqref{p.e.Suslintreesdef} follows by independence since for any $\xi <\omega_1$, $\prod_{\xi \neq \nu} T(\nu)$ does not add a branch through $T(\xi)$.
\end{exm}

A similar technique will ensure the requirements of Lemma~\ref{p.l.def}. 
We shall use an independent sequence $\langle T(\xi) \setdef \xi <\kappa \rangle$ of $\kappa^{++}$-Suslin trees in $L$, 
as the height should be at least $\kappa = {\omega_1}^{V[G]}$.
For height $\kappa^{++}$, such trees, with additional desirable properties, can be found conveniently using standard methods (see Section~\ref{p.s.trees}).

At stage $\xi$, given a real $r_\xi$ and faced with the task of forcing $\Psi(r_\xi)$, we will pick a set $I^*_\xi \subseteq\kappa$ (the \emph{coding area}) indexing a set of trees ``encoding'' $r_\xi$. 
As our iteration will be sufficiently homogeneous, we cannot proceed  in such close analogy to Example~\ref{p.exm.trees} that we simply let $I_\xi = \{\omega\cdot \xi +n \setdef n \in r_\xi\}$: 
For any automorphism $\Phi$ of $P$, 
$\Phi(I_\xi)= I_\xi$
and we also add a branch for each $n \in \Phi(r_\xi)$---destroying the coding.
This is related to the fact that $\Gamma$ in Example~\ref{p.exm.trees} becomes  definably well-orderable in order type $\omega_1$.
By a slight variant of Theorem~\ref{p.t.lm} \cite{shelah:amalgamation}, this stands in contradiction to our goal that sets of the same definitional complexity as $\Gamma^0$ be Lebesgue measurable.

\medskip

For this reason, 
force with $\kappa$-Cohen forcing to add a generic set $I_\xi \in [\kappa]^\kappa$ and let
\[
I^*_\xi = \{\omega\cdot \sigma +n \setdef \sigma \in I_\xi, n \in r_\xi\}.
\]
This solves the problem described in the previous paragraph:
If $r_\xi \neq \Phi(r_{\xi'})$ where $\xi, \xi' < \kappa$ and $\Phi$ is an automorphism of $P_\kappa$, we shall be able to show that
\begin{equation}\label{p.e.coding-areas}
\card{I^*_\xi \cap \Phi(I^*_{\xi'})} < \kappa
\end{equation}
so that the ``noise'' created by automorphisms, i.e., the degree to which the coding of one real disturbs the coding of another, is just manageable.
We show \eqref{p.e.coding-areas} analogously to how we analyze the action of $\mathcal F$ on the set of Cohen reals $c_\xi$.
In the case of Cohen reals this analysis uses the fact that Cohen reals are \emph{unbounded} over the ground model; for coding areas the analogous concept is the ad-hoc (but straightforward) notion of being \emph{remote} (see Definition~\ref{remote}).

The proof of Item~\ref{p.l.def.item2} of Lemma~\ref{p.l.def} follows the same idea as showing \eqref{p.e.Suslintreesdef} in Example~\ref{p.exm.trees} and uses, in essence, that $\langle T(\xi)\setdef\xi <\kappa \rangle$ is independent, along with the fact that the rest of the forcing is $\kappa^+$-linked and thus adds no branches through $\kappa^{++}$-Suslin trees. Here we see why we need our trees to have height at least $\kappa$.

\medskip

Unfortunately, the above sketch is oversimplified, due to the enormous tension between definability and homogeneity:
Let $B(\nu)$ be the branch added through the tree $T(\nu)$, for each $\nu < \kappa$.
We shall see that to ensure that $\Psi$ is of low complexity, we will make definable, or \emph{code}, each $B(\nu)$.
If $\Phi$ is an automorphism of $P_\kappa$,
then clearly $\Phi(B(\nu))$ must be a branch through $T(\nu)$ as well, and unless we should (definably) collapse $\kappa^{++}$ (and add definable branches through \emph{all} the trees in a single step) we should assume that $\Phi(B(\nu)) = B(\nu)$.

To ensure this, we add branches at the first step of our iteration, i.e., with $P_0$ and define amalgamation so that it only creates automorphisms of $P$ which act as the identity on $P_0$. 
Since we have not added any reals, we add branches through all trees---but we will only \emph{code} $B(\nu)$ if for some $\xi<\kappa$, $\nu \in I^*(r_\xi)$.
The proof of Item~\ref{p.l.def.item2} of Lemma~\ref{p.l.def}, i.e., that no unwanted branches are \emph{coded} (instead of added) is technical, and here the difficulty of reconciling definability and homogeneity comes back to haunt us; 
we sketch this proof at the end of the next section.

\subsection{Making the Definition Projective}\label{sec:overview:coding}

Jensen \cite{bjw:82} developed a forcing that under certain prerequisites, given a class $A$, creates $u\in \pset(\omega)$ so that $A \in L[u]$. We can say ``$A$ is coded by a real'' (compare Theorem~\ref{p.f.jensen.coding} below). 
Jensen coding is simplified by taking (a model close to) $L$ as the ground model; so we from now on assume $V=L$. 
A slight variant of Jensen's forcing, which called Jensen forcing \emph{with localization} or \emph{David's trick} will enable us to find a formula $\Psi(r)$ as in Lemma~\ref{p.l.def} which is \emph{projective}. 
For what follows, the reader is advised to keep in mind Definition~\ref{p.n.d.canonical} and Convention~\ref{d.localobject}.

\medskip

In the previous section, we sketched a method to make $\Gamma^0$ definable by a (non-projective) formula as in \eqref{p.e.Suslintreesdef}.

Now take $\Psi_0(r)$\index[notation]{Psi 0 r@$\Psi_0(r)$} to be the formula ``Letting $\bar T = \langle T(\xi) \setdef \xi<\bar\kappa\rangle$ be the canonical sequence of 
$\bar\kappa^{++}$-Suslin trees, for $\bar\kappa$ the least Mahlo, $\bar\kappa$-many trees with index in $I(r)$ have a branches,'' 
where 
\[
I(r) = \{\omega\cdot \sigma +n \setdef \sigma < \bar \kappa, n \in r\}
\]
In contrast to Example~\ref{p.exm.trees}, in our iteration we add branches through each tree of $\bar T$ at the beginning of our construction---so $\Psi_0(r)$ will always be true for any real $r$. But using Jensen coding, we can arrange that some branches are ``coded'', while others are not:

At stage $\xi$, we are given a real $r_\xi$.
As described in the previous section, we force a  generic coding area $I^*_\xi$ of size $\kappa$, $I^*_\xi \subseteq I(r_\xi)$.
Jensen coding now allows us to generically create a real $u$ such that 
\begin{equation}\label{p.e.def.u.0}
L[u,r_\xi] \vDash\Psi_0(r_\xi)
\end{equation}
with $u$ coding the branches indexed by $I^*_\xi$---and as we shall show, only those.

\medskip

Using a trick invented by Ren\'e David \cite{david:singleton} (see also \myplacecite{Theorem~6.18,~p.~129}{friedman:codingbook} for a very general formulation), we can arrange the following stronger statement to hold for $r = r_\xi$, 
\begin{equation}\label{p.e.def.u}
\exists u \in \pset(\omega) \; \forall \alpha <\kappa \;  \big( L_\alpha[u,r]\vDash\Delta \Rightarrow L_\alpha[u,r]\vDash\Psi_0(r) \big)
\end{equation}
where $\Delta$ is a specific, fixed (and recursive) set of sentences.
(Notice that interpreting $\Psi_0$ in $L_\alpha[u,r]$ yields an assertion about the local object corresponding to $\bar T$, i.e., $\bar T^{L_\alpha}$.) 
Taking $\Psi(r)$ to be \eqref{p.e.def.u}, we have forced $\Psi(r_\xi)$. 

Moreover, $\Psi(r)$ is clearly upwards absolute, as required in Lemma~\ref{p.l.def}.
It remains to show Item~\ref{p.l.def.item2} of Lemma~\ref{p.l.def}, that if $r\not\in \Gamma^0$, then $\Psi(r)$ is false in $V[G]$. 
We give a short sketch (the proof will be carried out in Section~\ref{sec:preserving:coding}):

We show, roughly, that if $r \not \in \Gamma^0$, then for any $u\in \pset(\omega)\cap V[G]$, the set of $\nu \in I(r)$ such that $L[u]$  contains a branch
through $T(\nu)$ has size less than $\kappa$.
(The argument for this involves writing the forcing as a product of $T(\nu)$ and a remainder, with $T(\nu)$ surgically removed; we then use the independence of the trees and the fact that the remainder is $\kappa^+$-linked.) 
By elementarity (since $L[u]\vDash \Delta$), \eqref{p.e.def.u} must fail.\index{independent!Suslin trees|)}\index{Suslin trees!independent|)}

Lastly, $\Psi(r)$ becomes $\Sigma^1_3$ in $L[G]$ (as $\kappa = \omega_1$ there).
We can apply the same procedure to $\Gamma^1=\omega^\omega\setminus \Gamma^0$ in the course of our iteration, so that $\Gamma^0$ is a  $\Delta^1_3$ set without the Baire property in the final model.

\section{Jensen Coding and David's Trick}\label{p.s.jensen.david}

We give a brief introduction to Jensen coding.
This forcing was developed to prove the famous Coding Theorem \myplacecite{Theorem~0.1,~p.~4}{bjw:82}.
It was much simplified over the course of the years (and has spawned an immense literature; see, e.g.,   \mycites{jensen1,coding_reshaping,jensen2,sh:st:1,jensen3,friedman:handbook,friedman:ea:coding}).
The reference we take as a starting point is \cite{friedman:codingbook}.

The setting in which we use Jensen coding simplifies matters further (cf.\ the role of the non-existence of $0^\sharp$, \myplacecite{p.~67}{friedman:codingbook}).  For this brief introduction, consider the following special case of Jensen's Coding Theorem\index{Jensen's Coding Theorem}:
\begin{fct}\label{p.f.jensen.coding}
Suppose $V=L[A]$, $\Card^V = \Card^{L[A\cap \omega]}$ and $\GCH$ holds.
There is a forcing $P_J$ such that the following holds in any $(L[A],P_J)$-generic extension $L[A][G]$:
$A$ is definable in $L[u_G]$, for some $u_G\in \pset(\omega)$.
\end{fct}

In the following section we start by reviewing (a version of) the forcing that serves as the main building block for Jensen coding, namely \emph{almost disjoint coding}.

\subsection{Almost Disjoint Coding}\label{s.almost}\index{almost disjoint coding}
For each $\alpha \in \Reg$ fix an indexed family
$\mathcal B_\alpha = \langle b^\xi \setdef \xi \in [\alpha, \alpha^+)\rangle$ with the property
that for each pair of distinct $\xi,\xi' < \alpha^+$, $\card{b^\xi \cap b^{\xi'}} < \alpha$, called an \emph{(indexed) almost disjoint family of subsets of $\alpha$}\index{almost disjoint family|textbf}. 
\begin{dfn}%
For a function $p$ with $\dom(p)\subseteq\On$ denote $\supp(\dom(p))$ by $\lvert p\rvert$\index[notation]{p@$\lvert p\rvert$}\index[notation]{ p @$\lvert p\rvert$}.
Suppose $\alpha \in \Reg$ and $B \in  {}^{[\alpha,\alpha^+)} 2$.
We define the partial order $R(B)$, called \emph{almost disjoint coding of $B$ relative to $\mathcal B$}, by case distinction:
\begin{itemize}
\item If $\alpha$ is a successor and $\alpha^-$ its predecessor, define
\[
R(B) = \{ (p, p^*) \setdef p \in {}^{[\alpha^-, \delta)} 2, \;\delta < \alpha,\; p^* \in [\{ b^{\xi}\setdef \xi \in [\alpha, \alpha^+) \}]^{<\alpha} \}.
\]
Order $R(B)$ as follows: $(q, q^*) \leq (p, p^*)$ if and only of $q \supseteq p$, $q^* \subseteq p^*$ and
for any $b \in p^*$ and $\nu \in b \cap [\lvert p\rvert, \lvert q\rvert)$, we have $q(\nu) \leq B(\nu)$.

\item For limit $\alpha$ define 
\[
R(B) = \{ (p, p^*) \setdef p \in {}^{<\alpha}2,\; p^* \in [\{ b^{\xi}\setdef \xi < \alpha^+ \}]^{<\alpha} \},
\]
ordered in precisely the same way.
\end{itemize}
\end{dfn}

Clearly $R(B)$ codes $B$ into a subset $a$ of $\alpha$ in the following sense: Letting $G$ be $R(B)$ generic over $L[B]$ and letting
\[
a =  \{ \xi < \alpha \setdef \exists (p, p^*) \in G\; p(\xi)=1 \}
\]
it holds that for each $\xi\in[\alpha, \alpha^{++})$,
\[
\xi \in B \iff \card{a \cap b^\xi} < \alpha.
\]

\subsection{Jensen Coding}\label{p.ss.jensen}
Going back to Fact~\ref{p.f.jensen.coding}: 
Roughly, $P_J$ resembles a product of almost disjoint codings, which together code $A$ into a real.
Recall from Section~\ref{sec:notation} that $\Card'$ denotes $(\Card \setminus\omega)\cup\{0\}$
and we write $0^+=\omega$.
 
\medskip

The conditions of the forcing $P_J$ are sequences $p$ with
$\dom( p) \subseteq \Card'$
 such that for each $\alpha \in \dom( p)$,
 $p(\alpha) = (p_\alpha, p^*_\alpha)$ is a condition in the almost disjoint coding of $p_{\alpha^{+}}$ relative to $\mathcal B_{\alpha^+}$.
Here, $p_\alpha$ is a function $p_\alpha \colon [\alpha, \delta) \to 2$, where $\delta<\alpha^+$---we call such $p_\alpha$ a  ``string''.
Simultaneously we arrange that also $A$ be coded.
This takes care of successor cardinals.

For limit cardinals, $P_J$ the idea is that $p_\alpha$ is ``coded'' by $p_{<\alpha}$, where 
\[
p_{<\alpha} = \bigcup_{\beta< \alpha} p_\beta.
\]
Such a $p_{<\alpha}$ is called a ``broken string''.\index{broken string}\index{string!broken}

For inaccessible $\alpha$, two different ways suggest themselves to effect a similar coding: either we demand that $p_{<\alpha}$ be bounded below $\alpha$ and use the almost disjoint coding of $p_\alpha$ relative to $\mathcal B_\alpha$, proceeding in analogy to the successor case; 
or otherwise, we may use so-called \emph{limit coding}\index{limit coding}. 
For singular limit cardinals we have no choice and  always have to use the latter method (if we are to preserve cardinals).
Note that using limit coding at a cardinal $\alpha$ makes it necessary to allow that $p_{<\alpha}$ be unbounded below $\alpha$.
Limit coding was greatly simplified in \cite{friedman:codingbook} by the introduction of \emph{coding delays}\index{coding delays} (the reader will find more details in Chapter~\ref{sec:coding}).

\medskip

To show that we do not add reals at the last stage $\kappa$ of the iteration, it seems we have to use the first, i.e., the successor-style coding\index{successor coding} method at every inaccessible; we call this \emph{Easton support}\index{Easton support}\index{support!Easton}  (more about this in the next section).
This Easton supported version of Jensen coding\footnote{Reportedly, Jensen first attempted to use Easton support, but this was abandoned and the limit coding method was used at every limit cardinal in the final version of \cite{bjw:82}.} is developed in Chapter~\ref{sec:coding}. 
See also \cite{bjw:82} and \myplacecite{4.2}{friedman:codingbook} for further details on Jensen coding (albeit the full-supported version, i.e., with limit coding at every regular limit cardinal). 

\subsection{David's Trick}\index{David's trick}\index{localization}\label{p.ss.david}

Jensen coding allows us to create a real $u$ which witnesses \eqref{p.e.def.u.0}.
A central idea to achieve that the generic $u$ witnesses the strengthening \eqref{p.e.def.u} is to amend the definition of ``string'': 
For $p$ to be a condition, require that for each $\delta \in \supp(p)$, and each 
$
\alpha \leq \lvert p_\delta \rvert
$, we have
\[
L_\alpha[p_\delta\res\alpha,r]\vDash\Delta_\alpha \Rightarrow L_\alpha[u,r]\vDash\Psi_0(r)
\]
where $\Delta_\alpha$ is a recursive set of formulas whose only free variable is $\alpha$.
Unfortunately, due to the special role of the Mahlo cardinal $\kappa$ and the ``discontinuous'' nature of Easton support, further  complications arise (e.g., \emph{virtual inaccessible coding}; see p.~\pageref{l.virtual}).

An illustration of David's trick in a simple setting is provided in Example~\ref{p.e.david} and Exercise~\ref{p.x.david} below.
See also \myplacecite{6.2}{friedman:codingbook} and \cite{david:singleton}.

\section{Stratification, Easton Support and the Mahlo Cardinal}\label{p.s.mahlo}

\subsection{Stratification}\index{preservation of kappa@preservation of $\kappa$}\index{cardinal!preservation}
Finally we discuss cardinal preservation.\index{preservation of kappa@preservation of $\kappa$}\index{cardinal!preservation}
Jensen coding preserves cardinals roughly because for each regular $\lambda$, the forcing may be written as an iteration
$P_J = (P_J)^\lambda * \dot{(P_J)_\lambda}$, where $(P_J)^\lambda$ (the ``upper part'' with respect to $\lambda$) is $\lambda^+$-distributive
and $\forces_{(P_J)^\lambda}$``$\dot{(P_J)_\lambda}$ is $\lambda$-linked'' (in fact, $\lambda$-centered). 

In fact, the upper part of $(P_J)_\lambda$ has a stronger property than being $\lambda^+$-distributive:  \emph{quasi-closure}, defined in Chapter~\ref{sec:stratified forcing}.
For this, we introduce an additional relation $\leqlol$ for each regular cardinal $\lambda$; 
Quasi-closure means that certain \emph{definable} $\leqlol$-descending sequences of length $\lambda$ posses a lower bound.\index{quasi-closure}\index{forcing!quasi-closed}
Quasi-closure (in contrast to distributivity) is preserved in forcing iterations provided the support on the upper part is large enough, which holds for diagonal support limit.

\medskip

An iteration of Jensen coding does not decompose in the same straightforward manner as above, since names for any ``upper part'' necessarily depend on
ealier ``lower parts.''
We can nonetheless salvage the essential feature of this decomposition through our notion of \emph{stratified forcing}:\index{stratification}\index{forcing!stratified}
For each $\lambda$, we will define a relation $\lequpl$ on our forcing that corresponds to ``extending on the upper part'' (with no relation whatsoever imposed on lower parts).
The linking function on the lower part simply becomes a multifunction $\Cl$ into $\lambda$ and is only defined on a dense subset.
Crucially, if $q \lequpl p$ and $\Cl( q)$ and $\Cl( p)$ have a common value, $p$ and $q$ are compatible.

\medskip

In such a setting, cardinal preservation arguments can be expressed as follows:
Given a name $\dot x$ and a regular cardinal $\lambda$, on a dense set of conditions the the interpretation of $\dot x$ will only depend on the ``lower part'' (w.r.t.\ $\lambda$), the dependence being expressed in terms of the centering multifunction $\Cl$.
As the ``upper part'' is quasi-closed, $\lambda$-many names can be ``reduced to the lower part'' simultaneously.
This ensures that any function in the extension with domain $\lambda$ can be covered by a multifunction of size $\lambda$ from the ground model.

\medskip

Stratification is preserved in forcing iterations with \emph{diagonal support limits}.
The support is called diagonal because for each $\lambda$, we need supports of size $\lambda$ on ``upper parts'' (w.r.t.\ $\lambda$) to preserve closure, but supports of size ${<}\lambda$ on lower parts to preserve the ``linkedness''.
Moreover, we must ask that the linking (multi-)functions $\Cl$ are continuous, in some sense.
We discuss stratified forcing in Chapter~\ref{sec:stratified forcing}.

\subsection{The Last Stage}
In Section~\ref{sec:stage kappa} we shall show that every element of $(\On^\omega)^{L[G]}$ is already in an initial segment $L[G_\xi]$, for some $\xi<\kappa$, proving Item~\ref{p.c.main.i.sequence} of Lemma~\ref{p.c.premain}.

We stress that this fact
does not follow directly from the theory of stratified forcing.
Instead, a kind of ``$\kappa$-properness''\index{forcing!proper} is shown in a kind of $\Delta$-systems argument involving some ideas from stratification, an archetypical use of $\Diamond$,\index{Diamond Principle} 
and the fact that enough cardinals are preserved in earlier stages of the iteration (this is, of course, is a consequence of stratification). 

The proof takes the general form described above at the end of the previous section:
Given a sequence $(\dot x_n)_{n < \omega}$ of $P_\kappa$-names, we reduce each of them to the ``lower part.''
But lower part with respect to which cardinal? 
No $\lambda < \kappa$ is preserved, and taking $\lambda = \kappa$ does not confine a name to an initial segment as antichains can have size $\kappa$ in the respective  ``lower part''.

The solution is, very vaguely, to let $\lambda$ vary throughout the argument:
One ``reduces'' diagonally as much as one can; we assume $\kappa$ to be Mahlo so that the construction reaches a closure point $\lambda < \kappa$ which is inaccessible and thus supports below it are bounded and we catch our tail. 
A parallel can be seen in the argument that the $\kappa$-chain condition is preserved in ${<}\kappa$-support iterations provided that $\kappa$ is Mahlo (see, e.g., \cite{baumgartner}).
It is to allow this proof to go through that we need an \emph{Easton supported}\index{Easton support}\index{support!Easton} version of localized Jensen coding (developed in Section \ref{sec:coding}).

\medskip

Whether the use of the Mahlo cardinal is necessary is currently unknown.
To the authors, it seems plausible that a similar construction go through starting only with an inaccessible. 
Again, the crux would be to show that every sequence of ordinals of countable length appears in an initial segment of the iteration where $\kappa$ is still a cardinal (in fact, still inaccessible).

\section{Amalgamation}\label{p.s.amalgamation}

To build iterations which are sufficiently homogeneous for random subalgebras (in the sense of Definition~\ref{p.d.hom})
we construct a variant of Shelah's amalgamation\index{amalgamation} from \cite{shelah:amalgamation}.

\medskip

Suppose $f \colon R_0 \to R_1$ is an isomorphism of preorders which are complete suborders of $\ro(P)$. Let $\pi_j\colon \ro(P) \to \ro(R_j)$ be the canonical projection. 

We now define a preorder $P^{\Int}_f$ which (under an additional assumption on $P$) has $P$ as a complete suborder and carries an automorphism which agrees with $f$ on $R_0$.
See  Chapter~\ref{sec:amalgamation} for proof of a refined version of the following theorem.
\begin{thm}\label{p.t.amalgamation}
Suppose $P$ is closed under $\pi_0$, $\pi_1$,  and the partial functions $f$ and $f^{-1}$.
Define 
\begin{equation}\label{p.e.amalgamation}
P^{\Int}_f = \{ \bar p\colon \Int \to P \setdef \forall i \in \Int\; f(\pi_0(\bar p(i)))=\pi_1(\bar p(i+1)) \}.
\end{equation}
Moreover let $\Phi\colon P^{\Int}_f\to P^{\Int}_f$ be defined as follows: For each $i\in \Int$,
\[
\Phi(\bar p) (i) = \bar p( i +1),
\]
i.e., $\Phi$ is given by the left-shift.
Then $\Phi$ is an automorphism of $P^{\Int}_f$ which agrees with $f$ on $R_0$ and $P$ is a complete suborder of $P^{\Int}_f$.
\end{thm}
This allows to build sufficiently homogeneous iterations: 
By bookkeeping, we can easily build an iteration $\langle P_\xi \setdef \xi \leq \kappa \rangle$ so that for each $\xi< \kappa$, any pair of $P_\xi$-names $\dot r_0$, $\dot r_1$ and each $\iota < \kappa$ such that $\dot r_0$, $\dot r_1$ are random over $V^{ P_\iota}$, 
 there is a sequence $\langle \eta_\xi \setdef \xi < \kappa\rangle$ cofinal in $\kappa$ and a sequence $\langle\Phi_\xi\setdef \xi \leq \kappa\rangle$ such that letting $B_\xi$ denote $\ro(P_\xi)$,
\begin{enumerate}
\item\label{p.amalgam.type1} $P_{\eta_0 +1} = (B_{\eta_0}\setminus\{0_{B_{\eta_0}}\})^{\Int}_f$, where $f$ is the isomorphism of suborders of $\ro(P_{\eta_0})$ which sends $\dot r_0$ to $r_1$ and keeps $P_\iota$ fixed.
\item $\Phi_0$ is the automorphism of $P_{\eta_0 +1}$ given by left-shift
\item\label{p.amalgam.type2} For each successor $\xi$ with $0 < \xi < \kappa$, $P_{\eta_{\xi} +1} = (\ro(B_{\eta_\xi})\setminus\{0_{B_{\eta_\xi}}\})^{\Int}_{\Phi_{\xi-1}}$
\item $\Phi_\xi$ is the automorphism of $P_{\eta_{\xi} +1}$ given by left shift, for $\xi$ as in the previous item
\item For each limit $\xi$ with $0 <\xi \leq \kappa$, $\Phi_\xi$ is the automorphism of $P_{\eta_\xi}$ given as the inverse limit of
$\langle \Phi_\nu \setdef \nu < \xi \rangle$.
\end{enumerate}
Clearly, $\Phi_\kappa$ is an automorphism of $P_\kappa$ sending $\dot r_0$ to $\dot r_1$ which is the identity on $P_\iota$.
If we succeed in showing Item~\ref{p.c.main.i.sequence} of Corollary\ref{p.c.main}, our iteration is easily seen to sufficiently homogeneous. %

\medskip

In order to guarantee cardinal preservation as well as Item~\ref{p.c.main.i.sequence} of Corollary~\ref{p.c.main}, i.e., that every sequence in our final model appears in some initial segment of the iteration, the construction of Theorem~\ref{p.t.amalgamation} shall be modified substantially in several ways:

One is the appropriate support condition in the above product---$P^{\Int}_f$ given above uses ``full support,'' and as we shall show in Section~\ref{p.s.digression}, Shelah's amalgamation is just $P^{\Int}_f$ with finite support. 
We must of course use \emph{diagonal support}.

In addition, we have to ``prepare''  by replacing $P$ with a carefully chosen dense subset in order to preserve quasi-closure. 
This precludes that we work with complete Boolean algebras---but note that the assumption that $P$ be closed under $\pi_0$, $\pi_1$, $f$, and $f^{-1}$ in Theorem~\ref{p.t.amalgamation} isn't entirely vacuous, as otherwise $P^{\Int}_f$ in \eqref{p.e.amalgamation} may be empty (cf.\ Remark~\ref{am:remark} and Section~\ref{p.s.digression} below).

Lastly, amalgamation produces iterations $(P_\iota,\dot Q_\iota)_{\iota<\theta}$ whose quotients $\dot Q_\xi$ aren't (forced to be) stratified.
For this reason we introduce \emph{stratified extensions} in Chapter~\ref{sec:ext}, describing how stratified initial segments of an iteration \emph{cohere} sufficiently so that diagonal support limit be stratified.

\section{An Example, and a Digression}\label{p.s.exm}

We now give an example presenting our coding method in very simplified form: 
We show how to make a single real $r$ projective, i.e., find a forcing extension in which ``$n\in r$'' is equivalent to a projective formula (without parameters).

\medskip

Let $( T(n) )_{n < \omega}$\index{independent!Suslin trees}\index{Suslin trees!independent} be the canonical sequence of $\omega_2$-Suslin trees (see Definition~\ref{p.n.d.canonical}).
As the sequence is independent,
forcing with $\prod_{n\neq m} T(n)$
does not add a branch through $T(m)$,  for $m\in \omega$.

For $\alpha \in \{ \omega, \omega_1\}$, let us now fix a constructible almost disjoint family $\mathcal B_\alpha = \langle b^\xi \setdef \xi \in [\alpha, \alpha^+)\rangle$ 
so that 
$\{ (\zeta, \xi) \setdef \zeta \in b^\xi,\; \xi \in [\alpha, \alpha^+)\}$
is locally semidecidable (cf.\ Definition~\ref{p.d.locally.definable}; we later explicitly define such a family, see Definition~\ref{reg_defs}, \emph{Coding apparatus}, p.~\ref{b^s:page}).

\begin{exm}\label{p.e.coding}\label{p.e.jensen}\label{p.e.david} 
Let some $r\in \pset(\omega)$ be given assume $(T(n) )_{n < \omega}$ remains independent in $L[r]$. 
Let $B_n$be the characteristic function of 
$\{ \# t \setdef t \in B\}$ 
where $B$ is the $T(n)$-generic branch through $T(n)$ and $\# t$ is the order-type of $t$ in $\leq_L$.
Consider the product
$\prod_{n\in r} T(n) *  R(\dot B_n)$, the composition (iteration) of $T(n)$ with the  almost disjoint coding of $B_n$ with respect to $\mathcal B_{\omega_1}$ (cf.\ Section~\ref{s.almost}).

\begin{claim}\label{p.claim}
If $G$ is $ \prod_{n\in r} T(n) *  R(\dot B_n)$-generic over $L[r]$, there is a $\Sigma_2$ formula $\Psi^*(n)$ such that
$\forall n\; \big( n\in r\iff \HSize(\omega_2)\vDash \Psi^*(n) \big)$ holds in $L[G]$ .
\end{claim}
We shall only sketch the proof. The reader is again advised to keep in mind Definition~\ref{p.n.d.canonical} and Convention~\ref{d.localobject}. 
\begin{proof}[Sketch of proof.]
We first show that in $L[G]$, $n\in r$ if and only if
\begin{equation}\label{p.e.definable}
\exists B^*_n\in\pset(\omega_1)\; L[B^*_n]\vDash \text{$T(n)$ has a branch.}
\end{equation}
Of course, the above is witnessed by
\begin{equation}\label{p.exm.B*_n}
B^*_n = \{ \xi < \omega_1 \setdef \exists t,p,p^* \; (n,t, (p,p^*)) \in G\wedge p(\xi) =1 \}.
\end{equation}
To see that \eqref{p.e.definable} defines $r$, it suffices to show that if $n\notin r$, $T(n)$ remains Suslin in $L[G]$.
For this, use that $( T(n) )_{n < \omega}$ is independent and each $R(\dot B_n)$ is forced by $T(n)$ to be $\omega_1$-linked.
We leave details to the reader.

To see that moreover, $n\in r$ is  equivalent to a $\Sigma_2$ formula in $\HSize(\omega_2)$, consider the formula $\Psi^*(n)$ given by
\begin{multline*}
\exists B^*_n\in \pset(\omega_1) \; \forall \alpha \in [\omega_1, \omega_2) \; \\
\text{ if } L_\alpha[B^*_n]\vDash\ZF^- + $``${\omega_2}^{L_\alpha[B^*_n]}\text{ exists and equals ${\omega_2}^{L_\alpha}$'' }\\
\text{then} L_\alpha[B^*_n] \vDash\text{$T(n)^{L_\alpha}$ has a branch.}
\end{multline*}
For a given $n$, $\Psi^*(n)$ implies \eqref{p.e.definable} and hence $n\in r$: Fix $B^*_n$ witnessing $\Psi^*(n)$ and assuming \eqref{p.e.definable} fails, consider the transitive collapse of a countable elementary submodel of $L_{\omega_3}[B^*_n]$, leading to a contradiction.
It remains to see that for all $n\in r$, $\Psi^*(n)$ holds. Here we use the local semidecidability of $\mathcal B_{\omega_1}$ and of $T(n)$:
Not only does $B^*_n$ ``code'' a branch through $T(n)$, this also holds locally inside $L_\alpha[B^*_n]$ whenever
$L_\alpha[B^*_n] \vDash \ZF^- + $``${\omega_2}^{L_\alpha[B^*_n]}$ exists and equals ${\omega_2}^{L_\alpha}$.'' 
\end{proof}
Now consider, for each $n\in \omega$, the forcing $P_D(B^*_n)$ consisting of conditions $(p_\omega, p^*_{\omega_1}) \in R(B^*_n)$ 
such that
\begin{multline}\label{p.e.killing}
\forall \xi \in [\omega,|p_\omega|] \; \forall \alpha < \omega_1 
\\ \text{ if } L_\alpha[p_\omega\res\xi]\vDash\ZF^- + 
\text{``} \Card^V=\Card^L \text{ and ${\omega_2}^{L_\xi}$ exists}\\
\text{ and } {\omega_1}^{L_\xi} = \xi\text{''}
\text{ then }L_\alpha[p_\omega\res\xi]\vDash\text{``}T(n)^{L_\alpha}\text{ has a branch.''}
\end{multline}
with the ordering inherited from $R(B^*_n)$ (where $R(B^*_n)$ is the almost disjoint coding of $B^*_n$ with respect to $\mathcal B_\omega$, cf.~Section~\ref{s.almost}).

The forcing $R(B_n) * P_D(\dot B^*_n)$ 
is essentially a ``toy version'' of Jensen coding with David's trick (to be revisited in Example~\ref{s.e.david} and Fact~\ref{s.f.jensen} to illustrate quasi-closure and stratification).
We leave the following to the reader as an exercise. 
\end{exm} 
\begin{exc}\label{p.x.david}
If $G$ is $ \prod_{n\in r} T(n) *  R(\dot B_n) * P_D(\dot B^*_n)$-generic over $L[r]$, there is a $\Sigma_2$ formula $\Psi(n)$ such that
$\forall n\; \big( n\in r\iff \HSize(\omega_1)\vDash \Psi(n) \big)$ holds in $L[G]$.
Thus, $r$ is a $\Sigma^1_3$ subset of $\omega$ in $L[G]$.
\end{exc}

\subsection{Digression on Shelah's Amalgamation}\label{p.s.digression}

We quickly review Shelah's original version of amalgamation \cite{shelah:amalgamation}\index{Shelah's amalgamation} and compare it to our variant discussed above. 
This section is not needed for the proof of our main result.

\medskip

The following account of Shelah's amalgamation is modeled after \cite{jr:amalgamation}.
Suppose we are given $P$, $e_i$, $R_i$ for $i\in\{0,1\}$, and $f$ as in Section~\ref{p.s.amalgamation}.
We seek a preorder $P'$ containing $P$ as a complete suborder and admitting an automorphism which extends $f\colon R_0 \to R_1$.

\medskip

For this, letting $\pi_i\colon \ro(P) \to \ro(R_i)$ denote the canonical projection, define
\[
P \times_{f} P = \{ (p_0, p_1) \in P^2 \setdef f(\pi_0(p_0)) \cdot \pi_1(p_1) \neq 0 \},
\]
carrying the natural order (that is, the product order).

As a forcing $P \times_{f} P$ is equivalent to 
$
R_0 * [(\quot{P}{R_0})^{G_{R_0}} \times (\quot{P}{R_1})^{f[G_{R_0}]}]
$,  
i.e., intuitively we take the product of two copies of $P$,
but identify $R_0$ with $R_1$.

Let $R^1_0 = \{1_P\}\times P$ and $R^1_1 = P \times \{1_P\}$, fix $j \in \{0,1\}$, and identify $P$ with  $R^1_j$.  For $r \in R_0$, the map $(1,r)\mapsto (r,1)$ agrees with $f$ in $\ro(P \times_{f} P)$, since $(p_0,p_1) = (p_0\cdot f^{-1}(\pi_1(p_1)), f(\pi_0(p_0))\cdot p_1)$ and hence $(r,1)=(1,f(r))$ in $\ro(P \times_{f} P)$ (which is the same as $\ro(\ro(P) \times_{f} \ro(P))$).

\medskip

We iterate this construction recursively:
Suppose we have $P_n$ together with two complete suborders $R^n_0, R^n_1$ of $\ro( P_n)$ and an isomorphism
$f_n \colon R^n_0 \to R^n_1$; for $n=0$ we let $P_0 = P$, $R^0_i = R_i$ for $i \in \{0,1\}$ and $f_0= f$.
Define
\begin{align*}
P_{n+1} &= P_n \times_{f_n} P_n,\\
R^{n+1}_0 &= \{ (1, p) \setdef p \in P_n \},\\
R^{n+1}_1 &= \{ (p, 1) \setdef p \in P_n \}.
\end{align*}
It is not hard to see that $R^n_0$ and $R^n_1$ are  complete suborders of $P_n$.
We have an isomorphism $f_{n+1} \colon R^{n+1}_0 \to R^{n+1}_1$ (it doesn't in general extend to $P_{n+1}$) given by
\[
f_{n+1} (1,p) = (p,1).
\]
Identifying $P_n$ with a $R^{n+1}_0$ for even $n$ and with $R^{n+1}_1$ for odd $n$, $f_n \subseteq f_{n+1}$.
This identification makes $(P_n)_{n< \omega}$ a directed system,  and its direct limit $P_\omega$
admits an automorphism $f_\omega$ obtained as the direct limit of the system of maps $( f_n)_{n\in \omega}$.
Thus, taking $P_\omega$ as $P'$ solves the problem described at the beginning of the section.

\medskip

Provided, e.g., $P$ is the positive part of a complete Boolean algebra, $P_\omega$ and $f_\omega$ can be defined equivalently in the following way:
\begin{thm}\label{p.t.amalgamation.comp}
Suppose $P$ is closed under $\pi_0$, $\pi_1$, and the partial functions $f$ and $f^{-1}$.
The following set is isomorphic to a dense subset of $P_\omega$:
\begin{equation}\label{p.e.set}
\{ \bar p \in P^{\Int}_f \setdef  \supp(\bar p) \text{ is finite}\}
\end{equation}
where for $\bar p\in P^{\Int}_f$ we let 
\[
\supp(\bar p) = \{ i \in \Int \setdef \bar p(i) \neq 1_P \}.
\]
Moreover $f_\omega$ and $\Phi$ agree on the set from \eqref{p.e.set} (modulo the isomorphism).
\end{thm}

We mention Theorem~\ref{p.t.amalgamation.comp} only for readers with some familiarity with Shelah's construction.
Since otherwise we will make no use of this theorem we shall not prove it, but we do provide the following hints:
Letting 
\[
\pi^{n}_i\colon \ro( P_n) \to \ro( R^n_i)
\]
be the canonical projection, 
calculate $\pi^n_i$ and $f_n$ for each $n \in \nat $ and $i\in\{0,1\}$. 

For instance, it is easy to see that
\begin{align*}
\pi^1_0 (p_0,p_1) &= f(\pi_0(p_0))\cdot p_1,\\
\pi^1_1 (p_0,p_1) &= p_0 \cdot f^{-1}(\pi_1(p_1)).
\end{align*}
Therefore, direct calculation shows 
\begin{gather*}
P_2 = ( P \times_{f} P ) \times_{f_1} ( P \times_{f} P ) =\\ 
\{ (p_0,p_1,p_2,p_3)\in P^4 \setdef f(\pi_0(p_0)) \cdot p_1 \cdot  p_2 \cdot f^{-1}(\pi_1(p_3)) \neq 0 \}.
\end{gather*}

Given any condition $(p_0,p_1,p_2,p_3) \in P_2$, by the above (and as by closure of $P$ under $\pi_0$, $\pi_1$, $f$, and $f^{-1}$) we may define $p_{1,2} \in P$ by
\[
p_{1,2} = f(\pi_0(p_0)) \cdot p_1 \cdot  p_2 \cdot f^{-1}(\pi_1(p_3)).
\]
It follows that (using again as $P$ is closed under $\pi_0$, $\pi_1$, $f$, and $f^{-1}$) %
\[
(p_0 \cdot f^{-1}(\pi_1(p_{1,2})), p_{1,2}, p_{1,2}, f(\pi_0(p_{1,2})) \cdot p_3) \in P_2,
\] 
extends $(p_0,p_1,p_2,p_3)$. 

Deleting the redundant coordinate and renaming, i.e., letting 
\[
(p'_{-1}, p'_0, p'_1) = \big(p_0 \cdot f^{-1}(\pi_1(p_{1,2})), p_{1,2}, f(\pi_0(p_{1,2}))\cdot p_3\big)
\]
we conclude that the following partial order is isomorphic (via an isomorphism that just redistributes coordinate indices) to a dense subset of $P_2 $:
\[
\{ (p'_{-1}, p'_{0}, p'_1)\in P^3 \setdef f(\pi_0(p'_{-1})) = \pi_1(p'_0)  \wedge f(\pi_0(p'_0))=\pi_1(p'_1)\}
\]
We have thus shown that $P_2$ has a dense set which is isomorphic to  
\[\{ \bar p \in P^\Int_f \setdef \supp(\bar p) \subseteq [-1,1]\}.\]

\medskip

An inductive argument generalizing the above shows the following dense inclusions modulo isomorphism for each $m\in \omega$: 
\begin{align*}
\{ \bar p \in P^\Int_f \setdef \supp(\bar p) \subseteq [-m,m]\} & \mathbin{\stackrel{\scriptscriptstyle{\mathrm{d}}}{\to}} P_{2m} \\
\{ \bar p \in P^\Int_f \setdef \supp(\bar p) \subseteq [-m-1,m]\} & \mathbin{\stackrel{\scriptscriptstyle{\mathrm{d}}}{\to}} P_{2m+1},
\end{align*}
where $\mathbin{\stackrel{\scriptscriptstyle{\mathrm{d}}}{\to}}$ is a shorthand for ``embeds densely''.
Moreover, $P$ is always identified with the $0$-component,
and $f_n$ acts  like the left-shift on the set of $\bar p$ such that $\supp(\bar p) \subseteq (-m,m]$ when $n=2m$, and such that $\supp(\bar p) \subseteq (-m-1,m]$ when $n=2m+1$.

\medskip

With these remarks and perhaps consulting \cite{jr:amalgamation}, the reader should have no difficulty in showing Theorem~\ref{p.t.amalgamation.comp}.
It is crucial that $P^\omega$ is a direct limit. 
The sketch here also makes crucial use of the assumption that $P$ is  is closed under $\pi_0$, $\pi_1$ and the partial function $f$, but of course it is easy to ensure such closure. 
In particular, the argument above will work for any $P$ using the \emph{``hybrid''} we work with in Chapter~\ref{sec:amalgamation}; see Definition~\ref{def:blowup}.

\chapter{Stratified Forcing}\label{sec:stratified forcing}

In this chapter we assume $V=L[A]$ for some class $A$.
We define \emph{stratified partial orders}, show such orders preserve cofinalities, give some examples and show that stratification is preserved under composition.
We also define diagonal support. These definitions have precursors in \mycites{friedman:codingbook,friedman-iterated}. 
We state that diagonal support iterations whose components are stratified are themselves stratified,  without proof, since we prove a slightly more general theorem in Chapter \ref{sec:ext} where we deal with iterations with stratified initial segments but where the components aren't necessarily stratified.

We present the definition of stratification in two parts: the first we dub \emph{quasi-closure}.
We treat this first part separately from the remaining axioms of stratification for the following reasons: firstly, the proofs that each of these two  groups of axioms is preserved in iterations are not only different but virtually independent of each other.

Secondly, we hope that the reader will agree that quasi-closure is interesting in its own right.
This view is in stark contrast to the fact that quasi-closure alone is not a very useful property---in fact, every partial order is quasi-closed. 
One should think of it as an incomplete notion, to which some other property has to be added in order to render it non-trivial.
Stratification is one example of this, closely connected to the notion of linked forcing.
There may be other examples, as well.

Before we define quasi-closure, we introduce preclosure systems; analogously we will define prestratification systems.
We can reuse these notions when we define \emph{quasi-closed and stratified extension}; see Chapter \ref{sec:ext} on page \pageref{sec:ext}.

\section{Quasi-Closure}\label{sec:qc}

Throughout, let $\langle R, \leq \rangle$ be a preorder and let $I$ be an interval of regular infinite cardinals\index{interval of regular infinite cardinals}, that is $I = \Reg \cap [\lambda_0,\lambda_1)$\index[notation]{I@$I$, an interval of regular cardinals} or $I = \Reg \cap[\lambda_0, \infty)$ for some regular $\lambda_0$ and some (if you like, regular) cardinal $\lambda_1$.
We introduce the notion of $R$ being \emph{stratified on $I$}, which implies that for any $\lambda \in I$, any ordinal of cofinality greater than $\lambda$ will remain so after forcing with $R$.
We have to use another property of $R$ to show cardinals greater $(\lambda_1)^+$ are preserved. 
Moreover, we want to allow for $R$ to collapse some cardinals up to and including $\lambda_0$.

In our application $\lambda_1 = \kappa^+$ throughout, while  $\lambda_0$ will increase throughout the iteration until at the final stage  we have $[\lambda_0, \lambda_1)=\{\kappa\}$.\footnote{That $\kappa$, $\kappa^{++}$, and $\kappa^{+++}$ are preserved  in our application will not simply follow from stratification; the first is hard to show, the latter is neither hard nor essential. 
Stratification fails at $\kappa^+$ since the Suslin trees are not quasi-closed but only distributive, and at $\kappa^{++}$ since they are not centered.
Each stage is also stratified on $[\kappa^{+++},\infty)$ (see Example~\ref{exm:centered}).}

\medskip

From now on, we write $\qcdefSeq(\{x\})$
\index[notation]{PiT1(x), PiT1(X)@$\qcdefSeq(\{x\})$, $\qcdefSeq(X)$} \index[notation]{SigmaT1(x), SimgaT1(X)@$\qcdefG(\{x\})$, $\qcdefG(X)$} 
for the set of formulas which are provably in $T$ equivalent to a $\Pi^A_1$ formula with parameter $x$, where $T$ is an adequate weak fragment\index[notation]{T, a weak fragment of ZFC@$T$, a weak fragment of $\ZFC$}
 of $\ZFC$.
The choice of $T$  is a matter of convenience, but for concreteness, let $T$ state that universe is closed under rudimentary functions and the function $x \mapsto \operatorname{tcl}(x)$ assigning to $x$ its transitive closure.
We shall not assume that this class of formulas is closed under bounded quantification; therefore we will painstakingly demonstrate there is an appropriate  $\qcdefSeq(\{x\})$ formula in \ref{adequate} and \ref{lem:D}.
For a set $X$, we write $\Pi^T_1(X)$
for the set of formulas which are $\Pi^T_1(\vec{x})$ for some finite tuple $\vec{x}$ of elements of $X$. Analogously for $\Sigma^T_1(X)$
\index[notation]{SigmaT1(z), SigmaT1(X)@$\Sigma^T_1(z)$, $\Sigma^T_1(X)$|textbf} 
etc.

We now make a few convenient definitions that facilitate the treatment of quasi-closed partial orders, which we define afterward. We shall also build on these definitions when we define \emph{quasi-closed extension} in Chapter~\ref{sec:ext}.
\begin{dfn}\label{def:pcs}
We say $\pcs=\pcsl$\index[notation]{s (preclosure system)@$\pcs$ (preclosure system)}\index[notation]{c aa(parameter in preclosure system)@$c$ (parameter in preclosure system)} is a preclosure system\index{preclosure system} for $R$ on $I$ if and only if $\D\subseteq \Reg\times V\times R^2$\index[notation]{D(lambda,x,r)@$\D(\lambda,x,r)$}\index[notation]{lessthanl  lambda@$\leqlol$}
is a $\qcdefD(\{\param\})$ class and for every $\lambda \in I$, $x \in V$, $p,q,r \in R$ we have
\begin{enumerate}[label=(C \arabic*), ref=C~\arabic*]
\item   if $p\leq q \in \D(\lambda, x, r)$ then $p \in \D(\lambda, x, r)$.\label{qc:D}%
\item  The relation $\leqlol$ is a preorder on $R$ and $p\leqlol q \Rightarrow p \leq q$. \label{qc:preorder}
\item  If $p \leq q \leq r$ and $p \leqlol r$ then $p \leqlol q$.\label{er}
\item  If $\bar \lambda \in I$, $\bar \lambda \geq \lambda$ then $q\leqlo^{\bar\lambda}p \Rightarrow q \leqlol p$.\label{qc:leqlo:vert}
\end{enumerate}
\end{dfn}
As a notational convenience, define $\leqlo^0$ to mean $\leq_R$.
Clause~(\ref{er}) can be dropped if one is not interested in iterations.
Observe that by (\ref{er}), $\leqlol$ is well-defined with respect to equivalence modulo $\approx$ (remember we say $p \approx q \iff p \leq q$ and $q \leq p$).

Think of each of the relations $\leqlol$ as a notion of \emph{direct extension}\index{direct extension}, as it is often called in the case of, e.g., Prikry-like forcings. Intuitively,  $p \leqlol q$ expresses that $p$ extends $q$ but some part ``below $\lambda$'' is left unchanged.
In our application, $ q \in \D(\lambda, x, r)$ will mean something like ``$q$ is partially generic over a Skolem-hull of $\lambda\cup\{x,r\}$''.
Think of $\D$ as providing a kind of strategy, as with strategically closed forcing (see Definition~\ref{d.strategic-closure} below). Together, this additional structure on $R$ allows us to construct sequences in such a way as to ensure there is a  lower bound in $R$.
The missing ingredient and distinct flavor of quasi-closure is the condition that these sequences be \emph{definable} in a sense.
The main point is that the definability and the use of $\D$ are intertwined in that they are coordinated by a common object $\bar w$ which we shall call the \emph{strategic guide and canonical witness}.

For the next two definitions, fix a preclosure system $\pcs$ for $R$ on $I$. 
All the notions in the next two definitions have their meaning \emph{with respect to $\pcs$}.
\begin{dfn}\label{strategic}\index{strategic guide|textbf}\index[notation]{w (strategic guide, canonical witness)@$\bar w$ {\small(strategic guide, canonical witness)}}
Let $\bar p = (p_\xi)_{\xi<\rho}$ be a sequence of conditions in $R$, $\lambda\in I$, and $\rho \leq \lambda$. 
We say $\bar w$ is a \emph{$(\lambda,x)$-strategic guide for $\bar p$} if and only if $\bar w= (w_\xi)_{\xi<\rho}$ is a sequence of the same length as $\bar p$ and for a tail\footnote{To require this only on a tail is natural; it is also necessary in the proof of \ref{thm:it:qc}, p.~\pageref{tail}.} of $\xi < \rho$,
\begin{enumerate}
\item For some $\lambda' \in I$, $p_{\xi+1}  \in \D(\lambda', (x,c, \bar w \res \xi+1) ,p_\xi)$
and $p_{\xi+1}  \leqlo^{\lambda'} p_\xi$.
\item $p_{\xi+1}\leqlol p_\xi$.
\item If $\xi$ is a limit, $p_{\xi}$ is a greatest lower bound of $(p_\nu)_{\nu<\xi}$.
\end{enumerate}
\end{dfn}
That we quantify out the $\lambda'$ is only a small trick to facilitate the treatment of diagonal support. 
In principle, this quantifier means that one could eliminate the mention of $\lambda$ from $\D$ altogether, although it is doubtful that it would make the presentation clearer.
Also note that in natural instances of quasi-closed forcing, we always have $\lambda'\geq \lambda$; in fact, until we treat transfinite iterations, the reader may safely assume $\lambda'=\lambda$.
\begin{dfn}~\label{canonical}
\begin{enumerate}
\item We say a sequence $\bar w = (w_\xi)_{\xi<\rho}$ is $\qcdefSeq(X)$
\index{sequence!PiT1(X)@$\qcdefSeq(X)$}
\index[notation]{PiT1(X)@$\qcdefSeq(X)$!sequence} 
if for some $\qcdefSeq(X)$ formula $\Psi$ we have $w = w_\xi \iff \Psi(w,\xi)$.
We say $G$ is a $\qcdefG(X)$ function\index{function, SigmaT1(X)@function, $\qcdefG(X)$}\index[notation]{SigmaT1(X)@$\Sigma^T_1(X)$!function} if $G(x)=y$ is a $\qcdefG(X)$ formula and $G$ is a partial function.
\item Let $\bar p = (p_\xi)_{\xi<\rho}$ be a sequence of conditions in $R$ and $\rho \leq \lambda$. 
We say $\bar w = (w_\xi)_{\xi<\rho}$ is a \emph{$(\lambda, x)$-canonical witness\index{canonical!witness}\index[notation]{lambda,x-canonical witness@$(\lambda, x)$-canonical witness}\index[notation]{ lambda,x-canonical witness@$(\lambda, x)$-canonical witness} for $\bar p$}
if and only if $\bar w$ is $\qcdefSeq(\lambda\cup\{x, \param\})$ and for some $\qcdefG(\lambda\cup\{(x, \param)\})$ (partial) function $G$, we have $p_\xi = G(\bar w\res \xi +1 )$ for every $\xi <\rho$.
\item We say $\bar p$ is $(\lambda,x)$-adequate\index{adequate sequence}\index[notation]{lambda,x-adequate@$(\lambda, x)$-adequate}\index[notation]{ lambda,x-adequate@$(\lambda, x)$-adequate} if and only if $\rho \leq \lambda$, $\{\param\}$ (from Definition~\ref{def:pcs}) is $\qcdefG(x)$ and there is $\bar w$ which is both a strategic guide and a canonical witness for $\bar p$.
\item\index[notation]{lambda-adequate@$\lambda$-adequate} If $\bar p$ is $(\lambda,x)$-adequate for some $x$ we say $\bar p$ is \emph{$\lambda$-adequate}.
\end{enumerate}
\end{dfn}

Note that a sequence $\bar p$ has a canonical witness just if it is $\mathbf{\Sigma}^T_2$.
Also note there is some flexibility with respect to the degree of definability we require for $G$; one could do with rudimentary $G$ or with one specific function (projection).
The demand that $\{\param\}$  is $\qcdefG(x)$ allows us to guarantee that we can always make use of some bare mininum of information when building adequate sequences. 
This will be helpful in the proof that diagonal support iterations preserve quasi-closure (see Theorem~\ref{thm:it:qc}). 
\index[notation]{c aa(parameter in preclosure system)@$c$ (parameter in preclosure system)|textbf}Concretely, we will have $c = (L_\mu[A], \theta, \hdots)$ where $L_\mu[A]$ is large enough and $\theta$ is the length of the iteration (for technical reasons and because we are not assuming that the iteration itself is definable, more parameters will be added---namely the system of projection maps). 
The presence of $\lambda$ will also play a role in showing quasi-closure is preserved in diagonal limits.
\begin{dfn}\label{def:qc}\index{forcing!quasi-closed|textbf}\index{quasi-closure|textbf}
We say $\langle R, \pcs \rangle$ is \emph{quasi-closed on $I$} if and only if for any $x$ and $\lambda\in I$, 
\begin{enumerate}[label=(C \Roman*), ref=C~\Roman*]
\item For any $p \in R$ there is $q \in \D(\lambda,x,p)$ such that $q \leqlol p$. 
In addition we can demand that $q \leqlo^{\bar\lambda} 1_R$ for any $\bar \lambda \in I$ such that $p \leqlo^{\bar\lambda} 1_R$.
\label{qc:redundant}
\item  Every $\lambda$-adequate sequence\index{adequate sequence}\index[notation]{lambda,x-adequate@$(\lambda, x)$-adequate}\index[notation]{ lambda,x-adequate@$(\lambda, x)$-adequate} $\bar p = (p_\xi)_{\xi<\rho}$ in $R$ has a greatest lower bound $p$ in $R$
and for all $\xi<\rho$, $p \leqlol p_\xi$. If $\bar \lambda\in I$ is such that for each $\xi < \rho$, $p_\xi \leqlo^{\bar \lambda} 1_R$, then $p \leqlo^{\bar \lambda} 1_R$.\label{qc:glb}
\end{enumerate}
We also use the expression \emph{$R$ is quasi-closed as witnessed by $\pcs$}. If we omit $\pcs$ and no preclosure system can be deduced from the context, we mean that there exists a  preclosure system $\pcs$ such that $\langle R, \pcs \rangle$ is quasi-closed.
When we say \emph{quasi-closed on $[\lambda_0, \lambda_1)$}, we mean of course
quasi-closed on $[\lambda_0, \lambda_1)\cap\Reg$.
\end{dfn}
The last sentences of Clauses~(\ref{qc:glb}) and  (\ref{qc:redundant}) are useful regarding infinite iterations of quasi-closed forcings.

\medskip

As has been mentioned, quasi-closure alone is not a very useful notion (it becomes useful in the context of stratified forcing):
\begin{rem}For arbitrary $R$, just define $p\leqlol q$ if and only if $p=q$ and $\D(\lambda,x,p)=\{p\}$ for all regular $\lambda \geq \lambda_0$ and all $x$.
Then $R$ is quasi-closed.
Quasi-closure becomes non-trivial under the additional hypothesis that certain questions about the generic extension can be decided by strengthening a condition in the sense of $\leqlol$, for some $\lambda$.
Stratified forcing satisfies such a hypothesis.
\end{rem}

Nonetheless, some forcings are naturally quasi-closed with a non-trivial, useful notion of direct extension:
\begin{exm}
Say $R$ is $\lambda^+$-closed and let $\leqlol$ be identical to $\leq$; then $R$ trivially satisfies all the conditions of \ref{def:qc} for $\lambda$.
\end{exm}

Recall the notion of strategically closed forcing:\index{strategic closure}\index{forcing!strategically closed}
\begin{dfn}\label{d.strategic-closure}
We say $R$ is \emph{$\lambda$-strategically closed} if and only if there exists a function $\sigma\colon  R \to R$ (called \emph{a strategy}) such that for any $p \in R$, $\sigma(p) \leq p$ and such that any descending sequence $\bar p = ( p_\xi )_{\xi < \rho}$ of conditions from $R$ of length $\rho < \lambda$
which satisfies $p_{\xi+1} \leq \sigma(p_\xi)$ for each $\xi <\rho$,  has a lower bound in $R$.
\end{dfn}

\begin{exm}
If $R$ is $\lambda^+$-strategically closed the conditions of \ref{def:qc} for $\lambda$ are satisfied: 
For if $\sigma\colon R\rightarrow R$ is a strategy for $R$,
define $\D(\lambda,x,p)= \{ q \in R \setdef q \leq \sigma(p) \}$.
$\D$ is clearly $\qcdefSeq(\{\sigma \})$. Define $\leqlol$ to be the same as $\leq$. 
Then every $\lambda$-adequate sequence has  a greatest lower bound.
\end{exm}

These are our first examples of forcings which non-trivially satisfy the definition of quasi-closed (albeit for just one fixed $\lambda$), since any statement about the generic  can be decided by extending in the sense of $\leqlol$.

The authors know of only two forcings which are non-trivially quasi-closed in the sense that they do not obviously satisfy a simpler iterable property: the \emph{reshaping} forcing\index{reshaping}\index{forcing!reshaping} of \cite{bjw:82}, and versions of Jensen coding.

\medskip

As we shall never need reshaping we re-visit our ``toy version'' of David's trick\index{David's trick}\index{toy version!of David's trick is quasi-closed} from Example~\ref{p.e.david} for further illustration.
For the reader's convenience, we define explicitly an equivalent forcing:

\begin{exm}\label{s.e.david}
Suppose $T$ is a tee from the canonical sequence of $\omega_2$-Suslin trees (cf.\ Definition~\ref{p.n.d.canonical}) and  
\[
B \colon [\omega_1, \omega_2) \to 2
\] 
represents an $(L,T)$-generic branch through $T$ in the same manner as in Example~\ref{p.e.david}.   
Assume $V=L[B]$, whence the cardinals of $V$ are precisely those of $L$.

Recall we have an almost disjoint family
$\mathcal B_\alpha = \langle b^\xi \setdef \xi \in [\alpha, \alpha^+)\rangle$
of subsets of $\alpha$, for each $\alpha \in \{ \omega, \omega_1\}$, such that 
$\{ (\zeta, \xi) \setdef \zeta \in b^\xi, \xi \in [\alpha, \alpha^+)\}$
is locally semidecidable (cf.\ Definition~\ref{p.d.locally.definable}).

Consider the partial order $P^*_D$ (equivalent to $R(B_n) * P_D(\dot B^*_n)$ from Example~\ref{p.e.david}), consisting of conditions $p = (p_0, p^*_\omega, p_{\omega}, p^*_{\omega_1})$ where
\begin{enumerate}
\item $p^*_{\omega_1}$ is a countable subset of $\mathcal B_{\omega_1}$;

\item $p_{\omega} \colon [\omega,\delta) \to 2$ for some $\delta < \omega_1$ which we denote by $| p_{\omega} |$;

\item\label{s.e.killing} $\forall \xi \in [\omega,|p_\omega|] \; \forall \alpha < \omega_1$ if $L_\alpha[p_\omega\res\xi]\vDash ZF^-$ together with ``$\Card^V=\Card^L $ and ${\omega_2}^{L_\xi}$ exists 
 and ${\omega_1}^{L_\xi} = \xi$''
 then $L_\alpha[p_\omega\res\xi]\vDash$``$T(n)^{L_\alpha}$ has a branch;''

\item $p^*_\omega$ is a finite subset of $\mathcal B_{\omega}$;
\item $p_0\colon n \to 2$ for some $n < \omega$ which we denote by $| p_0 |$.

\end{enumerate}
The ordering of $P^*_D$ is given by: $q \leq p$ if and only if for each $\alpha \in \{0,\omega\}$ (letting $0^+=\omega$ in the present context),
\begin{enumerate}
\item $q_{\alpha} \supseteq p_\alpha$,
\item $q^*_{\alpha^+} \supseteq p^*_{\alpha^+}$,
\item for any $b \in p^*_{\alpha^+}$ and $\nu \in b \cap [\lvert p_\alpha\rvert, \lvert q_\alpha\rvert)$, we have $q_\alpha(\nu) \leq B(\nu)$.
\end{enumerate}
\end{exm} 
Define for $q,p \in P^*_D$ that $q \mathbin{\leqlo^{\omega}} p$ if and only if
$p \leq q$, $p_0 = q_0$ and $p^*_\omega = q^*_\omega$.
It will become clear that a quasi-closure system involving $\leqlo^{\omega}$ is indeed useful. In particular, we shall see that this can be extended to a stratification system in Fact~\ref{s.f.jensen}.

 \begin{fct}\label{s.f.david.qc}
 The forcing $P^*_D$ is quasi-closed with respect to a preclosure system with $\leqlo^{\omega}$ as defined above.
 \end{fct}
 The proof is a special case of our proof that $P(A_0)$, defined in Chapter~\ref{sec:coding}, is quasi-closed. 
We give a rather detailed sketch.
 \begin{proof}[Proof sketch.]
We define $\D$ for $\lambda=\omega$:
Let $q \in \D(\omega, p, x)$ if and only if $q \in P^*_D$, $q \leqlo^{\omega} p$ and 
letting $H = h^{L[B]}_{\Sigma_1}( x)$ we have
\begin{enumerate}
\item $| q_\omega | > H\cap \omega_1$,
\item for all $\xi \in H \cap [\omega_1, \omega_2)$, $b^{\xi} \in q^*_{\omega_1}$,
\item  for all $\xi \in H \cap [\omega_1, \omega_2)$ there is $\zeta \in b^{\xi} \setminus |p_\omega|$ such that
$q_0(\zeta) = B(\xi)$.
\end{enumerate}
For $\lambda \in \Reg\setminus(\omega+1)$ define $\D(\lambda,p,x) = \{p\}$, and $p \leqlol p$ to mean $p=q$.

\medskip

To see that $P^*_D$ is quasi-closed, suppose $\bar p = (p^n)_{n< \omega}$ is $(\omega,x)$-adequate and fix a canonical witness and strategic guide $\bar w=(w^n)_{n< \omega}$.
Let
\[
X^n = h^{L[B]}_{\Sigma_1}( \bar w \res n+1\cup\{ x, \omega_2\})
\] 
and observe that $p^n \in X^n$ whence 
\[
| p^n_\omega | < \sup X^n \cap \omega_1 < |p^{n+1}_\omega |
\]
The last inequality holds since $p^{n+1} \in \D(\omega, p^n, \{ \bar w\res n+1 , x\})$.

Let $X = \bigcup_{n<\omega} X^n$ and note that $X \prec_{\Sigma_1} L[B]$.
Thus $\bar w$ is definable in $X$,
and hence also $\bar p$ is definable in $X$.

Let $\tilde X$ be the transitive collapse of $X$, let $\pi\colon  X \to \tilde X$ be the collapsing map
and write $\tilde X$ as $L_\mu[\tilde B]$. 
Further, let $\tilde p^n = \pi(p^n)$.
As the sequence $\bar p$ is definable in $X$,
 $(|\tilde p^n_\omega|)_{n< \omega}$ is definable in $L_\mu[\tilde B]$.

\medskip

To see that $\bar p$ has a greatest lower bound, it suffices to check that $p_\omega$ defined by 
\[
p_\omega = \bigcup_{n<\omega} p^n_\omega
\]
satisfies Clause~\ref{s.e.killing} in the definition of $P^*_D$ above. 
In fact, since each $p^n$ is a condition and $| p^n_\omega|$ is strictly increasing, it suffices to check
Clause~\ref{s.e.killing} for $\xi = | p_\omega |$.
So fix $\alpha$ as in Clause~\ref{s.e.killing}, i.e., so that $L_\alpha[p_\omega]\vDash | p_\omega | = \omega_1$.
As $(| \tilde p^n_\omega|)_{n< \omega}$ definably collapses $| p_\omega |$ in $L_\mu[\tilde B]$, $\alpha \leq \mu$.

Note that $\pi^{-1}(\xi) = \omega_1$. 
We leave it to the reader to check that because $\D$ was defined to emulate the coding of $B$ by $ p_\omega$ over Skolem hulls (cf.\ the proof of Theorem~\ref{jensen_qc_main}), $\tilde p_\omega = \bigcup_{n\in \omega} \tilde p^n_\omega$ codes $\tilde B$, and by local semidecidability of $\mathcal B_{\omega_1}$, $\tilde B \cap \Hhier(\omega_2)^{L_\alpha} \in L_\alpha[p_\omega]$. Likewise, by elementarity
$\pi^{-1}(\tilde B) = B$.  

Finally, $T'=\pi^{-1}(T^{L_\alpha})$  is an initial segment of $T$---in fact $T'=T$ if $\mu = \alpha$, and if $\mu < \alpha$,  $T'=T^{L_{\pi^{-1}(\alpha)}}$, which is an initial segment of $T$ since $T$ is locally semidecidable. 
So again by elementarity, $\tilde B \cap \Hhier(\omega_2)^{L_\alpha}$ is a branch through $T^{L_\alpha}$, verifying Clause~\ref{s.e.killing}.
\end{proof}
 
\subsection{A Word about Definability and Set Forcing}\label{rem:def:set}

The concept of quasi-closure is more natural  in a class forcing context.\index{class forcing}\index{forcing!class forcing} 
Since we only apply it for set forcing, we can make do with $\qcdefSeq(X)$ (as opposed taking into account $\Pi^T_n(X)$ for all $n > 1$).
We still have to use a form of $\Pi_1$-uniformization\index{uniformization}\index[notation]{Pi1-uniformization@$\Pi_1$-uniformization}, implicit in the construction of canonical witnesses (see \ref{adequate}). 
In class forcing, this uniformization can be
 achieved using the fact that conditions form a class, and the ``height'' of each condition in an adequate sequence effectively represents the canonical witness.
In the present application, we do not know a proof that $\kappa$ remains a cardinal using this approach instead of canonical witnesses.
For more details see \cite{friedman:codingbook}, Chapter~8.2, p.~175.

As has been mentioned, the presence of $\lambda$ and $\param$ will become clear when we prove quasi-closure is preserved under diagonal limits.
Insisting that, e.g., $c = (L_\mu[A], \theta, \hdots)$ be $\qcdefG(X)$ is perhaps somewhat analogous to the use of a large structure with predicates in the context of proper forcing.\index{proper forcing}\index{forcing!proper}
In our application we could let $\param=\{ L_{\kappa^{+++}}\}$, where $\kappa$ is the only Mahlo in $L$.
\section{Stratification}
\begin{dfn}\label{def:pss}\index{stratification|textbf}\index{forcing!stratified}
We say $\pss=\pssl$\index[notation]{S (prestratification system)@$\pss$ (prestratification system)}\index[notation]{lessthanu  lambda@$\lequpl$}\index[notation]{C lambda@$\Cl$|textbf} is a prestratification system\index{prestratification system} for $R$ on $I$ if and only if
$\pcsl$ is a preclosure system for $R$ on $I$ and for every $\lambda \in I$ the following conditions are met:
\begin{enumerate}[label=(S \arabic*), ref=S~\arabic*]

\item  \label{up:extra} The binary relation $\lequpl$ on $R$ satisfies $p \leq q\Rightarrow p \lequpl q$.
\item  \label{up}  If $p \leq q \lequpl r$ then $p \lequpl r$.\footnote{Note that we don't assume $\lequpl$ to be transitive, since this does not seem to be preserved by composition. If $\lequpl$ were transitive, Condition~(\ref{up}) would follow from (\ref{up:extra}). 
We need (\ref{up}) for Lemma \ref{density reduction}. We need that $\lequpl$ is reflexive (i.e., $p \lequpl p$ for all $p$) for \ref{pss:is}(\ref{pss:is:cdot:up}). In the context of (\ref{up}), reflexivity is the same as the last part of (\ref{up:extra}).}

\item \label{s:lequp:vert} If $\lambda \leq \bar \lambda$ and $\bar \lambda \in I$ then $p \lequpl q \Rightarrow p\lequp^{\bar\lambda} q$.

\item\label{density}
      \emph{Density:}\index{Density (stratification)} $\Cl \subseteq R \times \lambda$ is a binary relation such that $\dom(\Clink^\lambda)$ is dense in $R$. Moreover, for any 
                        $\lambda' \in I \cap \lambda$ and $p \in R$, there is $q \leqlo^{\lambda'} p$ such that $q \in \dom(\Clink^\lambda)$.

 \item\label{continuous}
      \emph{Continuity}:\index{Continuity (stratification)} If $\lambda'\in I\cap\lambda$ and $p$ is a greatest lower bound of the
                           $\lambda'$-adequate sequence $\bar p=(p_\xi)_{\xi<\rho}$ and for each $\xi<\rho$, $p_\xi \in \dom(\Cl)$, \emph{then} $ p \in \dom(\Cl)$.\footnote{In any known application, we could ask this for all $\lambda'\in I$, not just those smaller than $\lambda$.} 
                           If in addition $\bar q$ is another $\lambda'$-adequate sequence of length $\rho$ with greatest lower bound $q$ and for each $\xi <\rho$,
                           $\Cl(p_\xi)\cap\Cl(q_\xi)\neq\emptyset$, then 
                           $\Cl(p)\cap\Cl(q)\neq\emptyset$.
           
\end{enumerate}
\end{dfn}
The last part of Condition~(\ref{up:extra}), all of (\ref{s:lequp:vert}) and the ``moreover'' part of (\ref{density}) can be dropped if one is not interested in infinite iterations.
We \emph{do not} require that $\lequpl$ be a preorder; but (\ref{up}) does guarantee some regularity with respect to $\approx$.

\begin{dfn}\label{stratified:main}
We say a preorder  $\langle R, \leq \rangle$ is \emph{stratified on $I$} as witnessed by $\pss=\pssl$ if and only if $\pss$ is a prestratification system on $I$, $\langle R, \leqlol, \D \rangle_{\lambda \in I} $ is quasi-closed, and for each $\lambda\in I$
the following conditions hold:
\begin{enumerate}[label=(S \Roman*), ref=S~\Roman*]

     \item\label{exp} \emph{Expansion}:\index{Expansion (stratification)} If  $p \lequpl d$ and $d \leqlol 1_R$, then in fact $p \leq d$.

      \item \label{interpolation}
      \emph{Interpolation}:\index{Interpolation (stratification)} If $d \leq r$, there is $p \leqlo^\lambda r$ such that $p \lequp^\lambda d$. 
                            In addition, whenever $\bar \lambda\in I$ and $d \leqlo^{\bar\lambda} 1_R$, then also $p 									                               \leqlo^{\bar\lambda} 1_R$.

      \item \label{linking} 
      \emph{Linking}:\index{Linking (stratification)}  If $p \lequpl d$ and $\Cl(p)\cap\Cl(d)\neq\emptyset$ then $p$ and $d$ are compatible. In fact, there is $w$ 		        such that for any 
$\lambda' \in I \cap\lambda$, $w \leqlo^{\lambda'} p$ and  $w \leqlo^{\lambda'} d$.

\end{enumerate}
If we omit $\pss$ and no prestratification system can be deduced from the context, we mean that there exists a  prestratification system $\pss$ witnessing that $R$ is stratified.
When we say \emph{stratified on $[\lambda_0, \lambda_1)$}, we mean of course
stratified on $[\lambda_0, \lambda_1)\cap\Reg$.
\end{dfn}
Conditions~(\ref{continuous}) and (\ref{exp}) are important to preserve stratification in (infinite) iterations. 
The second part of (\ref{linking}) was introduced to allow for amalgamation (see Chapter \ref{sec:amalgamation}), but is also useful to control the diagonal support in iterations (see below).

We illustrate Definition~\ref{stratified:main} with some examples.
\begin{exm}\label{exm:centered}
A simple observation is that for any preorder $R$, $R$ is stratified above $\card{R}$. 
A little more generally, if $R$ is $\lambda_0$-linked, then $R$ is stratified on $[\lambda_0,\infty)$: 
for if $\lambda \geq \lambda_0$, we can simply define $p \leqlol q$ to mean $p=q$. Similarly, $\D(\lambda,x,p)=R$ for all $p \in R$. Thus, quasi-closure and \emph{Continuity} become vacuous.
Moreover, let $g:R\rightarrow \lambda_0$ be a a function such that if $g(p)=g(q)$ then $p$ and $q$ are compatible.
Set $\Cl(p)=\{ g(p) \}$ for any $p \in R$. Lastly, define $p \lequpl q$ to hold for \emph{ any } pair $p,q$.
Then the only non-vacuous condition in the definition of stratification is \emph{Linking}, which holds for every $\lambda\geq \lambda_0$ since $g$ witnessed that $R$ was linked.
\end{exm}

This example has a corollary:
\begin{cor}\label{lambda:is:big}
If a preorder $R$ is stratified, we can always assume that for $\lambda \geq \card{R}$,
$\D$, $\leqlol$, $\lequpl$ and $\Cl$ take the simple form discussed above in example \ref{exm:centered}.
\end{cor}

Observe that (\ref{qc:leqlo:vert}) and (\ref{s:lequp:vert}) remain valid if we modify a given prestratification system in such a way as to ensure that the above assumption holds.
A more interesting example: 
\begin{exm}\label{strat:ex:*}
Say $R=P*\dot Q$ where $P$ is $(\lambda_0)^+$-linked and $(\lambda_0)^+$-closed and $\forces_P \dot Q$ is $\lambda_0$-linked and $\lambda_0$-closed. Then $R$ is stratified on $\Reg$ --- ignoring (\ref{continuous}). If the linking functions for $P$ and $\dot Q_\xi$ in the extension are continuous in the sense of (\ref{continuous})---and it seems that for many linked forcings, this is the case---$R$ is actually stratified.

Define $\D$ as in the previous example.
For $\lambda < \lambda_0$, define $\leqlol$ to be identical to $\leq_R$;
define $p \lequpl q$ if and only if $p = q$ and let $\Cl(p)=\lambda$ for every $p\in R$. Then \emph{Interpolation} and \emph{Linking} hold at $\lambda$ for trivial reasons, and quasi-closure at $\lambda$ expresses the fact that $R$ is closed under sequences of length at most $\lambda$.
For $\lambda = \lambda_0$, fix a name for a linking function $\dot g$; set $(p,\dot q) \leqlol (p', \dot q')$ if and only if $(p,\dot q) \leq (p', \dot q')$ and $\dot q = \dot q'$; set $(p,\dot q) \lequpl (p', \dot q')$ if and only if $p \leq_P p'$. 
Let $\chi \in\Cl(p,\dot q)$ if and only if $p \forces \dot g(\dot q)=\check\chi$.
Lastly, $R$ has a subset $R'$ which is $(\lambda_0)^+$-linked and $\leqlo^{\lambda_0}$-dense. This allows us to define a stratification above $(\lambda_0)^+$, in a similar way to the previous example.
\end{exm}

A less trivial example is the forcing $P^*_D$ from Example \ref{s.e.david}~\ref{p.e.david} (or from Example~\ref{s.e.david}). 
 \begin{fct}\label{s.f.jensen}
 The forcing $P^*_D$ is stratified.
 \end{fct}
 \begin{proof}
The family of sets $D(\lambda, p, x )$ as well as $\leqlol$ was defined in Fact \ref{s.f.david.qc}.
 In addition, for each $p,q\in P^*_D$ we let 
 \begin{equation*} 
q \lequp^\omega p                \iff    \big( q_\omega \supseteq p_\omega \wedge q^*_{\omega_1} \supseteq p^*_{\omega_1} \big)
 \end{equation*} 
and for each $\lambda > \omega$, $q \lequp^\lambda p$ for any $p,q\in P^*_D$.
Lastly let $\Clink^\omega( p) = \{p_0\}$ and $\Clink^{\lambda} = \{ (p_0, p^*_\omega, p_\omega) \}$ for any $\lambda \geq \omega_1$.
We leave it to the reader to prove that this is a stratification system.
\end{proof}

Finally, we can discuss preservation of cofinalities and the $\GCH$.\index{GCH, preservation of}\index{cofinalities, preservation of}\index{preservation of GCH, cofinality}
\begin{thm}\label{s.t.preservation.cofinalities}
For any $\lambda \in I$, cofinalities greater than $\lambda$ remain greater than $\lambda$ after forcing with $R$ and $(2^\lambda)^V = (2^\lambda)^{V[G]}$ for any $R$-generic $V$.
\end{thm}

The proof will stretch across several lemmas. 
For it, we introduce following terminology (which will be convenient throughout this memoir).
\begin{dfn}\label{chromatic}
In the following, we fix a regular cardinal $\lambda\in I$ and drop the superscripts on $\Clink, \leqlo$ and $\lequp$.
\begin{enumerate}
\item Let $D, D^* \subseteq R$, and $r\in R$.
We say \emph{$r$ $\lambda$-reduces  $D$ to $D^*$}\index[notation]{lambda-reduces (a dense set)@$\lambda$-reduces (a dense set)}\index{reduction (stratification)} (often, we don't mention the prefix $\lambda$) exactly if 
\begin{enumerate}
\item $D^* \subseteq  \dom(\Clink)\cap D$ and $\card{D^*}\leq\lambda$;
\item for each  $d\in D^*$, $r \lequp d$;
\item for any $q \in \dom(\Clink)\cap D$, if $q\leq r$, there is $d \in D^*$ such that $\Clink(q)\cap\Clink(d)\neq 0$.
\end{enumerate}
\item %
Let  $\dot \alpha$ be a name for an element of the ground model, and let $r \in R$. 
We say $\dot \alpha$ is $\lambda$-chromatic\index[notation]{lambda-chromatic name@$\lambda$-chromatic name}\index{chromatic name} below $r$ just if there is a function $H$ with $\dom H \subseteq\lambda$ such that if $q \leq r$ decides $\dot \alpha$ and $q \in \dom(\Clink)$, 
then $\Clink(q)\cap \dom(H)\neq 0$ and for all $\chi \in \Clink(q)\cap \dom(H)$, $q \forces \dot \alpha = H(\chi)$ (to be pedantically precise, we mean the ``standard name'' for $H(\chi)$). 
We call such $H$ a $\lambda$-\emph{spectrum (of $\dot \alpha$)}.\index[notation]{lambda-spectrum@$\lambda$-spectrum}\index{spectrum}

\item If $\dot s$ is a name and $p \forces \dot s \colon \lambda \rightarrow  V$, then we say \emph{ $\dot s$ is $\lambda$-chromatic (with $\lambda$-spectrum $(H_\xi)_{\xi<\lambda}$) below $p$} if and only if for each $\xi<\lambda$, $\dot s(\check \xi)$ is chromatic with spectrum $H_\xi$ below $p$.

For notational convenience, we say $\dot x$ is $0$-chromatic\index[notation]{0-chromatic@$0$-chromatic} below $p$ if for some $x$, $p \forces_R \dot x =\check{x}$.
\end{enumerate}
Observe that if for some ground model set $x$, $p \forces \dot \alpha = \check x$ (i.e., $\alpha$ is $0$-chromatic), then $\dot \alpha$ is in fact $\lambda$-chromatic for every regular $\lambda$,
and the function with domain $\lambda$ and constant value $x$ is a $\lambda$-spectrum.
\end{dfn}
To illustrate $\lambda$-reduction, observe that if $r$ $\lambda$-reduces a dense open set $D$ to $D^*$, then $D^*$ is predense below $r$:
for if $q\leq r$, we can assume that $q \in \dom(\Cl)\cap D$, so that there is $d \in D^*$ with $\Cl(d)\cap\Cl(q)\neq \emptyset$. By \eqref{up}, $q \lequpl d$ and so by \emph{Linking} \eqref{linking}, $q$ and $d$ are compatible.

For the remainder of this section let us suppose $R$is stratified on $I$ and $\lambda\in I$.
\begin{thm}\label{single xdensity reduction}
Let $D$ be a dense open subset of $R$ and $p \in D$. Then there is $r \leqlol p$ such that
$r$ $\lambda$-reduces $D$ to some $D^*$.
\end{thm}
\begin{proof}
We build an adequate sequence $\bar p = (p_0)_{\xi \leq \lambda}$, starting with $p_0=p$. 
We shall first sketch a construction such that $r = p_\lambda$ is the desired condition, without specifying a canonical witness; then we argue how this construction can be carried out so as to obtain a canonical witness at the same time.
This will serve as a blueprint for later constructions, where we shall not explicitly carry out the construction of $\bar w$, as it is entirely analogous to the case at hand.

Let $x = (p_0, \leqlol, \lequpl, \Cl, R, \lambda, c)$, where $c$ is the parameter in the formula representing $\D$.
Say we have constructed $\bar w \res \nu$ and $\bar p \res \nu$. If $0<\nu \leq \lambda$ and $\nu$ is a limit ordinal, assume by induction that $\bar p \res \nu$ is adequate and let $p_\nu$ be a greatest lower bound.
If $\nu = \xi +1$ we will choose, in a manner yet to be specified, $w_{\xi+1}$, $p_{\xi+1}$, $p^*_\xi$ and $d_\xi$ which satisfy the following:
\begin{enumerate}
\item \label{build_ad_red_strategic}$p^*_\xi \in \D(\lambda,(x, \bar w\res\xi+1),p_\xi)$, and $p_{\xi+1}\leqlol p^*_\xi$.
\item \label{build_ad_red}
\begin{enumerate}
\item If there is no $d^* \leq p^*_\xi$ such that $d^* \in D$ and $\xi \in \Cl(d^*)$ we demand $p_{\xi+1}=p^*_\xi$; 
\item \label{work} else $p_{\xi+1}$ and $d_\xi$ satisfy: $d_\xi \leq p^*_\xi$, $d_\xi \in D$ and $\xi \in \Cl(d_\xi)$; moreover $p_{\xi+1} \leqlo p^*_\xi$ and $p_{\xi+1} \lequpl d_\xi$.
\end{enumerate}
\end{enumerate}
Observe this does not mention $w_{\xi+1}$; its role will be explained by Lemma \ref{adequate}, below.
If the first alternative of Item~\ref{build_ad_red} obtains, the choice of $d_\xi$ is completely irrelevant for the rest of the construction.
It is clear that $\bar p$ will have $\bar w$ as a strategic guide: 
since $p_{\xi+1} \leq p^*_\xi$, by Item~\ref{build_ad_red_strategic} and \eqref{qc:D}, $p_{\xi+1}\in\D(\lambda,x,p_\xi)$ and $p_{\xi+1}\leqlol p_\xi$.

Let $D^* = \{ d_\xi \setdef  \xi < \rho\text{ and \ref{work} obtained at stage $\xi+1$ } \}$. 
Clearly, $p_\lambda$ reduces $D$ to $D^*$.
For if $q \leq p_\lambda$, $q \in D$ and $\xi \in \Cl(q)$, then $q$ witnesses that at step $\xi+1$ of the construction of $\bar p$, \ref{work} obtained. 
So we have $d_\xi \in D$ with $\xi \in \Cl(d_\xi)$. 
Moreover, $q\leq p_\lambda \lequpl d_\xi$.
In order to conclude that this construction works, we need to show that for every $\nu \leq \lambda$, $\bar p \res \nu$ is $(\lambda,x)$-adequate. For this it is enough to explain how precisely we made our choices in the above construction:
\begin{lem}\label{adequate}
In the previous, $w_{\xi+1}$, $p_{\xi+1}$, $p^*_\xi$ and $d_\xi$ can be chosen so that $\bar w$ is a $(\lambda, x)$-canonical witness for $\bar p$.
\end{lem}
\begin{proof}
Observe that at successor stages, we promised that $p_{\xi+1}$, $p^*_\xi$ and $d_\xi$ are chosen so as to satisfy a  property which may be expressed by a $\qcdefSeq$ formula $\Phi_0$ as follows:
Let $\Phi_0(x,p,p^*,d, \tilde w)$ be the $\qcdefSeq(\{x\})$ formula  expressing that $\tilde w$ is a sequence and if $\dom(\tilde w)$ is a successor, then, 
for $\xi = \dom(\tilde w) -1$ it holds that
\begin{enumerate}
\item $p^* \in \D(\lambda,(x, \tilde w),p)$, and $p\leqlol p^*$;
\item 
\begin{enumerate}
\item if there is no $d^* \leq p^*$ such that $d^* \in D$ and $\xi \in \Cl(d^*)$, then $p=p^*$;
\item else $p$ and $d$ satisfy: $d \leq p^*$, $d \in D$ and $\xi \in \Cl(d)$; moreover $p \leqlo p^*$ and $p \lequpl d$.
\end{enumerate}
\end{enumerate}

In what follows, the choice of $\bar w$ will be such that for each $\xi < \lambda$,
$w_\xi$ is a quintuple, and $\bar p$ is obtained from $\bar w$ by projecting to the first coordinate.
In fact, we will have $w_{\xi+1} = (p_\xi, M, p^*_\xi, d_\xi, \bar w \res \xi+1)$ where $M$ is a model allowing us to uniformise $\Phi_0$ and the rest of the coordinates are as in the successor step of the above construction. At limits $\nu$ we shall have
$w_\nu = (p_\nu, M, p^*, d, \bar w \res \nu)$, where $p^*$ and $d$ are just dummies.
Using the initial segment $\bar w\res \xi$ as a last coordinate in $w_\xi$ ensures that $\bar w$ is $\qcdefSeq(\{x\})$ in the end (see below).

Let $\Phi_1(p,p^*,d,\tilde{w})$ be the formula expressing that if $\nu = \dom (\tilde{w})$ is a limit then $p$ is the greatest lower bound in $R$ of the sequence obtained from $\tilde{w}$ by projecting to the first coordinate (and $p^*=d=\emptyset$, if you want). %

Now let $\Phi(w, \tilde{w})$ be the $\qcdefSeq(\{x\})$ formula expressing: $w = (p,M,p^*,d,\tilde{w})$ is such that
\begin{enumerate}
\item $M$ is transitive, $M  \prec_{\Sigma_1} L[A]$ and $p,p^*,d,\tilde{w} \in M$,
\item for any initial segment of $M$ containing $p,p^*,d,\bar w$ we have $N\not\prec_{\Sigma_1}  M$,
\item $M \vDash (p,p^*,d)$ is $\leq_{L[A}$-least such that $\Phi_1(p,p^*,d, \tilde{w})$ and $\Phi_0(x,p,p^*,d, \tilde{w})$.
\end{enumerate}

We may choose $\bar w$ recursively such that for each $\xi < \lambda$, $\Phi(w_\xi, \bar w\res\xi)$ holds:
Observe that if such $w_\xi$ exists, it is unique.
For limit $\nu$, we may assume by induction that $\bar w\res \nu$ is a canonical witness for $\bar p\res\nu$, allowing us to infer $\bar p\res\nu$ is adequate and that $p_\nu$ and hence $w_\nu$ exists. At successor stages $\xi+1$, $w_{\xi+1}$ always exists, and so $\bar w$ is well-defined.

Lastly, we show $\bar w$ is $\qcdefSeq(\{x\})$: this is because $w = w_\xi$ if and only if $w$ is a quintuple with last coordinate $\bar w$ such that $\Phi(w, \bar w)$ holds, 
and $\bar w$ is a sequence such that for each $\xi \in \dom(\bar w)$, $\Phi(\bar w (\xi), \bar w \res\xi)$.
This finishes the proof of the lemma. 
\end{proof}
Having shown that $\bar p$ may be chosen as a $(\lambda,x)$-adequate sequence, we are finished with the proof of the theorem.
\end{proof}

\begin{lem}\label{density reduction}
For each $\xi<\lambda$, let $D_\xi$ be an open dense subset of $R$. Let 
\[
X = \{q \in R \setdef  \exists D^*   \; \forall \xi<\lambda \;\text{s.t.\ $q$ $\lambda$-reduces $D_\xi$ to $D^*$}\}.
\]
Then $X$ is dense in $\langle R, \leqlol \rangle$ and open in $\langle R, \leq \rangle$.

If $\dot s$ is a name such that $p \forces \dot s \colon \lambda \rightarrow  V$, the set of $q$ such that $\dot s$ is $\lambda$-chromatic below $q$ is dense in $\langle R(\leq p), \leqlol \rangle$ and open.
\end{lem}
\begin{proof}
Let $\bar D=(D_\xi)_{\xi<\lambda}$ be a sequence of dense open subsets of $R$. Build a sequence as before:
let 
\[ 
x = \langle p_0, \leqlol, \lequpl, \Cl,  R, \pset(R), \bar D, y \rangle. 
\]
At successor steps $\xi$, choose $p_\xi$ and $w_\xi$ such that $p_\xi \leqlol p_{\xi-1}$ and 
\[
p_\xi \in \D(\lambda,(x, \bar w\res\xi+1),p_{\xi-1})
\]
and such that we can pick $D^*_\xi$ such that $p_\xi$ reduces $D_\xi$ to $D^*_\xi$.
As in Lemma \ref{adequate}, argue this can be done in such a way that the resulting sequence $\bar p$ is $(\lambda,x)$-adequate.
So a greatest lower bound $p_\lambda$ exists and for each $\xi<\lambda$, $p_\lambda$ reduces $D_\xi$ to $\bigcup_{\xi<\lambda} D^*_\xi$.

Now let $p \forces \dot s \colon \lambda \rightarrow  V$. Let $D_\xi$ be the set of conditions $p \in R$ which decide $\dot f(\xi)$.
As above, find $q$ reducing all $D_\xi$ to $D^*$. 
We now find a spectrum for $\dot s$: For $\chi < \lambda$, if $w \leq q$ decides $\dot s(\xi)$ and $\chi \in \Cl(w)$, there is also $d \in D^*$ which decides $\dot s(\xi)$ and such that $\chi \in \Cl(d)$. 
Fix $z$ such that $d \forces f(\xi) = \check z$.
Then we may set $H_\xi(\chi)=z$.
It is easy to check that for each $\xi<\lambda$, $H_\xi$ is a spectrum for $\dot s(\xi)$ (and thus $(H_\xi)_{\xi<\lambda}$ is a spectrum for $\dot s$): 
Say $w \leq q$ decides $\dot s(\xi)$ and fix some $\chi \in \Cl(w)$. 
Then there is $d \in D^*$ with $\chi \in \Cl(d)$ such that $d \forces s(\xi) = H_\xi(\chi)$. 
As $d \in D^*$, $q \lequpl d$. So as $\chi \in \Cl(w)\cap\Cl(d)$, $w$ and $d$ are compatible and thus $w \forces s(\xi) = H_\xi(\chi)$.
\end{proof}
\begin{cor}
Theorem~\ref{s.t.preservation.cofinalities} holds, i.e., for any $\lambda \in I$,
cofinalities greater than $\lambda$ remain greater than $\lambda$ after forcing with $R$ and $(2^\lambda)^V = (2^\lambda)^{V[G]}$ for any $R$-generic $V$.
\end{cor}

\section{Composition of Stratified Forcing}

In the main theorem of this section, Theorem \ref{stratified:composition} below, we show stratification is preserved by composition. In the proof, we use ``guessing systems'',\index{guessing system} which we shall motivate now, before we state and prove the theorem.

Say $P$ is stratified and $\dot Q$ is forced by $P$ to be stratified on $I$, and let $\lambda\in I$ be fixed.
We know $P$ has a linking relation $\Clink$ and $\dot Q$  is forced to have a linking relation $\dot \Clink$ in the extension.
Similar to the proof that composition of linked forcing stays linked, we want to gain some control over $\dot \Clink$ in the ground model.
If we ignore the requirements \ref{def:pss}(\ref{density}) \emph{density} and \ref{stratified:main}(\ref{continuous}) \emph{Continuity},
we could define $\bar \Clink$ on $P*\dot Q$ in the following way:
\[ (\chi,\xi) \in \bar \Clink(d,\dot d) \iff \chi \in \Clink(d)\text{ and }d \forces \check \xi \in \dot \Clink(\dot d) \]
Then $\dom(\bar\Clink)$ is dense and \ref{stratified:main}(\ref{linking}) \emph{Linking} holds.

The following definition also satisfies \ref{def:pss}(\ref{density}) \emph{density}: let
\begin{multline}\label{C:second:attempt}
(\chi,X) \in \bar \Clink(d,\dot d) \iff \Big( \chi \in \Clink(d)\text{ and for some $\dot \xi$ and }\lambda' \in\Reg\cap[\lambda_0,\lambda), \\
\text{ $X$ is a $\lambda'$-spectrum for $\dot \xi$ below $d$ and }
d \forces \dot \xi \in \dot \Clink (\dot d) \Big).
\end{multline}
Let's check \ref{def:pss}(\ref{density}) \emph{density} holds:
Given a condition $\bar p= (p,\dot p)$  and $\lambda'\in\Reg\cap[\lambda_0,\lambda)$, we can find $\bar d=(d,\dot d)$
such that $\bar d \alo^{\lambda'}\bar p$, $d \in \dom(\Clink)$ and $d \forces \dot d \in \dom(\dot \Clink)$. Moreover, we can assume that for some name $\dot \chi$, $d\forces \dot \chi \in \dot \Clink(\dot d)$ and $\dot \chi$ is $\lambda'$-chromatic. We have $\bar d \in \dom(\bar \Clink)$.
Let's also check that \ref{stratified:main}(\ref{linking}) \emph{Linking} holds: say $\bar d \aupl \bar p$ and $(\chi,X)\in\bar\Clink(\bar p)\cap \bar\Clink(\bar d)$. First, observe that $p$ and $d$ are compatible. Fix $\dot \chi_0$, $\dot\chi_1$ such that both $p\forces\dot \chi_0 \in \dot \Clink(\dot p)$  and 
$d\forces\dot \chi_1 \in \dot \Clink(\dot d)$
and $X$ is a spectrum for $\dot \chi_0$ below $p$ and for $\dot \chi_1$ below $d$. 
As $p \cdot d \forces \dot \chi_0 = \dot \chi_1 \in \dot \Clink(\dot d) \cap \dot \Clink(\dot p)$, by \emph{Linking} for $\dot Q$ in the extension,
 $p \cdot d \forces \dot d$ and $\dot p$ are compatible, whence $\bar d$ and $\bar p$ are compatible.

To show stratification is preserved at limits, we will have to use \emph{Continuity} of the linking function; 
Unfortunately, the approach described above does not yield a \emph{continuous} linking function in the sense of (\ref{continuous}). 
For say $(d_\xi, \dot d_\xi)_{\xi<\rho}$ is a $\lambda'$-adequate sequence, and for each $\xi<\rho$,
$d_\xi\forces \dot \chi_\xi \in \dot \Clink(\dot d_\xi)$ and $X_\xi$ is a $\lambda_\xi$-spectrum for $\dot \chi_\xi$ below $d_\xi$. 
By \emph{Continuity} for the components of the forcing,
if $(d,\dot d)$ is a greatest lower bound, we know $d \forces \dot d \in \dom(\dot\Clink)$; 
but there is no reason to assume that there exists a $P$-name $\dot \gamma$, such that $d \forces \dot \gamma \in \dot\Clink(\dot d)$ and $\dot \gamma$ is $\lambda''$-chromatic for some $\lambda'' < \lambda$.

The solution to this problem is to allow a more general set of values for $\bar C(\bar d)$: in the situation described above, e.g. the sequence $(X_\xi)_{\xi<\rho}$ can be used in much the same way as the single spectrum $X$.
This leads to the notion of a guessing system, which will be precisely defined in \ref{gs}. 
 
\begin{thm}\label{stratified:composition}\index{composition (forcing)!stratification}\index{stratification!for composition}
Say $P$ is stratified on $I$ and $\dot Q$ is forced by $P$ to be stratified on $I$. Then $\bar P=P * \dot Q$ is stratified (on $I$).
\end{thm}
\begin{proof}
Say stratification of $P$ is witnessed by $\D, \Cl, \leqlol$, $\lequpl$ for each regular $\lambda\in I$, and we have class $\dot \D$, definable with parameter $\dot c$ and names $\dot\Clink^\lambda, \dot\leqlo^\lambda$, $\dlequpl$ for $\lambda$ regular which are forced by $P$ to witness the stratification of $\dot Q$. We now define  $\bar \D, \bCl, \bleqlol$ and $\blequpl$ for regular $\lambda \geq \lambda_0$ to witness stratification of $P * \dot Q$.

\medskip

\textbf{The auxiliary orderings:}

Let $\lambda \in I$ be regular.
We say $(p,\dot q)\bar \leqlo^\lambda(u,\dot v)$ if and only if $p \leqlo^\lambda u$ and $p\forces_P \dot q \dot \leqlo^\lambda \dot v$.
This defines a preorder stronger than the natural ordering on $P*\dot Q$ (i.e., \ref{def:pcs}(\ref{qc:preorder}) holds).
Define  $\bar p \blequpl \bar q$ if and only if $p \lequpl q$ and if $p \cdot q \neq 0$,  $p \cdot q \forces_P \dot p \lequpl \dot q$.

\medskip

\textbf{The ordering axioms:}

Let $\bar p=(p,\dot p)$, $\bar q=(q,\dot q)$ and $\bar r=(r,\dot r)$ be conditions in $\bar P$.

We check that \ref{def:pcs}(\ref{er}) holds: Say $\bar p\leq_{\bar P} \bar q \leq_{\bar P} \bar r$ and $\bar p \bleqlol \bar r$.
Then $p \leqlol r$ by \ref{def:pcs}(\ref{er}) for $P$. Moreover, $p$ forces \ref{def:pcs}(\ref{er}) for $\dot Q$ as well as $\dot p\leq\dot q \leq\dot r$ and $\dot p \dot\leqlo^\lambda \dot r$. So $p \forces_P \dot p \dot\leqlol \dot q$, and we conclude $\bar p \bleqlol \bar q$.

Check that \ref{def:pss}(\ref{up}) holds:
Say $\bar p \leq_{\bar P} \bar q \blequpl \bar r$. By (\ref{up}) for $P$, $p \lequpl r$. If $p \cdot r \neq 0$,  
\[ p \cdot r \forces_P \dot p \leq_{\dot Q} \dot  q \lequpl \dot r, \]
and so $p \cdot r\forces_P \dot p \lequpl \dot r$. Thus $\bar p \blequpl \bar r$.

Next, check \ref{def:pss}(\ref{exp}). Say $\bar p \blequpl \bar q$ and $\bar q \bleqlol 1_{\bar P}$. By (\ref{exp}) for $P$, $p \leq q$.
So $p \forces \dot p \dlequpl q \dleqlol 1_{\dot Q}$, so by (\ref{exp}) applied in the extension, 
$p \forces \dot p \leq_{\dot Q} \dot q$, whence $\bar p \leq_{\bar P} \bar q$.
We leave it to the reader to check \ref{def:pcs}(\ref{qc:leqlo:vert}) and \ref{stratified:main}(\ref{s:lequp:vert}). 

\medskip

\textbf{Quasi-Closure:}\index{composition (forcing)!quasi-closure}

Define $(p, \dot p) \in \bar \D( \lambda, x, (q, \dot q))$ if and only if $p \in \D(\lambda,x,q)$ and $p\forces \dot p \in \dot \D( \lambda, x, \dot q)$.
It is straightforward to see that this definition is $\qcdefD( (\param, \dot c))$.
Clearly, this defines a ``dense and open'' set, i.e., (\ref{qc:D}) and \eqref{qc:redundant} are  satisfied.

Now say $(p_\xi, \dot q_\xi)_{\xi<\rho}$ is $(\lambda,x)$-adequate. We show this sequence has a greatest lower bound. Let $\bar w$ be a strategic guide and a canonical witness for $(p_\xi, \dot q_\xi)_{\xi<\rho}$.
We can immediately infer by the definition of $\bar \D$ that $\bar w$ is a strategic guide for $(p_\xi)_{\xi<\rho}$.
It is also clear that $\bar w$ is a canonical witness, since $p_\xi$ can be obtained from $(p_\xi, \dot q_\xi)$ by projecting to the first coordinate, and this map is $\Delta_0$.
Thus there is a greatest lower bound $p_\rho$ of $(p_\xi)_{\xi<\rho}$.

It is easy to see now that $p_\rho\forces_P$``$(\dot q_\xi)_{\xi<\rho}$ is $(\lambda,x)$-adequate'': 
Fix a $\qcdefSeq(\lambda\cup\{x\})$ formula $\Phi(x,y)$ such that for $\xi<\rho$ we have
\[ \Phi(x,\xi) \iff x=w_\xi. \]
Then the relativization $\Phi(x,y)^{L[A]}$ witnesses that $1_P$ forces that $\bar w$ is $\Pi_1(\lambda\cup\{x\})$ in the extension by $P$, as well.
For a $\qcdefG(\lambda\cup\{x\})$-function $G$,
$(p_\xi, \dot q_\xi)=G(\bar w\res\xi+1)$ for each $\xi<\rho$,  and  so
$\dot q_\xi = \pi_1 (G (\bar w\res\xi+1))$, where $\pi_1$ is the projection to the second coordinate.
As $G(x)=y$ is $\qcdefG$, clearly the relativized formula $(\pi_1(G(x))=y)^{L[A]}$ is also $\qcdefG(\lambda\cup\{x\})$ in the extension by $P$.
So $\bar w$ is forced to be canonical witness. 
Moreover, it is clear that $p_\rho$ forces that
$\bar w$ is a strategic guide for $(\dot q_\xi)_{\xi<\rho}$, by the definition of $\bar \D$.

So we can find $\dot q_\rho$ such that $p_\rho \forces_P$``$\dot q_\rho$ is a greatest lower bound of $(\dot q_\xi)_{\xi<\rho}$'', whence $(p_\rho, \dot q_\rho)$ is a greatest lower bound of the original sequence.
Leaving the last sentence of (\ref{qc:glb}) to the reader, we conclude that $\langle P*\dot Q, \bar \lequpl, (\param, \dot c), \bar \D\rangle$ is $\lambda$-quasi-closed above on $I$.

To define $\bCl$, we first define the notion of a \emph{guessing system}.
Roughly speaking, a guessing system consists of conditions which are organized in levels; the conditions on the bottom have a $\bCl$-value in the sense of \eqref{C:second:attempt}. Conditions on higher levels are greatest lower bounds of conditions on the levels below, and we have some control over their $\dot\Clink^\lambda$-value by \emph{Continuity} for $\dot Q$. 
\begin{dfn}\label{gs}\index{guessing system|textbf}
Say $(p,\dot q)\in P*\dot Q$ and $\lambda$ is regular and uncountable.
A \emph{$\lambda$-guessing system for $\dot q$ below $p$} is a quadruple $(\gT,\gH, \gl, \gq)$ such that
\begin{enumerate}[label=(\Alph*), ref=(\Alph*)]
\item $\gT$ is a tree, $\gT \subseteq {}^{<\omega}\gamma$, where $\gamma=\width(T) <\lambda$ and $\is$ (initial segment) is reversely well founded on $\gT$.
The root of $\gT$ is $\emptyset$ (i.e., the empty sequence).
\item For $s \in \gT$,  $\gr(s) = \{\xi \setdef s\conc\xi \in \gT \}$
is an ordinal. Write $\gT^0$ for the set of $\is$-maximal $s \in \gT$, i.e.,
$\gT^0 = \{ s \in \gT \setdef \gr(s)=0 \}$. 
\item $\gq$ is a function from $\gT$ into the set of $P$-names for conditions in $\dot Q$ and $\gl\colon \gT \rightarrow \lambda\cap (I\cup \{0\})$. 
\item\label{approximate} For $s \in \gT \setminus \gT^0$, $\gl(s) \in I$ and $\{\gq(s\conc\xi)\}_{\xi<\gr(s)}$ is a $\gl(s)$-adequate sequence and $p$ forces that
$\gq(s)$ is a greatest lower bound of 
\[
\{\gq(s\conc\xi)\}_{\xi<\gr(s)}.
\]
\item $\dom(\gH)=\gT^0$. 
\item \label{spectrum} For $s \in \gT^0$, there is a $P$-name $\dot \chi$ such that
$p \forces_P \dot \chi \in \Cl(\gq(s))$
and $\gH(s)$ is a $\gl(s)$-spectrum of $\dot \chi$ below $p$.
\item $\gq(\emptyset) = \dot q$.
\end{enumerate}
\end{dfn}
 Now we are ready to define ${\bar \Clink}^{\lambda}$: let $s \in {\bar \Clink}^{\lambda} (p,\dot q)$ if and only if \emph{either}
\begin{enumerate}[(i)]
\item \label{short:center} $s \in \Cl(p)$ and $p \forces \dot q \dleqlol 1_{\dot Q}$ holds \emph{or else}
\item \label{long:center} if $\lambda>\min I$,  $s=(\chi, \gT,\gH,\gl)$ where $\chi \in \Cl(p)$ and for some $\gq$, $(\gT,\gH,\gl, \gq)$ is a $\lambda$-guessing system for $\dot q$ below $p$.
\item if $\lambda=\min I$, $s=(\chi,\xi)$, where $\chi \in \Cl(p)$ and $p \forces_P \check \xi \in \dot \Cl(\dot q)$.
\end{enumerate}
It is straightforward to check that  $\ran(\bCl)$ has size at most $\lambda$. Thus we may assume $\bCl \subseteq  (P*\dot Q)\times \lambda$,
although this is not literally the case.

We have finally defined the stratification of $\bar P=P*\dot Q$. Let's check the remaining axioms.

\medskip

\textbf{Continuity:}\label{proof:composition:cont}
Say $\lambda'\in I$, $\lambda'<\lambda)$, both $\bar p=(p_\xi,\dot p_\xi)_\xi$ and $\bar q=(q_\xi,\dot q_\xi)_\xi$ are $\lambda'$-adequate sequences of length $\rho$ and for each $\xi<\rho$, ${\bar \Clink}^{\lambda}(p_\xi,\dot p_\xi)\cap {\bar \Clink}^{\lambda}(q_\xi,\dot q_\xi)\neq\emptyset$. 
Moreover, let $(p,\dot p)$ and $(q,\dot q)$ denote greatest lower bounds of $\bar p$ and $\bar q$, respectively.

First, by \emph{Continuity} for $P$, we can find $\chi \in \Cl(p)\cap \Cl(q)$.
For each $\xi<\rho$, fix $(\gT^\xi,\gH^\xi,\gl^\xi)$ such that for some $\chi'$,
\[  (\chi', \gT^\xi,\gH^\xi,\gl^\xi) \in \bCl(p_\xi, \dot p_\xi)\cap\Cl(q_\xi, \dot q_\xi). \]
and find $p_\g^\xi$ such that $(\gT^\xi,\gH^\xi,\gl^\xi,p_\g^\xi)$ is a guessing system for $\dot p_\xi$ below $p_\xi$.

Now construct a guessing system $(\gT,\gH,\gl,p_\g)$ for $\dot p$ below $p$, showing $(p,\dot p)\in\dom(\bCl)$.
It will be clear from the construction that $\gT,\gH$ and $\gl$ do not depend on the sequence of $p_\g^\xi$, $\xi<\rho$.
Let $s \in \gT$ if and only if $s=\emptyset$ or $\xi\conc s \in \gT^\xi$. 
Let $\gl(\emptyset)=\lambda'$, and of course $p_\g(\emptyset)= \dot p$.
Now let $s \in \gT\setminus \{ \emptyset \}$ be given and define $\gl(s)$, $p_\g(s)$ and, in the case that $s \in \gT^0$, also define $\gH(s)$. 
Find $s'$ such that $s= \xi\conc s'$. Let $\gl(s) = \gl^\xi(s')$ and let $\gH(s) = \gH^\xi(s')$ if $s \in \gT^0$ (or equivalently, if $s' \in (\gT^\xi)^0$). Let $p_\g(s)=p_\g^\xi(s')$. 

To check that $(\gT,\gH,\gl,p_\g)$ is a guessing system, first observe that $\is$ is reversely well-founded on $\gT$.
Moreover, $\gr(\emptyset) = \rho$ is an ordinal and $\gl(\emptyset)<\lambda$. Also, Item~\ref{approximate} holds for $s=\emptyset$, by construction. The rest of the conditions are straightforward to check; they hold by construction and because for each $\xi<\lambda'$, $(\gT^\xi,\gH^\xi,\gl^\xi, p_\g^\xi)$ is a guessing system.

The same construction for $(q,\dot q)$ yields a function $q_\g$ such that 
$(\gT,\gH,\gl,q_\g)$
is a guessing system for $\dot q$ below $q$. Thus,
\[  (\chi,\gT,\gH,\gl) \in {\bar \Clink}^{\lambda}(p,\dot p)\cap {\bar \Clink}^{\lambda}(q,\dot q). \]

\medskip

\textbf{Interpolation:} Say $(d,\dot d)\leq_{\bar P} (r,\dot r)$. First find $p \in P$ such that $p \lequpl d$ and $p \leqlol r$.
If $p \cdot d \neq 0$, then $p \cdot d \forces_P \dot d \leq_{\dot Q} \dot r$, so we can find $\dot p$ such that $p \cdot d \forces_P \dot p \dlequpl \dot d$ and $p \forces_P \dot p \dleqlol \dot r$.

\medskip

\textbf{Linking:} Say $\bar p\blequpl \bar d$, where $\bar p=(p,\dot p)$ and $\bar d=(d,\dot d)$, 
and assume $\bCl(\bar p)\cap\bCl(\bar d)\neq\emptyset$. 
First assume we can find $(\chi,\gT,\gl,\gH) \in \bCl(\bar p)\cap\bCl(\bar d) $ (i.e., \ref{long:center} holds in the definition of $\bCl$). 
As $\chi \in \Cl(p)\cap\Cl(d)$, by \emph{Linking} for $P$ there exists $w$ such that
for all regular $\lambda' \in 0\cup I \cap \lambda$, both $w\leqlo^{\lambda'}p$ and $p\leqlo^{\lambda'}d$.

Now fix $p_g$ and $d_g$ such that 
$(\gT,\gl,\gH,p_g)$ is a guessing system for $\dot p$ below $p$ and 
$(\gT,\gl,\gH,d_g)$ is a guessing system for $\dot d$ below $d$.
We show by induction on the rank of $s$ (in the sense of the reversed $\is$-order) that for each $s \in \gT$,
\begin{equation}\label{share:colour:ext}
p\cdot d \forces_P \dot \Cl (p_g(s)) \cap \dot \Cl (d_g(s)) \neq 0.
\end{equation}
First, let $s \in \gT^0$. By Definition~\ref{gs} Item~\ref{spectrum}, we can find $P$-names $\dot \alpha$ and $\dot \beta$ such that both have
spectrum $\gH(s)$ below $p$ and $d$, respectively, and moreover:
\[ p \forces_P \dot \alpha \in \dot \Cl(p_g(s)) \]
and 
\[ d \forces_P \dot \beta \in \dot \Cl(d_g(s)). \]
Thus, as $\dot \alpha$ and $\dot \beta$ have a common spectrum below $p \cdot d$, \eqref{share:colour:ext} holds.

For $s$ of greater rank, we may assume by induction that for each $\xi < \gr(s)$, 
\begin{equation}\label{share:colour:ext:ind}
p\cdot d \forces_P \dot \Cl (p_g(s\conc\xi)) \cap \dot \Cl (d_g(s\conc\xi)) \neq 0.
\end{equation}
As $p$ forces that
\begin{eqpar}
$\{p_g(s\conc\xi)\}_{\xi<\gr(s)}$ is a $\gl(s)$-adequate sequence and
$p_g(s)$ is a greatest lower bound of $\{p_g(s\conc\xi)\}_{\xi<\gr(s)}$,
\end{eqpar}
and as $d$ forces the corresponding statement for $d_g(s)$ and $\{d_g(s\conc\xi)\}_{\xi<\gr(s)}$, \emph{Continuity} for $\dot Q$ in the extension allows us to infer \eqref{share:colour:ext} for this $s$. This finishes the inductive proof on the rank of $s$.

Finally, \eqref{share:colour:ext} holds for $s = \emptyset$, so as $p_g(\emptyset)=\dot p$ and $d_g(\emptyset)=\dot d$, by \emph{Linking} for $\dot Q$ in the extension, $w \forces$ there exists $\dot w$ such that for all regular $\lambda' \in 0\cup I\cap\lambda$, both $\dot w\dot\leqlo^{\lambda'}\dot p$ and $\dot w\dot\leqlo^{\lambda'}\dot d$. Then $\bar w =(w,\dot w)$ is as desired.

Now secondly assume we have $\chi \in \bCl(\bar p)\cap\bCl(\bar d)$ and \ref{short:center} holds in the definition of $\bCl$.
In this case $\chi \in \Cl(p)\cap\Cl(d)$. 
Let $w \in P$ such that $w \leqlo^{<\lambda} p$ and $w\leqlo^{<\lambda} d$.
By assumption, $w \forces \dot d \dleqlol 1_{\dot Q}$ and $\dot p \dlequpl \dot d$. 
By \emph{Expansion}~(\ref{exp}) for $\dot Q$, we conclude $w \forces \dot p \leq \dot d$. 
We claim $\bar w=(w,\dot p)$ is the desired lower bound:
$\bar w \bleqlo^{<\lambda} \bar p$ holds because $\dleqlol$ is a preorder.
We show $\bar w \bleqlo^{<\lambda} \bar d$: we have $w \forces \dot p \leq \dot d \leq 1_{\dot Q}$ and by assumption $w \forces \dot p \dleqlol 1_{\dot Q}$. So by (\ref{er}), we conclude $w \forces \dot p \dleqlo^{<\lambda} \dot d$ and are done.

\medskip

\textbf{Density:}
Let $(p_0,\dot q_0) \in \bar R$. First, assume $min I < \lambda$ and fix a regular $\lambda' \in I \cap \lambda$.
By \emph{Density} for $\dot Q$ in the extension, we can find $P$-names $\dot \chi$ and $\dot q_1$ such that $\forces_P$``$\dot q_1 \dot \leqlo^{\lambda'} \dot q_1$ and $\dot \chi \in \Cl(\dot q_1)$''. 
By Lemma \ref{density reduction}, we can find $p_1 \leqlo^{\lambda'} p_0$ such that $\dot \chi$ is $\lambda'$-chromatic below $p_1$,
and by  \emph{Density} for $P$ we can find $p_2 \leqlo^{\lambda'} p_1$ and $\zeta$ such that $\zeta \in \Cl(p_2)$.

Let $\gT=\{ \emptyset \}$, $q_g(\emptyset)=\dot q_1$, $\gl(\emptyset)= \lambda'$ and let $\gH(\emptyset)$ be a $\lambda'$-spectrum of $\dot \chi$ below $p_2$. 
Thus $(\gT,\gl,\gH,q_g)$ is a guessing system for $\dot q_1$ below $p_2$---the only non-trivial clause is Item~\ref{spectrum}, 
which holds as $\gH(\emptyset)$ is a $\lambda'$-spectrum of $\dot \chi$ below $p_1$ and $p_2 \leq_P p_1$. 
So $(\zeta, \gT, \gl, \gH) \in \bCl(p_2,\dot q_1)$, and $(p_2, \dot q_1)\bar \leqlo^{\lambda'} (p_0, \dot q_0)$.

It remains to show $\dom(\Cl)$ is dense in the case that $min I =\lambda$. Find $(p_1,\dot q_1) \leq_{\bar R}(p_0,\dot q_0)$ such that for some ordinals $\zeta, \chi< \lambda$, $\zeta \in \Cl(p)$ and $p_1\forces_P \check{\chi}\in \dot \Cl(\dot q_1)$. 
Then $(\zeta, \chi) \in \bCl(p_1,\dot q_1)$.
\end{proof}

\section{Stratified Iteration and Diagonal Support}

Stratified forcing is iterable, if the right support is used.
We now define this type of support for iterations in which the quotients are forced to be stratified.
A more general theory (suitable for amalgamation) will be developed in Chapter~\ref{sec:ext}.

\medskip

Consider a product of forcings $P_\xi * \dot Q_\xi$, of the type of example $\ref{strat:ex:*}$. 
To see that such a forcing preserves cofinalities we may use the fact that for large enough $\lambda_0$, $P_\xi$ is closed under sequences of length $\lambda_0$ while $\dot Q_\xi$ has a (strong form of) $(\lambda_0)^+$-chain condition. 
To preserve the latter, we should use support of size less than $\lambda_0$; to preserve the former, our choice would be to use support of size $\lambda_0$.

\medskip

This calls for a kind of mixed support:
Define 
\[ \Pi^{d}_{\xi<\lambda_0} ( P_\xi * \dot Q_\xi) \]
to be the set of all sequences $(p(\xi), \dot q(\xi))_{\xi<\lambda_0} \in \Pi_{\xi<\lambda_0} ( P_\xi * \dot Q_\xi)$
such that for all but less than $\lambda_0$ many $\xi$, 
\begin{equation}\label{ex:support}
p(\xi)\forces_P \dot q(\xi)=\dot 1_\xi.
\end{equation}
Using the stratification of $P * \dot Q$, \eqref{ex:support} 
 may be written as
 \[ (p(\xi), \dot q(\xi)) \leqlo^{\lambda_0} 1_{P_\xi*\dot Q_\xi}. \] 
The use of the term ``diagonal'' is motivated by the intuition that we allow large support on the ``upper'' part $P_\xi$, and small support on the ``lower'' part $\dot Q_\xi$.
\begin{dfn}~\label{def:it:strat:comp}
\begin{enumerate}
\item We say the iteration $\bar Q^\theta = \langle P_\nu; \dot Q_\nu, \leq_\nu , \dot 1_\nu\rangle_{\nu<\theta}$ \emph{has stratified components}\index{iteration (forcing)!with stratified components} if and only if for every $\nu<\theta$,
$\dot Q_\nu$ is a $P_\nu$-name and $P_\nu$ forces $\dot Q_\nu$ is a stratified partial order as witnessed by the system 
\[
\bar \pss=(\dot\leqlol_\nu, \dot c_\nu, \dlequpl_\nu,  \dot \D_\nu, \dot \Clink_\nu)_{\lambda\in \Reg, \nu<\theta}.
\] 
(which is called its \emph{stratification}). 
Formally, the reader may wish to replace $\dot \D_\nu$ in the above by a name for the G\"odel number of a formula defining the class $\dot \D_\nu$ with parameter $\dot c_\nu$.
Moreover, we demand that for all regular $\lambda$ there is $\bar\lambda<\lambda^+$ such that $\supp_\lambda(p) \subseteq \bar \lambda$, where $\supp_\lambda(p)$ is defined as
\[
\supp^\lambda(p)=\{ \xi  \setdef \pi(p)_\xi \not \forces_\xi p(\xi) \dot \leqlo_\xi \dot 1_\xi\}.
\]
\item $P_\theta$ is the \emph{diagonal support limit }\index{diagonal support}\index{support!diagonal} of the iteration with stratified components $\bar Q^\theta$ with stratification $\bar \pss$ if and only if
$P_\theta$ is the set of all threads though $\bar Q^\theta$ such that for each regular $\lambda$, $\supp_\lambda(p)$ has size less than $\lambda$.
\item We say $\bar Q_\theta$ is an iteration with diagonal support if for all limit $\nu < \theta$, $P_\nu$ is the diagonal support limit of $\bar Q\nu$.
\end{enumerate}
\end{dfn}

We state the following theorem here for completeness; it will follow from Theorem \ref{thm:it:strat} and Lemma \ref{stratified:comp:implies:ext} as Corollary \ref{cor:it:strat:comp}.
\begin{thm}\label{thm:it:strat:comp}\index{support!diagonal}\index{diagonal support}
Say $\bar Q = \langle P_\nu, \dot Q_\nu \rangle_{\nu<\theta}$ is an iteration with stratified components and diagonal support.
Then $P_\theta$ is stratified.
\end{thm}

\chapter{Easton Supported Jensen Coding}\label{sec:coding}

In this chapter we shall discuss Easton supported Jensen coding with localization (David's trick).
See Chapter~\ref{s.overview.proof} (especially Section~\ref{p.s.jensen.david})  for an introduction.

\section{A Variant of the Square Principle}\label{varsquare}\index{Square Principle}

The variant of square we discuss in this section is a technical prerequisite which we use to obtain a smooth transition from inaccessible to singular coding
at certain points in our construction, an approach we shall refer to as \emph{virtual inaccessible coding}. 
We will say more about this when we use it.
\begin{lem}
There is a class $(E_\alpha)_{\alpha\in\Card}$\index[notation]{ E alpha alpha in Card@$(E_\alpha)_{\alpha\in\Card}$}
\index[notation]{E alpha alpha in Card@$(E_\alpha)_{\alpha\in\Card}$} such that for all $\alpha \in \Card$ which are not Mahlo\index{Mahlo cardinal}\index{cardinal!Mahlo}, $E_\alpha$ is club in $\alpha$,
 $E_\alpha \subseteq \Sing$ and
whenever $\beta$ is a limit point of $E_\alpha$ we have  $E_\beta= E_\alpha \cap\beta$
 and $E_\alpha \in J^{A\cap\alpha}_{\delta+2}$ whenever $J^{A\cap\alpha}_\delta\models$``$\alpha$ is not Mahlo.''
\end{lem}
\begin{proof}
Let $(C_\alpha)_{\alpha \in \Sing}$ be the standard global square\index{Global Square Principle} on singulars, constructed as in \myplacecite{11.63,~p.~228}{ralfbook}.
For $\alpha \in \Sing$, let $\eta^*(\alpha)$ be the maximal limit ordinal such that $\alpha$ is regular in $J^{A\cap\alpha}_{\eta^*(\alpha)}$ and let $M^*_\alpha = J^{A\cap\alpha}_{\eta^*(\alpha)}$.
Observe that if $\beta$ is a limit point of $C_\alpha$, there is
$\sigma\colon M^*_\beta \rightarrow M^*_\alpha$ which is (at least) $\Sigma_0$-elementary such that $\crit(\sigma)=\beta$ and $\sigma(\beta)=\alpha$.

Suppose $\alpha$ is not Mahlo. 
Let $\eta$ be least such that $J^{A\cap\alpha}_{\eta}\models$``$\alpha$ not Mahlo'', let $m_\alpha = J^{A\cap\alpha}_{\eta}$, $\eta(\alpha)=\eta+1$ and let $M_\alpha = J^{A\cap\alpha}_{\eta(\alpha)}$.\footnote{
Observe we could write $\eta(\alpha)=\eta$ above. 
Then $m_\alpha$ would be minimal with a definable club of definably singular cardinals. 
This is still enough to make the rest of the argument go through, as our goal was to witness the non-Mahlo-ness in a way that is preserved with $\Sigma_0$-embeddings with large enough range.}

\begin{description}
\item[Case 1] If $\alpha \in \Sing^{M_\alpha}$ we let $E_\alpha = C_\alpha$, where the latter comes from the standard square.

\item[Case 2] Otherwise, if $\alpha$ is $\mathbf{\Sigma}_1$-singular in $M_\alpha$, to ensure coherency we define $E_\alpha$ to be the tail of $C_\alpha$ obtained by requiring that the maps witnessing that $\xi \in C_\alpha$ have $m_\alpha$ in their range.

In detail:
Note that $\rho_1(M^*_\alpha)=\alpha$. 
Let $p(\alpha)$ be the 1st standard parameter and let 
\[
W(\alpha)= \{ W^{\nu, p_1(M^*_\alpha)}_{M^*_\alpha}\setdef \nu \in p_1(M^*_\alpha) \}
\]
be the appropriate ``solidity witness'' (following the notation in \cite{hbst:fs}). 
Observe that by construction, we can find a minimal $\theta(\alpha)<\alpha$ such that
$h_{\Sigma_1}^{M^*_\alpha}(\theta(\alpha)\cup\{p(\alpha)\})$ is unbounded in $\alpha$.\footnote{Pick a minimally definable function $F$ witnessing that $\alpha$ is singular; 
the parameter of $F$, if any, can be absorbed into $\theta(\alpha)$ since $M_\alpha$ projects to $\alpha$; 
now the Skolem hull is a fortiori unbounded.}
Let
$E^*_\alpha$ consist of those $\beta \in\alpha\cap\Sing$ such that there is a %
$\Sigma_0$-elementary
map $\sigma\colon J^{A\cap\beta}_{\bar \eta}\rightarrow M^*_\alpha$ such that $\{ \alpha, p(\alpha), m_\alpha \} \cup W(\alpha)\subseteq \ran(\sigma)$, $\crit(\sigma)=\beta$ and $\sigma(\beta)=\alpha$.
As in the proof of $\Box$, we show that if $E^*_\alpha$ is bounded in $\alpha$,  $\cof(\alpha)=\omega$.
In this case, we
can set $E_\alpha = C_\alpha$.
\begin{fct}
If $E^*_\alpha$ is bounded in $\alpha$,  $\cof(\alpha)=\omega$.
\end{fct} 
\begin{proof}[Sketch of the proof.]
We need to find an embedding $\sigma$ witnessing $\beta \in E^*_\alpha$ for some large enough $\beta$.
Let $\xi <\alpha$ be given and let $M \prec M_\alpha$ be a countable elementary submodel such that $\{\xi, m_\alpha \} \subseteq M$ and let $\pi \colon\bar M \rightarrow M$ be the inverse of the collapsing map.
Let $E$ be the $(\crit(\pi), \beta)$-extender, where $\beta=\sup(\ran(\pi)\cap\alpha)$, derived from $\pi$.
Let $\sigma\colon \Ult(\bar M, E) \rightarrow M_\alpha$ be the factor map. 
Check that $\sigma$ is as required, in particular it has critical point $\beta$ and $m_\alpha \in \ran(\sigma)$.
\end{proof}

If $E^*_\alpha$ is unbounded in $\alpha$, we define two sequences as follows.
Let $\gamma_0 = \min E^*_\alpha$.
Given $\gamma_\xi$, let $\delta_\xi$ be the least $\delta$ such that $h^{M^*_\alpha}_{\Sigma_{n(\alpha)}}(\delta\cup\{p(\alpha)\})\setminus (\gamma_\xi+1)\neq\emptyset$.
Let 
\[
\gamma_{\xi+1} = \min [ E^*_\alpha \setminus  h^{M^*_\alpha}_{\Sigma_{n(\alpha)}}(\delta_\xi\cup\{p(\alpha)\}) ].
\]
Here, assume that $h^{M_\alpha}_{\Sigma_{n(\alpha)}}$ to be defined such that its range grows by one element for every ordinal, slightly abusing notation.
 
At limit points $\lambda$, let $\gamma_\lambda= \bigcup_{\xi<\lambda} \gamma_\xi$, if
this yields a point below $\alpha$. 
Otherwise set $\lambda=\bar\theta(\alpha)$ and stop the construction.
Observe that $\delta_\xi<\theta(\alpha)$ and the $\delta_\xi$ are increasing, so the ordertype $\bar \theta$ of the sequence constructed is at most $\theta(\alpha)<\alpha$.
Set
\[
E_\alpha = \{ \gamma_\xi \setdef \xi <\bar\theta(\alpha) \}\cap\Card.
\]
Note the only difference to $\Box$ is the requirement $m_\alpha \in \ran(\sigma)$ when $\eta^*(\alpha) = \eta(\alpha)$, i.e., $M_\alpha=M^*_\alpha$.

\item[Case 3] $\alpha$ is $\mathbf{\Sigma}_1$-regular over $M_\alpha$.
So
\[
E_\alpha = \{ \beta <\alpha \setdef h_{\Sigma_1}^{M_\alpha}(\beta\cup\{m_\alpha\})\cap\alpha=\beta \}
\]
defines a club.
Observe that for some $C\in m_\alpha$ , we have $m_\alpha\models C\subseteq \Sing$ and $C$ is club in $\alpha$;
moreover, $\{\alpha, C\} \subseteq h_{\Sigma_n}^{M_\alpha}(\emptyset\cup\{m_\alpha\})$ ($\alpha$ is the cardinality of $m_\alpha$).
So for every $\beta \in E_\alpha$, we have $\beta \in C$ by elementarity, and thus $\beta \in \Sing$.\footnote{The same argument would work in case 1.}
\end{description}

It remains to see that $(E_\alpha)_{\alpha\in \Sing}$ is coherent.
So let $\alpha\in\Sing$ and $\beta$ be a limit of $E_\alpha$.
We must check that $\beta$ falls into the same case as $\alpha$. 

Assume $\alpha$ falls into case 1.
We know that there is a $\Sigma_0$-elementary $\sigma\colon M^*_\beta\rightarrow M^*_\alpha$.
Since $\eta^*(\alpha) < \eta(\alpha)$, it must be the case that
$\eta^*(\beta) < \eta(\beta)$: otherwise, $m_\beta$ is a $J$-structure which is a model of ``$\beta$ is not Mahlo'' and by elementarity $\sigma(m_\beta)$ is a $J$-structure and $\sigma(m_\beta)\models$``$\alpha$ is not Mahlo,''
contradicting $\eta^*(\alpha) < \eta(\alpha)$.
Thus $\beta$ also falls into case 1 and coherency follows by the coherency of $\Box$.

Assume $\alpha$ falls into case 2.
It follows that $\eta^*(\alpha) = \eta(\alpha)$.
We know that there is $\sigma\colon J^{A\cap\beta}_{\bar \eta}\rightarrow M^*_\alpha$ as in the definition.
Standard fine structural arguments (e.g. see \cite{ralfbook}) show that $\eta^*(\beta)=\bar\eta$, i.e., $J^{A\cap\beta}_{\bar \eta} = M^*_\beta$ and $\beta$ is $\Sigma_1$-singular in $M^*_\beta$.
We have $m_\alpha \in \ran(\sigma)$, so
by $\Sigma_0$-elementarity $\sigma^{-1}(m_\alpha)$ must be a $J$-structure and $\sigma^{-1}(m_\alpha) = J^{A\cap\beta}_{\eta^*(\beta)-1}$, for otherwise the rudimentary closure of $m_\alpha\cup\{m_\alpha\}$ would have to exist in $M_\alpha$.
Thus $\eta^*(\beta)=\eta(\beta)$, for by a similar argument, $J^{A\cap\beta}_{\eta^*(\beta)-2}\models$``$\beta$ is Mahlo.''
Thus $\beta$ falls into case 1.
Now the same arguments as in the proof of $\Box$ show that $E_\beta=E_\alpha\cap\beta$  (maps factor because $W(\alpha) \subseteq \ran(\sigma)$).

Now assume $\alpha$ falls into case 3.
Let $\sigma\colon J^{A\cap\beta}_{\bar\eta}\rightarrow M_\alpha$ be the inverse of the collapsing map of $h_{\Sigma_1}^{M_\alpha}(\beta\cup\{m_\alpha\})$, and let
$\bar m =\sigma^{-1}(m_\alpha)$.
Clearly $\sigma(\beta)=\alpha$, $\crit(\sigma)=\beta$. 
By elementarity,  $\bar m$ is the first $J$-structure witnessing that $\beta$ is not Mahlo.
Thus, $\bar\eta = \eta(\beta)$.
Also, by elementarity and since $\beta$ is a limit point of $E_\alpha$, $\beta$ is $\mathbf{\Sigma}_1$-regular in $M_\beta = J^{A\cap\beta}_{\bar\eta}$.
Thus $\beta$ falls into case 2.
It is easy to check that $E_\beta = E_\alpha\cap\beta$.
\end{proof}

\section{Notation and Building Blocks}\label{sec:full:setting}

\begin{dfn}[Notation]
In the following we shall use a convenient partition of the ordinals into $4$ components:
for $0\leq i\leq3$, write $[\On]_i$\index[notation]{On i@$[\On]_i$}\index[notation]{ On i@$[\On]_i$} for the set of $\xi$ which is equal to $i$ modulo $4$.
Write $(\xi)_i$\index[notation]{xi i@$(\xi)_i$}\index[notation]{ xi i@$(\xi)_i$} for the $\xi$-th element of $[\On]_i$.
Given a set (or class) $B$,
write $[B]_i = \{ (\xi)_i \setdef \xi \in B \}$;
this is consistent with the notation $[\On]_i$ for the $i$-th component.
When we say \emph{$A$ codes $B$ on its $i$-th component},
we mean $A \cap [\On]_i = [B]_i$; that is $\xi \in B \iff (\xi)_i \in A$.
We also write $[A]^{-1}_i$ for $\{ \xi \setdef (\xi)_i \in A \}$,
so that $[A]^{-1}_i$ is the set coded by the $i$-th component of $A$.
We shall sometimes write $B\oplus C$\index[notation]{plus@$\oplus$}\index[notation]{oplus@$\oplus$}\index[notation]{  oplus@$\oplus$} for $[B]_0 \cup [C]_1$.%

Recall that  we write $s \conc t$ for the concatenation of sequences $s$ and $t$. 
If $X\subseteq [\lh(s), \gamma)$ for $\gamma\in \On$, we shall use the shorthand $s\conc X$\index[notation]{s^X@$s \conc X$} for $s\conc\chi_X$\ where $\chi_X$ is the characteristic function of $X$ on $[\lh(s), \sup(X))$. 
We use $s \conc i$\index[notation]{s^i@$s \conc i$} to mean ``$s$ with the single value $i$ appended'' when for $i \in \{0,1\}$. 
The two uses stand in conflict in a few extreme cases, but we trust the reader to find out the correct interpretation each time.
\end{dfn}

Let $\langle \Diamond_\delta \setdef \delta \in \On \rangle$ be the global $\Diamond$-sequence\index{Diamond Principle}\index{Global Diamond Principle}
of $L$ concentrating on the inaccessibles below $\kappa$---i.e., the following holds in $L$: If $X \subseteq \kappa$, the set 
\[
\{ \alpha \setdef \Diamond_\alpha = X\cap\alpha \}\cap\Inacc
\] 
is stationary.
Let $(E_\delta)_{\delta\in \Card}$\index[notation]{E alpha alpha in Card@$(E_\alpha)_{\alpha\in\Card}$} be the variant of square constructed in Section~\ref{varsquare}, except that we set $E_{\kappa^+}=E_\kappa = \emptyset$, for a more uniform notation.

Recall that $\langle .\,, .\rangle$ for us denotes the G\"odel pairing function, that $\Card'$ denotes $(\Card \setminus \omega) \cup\{ \emptyset \}$, and that we write $\emptyset^+=\omega$. 

\medskip

We work in the following setting, which captures the essence of the situation we shall find ourselves in at (some of the) successor stages in our iteration.
The reader should consult Chapter~\ref{sec:main} (especially Section~\ref{def:it:succ:stage}) and Chapter~\ref{s.overview.proof} for a comprehensive motivation.

We shall, throughout this section, suppose that $\kappa$ is the only Mahlo and $V=L[G^o][\bar B^-]$ is a model of $\GCH$, 
where $G^o \subseteq \HSize(\kappa^+)$ and $\bar B^- \subseteq \HSize(\kappa^{++})$.
We also assume $\HSize(\kappa^{++})=L_{\kappa^{++}}[G^o]$ (intuitively, $\bar B^-$ contributes no subsets of $\kappa^+$ --- in fact, it comes from a $\kappa^{++}$-distributive forcing).
Lastly, for some real $x_0$, $\omega_1 = {\omega_1}^{L[x_0]}$ and 
 $\Card^V\setminus \omega_1=\Card^{L}\setminus \omega_1$.

Suppose $\bar T=\langle T(\sigma, n, i ,j) \setdef (\sigma, n, i,j) \in {}^{<\kappa}2 \times \omega\times \omega \times 2\rangle$ is  our canonical sequence\index{canonical!sequence of trees}\index{Suslin trees!canonical sequence of} of $\kappa^{++}$-Suslin trees\index{canonical!sequence of trees}\index{Suslin trees!canonical sequence of} so that nodes of height $\xi$ in each tree are elements of ${}^{\xi}2$ and so that $\bar T$ is definable, in fact locally semidecidable (see Section~\ref{p.s.trees}).
Further, assume that in $L[G^o][\bar B^-]$, there is a set $I$ of size $\kappa$ and a real $r$ such that 
$\bar B^-(\sigma,n,i,j)$ is a branch through $T(\sigma,n,i,j)$ if and only if $r(n)=i$ and $\sigma \in I$;
otherwise, we assume $\bar B^-(\sigma,n,i,j) = \emptyset$, and $T(\sigma,n,i,j)$ remains $\kappa^{++}$-Suslin in $L[G^o][\bar B^-]$.

\medskip

Find $A_0\subseteq \kappa^{++}$\index[notation]{A0@$A_0$} such that for each cardinal $\alpha \leq \kappa^{++}$,
we have $\HSize(\alpha) = L_\alpha[A_0]$.
We can also assume that $A_0\cap [\On]_0=[I]_0$, 
\[
A_0 \cap \omega \cap [\On]_1= [x_0 \oplus \{ 2n + i \setdef r(n)=i \}]_1,
\] 
and $A_0 \cap [\kappa^+, \kappa^{++}) \cap [\On]_1$ is the set 
\[
\{(\#(\sigma,n,i,j,t))_1 \setdef r(n)=i, \sigma \in I, t\in \bar B^-(\sigma,n,i,j) \setminus \HSize(\kappa^+)\}.
\]
Thus we have $V=L[G^o][\bar B^-]=L[A_0]$ and for each cardinal $\alpha$,
we have $\HSize(\alpha) = L_\alpha[A_0]$.
In fact, $\Card^V=\Card^{L[A_0\cap\omega]}$.

We shall write $s_{\kappa^+}$ for the characteristic function of $A_0 \cap [\kappa^+, \kappa^{++})$, as we will sometimes treat this set similar to conditions in our forcing (the top ``string'').
Our goal is to find a forcing $P$ which codes $L[A_0]$ into a subset of $\omega$ in the sense that if $G$ is $P$-generic over $L[A_0]$,
$L[G] = L[G\cap\omega]$, where we view $G$ as a set of ordinals. 
Moreover, we want that $A_0$ is ``locally'' definable in $L[G]$ (localization or David's trick).

\medskip

We now define the building blocks of our forcing.
These consist of three types of almost-disjoint coding: one for successor cardinals, one for inaccessible cardinals, and a slight variation of the latter for the coding from $\kappa^+$ into $\kappa$.
Also, we define the singular limit coding using \emph{coding delays}. 
For all this, we set up the following \emph{coding apparatus}.

We will make use of our partition in the following way:
the coding of $A_0$ by $G$ uses $[\On]_0$, the ``successor coding'' of $G\cap[\alpha^+, \alpha^{++})$ into $G\cap[\alpha, \alpha^{+})$ uses $[\On]_1$,
the ``inaccessible limit coding'' from $G\cap[\alpha, \alpha^+)$ to $G \cap \alpha$ for inaccessible $\alpha$ uses $[\On]_2$.
Finally, the singular limit coding from $G\cap[\alpha, \alpha^+)$ to $G \cap \alpha$ for singular $\alpha$ uses $[\On]_3$.

The definition of the building blocks will be relative to a set $A$.
In later use, $A$ will either stand for $A_0$---this will be the case when $\alpha \in \{ \kappa, \kappa^+\}$---or for a set
$A_p$, which we define in the next section, 
depending on a condition $p \in P$.
This is due to the nature of the Mahlo coding from $\kappa^+$ into $\kappa$ and can be circumvented by treating the forcing as an iteration in three steps (see Section~\ref{3steps}).

\begin{dfn}\label{reg_defs}
Let $\alpha$ be an uncountable cardinal $\leq\kappa^{+}$ and let $A \subseteq \On$. 
 
\begin{description}%

\item [Basic Strings]\index{basic strings}\index{string!basic}

Let $S_\alpha$ denote the set of \index[notation]{s@$\lvert s\rvert$}\index[notation]{ s @$\lvert s\rvert$}
\index[notation]{p alpha@$\lvert p_\alpha\rvert$}\index[notation]{ p alpha @$\lvert p_\alpha\rvert$}$s\colon [\alpha, |s|) \rightarrow 2$, where $\alpha \leq |s| < \alpha^+$.
We abuse notation by writing $\emptyset_\alpha$ for the empty string at $\alpha$ and note $|\emptyset_\alpha| = \alpha$.

Also, let $S_{<\kappa}$\index[notation]{S<kappa@$S_{<\kappa}$} denote the set of $s\colon [0, |s|) \rightarrow 2$, where $|s| < \kappa$.

\item[Steering Ordinals]\index{steering ordinals}
\index[notation]{mu xi 0, mu <xi 0@$\mu_0^\xi$, $\mu_0^{<\xi}$}
For $\xi\in [\alpha, \alpha^+)$, define $\mu_0^\xi=\mu(A)_0^\xi$ 
and $\mu_0^{<\xi}$  simultaneously by induction:
Let $\mu_0^{<\alpha}$ be least $\mu$ such that\footnote{We want $E_\alpha$ in the coding structures of locally inaccessibles, as this will give us a way to distinguish between conditions as they turn out in the distributivity proof and those as built to show extendibility.
The presence of $\Diamond_\alpha$ is important for a $\Delta$-system argument in \ref{it:reals:are:caught}, when we show that $\kappa$ is not collapsed in our iteration and stays Mahlo until the last stage.} 
$E_\alpha, \Diamond_\alpha \in L_\mu[A\cap\alpha]$ 
and for $\xi > \alpha$ let
\[
\mu_0^{<\xi} = \sup_{ \nu < \xi} \mu^{\nu}.
\]
Define $\sigma(\xi)$ to be least above $\mu_0^{<\xi}$ such that $L_{\sigma(\xi)}\vDash$``$\alpha$ is the greatest cardinal.'' 
Let $\mu_0^\xi= \sigma(\xi) +  \alpha$.

Note that by induction $\xi \leq \mu_0^{<\xi}$.
If $\alpha \in \Reg$, we write $\mu^\xi$ and $\mu^{<\xi}$ for $\mu^\xi_0$ and $\mu_0^{<\xi}$.
We write $\mu_0^s$ and $\mu_0^{<s}$ for $\mu^{|s|}_0$ and $\mu_0^{<|s|}$; similarly for $\mu^s$ and $\mu^{<s}$.

\item [Coding structures]\index{coding structures}

We let $\mathcal{A}_0^s=\mathcal{A}(A)_0^s$ denote $L_{\mu_0^s}[A\cap\alpha,s]$.
If $\alpha \in \Reg$, we write $\mathcal{A}^s$\index[notation]{A s, A s 0, A xi@$\mathcal{A}^s$, $\mathcal{A}_0^s$, $\mathcal{A}^\xi$} for $\mathcal{A}_0^s$ and $\mathcal{A}^\xi$ 
for $\mathcal{A}_0^\xi$.

\item [Coding apparatus]\index{coding apparatus}

For cardinals $\alpha < \kappa$, $s \in S_\alpha$ such that $\alpha$ is regular in $\mathcal{A}_0^s$, let\index[notation]{H s i@$H^s_i$} 
\[
H^s_i=H(A)^s_i = h^{\mathcal{A}_0^s}_{\Sigma_1}(i \cup \{ A\cap \alpha, D_\alpha, E_\alpha, s\})
\]
and let
$f^s(i)$\index[notation]{f s i@$f^s(i)$} be the order type of $H^s_i \cap \On$.
If $\alpha=\beta^+$ for $\beta\in\Card$, let $B^s$\index[notation]{B s@$B^s$} be %
$\{ i \in [\beta,\alpha) \setdef H^s_i \cap \alpha= i\}$ and
let \label{b^s:page}$b^s = \{ \langle i, f^s(i)\rangle \setdef i \in B^s\}$\index[notation]{b s@$b^s$} (where $\langle .\,, .\rangle$ denotes the G\"odel pairing function).
If $\alpha$ is inaccessible in $\mathcal{A}_0^s$, let
$B^s= \{ \beta^+ \setdef   H^s_\beta \cap \alpha= \beta, \beta \in \Card\cap\alpha\}$
and let $b^s = \ran(f^s\res B^s)$.\footnote{For $\alpha < \kappa$ inaccessible in $\mathcal{A}_0^s$, we could also just let $B^s = \{ \beta^+ \setdef \beta \in \Diamond_\alpha \cap E_\alpha \}$ if this is unbounded in $\alpha$ and $B^s = \{ \beta^+ \setdef \beta \in E_\alpha \}$ otherwise. We could let $B^s = \Succ\cap\kappa$ for $s \in S_\kappa$.}

\end{description}
\end{dfn}

The intuition is that we want to code $s$ via almost-disjoint coding, using the almost disjoint family\index{almost disjoint family} $\{ b^{s \res \nu}\setdef \nu \in [\alpha, |s|)\}$. 
The particular form of the $b^s$ is very convenient for the proof of Extendibility~\ref{ext1} when $\alpha\in\Sing$.
For successor $\alpha$, using the pair $\langle i, \otp (H_i\cap\On)\rangle$ is vital to ensure that the $b^s$ are almost disjoint.
Observe that by definition, 
for $s\in S_\alpha$ when
$\alpha < \kappa$ is inaccessible in $\mathcal{A}^s_0$, we have $\lVert \xi \rVert^- \in E_\alpha$ for all $ \xi\in b^s$.
Similarly, $\lVert \xi \rVert^- \in \Diamond_\alpha$ if in addition $\Diamond_\alpha$ is club in $\alpha$.%

\medskip

The phenomenon we call virtual inaccessible coding makes it necessary that we be able to put all of $b^s$ into the support of a condition in $P$;
otherwise the coding structures will change as conditions are extended (we discuss this below). 
Thus we want $b^s$ to be Easton. 
This, together with coherency issues arising in the proof of \emph{Extendibility} at singular $\alpha$ (see \ref{ext1} below), is why we use $E_\alpha$. 
In said proof, we shall have no control over what happens at limits of $E_\alpha$, which is why we use successors for $B^s$.
We want $E_\alpha\in \mathcal{A}^{\emptyset_\alpha}_0$ as this will give us a way to distinguish whether virtual inaccessible coding has occurred (again, see below).
The presence of $\Diamond_\alpha$ is important for a $\Delta$-system argument in \ref{it:reals:are:caught}, when we show that $\kappa$ is not collapsed in our iteration and stays Mahlo until the last stage.

\medskip

A condition $p \in P$ will in part consist of a sequence of $p_\alpha \in S_\alpha$, for $\alpha \in \Card \cap \kappa^{++}$.
Abusing notation, we identify the generic $G$ for $P$ with the set 
$$\{ \xi < \kappa^{++} \setdef \exists p\in G \; \exists \alpha \; p_\alpha(\xi)=1\}.$$
We will have that $L[G\cap \omega] = L[G]$ and $A_0$ is definable in $L[G]$.
For regular $\alpha<\kappa$, $G\cap \alpha$ will code $G\cap[\alpha,\alpha^{+})$ via almost almost disjoint coding (we discuss the case $\alpha=\kappa$ later).
For singular $\alpha$, $G\cap \alpha$ will code $G\cap[\alpha,\alpha^{+})$ via \emph{coding delays} which we introduce in Definition~\ref{sing_defs} below.

\medskip

Easton support presents a further complication for singular coding: 
In the proof of distributivity, there may arise a structure $\bar X$ such that $\bar X\models \alpha \in\Inacc$ for some singular $\alpha$ and the conditions we construct inevitably use inaccessible coding at such $\alpha$.
This is the phenomenon we call \emph{virtual inaccessible coding}\index{virtual inaccessible coding}\index{inaccessible coding!virtual}\label{l.virtual}.
The solution is based on the following intuition: We make singular coding structures large enough to 
``see'' the virtual inaccessible coding.
This requires that we can distinguish whether or not virtual inaccessible coding has occurred at all. 
Also the coding structures now depend on the condition \emph{below} $\alpha$ and we have to be careful they be stable with respect to extending $p$.

\begin{dfn}[The singular case]\label{sing_defs}
Except for the first item of this definition, we assume
$\alpha$ is singular now.
The following definitions will be relative to some
(partial) function $p \colon \alpha \rightarrow 2$ (it will be $\bigcup_{\beta \in \Card \cap \alpha} q_\beta$ for some $q\in  P$).
We extend the definition of support to such $p$ by $\supp(p)= \Card \cap \dom(p)$.

Given $p$ as above and a function $f$ with $\dom(f)\subseteq \Card$ such that $f \in \prod_{\beta \in \dom(f)} [\beta,\beta^+)$, we let $f_p$ be the partial function such that
$f_p(\beta)$ is the least $\eta \in [\On]_3$ with $f(\beta) < \eta$ and $p(\eta)=1$.

\begin{description}
\item [Regular decoding]\index{decoding!regular}
First, define $s(A,p)=s(p) \in S_\alpha$\index[notation]{s(p)@$s(p)$}.
To this end, define a sequence $(s_\gamma)_{\gamma < \gamma(p)}$ of basic strings in $S_\alpha$. 
Let $k=1$ if $\alpha \in \Succ\cap(\kappa^{+}+1)$ and $k=2$ if $\alpha \in \Inacc\cap(\kappa+1)$.
Let $s_0 = \emptyset$.
Given $s_\gamma$, first assume
$b^{s_\gamma} \subseteq \dom(p)$ and $\alpha$ is regular in $\mathcal{A}_0^{s_\gamma}$. 
In this case let $s_{\gamma+1}$ be ${s_\gamma} \conc j$, where $j$ is such that $p((i)_k) =j$ for cofinally many $i \in b^{s_\gamma}$.
This is well-defined as long as $\alpha$ is regular in $\mathcal{A}_0^{s_\gamma}$ ($b^s$ is not defined otherwise).

If $b^{s_\gamma} \not \subseteq \dom(p)$, or if $\mathcal{A}_0^{s_\gamma} \vDash \alpha$ is singular, end the construction and let $s(p)= s_\gamma$, $\gamma(p)=\gamma$
and $\mu(p)=\mu^{s_\gamma}$.

\item [Steering Ordinals at singulars]\index{steering ordinals}

We define by recursion on $\xi \in [\alpha,\alpha^+)$ the ordinals $\mu(A,p)^\xi$, $\tilde \mu(A,p)^\xi$ and $\mu(A,p)^{<\xi}$, denoted by
\index[notation]{mu(p) xi, mu(p) xi z, mu(p) <xi@$\mu(p)^\xi$, $\tilde \mu(p)^\xi$, $\mu(p)^{<\xi}$}
$\mu(p)^\xi$, $\tilde \mu(p)^\xi$ 
and $\mu(p)^{<\xi}$ when $A$ is clear from the context :

If %
$p((\beta^+)_2)=1$ for a cofinal set of $\beta \in E_\alpha$, we let
$\mu(p)^{<\alpha} =  \mu(p)$ and we say $p$ \emph{uses virtual inaccessible coding}\index{virtual inaccessible coding|textbf}\index{inaccessible coding!virtual|textbf}.
Conditions with this property have to be dealt with when we prove quasi-closure.

Otherwise, let
$\mu(p)^{<\alpha}$ be the least $\mu$ such that $E_\alpha, \Diamond_\alpha \in L_{\mu}[A\cap\alpha]$, i.e., 
$\mu(p)^{<\alpha} = \mu^{<\alpha}$.
We say in this case that $p$ \emph{immediately uses singular coding}\index{immediately uses singular coding}.
Such  conditions are built in a straightforward manner in the proof of Extendibility~\ref{ext1}, below.

For $\xi > \alpha$, let 
\[
\mu(p)^{<\xi} = \sup_{\nu<\xi} \mu^\nu.
\]
and for $\xi \geq \alpha$ let, first let
$\sigma(p)^\xi > \mu^{<\xi}$ be least such that $L_{\sigma(\xi)}\vDash$``$ \alpha$ is the greatest cardinal and $\alpha\in \Sing$''.
Now let $\tilde \mu(p)^\xi = \sigma(p)^\xi + \alpha$ and $\mu(p)^\xi= \sigma(p)^s + \omega\cdot \alpha $.
Notice that by induction $\xi \leq\mu(p)^{<\xi}$ (equality may hold if $\xi$ is a limit)
and $\mu(p)^{\xi}\geq \mu_0^\xi$.

Again, write $\mu(p)^s, \tilde \mu(p)^s, \mu(p)^{<s}$
\index[notation]{mu(p) s, mu(p) s z, mu(p) <s@$\mu(p)^s$, $\tilde \mu(p)^s$, $\mu(p)^{<s}$} 
for $\mu(p)^{|s|}, \tilde \mu(p)^{|s|}, \mu(p)^{<|s|}$.

We say \emph{$p$ recognizes the singularity of $\alpha$}\index{recognizes the singularity of $\alpha$} if and only if $\alpha$ is singular in $\mathcal{A}_0^{s(p)}$. 
Observe this is never relevant when $p$ immediately uses singular coding.

\item[Coding structures at singulars]\index{coding structures}

We let\index[notation]{B(A,p,) s, B(p) s, B(p) s tilde@$\mathcal{B}(A,p)^{s}$, $\mathcal{B}(p)^{s}$, $\mathcal{\tilde B}(p)^s$}
\begin{gather*}
\mathcal{B}(A,p)^{s}=\mathcal{B}(p)^{s} = L_{\mu(p)^s}[A\cap\alpha, s(p), s]\\
\mathcal{\tilde B}(p)^s = L_{\tilde\mu(p)^s}[A\cap\alpha, s(p), s],
\end{gather*}
where we set $s(p)=\emptyset$ if $p$ immediately uses singular coding.
Note that $\mathcal{B}(p)^{s}\models \alpha\in \Sing$ whenever $|s| > 0$ or $p$ immediately uses singular coding.
Again note that $E_\alpha \in \mathcal{B}(p)^{\emptyset_\alpha}$ and $\mathcal{A}_0^s \subseteq \mathcal{B}(p)^s$.

\item[Coding apparatus at singulars]\index{coding apparatus}

For $s \in S_\alpha$, let 
\[
H(A,p)_i =H(p)_i = h^{B(p)^s}_{\Sigma_1}(i \cup \{ A\cap \alpha, s(p), s\})
\]
and let
$f(p)^s(i)$ be the order type of $H(p)_i \cap \On$.
Note again that $f(p)^{\emptyset_\alpha} = f^{\emptyset_\alpha}$.

\item[Precoding]\index{precodes}

We say \emph{$X$ precodes $s$} (relative to $p$) to mean that $X$ is set of (G\"odel-codes for) $\Sigma_1$-formulas
with parameters from $\alpha\cup\{A\cap\alpha,s\}$ which are true in
$\mathcal{\tilde{B}}(p)^{s}$.

\item [Singular coding with delays]\index{coding delays|textbf}\index{limit coding!singular}

We define $t(A,p)=t(p)\in S_\alpha$.
To this end, we define a sequence $(t(p)_\xi)_{\xi<\delta(p)}$.
Let $t(p)_0 = \emptyset$. 
For limit $\delta$, let $t(p)_\delta = \bigcup_{\delta'<\delta}t(p)_{\delta'}$.
Now suppose $t=t(p)_\delta$ is defined.
For $\gamma\in \Card$ let
$f^{t}_p(\gamma)$ be the least $\eta$ such that $f(p)^{t}(\gamma) \leq \eta < \gamma^+$ and $p((\eta)_3)=1$.
If $f^{t}_p(\beta^+)$ is undefined for cofinally many $\beta \in E_\alpha$, let $\delta(p)=\delta$ and $t(p) = t(p)_\delta$.
Otherwise, if possible define $X=X(p,\delta) \subseteq \alpha$ by $\eta \in X$ if and only if $p( (f^{s_\gamma}_p(\beta) + 1 + \eta )_3 ) = j$ for a tail of successor cardinals $\beta <\alpha$.
If $X(p,\delta)$ is undefined, again stop the construction at $\gamma$ and let $\delta(p)=\delta$ and $t(p)=t_\delta$.\footnote{%
This will never occur.
Note that it follows from later definitions that the construction of $t(p)$ finishes only for one reason, namely that $f^{t}_p(\beta^+)$ is undefined for cofinally many $\beta \in E_\alpha$. 
For since we require $p\res\alpha$ exactly codes $p_\alpha$, the construction cannot stop before we have $t(p)=p_\alpha$, where we 
halt for the given reason, by the requirement that $p\res\alpha \in \mathcal{B}(p_{<\alpha})^{p_\alpha}$.}

If $[X]^{-1}_0$ precodes, relative to $p$, a string $t \in S_\alpha$ which end-extends $t(p)_\delta$ and such that $f^{s_\gamma}_p \in \mathcal{B}(p)^{t}$, let
$t(p)_{\delta+1} = t$.

Otherwise let $t= {t(p)_\delta }^\frown  [X]_3$.
If $f^{s_\gamma}_p \in \mathcal{B}(p)^{t}$, let
$t(p)_{\delta+1} = t$. 
If not, again stop the construction at $\gamma$ and let $\delta(p)=\delta$ and $t(p)=t_\delta$ (we will later see this case never occurs).
We say \emph{ $p$ exactly codes $t$} if and only if $t = t(p)$.
\end{description}
\end{dfn}

\begin{dfn}[Decoding]\label{decoding_def}\index{decoding}
Assume $\beta_0\in \Card$, $A'\subseteq \beta_0$ and $s  \colon [\beta_0,{\beta_0}^+) \rightarrow 2$, or $s\colon \beta_0 \rightarrow 2$ and $\beta_0$ is the least Mahlo.
We now describe the process of \emph{decoding} a
set $A^*_0=A^*_0(A', s)$, from $A' $ and $s$.

In a $P$-generic extension $L[A_0][G]$, setting $\beta_0=\emptyset$ and $s= \bigcup_{p \in G} p_0$, decoding allows us to recover $A_0$ (making it definable). 
We shall also run this process ``locally'' in a transitive model $\bar M$, i.e.,
we shall look at $A^*_0( A', p_{\beta_0})^{\bar M}$ (in order to achieve localization, and in the proof of \ref{jensen_qc_main}).
In this case, of course $\kappa$ should be interpreted to mean the least Mahlo in $\bar M$. 
Likewise, the definition of coding structure is to be interpreted from the point of view of $\bar M$, etc.

The definition of $A^*_0$ is by induction on cardinals $\beta \in [\beta_0, \kappa^+)$.
If $\beta_0 < \kappa$, we start by constructing $A^*\subseteq \kappa$ and $s_\beta$, for $\beta <\kappa$. 
Set $A^*\cap\beta_0 = A'$ and $s_{\beta_0} = s$.

If $\beta \in [\beta_0, \kappa)$ and we have constructed $A^*\cap \beta$ and $s_\beta \colon [\beta,{\beta}^+) \rightarrow 2$, define
$s_{{\beta}^+} \colon [{\beta}^+,{\beta}^{++}) \rightarrow 2$ via the decoding process at regulars described in Definition~\ref{sing_defs} from $s_{\beta}$, i.e., let $s_{{\beta}^+}= s(A^*\cap\beta, s_{\beta})$.
Also, let $\xi \in A^*\cap [\beta, \beta^{+3})\iff s_\beta( ( \langle \xi, \nu \rangle )_0)=1$ for $\beta^+$-many $\nu \in [\beta, \beta^+)$.

At limit $\beta < \kappa$, let $s_{<\beta} = \bigcup_{\beta' < \beta} s_{\beta'}$ and let
$s_\beta = s(A^*\cap\beta,s_{<\beta})$ when $\beta$ is inaccessible 
and $s_\beta=t( A^*\cap\beta, s_{<\beta})$ when $\beta$ is singular.

When $\beta > \beta_0$ and $\beta$ is the least Mahlo, supposing we have already constructed $A^* \cap \beta$ first extract $A^*_0\cap \beta$ and the generic for the Mahlo-coding from $A^* \cap \beta$: Let
$A^*_0 \cap \beta = [A^*\cap\beta]^{-1}_0$, let $s_{<\beta} = [A^*\cap\kappa]^{-1}_1$, 
and let $s_\beta = s(A^*_0 \cap \beta, s_{<\beta})$.
When $\beta = \beta_0$ is the Mahlo and we are at the beginning of the induction, we just set $A^*_0\cap\beta=A'$ and $s_{<\beta} = s$.
Let $\xi \in A^*_0\cap [\kappa, \kappa^{+})\iff s_\kappa( ( \langle \xi, \nu \rangle )_0)=1$ for $\kappa^+$-many $\nu \in [\kappa, \kappa^+)$.

Finally, for $\beta^+= \kappa^+$, we continue with $A^*_0$ instead of $A^*$: let
$s_{\kappa^{+}} = s(A^*_0\cap\kappa^+, s_{\kappa})$
and let $A^*_0\cap [\kappa^+, \kappa^{++})$ be the set whose characteristic function is $s_{\kappa^{+}}$.
The set $A_0^* \subseteq \kappa^{++}$ is the outcome of the decoding procedure run up to $\kappa^{++}$ and we write $A^*_0(A', s)$ for this set.

\end{dfn}
Localization\index{localization|textbf} is achieved by ``thinning out'' the sets $S_\alpha$ and $S_{<\kappa}$ (this is also called  \emph{David's trick}, or \emph{killing universes})\index{David's trick|textbf}:
\begin{dfn}[Strings, localization]\label{strings_dfs}\index{string}
We now define $S^*_\alpha= S(A)^*_\alpha$\index[notation]{S* alpha@$S^*_\alpha$}.
Let $\Psi_0(r)$ be the statement ``$ r(n)=i \Rightarrow$ for cofinally many $\sigma <\bar\kappa$, $T(\sigma,n,i,j)$ has a branch, where $\bar T$ is the canonical\index{canonical!sequence of trees}\index{Suslin trees!canonical sequence of} $\bar\kappa$-sequence of $\bar\kappa^{++L}$-Suslin trees and $\bar \kappa$ is the least Mahlo in $L$.''
We also say \emph{$r$ is coded by branches}\index{coded by branches} for $\Psi_0(r)$\index{Psi_0 r@$\Psi_0(r)$}.

We say $s \in S^*_\alpha$ if and only if $s \in S_\alpha$ and for all $\zeta \leq |s|$ and all $\eta, \bar \kappa$
such that $N= L_\eta[A\cap\alpha, s\res \zeta]\vDash$``$\zeta =\alpha^+$, $\bar \kappa$ is the least Mahlo in $L$, $\bar \kappa^{++}$ exists, $\Card = \Card^{L[A\cap\omega]}$, $\Card\setminus\omega_1=\Card^{L}\setminus\omega_1$ and $\ZF^-$ holds'',
we have that $L[A^*]\vDash \Psi_0(r)$, where $A^*=A^*(A\cap\alpha, s\res\zeta)$ (i.e., the predicate coded by $s\res\zeta$ in $L_\eta[A\cap\alpha,s\res\zeta]$).

When $s \in S_\alpha$, we call $N$ as in the hypothesis a \emph{test model for $s$}\index{test model}.
Thus, $s \in S^*_\alpha$ if and only if for every test model $N$ for $s$, $\big(L[A^*(s\res\alpha^{+})]\big)^N\models r$ is coded by branches.\footnote{Observe that requiring $N \models \ZF^-\wedge\exists  \bar\kappa^{++}$ ensures that $L^N$ thinks that the canonical $\bar\kappa$-sequence of $\bar\kappa^{++}$-Suslin trees exists, but this is not the only reason to make this requirement.
Rather, by finding $s$ which collapses $|s|$ to $\alpha$ ``quickly'', this requirement
gives us some control over the set of test models. Compare Example~\ref{s.e.david}.}

We say $s \in S^*_{<\kappa}$ if and only if $s \in S_{<\kappa}$ (i.e., $s\colon [0,|s|) \rightarrow 2$, $|s| < \kappa$) and for all $\bar \kappa \in \Card\cap |s|+1$ and all $\eta$
such that $N= L_\eta[A_0\cap\bar\kappa, s\res \bar\kappa]\vDash$ ``$\bar \kappa$ is the least Mahlo in $L$, $\bar \kappa^{++}$ exists, $\Card = \Card^{L[A_0\cap\omega]}$, $\Card\setminus\omega_1=\Card^{L}\setminus\omega_1$ and $\ZF^-$ holds'', we have that $N\vDash \Psi_0(r)$.
Similarly to the above, when $s \in S_{<\kappa}$ we call $N$ as in the hypothesis a \emph{${<}\kappa$-test model for $s$}\index[notation]{<kappa-test model@${<}\kappa$-test model}\index[notation]{kappaa-test model@${<}\kappa$-test model}.
\end{dfn}

\begin{dfn}[Building blocks]\label{reg_codings_dfs}
We now define three partial orders, taking care of coding and localization at regulars. 
These partial orders serve as building blocks for the final forcing.
We distinguish $\kappa$, inaccessibles below $\kappa$, and successors.
\begin{description}
\item [Successor coding]\index{successor coding|textbf}

Let either $\beta=\kappa$, $\alpha = \kappa^+$ and $s=s_{\kappa^+}$ or let $\beta \in \Card \cap \kappa$, $\alpha= \beta^+$ and $s \in S^*_\alpha$ and let $A\subset\alpha$.
We define the partial order $R^s=R(A)^s$ to consist of conditions $(p,p^*)$ such that
$p \in S^*_\beta$ and $p^* \subseteq \{ b^{s\res\xi} \setdef \xi \in [\alpha, |s|)  \} \cup |p|$ of size at most $\beta$.
It is ordered by: $(q,q^*)\leq (p,p^*)$ if and only if $q$ end-extends $p$, $p^* \subseteq q^*$
and
\begin{enumerate}
\item If $b^{s\res \xi} \in p^*$ and $s(\xi)=0$ then for any $\gamma \in (|p|,|q|)\cap b^{s\res\xi}$ we have $q((\gamma)_1)=0$. 
\item If $\xi \in p^* \cap |s|$ and $\xi \in A$ then if $\gamma$ is such that $\langle \xi, \gamma \rangle \in (|p|,|q|)$, we have $q((\langle \xi, \gamma \rangle)_0)=0$. 
\end{enumerate} 

\item[Inaccessible coding]\index{inaccessible coding}\index{limit coding!regular}
Let $\alpha\leq \kappa$ be inaccessible, $s \in S^*_\alpha$.
We define the partial order $R^s=R(A)^s$\index[notation]{R s@$R^s$} to consist of conditions $(p,p^*)$ such that
$p$ is a partial function $p \colon \alpha \rightarrow 2$ of size less than $\alpha$ (note here that we allow ``broken'' strings!)\index{string!broken}\index{broken string}
 and $p^* \subseteq \{ b^{s\res\xi}\setminus\eta \setdef \xi \in [\alpha, |s|), \eta < \alpha \} \cup \alpha$ of size less than $\alpha$ and $\alpha\cap p^*$ is an ordinal.
We write $\rho(p^*)$\index[notation]{rho(p*)@$\rho(p^*)$} for this ordinal.
For $\alpha = \kappa$,  we additionally demand that $p \in S^*_{<\kappa}$.
The ordering on $R^s$ is given by: $(q,q^*)\leq (p,p^*)$ if and only if $q$ end extends $p$, $p^* \subseteq q^*$
and
\begin{enumerate}
\item If $b \in q^* \setminus p^*$ then 
$$b \cap (\rho(p^*)\cup \sup\dom(p)) \subseteq \bigcup \{ b'\setdef b' \in p^*\}.$$
\item If $b^{s\res \xi} \setminus\eta\in p^*$ and $s(\xi)=0$ then for any $\gamma \in (|p|,|q|)\cap (b^{s\res\xi} \setminus \eta)$ we have $q((\gamma)_2)=0$. 
\end{enumerate} 
\end{description}
\end{dfn}
The additional construct $\rho(p^*)$ and its use in the definition of $\leq$ for inaccessible coding is necessary to preserve Requirement~\ref{coding_areas_bounded} in the definition of $P(A_0)$ (see below) when taking greatest lower bounds (see also Definition~\ref{D}, item \eqref{restraints_start_high}).
With the intuition that our iteration can \emph{almost} be decomposed into an upper and a lower part, this device limits the interaction caused by restraints between these two parts. 
This also means that the part of the restraint below $ \rho(p^*)\cup \sup\dom(p_{<\alpha})$ as well as $\rho(p^*)$ itself must part of the value of the ``linking function'' $\Cl(p)$ (see \ref{sec_coding_strat}).
Note that $\sup(\dom(p_{<\alpha}))$ is used similarly (but affects all inaccessible codings above it); this is convenient in the proof that $\kappa$ is not collapsed in our iteration (see \ref{it:reals:are:caught}).
The use of $\sup(\dom(p_{<\alpha}))$ renders the argument more elegant but can be eliminated.
The use of $\sup(\dom(p_{<\alpha}))$ cannot obviate the use of $\rho(p^*)$ because of the last clause in \eqref{qc:redundant}.

\section{Definition of the Forcing}\label{3steps}

\begin{dfn}\label{conditions}\index{Easton supported Jensen coding}\index{coding}\index{Jensen coding!with Easton support and localization}
Remember $\Card'$ denotes $(\Card \setminus \omega) \cup\{ \emptyset \}$ and $\emptyset^+=\omega$. 
A condition in $P=P(A_0)$\index[notation]{P(A_0)@$P(A_0)$} is a  sequence 
$$p = (p_{<\alpha}, p_\alpha, p^*_\alpha)_{\alpha \in \Card' \cap (\kappa^{+}+1)}$$
such that:
\begin{enumerate}[label=P.\arabic*), ref=P.\arabic*] 
\item $p_{\kappa^+}=s_{\kappa^+}$, where $s_{\kappa^+}$ is the characteristic function of $A_0\cap [\kappa^+, \kappa^{++}]$ and $p_\alpha \in S^*_\alpha$\index[notation]{p alpha@$\lvert p_\alpha\rvert$}\index[notation]{ p alpha @$\lvert p_\alpha\rvert$} for all $\alpha \leq \kappa$, while $p_{<\kappa} \in S^*_{<\kappa}$ (this is made redundant by \ref{P.succ} and \ref{P.inacc}). 
For cardinals $\alpha< \kappa$, we demand $p_{<\alpha} = \bigcup_{\delta < \alpha} p_\delta$ (so $p_{<\alpha}$ is redundant unless $\alpha = \kappa$).

We write $\supp(p) = \{ \alpha \setdef p_\alpha \neq \emptyset_\alpha \}$ and $\alpha(p) = \sup (\supp(p) \cap \kappa)$.
\item\index{Easton support|textbf}\index{support!Easton|textbf} $\supp(p) \cap \kappa$ is a subset of $|p_{<\kappa}|$ and an \emph{Easton set}, i.e., 
 whenever $\gamma \leq \kappa$ is inaccessible, $\sup(\supp(p) \cap \gamma) < \gamma$.

For $\alpha \in \{ \kappa, \kappa^+\}$ we introduce the abbreviations $\mathcal{A}^{p_\alpha} = \mathcal{A}(A_0)^{p_\alpha}$ (used when $A$ is clear from the context).  
For $0<\alpha \leq |p_{<\kappa}|$, let $\mathcal{A}^{p_\alpha}_0 = \mathcal{A}(A_p)^{p_\alpha}_0$  and $\mathcal{B}(p_{<\alpha})^{p_\alpha}=\mathcal{B}(A_p,p_{<\alpha})^{p_\alpha}$, where
\begin{equation*}
A_p = A_0\oplus \{ \xi < |p_{<\kappa}| \setdef p_{<\kappa}(\xi)=1\}.
\end{equation*}  
These coding structures (when defined) do not change when $p$ is extended unless $p_\alpha$ is extended---in particular the dependence on $A_p$ is ``static''. In the singular case, this is not apparent now but is ensured by \ref{P.recognizes} below.
In the following, $R(A_p)^s$ is defined relative to these coding structures; we shall write $R^s$ instead of $R(A_p)^s$.
\item \label{P.succ} For all $\alpha \in \Card'\cap(\kappa+1)$, $(p_\alpha, p^*_{\alpha^+}) \in R^{p_{\alpha^+}}$, where we let $R^{p_\omega}$ denote the standard almost disjoint coding of $p_\omega$ by a real relative to some convenient almost disjoint family in $L[A_p\cap\omega]$.
\item \label{P.inacc} For all inaccessible $\alpha \leq\kappa$, $(p_{<\alpha}, p^*_{\alpha}) \in R^{p_\alpha}$, 
remembering $p_{<\alpha}=\bigcup_{\alpha'<\alpha} p_\alpha$ for $\alpha < \kappa$.
\item\label{P.recognizes} For all singular $\alpha <\kappa$, $p\res \alpha \in \mathcal{B}(p_{<\alpha})^{p_\alpha}$, $t(p_{<\alpha})=p_\alpha$ and if $p_{<\alpha}$ is unbounded below $\alpha$ and uses virtual inaccessible coding, then $p_{<\alpha}$ recognizes singularity of $\alpha$.\footnote{\label{why_stop_here}We let the virtual inaccessible coding stop when the singularity of $\alpha$ is realized; without this natural stopping point, the coding structure for singulars will change when a condition is extended below $\alpha$, i.e., $\mathcal{B}(q_{<\alpha})^{\emptyset_\alpha} \neq \mathcal{B}(p_{<\alpha})^{\emptyset_\alpha}$ for $q\leq p$, because $q$ might carry new information in $t(q_{<\alpha})$.
Intuitively, $p$ hasn't exhausted the room for virtual inaccessible coding.}
\item \label{coding_areas_bounded} For $\alpha\in\Reg$, there is a $\gamma < \alpha$ such that for all $\delta\in \Inacc\setminus \alpha+1$ and all  $(\xi, \eta)\in p^*_\delta$ we have $b^{p_\delta\res\xi}\setminus \eta \cap \alpha \subseteq \gamma$.

\end{enumerate} 
Note again we have $p_{\kappa^+}=s_\kappa$ for every condition $p \in P$.
We say $q \leq p$ if and only if for all $\alpha \in \supp(p)$, $(q_\alpha, q_{\alpha^+})\leq (p_\alpha, p_{\alpha^+})$ in  $R^{p_{\alpha^+}}$ and for all inaccessible $\alpha \leq \kappa$,
$(q_{<\alpha}, q_{\alpha})\leq (p_{<\alpha}, p_{\alpha})$ in $R^{p_\alpha}$.
\end{dfn}
Observe that the limit coding from $\kappa^+$ into $\kappa$ takes a slight detour, via $p_{<\kappa}$;
the reason for this is that the coding into $\kappa$ is easier if we use strings (because of localization). 
In contrast, $p_{<\alpha}$ can be seen as a shorthand for $\bigcup_{\delta < \alpha} p_\delta$; 
we use these ``broken strings'' to code into inaccessible $\alpha < \kappa$,
which is the natural choice.

We will need that being a condition, and in fact all of the definitions in \ref{reg_defs} and \ref{sing_defs} are absolute for $\Sigma^A_1$-correct models.
In fact, all these notions are Boolean combinations of $\Sigma^A_1$ statements.\footnote{So is $\perp$, the incompatibility relation. Of course, we could just use $\langle P, \leq, \perp\rangle$ as a parameter.}

For $p \in P$ and $\alpha \in \Card$, 
let $B_p^\alpha\subseteq \alpha$ be defined by
\[
B^\alpha_p=\bigcup\{ b\cap\alpha \setdef b=b^s\setminus \eta, b \in p^*_\gamma,  \gamma\in \Inacc\setminus\alpha+1\},
\]
and let $b_p \in \prod _{\alpha \in \Card\cap\kappa} \alpha$ be defined by
\begin{align*}
b_p(\alpha) =\sup B^\alpha_p
\end{align*}
That is, $B_p^\alpha$ is the set of coding ordinals used by inaccessibles above $\alpha$, and 
$b_p$ locally bounds the height of this set.
\begin{rem}
Note that \ref{coding_areas_bounded} is equivalent to asking
$b_p(\alpha) < \alpha$ for every $\alpha \in \Reg$.
\end{rem}

\subsection{Doing the same in three steps.}

To aid the readers intuition, we describe the same forcing as a three-step iteration.
This illustrates further the special role of $p_{<\kappa}$ and the detour in the Mahlo coding, 
but will not be needed for the rest of the proof.

The first step is to force with the successor coding $R(A_0)^{s_{\kappa^+}}$.
It follows from later theorems that this forcing is stratified.
Let $G_1$ be the generic and let 
\begin{gather*}
s_{\kappa} = \bigcup \{ p \setdef \exists p^* \; (p, p^*) \in G_1 \}\\
A_1 = \{ \xi \in [\kappa,\kappa^+) \setdef s_\kappa(\xi)=1\} \cup (A_0 \cap \kappa).
\end{gather*}

In the next step we force with the Mahlo coding $R(A_1)^{s_\kappa}$ in $L[A_1]=L[A_0][G_1]$.
It will be implicit in work below that no cardinals collapse and $\kappa$ is still Mahlo.
Let $G_2$ be the generic and let 
\begin{gather*}
s_{<\kappa} = \bigcup \{ p \setdef \exists p^* \; (p, p^*) \in G_2 \}\\
A_2 = \{ \xi \in \kappa \setdef s_{<\kappa}=1\} \oplus (A_0\cap\kappa).
\end{gather*}
Observe that $A_2$ is `localized', i.e., satisfies the following:
\begin{eqpar}
for all $\bar \kappa\in\Card$, $\eta \in \On$
such that $N= L_\eta[A_2\res \bar\kappa]\vDash$ ``$\bar \kappa$ is the least Mahlo, $\bar \kappa^{++}$ exists and $\ZF^-$ holds'', we have that $N\vDash \Psi_0(r)$.
\end{eqpar}

Lastly, we define $P(A_2)_{<\kappa}$ which codes $A_2$ by a subset of $\omega$.
For this sake, let all coding structures and the partial orders from \ref{reg_codings_dfs} be defined relative to $A=A_2$.

\begin{dfn}[The coding below $\kappa$]\label{conditions_3steps}
A condition in $P(A_2)_{<\kappa}$ is a function $p\colon \Card'\cap \kappa\rightarrow V$, $p(\alpha)=(p_\alpha, p^*_\alpha)$ such that
\begin{enumerate} 
\item $\supp(p) = \{ \alpha < \kappa \setdef p_\alpha \neq \emptyset \}$ is an Easton set.
\item For all $\alpha \in \Card'\cap\kappa$, $(p_\alpha, p^*_{\alpha^+}) \in R^{p_{\alpha^+}}$, where $R^{p_\omega}=R(A)^{p_\omega}$ is again some conveniently chosen partial order to code $p_\omega$ by a real.
\item For all inaccessible $\alpha <\kappa$, $(p_{<\alpha}, p^*_{\alpha}) \in R^{p_\alpha}$, 
where we define $p_{<\alpha}=\bigcup_{\alpha'<\alpha} p_\alpha$ for any $\alpha$.
\item For all singular $\alpha <\kappa$, $p\res \alpha \in \mathcal{B}(p_{<\alpha})^{p_\alpha}$, $t(p_{<\alpha})=p_\alpha$ and if $p_{<\alpha}$ is unbounded below $\alpha$ and uses virtual inaccessible coding, then $p_{<\alpha}$ recognizes singularity of $\alpha$.\footnote{See footnote~\ref{why_stop_here}.}
\item %
For $\alpha\in\Reg$, there is a $\gamma < \alpha$ such that for all $\delta\in \Inacc\setminus \alpha+1$ and all  $(\xi, \eta)\in p^*_\delta$ we have $b^{p_\delta\res\xi}\setminus \eta \cap \alpha \subseteq \gamma$.

\end{enumerate} 
We say $q \leq p$ if and only if for all $\alpha \in \supp(p)$, $(q_\alpha, q_{\alpha^+})\leq (p_\alpha, p_{\alpha^+})$ in  $R^{p_{\alpha^+}}$ and for all inaccessible $\alpha <\kappa$,
$(q_{<\alpha}, q_{\alpha})\leq (p_{<\alpha}, p_{\alpha})$ in $R^{p_\alpha}$.
\end{dfn}
This ends our description of $P$ as a three-step iteration.
We find it more convenient to talk about $P$ instead of this iteration,
and this is the forcing we work with in all of the following.

\section{Extendibility}

\begin{lem}[Extendibility for the Mahlo coding]\label{ext_mahlo}\index{Mahlo coding}
Let $t \in S_\kappa$ or 
$$t\colon [\kappa, \kappa^+) \rightarrow 2.$$
For any $\alpha< \kappa$ and any 
$p \in R^{t}$ there is $q \in  R^{t}$ such that $q \leq p$ and $|q_{<\kappa}| \geq \alpha$.
\end{lem}
\begin{proof}
Let $\alpha$ be least such that the above fails, and let $p$ be given.
Of course, the only difficulty is to meet the requirement $q\in S^*_{<\kappa}$.
If $\alpha \not \in \Card$ it suffices to extend $p_{<\kappa}$ to $p_{<\kappa}^1$ with 
$|p_{<\kappa}^1| \geq \lVert \alpha \rVert$.
We can then further extend $p_{<\kappa}^1$ by appending 0s to obtain $q$.
If $\alpha \in \Succ$, the proof is similar (as no test model will think that $\alpha$ is Mahlo).
So assume $\alpha$ is a limit cardinal.

Let $C \subset \alpha$ be club in $\alpha$, $\otp C = \cof \alpha $, $C \subseteq \Sing$ and
let $(\beta_\xi)_{\xi \leq \rho}$ be the increasing enumeration of $C\cup\{\alpha\}$.
We can assume $\beta_0 > |p_{<\kappa}|$ and $\beta_0 > \rho$ if $\alpha$ is singular.
Let $p^0 \leq p$ such that $|p^0_{<\kappa}|=\beta_0$ and build a descending chain of conditions.
Assume you have $p^\xi$ such that $|p^\xi_{<\kappa}|=\beta_\xi$.
Extend to get $p'$ with $|p'_{<\kappa}| = \beta_{\xi+1}$.
Let $p^{\xi+1}$ be obtained from $p'$ by shifting values of $p'_{<\kappa}$ above $\beta_\xi$ away from the cardinals in a gentle manner, putting a 1 on $\beta_{\xi}$ and padding with 0s, as follows: let $p^{\xi+1}_{<\kappa} \res\beta_\xi = p'_{<\kappa} \res\beta_\xi = p^{\xi}_{<\kappa} \res\beta_\xi$ and for $\nu \in [ \beta_\xi, \beta_{\xi+1})$, let
\[
p^{\xi+1}_{<\kappa}(\nu)=\begin{cases}

    p'_{<\kappa}(\delta + k)&\text{if $\nu = \delta + k +1 $ for some $\delta \in 
\Card$, $k \in \omega$,}\\
1 &\text{if $\nu = \beta_{\xi}$,}\\
0 &\text{if $\nu \in \Card\cap (\beta_\xi, \beta_{\xi+1})$.}\\
  p'_{<\kappa}(\nu) &\text{otherwise.}
\end{cases}
\]
Of course we let $(p^{\xi+1})^*=(p')^*=(p^\xi)^*$.
Observe that $p^{\xi+1}\leq p^{\xi}$ since no restraints $b \in (p^\xi)^*$ are violated, as we have $\delta + k \not\in [b]_3$ for $\delta \in\Card$, $k \in \omega$.  
We still have $p^{\xi+1}_{<\kappa} \in S^*_{<\kappa}$, as $L_\eta[A\cap\bar \kappa, p^{\xi+1}_{<\kappa}\res\bar\kappa] =L_\eta[A\cap\bar \kappa, p'_{<\kappa}\res\bar\kappa] $ for all relevant $\eta, \bar\kappa$.
At limit $\xi\leq \rho$, if $N = L_\eta[ p^{\xi}_{<\kappa}\res\bar\beta_\xi]$ is a test model,
note that $C\cap \beta_\xi \in N$ and so $N \models \beta_\xi$ is not Mahlo.
Thus $q = p^\rho$ is as desired.
\end{proof}

\begin{lem}[Extendibility for the successor coding.]
For each $p \in P$, $\xi \in \kappa^+$ there is $q\leq p$ such that $\xi \in \dom(q_{<\kappa^+})$.
\end{lem}
\begin{proof}
Note that we can assume that $\lVert\xi \rVert\in \supp(p)$.
To avoid repetition, we leave the rest of the proof as a by-product of Lemma~\ref{D}:
any $q \in \D^{L[A]}_{[\lVert\xi \rVert, \xi^+)}(p,\{\xi\})$ will do.
\end{proof}

The next lemma allows us to extend conditions at a singular cardinal $\alpha$ without violating that the extension be coded exactly below $\alpha$.
By the last item below, we may at the same time capture a set $X$ locally (in a sense);
this included for completeness and is not needed in the rest of the proof.

\begin{lem}[Extendibility at singulars]\label{ext1}
Let $p\in P$, $\alpha \in \Sing\cap |p_{<\kappa}|$ and $s \in S(A_p)_\alpha$ such that $p_\alpha \is s$, further suppose $X \subseteq \On$ and %
$ X\in \mathcal{B}(p_{<\alpha})^s$.
We can find $q \in P$, such that 
\begin{itemize}
\item $q\res(\alpha, \infty] = p\res(\alpha, \infty]$, 
\item $q_\alpha = s$, 
\item for each each $\beta$ which is a limit point of $E_\alpha$ (in fact, for all $\beta=\delta^+$ for some $\delta \in E_\alpha$), $X\cap\beta \in \mathcal{B}(q_{<\alpha})^{q_\beta}$. \end{itemize}
\end{lem}
We write $A=A_p$ for the discussion of the above lemma.
Before we proof it, we make two technical observations:
\begin{lem}\label{codingapp_define}
Let $s,t \in S_\alpha$, $p$ be a condition such that $p\res \alpha \in \mathcal{B}(p_{<\alpha})^s$ and $s$  a proper initial segment of $t$. Then
$f(p_{<\alpha})^s \in \mathcal{B}(p_{<\alpha})^t$. In fact, 
$f(p_{<\alpha})^s \in L_{\mu(p_{<\alpha})^s + \alpha +1}[A\cap\alpha,s(p),s]$.
\end{lem}
\begin{proof}
This is because the Skolem hulls are definable over $\mathcal{B}(p_{<\alpha})^s$, 
and the transitive collapse of the Skolem hull of $\gamma < \alpha$ is constructible at most $f(p_{<\alpha})^s(\gamma) + 1$ steps above $\mathcal{B}(p_{<\alpha})^s$, 
by the recursive definition of the transitive collapse. 
Thus, $f(p_{<\alpha})^s$ is a definable subset of $L_{\nu(p_{<\alpha})^s + \alpha}[A\cap\alpha,s]$, and
\[
f(p_{<\alpha})^s \in L_{\nu(p_{<\alpha})^s + \alpha+1}[A\cap\alpha,s].
\]
Since possibly, $\nu(p_{<\alpha})^s + \alpha = \tilde\nu(p_{<\alpha})^t$, it needn't be the case that $f(p_{<\alpha})^s \in \mathcal{\tilde B}(p_{<\alpha})^t$
but clearly, $f(p_{<\alpha})^s \in \mathcal{B}(p_{<\alpha})^t$.
\end{proof}

Note that if the length of $\lh(t) > \lh(s)+1$, then we even have $f(p_{<\alpha})^s \in \mathcal{B}(p_{<\alpha})^{<t}$.
\begin{lem}\label{coding_app_stable}
If  $\beta\in\Card\cap\alpha+1$, $t \in S_\beta$, $\mathcal{B}(p_{<\beta})^{t} = \mathcal{B}(q_{<\beta})^{t}$ and $f(q\res\beta)^t =  f(p\res \beta)^t$.
\end{lem}
\begin{proof}
Fix $t$ and $\beta$ as in the first statement.
First, assume $p\res\beta$ does not use virtual inaccessible coding.
Then neither does $q$.
In this case, $\mu(q_{<\beta})^{\emptyset_\beta}=\mu^{\emptyset_\beta}$ and $p\res\beta$ and $q\res\beta$ play no role at all in the definition of $f(q_{<\beta})^t$ and $f(p_{<\beta})^t$.

In the other case, when $p\res\beta$ uses inaccessible coding, by definition we have
$b^{\bar s} \subseteq \dom(p)$ for any $\bar s \is s(p\res\beta)$. 
But then as $q \leq p$,
$s(p\res\beta)=s(q\res\beta)$ and again $\mu(q_{<\beta})^{\emptyset_\beta}=\mu(p_{<\beta})^{\emptyset_\beta}$ by definition.
Thus in this case as well, we have $f(q_{<\beta})^t = f(p_{<\beta})^t$.
\end{proof}

We can now prove extendibility at singulars cardinals.
\begin{proof}[Proof of Lemma~\ref{ext1}]
The proof is by induction on $\alpha$.
Find
\[
M = L_\mu[A\cap\alpha,s]
\] 
so that $\mu$ is a limit ordinal, $\tilde \mu(p_{<\alpha})^s < \mu < \mu(p_{<\alpha})^s$
and $p, X, b^\alpha_p \in M$.
Let
\[
H_\beta = h^M_{\Sigma_1}(\beta\cup\{ A\cap\alpha, s\}),
\]
and let $\pi_\beta\colon H_\beta \rightarrow \bar M_\beta$ be the transitive collapse. 
Let $g$ be defined by $g(\beta)=(\beta^+)^{\bar M_\beta}$ for successor cardinals $\beta <\alpha$, noting that this is well defined since $\beta \in H_\beta$.

Pick $Y$ such that $[Y]^{-1}_0=\Th_{\Sigma_1}^{\mathcal{\tilde B}(p_{<\alpha})^s}(\alpha \cup \{ A\cap\alpha, s\})$, 
while $[Y]^{-1}_1=\Th_{\Sigma_1}^{M}(\alpha \cup \{ A\cap\alpha, s\})$. 
Moreover, demand that $[Y]^{-1}_2 = X$.

Observe that $Y$ precodes $s$.
Observe also that for limit cardinals $\beta < \alpha$, $[Y]^{-1}_0\cap\beta$ correspond to the G\"odel-numbers of true $\Sigma_1$ sentences of $\mathcal{\tilde B}(p_{<\alpha})^s$ with parameters from $\beta\cup\{A\cap\alpha, s\}$. Similarly for $[Y]^{-1}_1\cap\beta$.
Also observe that for large enough $\beta < \alpha$, we never have $H_\beta \cap\alpha = \beta$.
This means also that $Y\cap \beta$ does not precode a $t\in S_\beta$, for $[Y]^{-1}_0\cap\beta$ codes a model where $\beta^+$ exists.

\medskip

We now define $q \leq p$; the same construction works, whether or not $p_\alpha \neq \emptyset_\alpha$. 
For $\beta \in E_\alpha$, let 
$q_{\beta^+} = p_{\beta^+} \conc 0^{g(\beta)} \conc 1\conc ( [Y]_3\cap\beta)$ (where $ 0^{\nu}$ denotes a string of $0$'s of length $\nu$).
For $\beta$ which is a limit point of $E_\alpha$, let $q_\beta = p_\beta \conc ( [Y]_3\cap\beta)$.
Note that clearly, if $p$ didn't use virtual inaccessible coding, neither does $q$.

We now show that this definition works for $\beta$ large enough.
First, note that
\[
H_\alpha = h^M_{\Sigma_1}(\alpha\cup\{ A\cap\alpha, s\}) = M,
\]
since $M$ can be characterized as the minimal model such that the $\Sigma_1$ statement
asserting the existence of $\mathcal{\tilde B}(p_{<\alpha})^s$ and certain statements describing the height of the $M$ via ordinal addition hold inside $M$. Thus, $M = \bigcup_{\beta < \alpha} H_\beta$.
\begin{lem}\label{exact_coding_collapse}
Assume $\beta$ is large enough so that $p \in H_\beta$, and also large enough so that letting
$\tilde \beta = \min H_\beta \cap [\beta, \alpha]$, we have $\tilde \beta \neq \alpha$.
Then $\pi_\beta(p_{\tilde \beta}) = p_\beta$.
\end{lem}
\begin{proof}
If $\tilde\beta=\beta$, there is nothing more to prove.
Otherwise, first note that $E_{\tilde\beta} \in H_\beta$ by elementarity and
$\pi_\beta(E_{\tilde\beta})=E_{\tilde\beta}\cap\beta=E_\beta$.
In fact, that $p\res\tilde\beta$ exactly codes $p_{\tilde\beta}$
is expressible as a $\Sigma_1$ statement inside $L_\alpha[A]$.
By elementarity, this statement also holds of $\pi_\beta(p \res\tilde\beta)$ and $\pi_\beta(p_{\tilde\beta})$ in $\bar M_\beta$ and is upwards absolute, so $\pi_\beta(p \res\tilde\beta)=p\res\beta$ exactly codes $\pi_\beta(p_{\tilde\beta})$.
But since by definition of $P$, $p\res\beta$ exactly codes $p_\beta$, we have $p_\beta = \pi_\beta(p_{\tilde\beta})$.
 \renewcommand{\qedsymbol}{{\tiny  Lemma~\ref{exact_coding_collapse}~}$\Box$}\end{proof}

\begin{lem}\label{l.sublemma:ext}
The following hold:
\begin{enumerate} 
\item 
$g\res \beta$ eventually dominates  $f(q\res\beta)^{p_{\beta}}$
for large enough $\beta \in\Card\cap\alpha$.
\item If $p_\alpha$  is a proper initial segment of $s$, $g$ eventually dominates  $f(q\res\alpha)^{p_\alpha}$.
\item On the other hand $f(q_{<\alpha})^s$ eventually dominates  $g$, 
in fact $g\in \mathcal{B}(q_{<\alpha})^{s}$.
\end{enumerate}
\end{lem}
\begin{proof}

Let $\beta \leq \alpha$ be a limit cardinal large enough so that $p \in H_\beta$. 
In case $\beta = \alpha$, note that $p_\alpha \in M$.
We may assume that either $\beta = \alpha$ or $H_\beta \cap \alpha \neq \beta$ (as $\alpha$ is singular in $M$).

Let $\tilde \beta = \min H_\beta \cap[\beta,\alpha]$.
For the case $\alpha = \beta$, we may assume $p_\alpha$ is a proper initial segment of $s$.
Then $f(p\res\alpha)^{p_{\tilde\beta}}\in M$ (by Lemma~\ref{codingapp_define} and the previous remark if $\beta = \alpha)$, and in fact it is the solution to a $\Sigma_1$-formula in $M$ with parameters $\tilde \beta$ and $p$ - or $p_\alpha$, in case $\beta = \alpha$. 
Thus $f(p\res\alpha)^{p_{\tilde\beta}} \in H_\beta$.
Suppose $\gamma<\beta$ is a successor large enough so that $f(p\res\alpha)^{p_{\tilde\beta}} \in H_\gamma$.
Observe $\gamma \in H_\gamma$.
Thus by elementarity, $\bar M_\gamma \vDash$``$f(p\res\alpha)^{p_\beta}(\gamma)$ has size $\gamma$''.

Finally, 
as $M \in \mathcal{B}(p_{<\alpha})^{s}$, clearly $g (\gamma) < \otp(\On \cap H_\gamma) < \otp (\On \cap H^s_\gamma)$ and so $g(\gamma) < f(p\res\alpha)^{s}(\gamma)$ for all $\gamma < \beta$.
Thus $g$ is slower than $f(p\res\alpha)^s$. \renewcommand{\qedsymbol}{{\tiny  Lemma~\ref{l.sublemma:ext}~}$\Box$}
\end{proof}

\medskip

We now show that for any limit $\beta \leq \alpha$, $q\res\beta \in \mathcal{B}(p_{<\alpha})^{q_\beta}$ and exactly codes $q_\beta$.
First, let $\beta=\alpha$.
Since $\mathcal{\tilde B}(p_{<\alpha})^{s} \in\bar M \in \mathcal{B}(p_{<\alpha})^{s}=\mathcal{B}(q_{<\alpha})^{q_\alpha}$, we have
$Y\in \mathcal{B}(q_{<\alpha})^{q_\alpha}$and $g \in \mathcal{B}(q_{<\alpha})^{q_\alpha}$ (because the height of the latter model is a limit mutiple of $\beta$, we can argue as in Lemma~\ref{codingapp_define}). 
So also $q\res\alpha \in \mathcal{B}(q_{<\alpha})^{q_\alpha}$.

Suppose $p_\alpha=s=q_\alpha$.
Then we have seen that $f(p\res\alpha)^{s}$ dominates $g$, so $q\res\alpha$ exactly codes $p_\alpha$ just as $p\res\alpha$ does.

Now suppose $p_\alpha$ is a proper initial segment of $s=q_\alpha$.
Moreover, we have seen that $g$ eventually dominates $f(p\res\alpha)^{p_\beta}$.
So $f(p\res\alpha)^{p_\alpha}_{q\res\alpha}$ is defined, 
the limit coding for $q\res\alpha$ goes on for one more step than for $p\res\alpha$ and we obtain $Y$ precoding $s=q_\alpha$.
Also, the argument of Lemma~\ref{codingapp_define} shows that $f(p\res\alpha)^{p_\beta}_{q\res\alpha} \in \mathcal{B}(p_{<\alpha})^{s}$.
Observe that also $\mathcal{B}(p_{<\alpha})^{s}= \mathcal{B}(q_{<\alpha})^{q_\alpha}$ by Lemma~\ref{coding_app_stable}.
Since also $Y \in \mathcal{B}(q_{<\alpha})^{q_\alpha}$, $q\res\alpha$ codes $q_\alpha$.
As $q\res\alpha \in \mathcal{B}(q_{<\alpha})^{q_\alpha}$, the coding stops there and $q\res\alpha$ exactly codes  $q_\alpha$.

Now suppose $\beta < \alpha$. 
Remember $H_\beta \cap\alpha\neq\beta$.
Clearly, $Y\cap\beta \in \mathcal{B}(q_{<\beta})^{q_\beta}$ as $q_\beta =  p_\beta \conc [Y\cap\beta]_3$.
The $\Sigma_1$-theory of $\bar M_\beta$ allows us to reconstruct $g$ inside 
$\mathcal{B}(q_{<\beta})^{q_\beta}$ using an argument like that of Lemma~\ref{codingapp_define}.
Thus, $q\res\beta \in \mathcal{B}(q_{<\beta})^{q_\beta}$.
Moreover, in this case we have also shown that $g\res\beta$ eventually dominates $f(p\res\alpha)^{p_\beta}$,
and as before we obtain $Y\cap\beta$ in the next step of the limit coding (after the exact coding of $p_\beta$ by $p_{<\beta}$, which could of course be trivial if $p_\beta=\emptyset_\beta$).
We have seen $Y\cap\beta$ does not precode an element of $S_\beta$,
so $q\res\beta$ exactly codes $q_\beta =  p_\beta \conc [Y\cap\beta]_3$,
provided that $f(p_{<\beta})^{p_\beta}_{q\res\beta} \in \mathcal{B}(q_{<\beta})^{q_\beta}$.
This holds again by an argument similar to that of Lemma~\ref{codingapp_define}.
Again, as $q\res\beta \in \mathcal{B}(q_{<\beta})^{q_\beta}$, the coding stops there and $q\res\beta$  exactly codes  $q_\beta$.

Finally, we obtain a $q$ such that the desired properties hold except for the capturing of $X$, which only holds for a tail of $\beta$;
we can use induction to get $q$ that works for all $\beta$.
Obviously, $q \leq p$ (we have never put 1s on any partition affected by restraints).
For the beginning of the induction, assume $\alpha$ is the least limit cardinal.
Argue as in the general limit case described above to get a condition $q$ that works on a tail below $\alpha$. Now make finitely many extensions to obtain a $q$ that works everywhere.  \renewcommand{\qedsymbol}{{\tiny  Lemma~\ref{ext1}~}$\Box$}\end{proof}

Note that the proof of Lemma~\ref{ext1} as given here relies heavily on the idea to work with the $E_\beta$'s. 
Perhaps the reader's initial impression is that one could have easily done with the standard $C_\beta$ from $\Box$ instead;
virtual inaccessible coding requires an Easton set, but this is also no reason to go beyond $\Box$.
It is the mechanism of distinguishing between virtual inaccessible coding and immediate singular coding
which requires that $E_\beta$ appear in the smallest relevant type of structure---where $\beta$ is seen to be non-Mahlo,
but not necessarily singular.
Without this distinction, we could choose to always do virtual inaccessible coding---but then the proof of Lemma~\ref{ext1} seems to fail.

\section{The Main Theorem for Quasi-Closure}\label{jensen_qc}\index{quasi-closure!of Jensen coding@of Jensen coding, $P(A_0)$}

We shall now define $\D$ and $\leqlol$ witnessing that $P$ is quasi-closed.
\begin{dfn}\label{d.lessthan}\index[notation]{lessthanl  lambda@$\leqlol$}
We define $p \leqlol q$ just if $p\leq q$, $p^*_{\lambda} = q^*_{\lambda}$, and for all $\delta \in \Card' \cap\lambda$ we have $p_\delta = q_\delta$ and $p^*_\delta = q^*_\delta$.
In accordance with Section~\ref{sec:s:ext}, define $p \leqlo^{<\lambda} q$ just if $p\leq q$ and both $p_\delta = q_\delta$ and $p^*_{\delta} = q^*_{\delta}$ for all $\delta \in \Card' \cap\lambda$.
\end{dfn}
The class $\D$ is defined in a more elementary way than in \cite{bjw:82} and \cite{friedman:codingbook}, making do with more basic genericity over Skolem-hulls, owing to the fact that we've assumed $\Card = \Card^{L[A_0\cap\omega]}$. 

The proof in Section~\ref{sec:D} that $\D$ is dense, i.e., \eqref{qc:redundant}, takes the form of an intricate induction.\footnote{See Theorems 3.1, 3.2 and Lemma 3.10 in \cite{bjw:82} and Lemma 4.5 in \cite{friedman:codingbook}, where $\Sigma^p_f$ plays the role of $\D$.}
The next theorem treats two situations simultaneously:
firstly, a local version of quasi-closure which will carry us through singular limits in the aforementioned induction, assuming we already know ``smaller fragments'' of $\D$ to be dense.
Secondly, the global version of quasi-closure (sketched in \myplacecite{chapter 3.7}{bjw:82} and \myplacecite{after 4.11}{friedman:codingbook}).
The proof readily suggests such an aggregation. 
The following assumption captures both situations:
\begin{ass}\label{M}
Let $\beta \in \Card\cap\kappa$ or $\beta=\infty$.
In the first case, let $q \in P$, $M=\mathcal{A}^{q_{\beta^+}}$ and $A=A_q$;
also demand that $|q_{<\kappa}| \geq \beta^+$.
Define 
\[
P(q)^{\beta^+} = \{  p \res\beta^+ \setdef p \in P\wedge p \leq q \wedge p^{\beta^+}=q^{\beta^+}\},
\]
and let $R=P(q)^{\beta^+}$.
In the second case let $M=L[A_0]$, $A=A_0$ and $R=P$.
\end{ass}
The definition of $P(q)^{\beta^+}$ was chosen so that for $p\in P(q)^{\beta^+}$, $p\cdot q \in P$ and $P(q)^{\beta^+}$ generically codes both $A\cap\beta$ and $q^{\beta^+}$ over $\mathcal{A}^{q^{\beta^+}}$.
In fact, $(p\cdot q) \res \beta^{++} = p$.
Equivalently, $p \in R$ if and only if 
$p \leq q \res \beta^{++}$, $p^{\beta^+}=q^{\beta^+}$ and $p$ obeys all restraints from $q$ for inaccessibles $\delta > \beta$.

\medskip

We now define approximations to $\D$ relativized to $M$, and then $\D$ itself.

\begin{dfn}\label{D}
Given $p\in P$, $\lambda\in \Reg$, $\lambda < \beta'\in\Card$, $x\in M$ arbitrary,
we now define
$\D^M_{[\lambda,\beta')}(p,\vec{x}) \subseteq P$\index[notation]{D M [lambda,beta](p,x)@$\D^M_{[\lambda,\beta')}(p,\vec{x})$} (we only need to consider $\beta'\leq\beta$).
For $\delta \in \On$, let
\begin{align*}
H_\delta\index[notation]{H delta, H <delta@$H_\delta$, $H_{<\delta}$} &= H^M_\delta (p,\vec{x}) = h^M_{\Sigma_1}(\delta\cup\{\vec{x}\}),\\
H_{<\delta} &= H^M_{<\delta} (p,\vec{x}) = h^M_{\Sigma_1}(\sup(\supp(p)\cap \delta)\cup\{ \vec{x}\}).
\end{align*}
We define $\D^M_{[\lambda,\beta)}(p,\vec{x})$ as
the set of $q \in P$ such that if $\tau = \min( \supp(p)\cap[\lambda,\beta'))$ exists then $q \leqlo^{\tau} p$ and
\begin{enumerate}[label=(D \arabic*), ref=D \arabic*]
\item\label{restraints_start_high} if $\tau > \lambda$ and $\tau\in\Inacc$ then $\rho(q^*_\tau)\geq \lambda$ (of course, this is vital to preserve \ref{coding_areas_bounded} in the definition of $P(A_0)$; see also the last clause of \eqref{qc:redundant});

\item\label{weak_genericity} for all $\delta\in [\tau,\beta')$ such that $\delta \in H_{<\delta} \cup \supp(p)$

\begin{enumerate}[label=(\alph*), ref=\ref{weak_genericity}\alph*]
\item\label{grow} $|q_\delta| > H_\delta \cap \delta^+$;
 for $\delta=\kappa$ in addition, $|q_{<\kappa}| > \sup (H_{|p_{<\kappa}|}\cap\kappa)$.
\item\label{restraints} if $\nu \in H_{<\delta}\cap[\delta,\delta^+)$ then $b^{p_\delta\res\nu}\setminus \eta \in q^*_\delta$ for some $\eta \in [\lambda,\delta)$;

\item\label{A_coding} if $\xi \in H_\delta \cap [\delta,\delta^+)$ there is $\nu > |p_\delta|$ such that
$p_\delta((\langle \xi, \nu \rangle)_0) = 1$ if $\xi\in A$ and $p_\delta((\langle \xi, \nu \rangle)_0) = 0$ if $\xi\not\in A_p$;

\item \label{succ_coding} if $b^{p_{\delta^+}\res\nu}\in p^*_{\delta^+}$ there is $\zeta > |p_\delta|$ such that $q_\delta((\zeta)_1)=p_{\delta^+}(\nu)$;

\item \label{fake_inacc_coding} if $\delta\in\Inacc\cap\kappa$ 
there is $\beta \in E_\delta \setminus \sup(\supp(p)\cap\delta))$ such that
$q_{\beta^+}((\beta^+)_2)=1$;\footnote{This somewhat technical requirement makes it easy to distinguish the virtual inaccessible coding from the singular coding. Conditions which immediately use singular coding, i.e., those constructed in the Extendibility Lemma will have $q_{\beta^+}((\beta^+)_2)=1$ on a tail of $\beta \in E_\delta$.}

\item \label{inacc_coding} if $\delta\in\Inacc\cap\kappa^+$ and $b^{p_\delta\res\nu}\setminus \eta \in p^*_\delta$ then
there is $\xi \in b^{p_\delta\res\nu}\setminus\eta$ such that $\xi >  \sup(\supp(p)\cap\delta))$ and
$q_{<\delta}((\xi)_2)=p_\delta(\nu)$ (note that this clause elegantly covers both the Mahlo coding and the inaccessible, non-Mahlo coding);

\item\label{coding_areas_in_supp} if $\delta\in\Inacc$ and $b \in p^*_\delta \setminus\On$ then $b\cap\sup(\supp(p)\cap\delta) \subseteq \dom(q_{<\delta})$;
\end{enumerate}
\end{enumerate}
We also write
\[
U(p)\index[notation]{U(p)@$U(p)$}=U^M(p, \vec{x})\index[notation]{U M (p,x)@$U^M(p, \vec{x})$}=\{ \delta \in \Card\cap\delta \setdef \delta \in H_{<\delta}\cup\supp(p)\}.
\]
Finally, for the proof of quasi-closure we set $\D(\lambda, \vec{x}, p) = \D^{L[A_0]}_{[\lambda, \infty)}(p,\vec{x})$.
\end{dfn}
Note that we define $\D^M_{[\lambda,\beta)}(p,\vec{x})$ as a subset of $P$, and for $p \in P$ rather than as a subset of $R$ and for $p \in R$. 
This is a notational convenience we will make use of when we show that these sets are non-empty.
To build sequences with greatest lower bounds, it is only the restriction to $R$ of  $\D^M_{[\lambda,\beta)}(p,\vec{x})$ which is useful.

\medskip

We also introduce the following terminology, which provides good intuition and will be useful when we show the least Mahlo is not collapsed in our iteration (see \ref{it:reals:are:caught}).
\begin{dfn}
Let $H$ be any set and let $p, q \in P$, $q \leq p$.
We say that $q$ is basic generic for $(H,p)$ at $\delta^+$ if and only if
\begin{enumerate}
\item%
$|q_\delta| > H \cap \delta^+$;

\item%
if $\nu \in H\cap[\delta^+,\delta^{++})$ then $b^{p_{\delta^+}\res\nu} \in q^*_{\delta^+}$;

\item %
 if $\nu \in H\cap[\delta^+,\delta^{++})$ then there is $\zeta > |p_\delta|$ such that $q_\delta((\zeta)_1)=p_{\delta^+}(\nu)$;

\item%
 if $\xi \in H \cap [\delta,\delta^+)$ there is $\nu > |p_\delta|$ such that
$q_\delta((\langle \xi, \nu \rangle)_0) = 1$ if $\xi\in A$ and $q_\delta((\langle \xi, \nu \rangle)_0) = 0$ if $\xi\not\in A_p$;

\end{enumerate}

When $\delta\in\Inacc$, we say that $q$ is basic generic\index{basic generic} for $(H,p)$ at $\delta$ if and only if
\begin{enumerate}
\item%
$|q_{<\delta}| \geq \sup (H\cap\delta)$.

\item%
if $\nu \in H\cap[\delta,\delta^{+})$ then $b^{p_{\delta}\res\nu}\setminus \eta \in q^*_\delta$ for some $\eta$;

\item%
if $\nu \in H\cap[\delta,\delta^{+})$ then
there is $\xi \in b^{p_\delta\res\nu}\setminus\eta$ such that $\xi \geq  |p_{<\delta}|$ and
$q_{<\delta}((\xi)_2)=p_\delta(\nu)$;

\item%
if $\xi \in H \cap \delta$ there is $\nu > |p_\delta|$ such that
$q_{<\delta}((\langle \xi, \nu \rangle)_0) = 1$ if $\xi\in A$ and $q_{<\delta}((\langle \xi, \nu \rangle)_0) = 0$ if $\xi\not\in A_p$;
\end{enumerate}
The second definition we shall only use for $\delta = \kappa$ (in 
\ref{it:reals:are:caught}).
\end{dfn}

\medskip

That $P$ is quasi-closed will follow from the next theorem.

\begin{thm}\label{jensen_qc_main}
Let $\rho \leq \lambda \leq \beta$, $\lambda\in \Reg$.
If $\beta<\infty$, let $(\beta_\xi)_{\xi < \rho} \in M$ be an increasing sequence of cardinals such that $\sup_{\xi < \rho} \beta_\xi = \beta$ 
and set $\beta_\xi = \infty$ for $\xi<\rho$ if $\beta=\infty$.
Suppose $(p^\xi)_{\xi < \rho} \in M$ is a sequence of conditions in $R$ which has $\bar w= (w^\xi)_{\xi < \rho} \in M$ s.t. $M\models\bar w$ is a $(\lambda, x)$-canonical witness.
Moreover suppose $(p^\xi)_{\xi < \rho}$ is strategic in the following sense:
For a tail of $\xi < \rho$, 
\begin{enumerate}[label=(\Alph*), ref=\Alph* ]
\item \label{lambda} for all $\bar \xi$ such that $\xi < \bar \xi< \rho$ we have $p^{\bar \xi} \leqlol p^\xi$,
\item \label{lambda_xi} there is $\lambda_\xi$ such that $p^\xi \in \D_{[\lambda_\xi, \beta_\xi)}(p^\xi, \{ \bar w\res\xi+1,x\})$ and
$p^{\xi +1} \leqlo^{\lambda_{\xi+1}} p^\xi$,

\item \label{dense_top}  in case $\beta<\infty$ , letting $H = h^M_{\Sigma_1}(\beta_\xi \cup\{ \bar w\res\xi+1, x\})$,
$p^{\xi+1}$ is $(H,p^\xi)$-generic at $\beta^+$, which means in the present case:
\begin{enumerate}[label=(\arabic*), ref=\ref{dense_top}\arabic*]
\item $|p^{\xi+1}_\beta| > H \cap \beta^+$;

\item if $\nu \in H\cap[\beta^+, |s|)$ then $b^{s\res\nu} \in (p^{\xi+1})^*_\beta$;

\item if $\xi \in H \cap [\beta, \beta^+)$ there is $\nu > |p^\xi_\beta|$ such that
$p^{\xi+1}_\beta((\langle \xi, \nu \rangle)_0) = 1$ if $\xi\in A_p$ and $p^{\xi+1}_\beta((\langle \xi, \nu \rangle)_0) = 0$ if $\xi\not\in A_p$;

\item if $b^{s\res\nu}\in (p^\xi)^*_{\beta^+}$ there is $\zeta > |p^\xi_\beta|$ such that $p^{\xi+1}_\beta((\zeta)_1)=s(\nu)$;
\end{enumerate}
\end{enumerate}
Then $\bar p$ has a greatest lower bound $p^\rho \in R\cap M$.
\end{thm}

Before we prove the theorem, we note an obvious corollary:
\begin{cor}\label{jensen_qc_main_cor}
Suppose $\bar p$ is a $(\lambda,x)$-adequate sequence.
Then $\bar p$ has a greatest lower bound.
\end{cor}
\begin{proof}[Proof of the corollary]
In case $\beta = \infty$, strategic in the sense of the theorem means exactly that $\bar w$ is a strategic guide.
\end{proof}

For quasi-closure, it thus remains to show that $\D$ is dense. 
But first, let us prove Theorem~\ref{jensen_qc_main}.
\begin{proof}[Proof of Theorem~\ref{jensen_qc_main}]
The proof is split over several lemmas.
Fix a sequence $(p^\xi)_{\xi<\rho}$ and $\bar w$ as in the hypothesis.
For each $\xi < \rho$, pick $\lambda_\xi\in \Reg$ as in \eqref{lambda_xi}.
Note we may assume $\lambda_\xi \geq \lambda$, for we may replace $\lambda_\xi$ by $\lambda$ whenever this assumption fails: 
firstly, \eqref{lambda} holds, and secondly, 
$\D^M_{[\lambda_\xi, \beta_\xi)}(p^\xi, \{ x, \bar w \res\xi+1\}) \subseteq \D^M_{[\lambda, \beta_\xi)}(p^\xi, \{ x, \bar w \res\xi+1\})$ if $\lambda_\xi \leq \lambda$.

Let $p= p^\rho$ be the obvious candidate for a greatest lower bound, i.e., for $\delta \in \bigcup_{\xi<\rho} \supp(p^\xi)$,
\begin{equation}\label{obvious_candidate}
\begin{gathered}
(p_\delta, p^*_\delta)= ( \bigcup_{\xi<\rho} p^\xi_\delta, \bigcup_{\xi<\rho} (p^\xi)^*_\delta),\\
p_{<\kappa} = \bigcup_{\xi<\rho} p^\xi_{<\kappa}\text{, if $\kappa\leq\beta$.}
\end{gathered}
\end{equation}

Most of this proof will now be devoted to checking that for each $\delta \in \Card \cap \kappa^+$ both
\begin{gather}
p_\delta \in S^*_\delta  \label{is_string}\\
p\res\delta \in 
\mathcal{B}(p_{<\delta})^{p_\delta} \text{ if }\delta \in \Sing.\label{growth}
\end{gather}
After that we conclude by checking that $p^*_{<\kappa} \in S^*_{<\kappa}$.

Let $\bar \lambda$ be minimal such that for an unbounded set of $\xi<\rho$, we have $\lambda_\xi \leq \bar \lambda$.
Obviously, if $\delta \in \Card \cap \bar \lambda$, $p^\rho_\delta = p^\xi_\delta$ for some $\xi$, and so \eqref{is_string} and \eqref{growth} hold.
So let $\delta \geq \bar \lambda$.
Let %
\[
X^\xi = \begin{cases}
 h^M_{\Sigma_1}(\delta \cup \{x, \bar w\res\xi+1\}) &\text{for }\delta<\beta,  \\
h^M_{\Sigma_1}(\beta_\xi \cup\{x, \bar w\res\xi+1\})&\text{for }\delta=\beta.
\end{cases}
\]
Let $X= \bigcup_{\xi<\rho} X^\xi$, let $\pi\colon \bar X \rightarrow X$ be the transitive collapse.
Let $j = \pi^{-1}$.
Write $\bar A = \bigcup_{\nu\in X} \pi(A\cap\nu)$ and (if $\beta<\infty$) $s= q_{\beta^+}$, $\bar s = \pi(s)$, so that
\[
\bar X = \begin{cases}
L_{\bar\mu}[\bar A]&\text{ in case }\beta=\infty, M=L[A]\\
L_{\bar\mu}[\bar A,\bar s]&\text{ when }\beta<\infty, M=\mathcal{A}^s.
\end{cases}
\]
We now embark on a series of lemmas which will be used several times when proving \eqref{is_string} and \eqref{growth}
in various cases.
\begin{lem}[Definability]\label{def}
The sequence $(\pi(p^\xi))_{\xi<\rho}$ is a definable class in $\bar X$.
\end{lem}
\begin{proof}[Proof of Lemma~\ref{def}]
The notion of canonical witness was chosen precisely to ensure this definability, and in a sense, the lemma is trivial.
We show that there is a formula $\Theta$ such that $M\vDash \Theta(\xi,p^*) \iff p^*= p_\xi$ for $\xi \in [\xi_0, \rho)$.

Let $\Psi$ and $G$ be as in the definition of canonical witness, i.e., $\Psi$ is the formula defining $\bar w$ and $G$ is the $\Sigma^T_1(\lambda\cup\{x\})$ function s.t. $p^\xi=G(\bar w\res\xi+1,\vec{x})$ for $\xi<\rho$.
Let  $\Gamma$ be a $\Sigma^T_1(\lambda\cup\{x\})$ formula representing $G$.
The  formula $\Theta(\xi,p^*)$ is
\begin{multline*}
\exists\bar w^*= (w^*_\nu)_{\nu < \xi+1}\text{ s.t. }
 \bar w^* \res\xi_0 = \bar w\res \xi_0 \\
\wedge \; [ \forall \nu \in [\xi_0, \xi +1] \quad \Psi(w^*_\nu, \bar w^*\res\nu,x) ]\\
\wedge \;\Gamma(p^*, (w^*_\nu)_{\nu < \xi+1}, x)
\end{multline*}
As $\rho\subseteq X \prec_{\Sigma_1} M$ and $\{ p_\xi , \bar w\res\xi+1\} \subseteq X$ for each $\xi < \rho$, the same formula defines $(p^\xi)_{\xi<\rho}$ in $X$.
Now apply $\pi$.
 \renewcommand{\qedsymbol}{{\tiny  Lemma~\ref{def}~}$\Box$}\end{proof}
Observe that only $\delta$ such that %
\begin{equation}
\delta \in \supp(p^\rho) \text{ or } \delta = \bigcup_{\xi<\rho} \sup(\supp(p^\xi)\cap\delta) \label{delta_relevant}
\end{equation}
are relevant to the proof of \eqref{is_string} and \eqref{growth}.
Thus we restrict our attention to such $\delta$ in the following.

We shall now show that because of \eqref{lambda_xi}, $p^\rho$ enjoys a basic type of genericity over $X$.
First, note that $\Card \cap X \cap(\kappa^++1)\subseteq  \supp(p^\rho)$, because of \eqref{grow} and the following:
\begin{lem}\label{sublemma.U}
If $\tilde \gamma \in \Card \cap X$, then
$\tilde \gamma \in U^M(p^\xi, \{ \bar w\res\xi+1, \vec{x}\})$ for large enough $\xi<\rho$.
\end{lem}
\begin{proof}
Fix $\xi$ large enough such that $\tilde\gamma\in X^\xi$.
If first alternative of \eqref{delta_relevant} obtains, we may assume $\xi$ is large enough such that $\delta \in \supp(p^\xi)$. 
We can assume $\tilde\gamma > \delta$ as trivially, 
$ \supp(p^\xi)\subseteq U^M(p^\xi, \{ \bar w\res\xi+1, \vec{x}\})$.
Then $\sup(\supp(p^\xi)\cap\tilde\gamma) \geq \delta$,
and so $\tilde\gamma \in H^M_{\tilde\gamma}(p^\xi, \{ \bar w\res\xi+1, \vec{x}\})$.
Now suppose the second alternative of \eqref{delta_relevant} obtains.
Then we may assume $\xi$ is large enough so that 
$\tilde\gamma \in h^M_{\Sigma_1}(\sup(\supp(p^\xi)\cap\delta) \cup\{ \bar w\res\xi+1, \vec{x}\})$.
Since $\delta\leq\tilde\gamma$, we immediately conclude $\tilde\gamma \in H^M_{\tilde\gamma}(p^\xi, \{ \bar w\res\xi+1, \vec{x}\})$.
 \renewcommand{\qedsymbol}{{\tiny  Lemma~\ref{sublemma.U}~}$\Box$}
 \end{proof}
For $\gamma \in \Card^{\bar X}$, let
 \[
\bar p_\gamma = \bigcup_{\xi<\rho}\pi(p^\xi)_{\gamma}.
\]
In fact, \ $j(\gamma)^+\cap X =\bigcup_{\xi < \rho} |p^\xi_\sigma(\gamma)|$, i.e., 
\begin{equation}\label{total}
|\bar p_\gamma |= \gamma^{+L_{\bar \mu}},
\end{equation}
by the following:
\begin{lem}\label{nu_in_H}
If $\nu \in [\gamma, \gamma^{+\bar X})$ and either $\nu > \delta$ or the second alternative of  \eqref{delta_relevant} obtains, then $j(\nu) \in H^M_{<j(\gamma)}(p^\xi, \{ \bar w\res\xi+1, \vec{x}\})$, for large enough $\xi<\rho$.
\end{lem}
\begin{proof}
Exactly as the previous lemma. \renewcommand{\qedsymbol}{{\tiny  Lemma~\ref{nu_in_H}~}$\Box$}
\end{proof}

So clearly by \eqref{grow}, we have $\nu < |\bar p^{\xi+1}_\gamma |$ for  $\xi, \nu$ as above, prooving \eqref{total}.
Similarly, \eqref{grow} makes sure $|\bar p^\xi_\gamma|$ strictly increases with $\xi < \rho$.
\begin{rem}
Observe for \eqref{total} we can assume $\nu > \delta$ in the hypothesis of the lemma, but the full generality of the lemma will be needed in \ref{local_coding}.
\end{rem}

Thus, letting
\[
\bar p = \bigcup_{\gamma \in \Card^{\bar X}} \bar p_\gamma
\]
we see $\bar p\colon\bar \mu \rightarrow 2$ is a total function.
We continue exploiting the basic genericity of $\bar p$ over $\bar X$, in the sense that $\bar p_\delta$ codes all of $\bar A$ and $\bar p$:
\begin
{lem}[Local coding]\label{local_coding}\index{local coding}
Provided $\delta$ satisfies \eqref{delta_relevant} and letting 
$$Y= L_{\bar\mu}[A\cap\delta, \bar p_\delta] $$ 
we have
$\bar X\subseteq Y$ and $\bar X$ is a definable class in $Y$.
In fact, for each $\nu < \bar \mu$,
$\bar A \cap \nu$ and $\bar p \res \nu$ are definable in
$(\mathcal{A}^{<\bar p_{\lVert \nu \rVert}\res\nu})^Y$ if $Y \models \lVert \nu \rVert \in \Reg$ and definable in $(\mathcal{B}(\bar p_{<\lVert \nu \rVert})^{<\bar p_{\lVert \nu \rVert}\res\nu})^Y$ if $Y \models \lVert \nu \rVert \in \Sing$.
\end{lem}
\begin{proof}[Proof of Lemma~\ref{local_coding}]

We show the definability of $\bar p$ and $\bar A$ by checking they can be reconstructed from $\bar A \cap \delta$ and $\bar p_\delta$ by the decoding procedure run over $\bar X$ (see the last item of Definition~\ref{sing_defs}).
We now describe this procedure, which is carried out in $L_{\bar \mu}[\bar A\cap \delta, \bar p_\delta]$, making three claims which we prove thereafter.\footnote{A lot of the technical complexity of the following stems from the fact that we do not know the cardinals of $L_{\bar \mu}[A\cap \delta, \bar p^\rho_\delta]$.
For even though $\Card=\Card^{L[A_0\cap\omega]}$,  $\bar p^\rho_\delta$ might collapse cardinals
of $L_{\bar \mu}[A\cap \delta]$ when added to that model. 
Thus we have to work in a mixture of models below.
If the sets $\D$ are defined to provide stronger genericity over $X$, the construction can be carried out entirely over the natural model, namely $L_{\bar \mu}[A\cap \delta, \bar p^\rho_\delta]$.
}
Observe that $\Card^{\bar X} = \Card^{L_{\bar\mu}[A_0\cap\omega]}$.
\begin{enumerate}[label=(\roman*), ref=\roman*]

\item \label{local_sing_coding}
Suppose for $\gamma\in \Sing^{L_{\bar \mu}}$ we have already constructed $\bar A \cap \gamma$ and $\bar p\res\gamma$ inside $L_{\bar \mu}[\bar A\cap \delta, \bar p_\delta]$. 
We claim that $\bar p\res\gamma$ codes $\bar p_\gamma$ via singular limit coding using as coding apparatus the functions $f^{\bar p_\gamma\res\nu}$ as defined in $L_{\bar \mu}[\bar A \cap \gamma, \bar p_\gamma\res\nu]$, for $\nu < \gamma^{+L_{\bar\mu}}$.
Observe this is well defined as $L_{\bar \mu}[\bar A \cap \gamma, \bar p_\gamma\res\nu] \subseteq \bar X$ and so $\gamma$ is a singular cardinal in that model (this needn't be the case for $L_{\bar \mu}[\bar A \cap \gamma, \bar p \res\gamma, \bar p_\gamma\res\nu]$).

\item \label{local_reg_coding}
Suppose for $\gamma \in \Reg^{L_{\bar \mu}}$ below the Mahlo of $L_{\bar \mu}$ (if there is one) and $\nu \in [\gamma,\gamma^{+L_{\bar \mu}})$ we have already constructed $\bar A \cap \gamma$ and $\bar p\res\nu$ inside $L_{\bar \mu}[\bar A\cap \delta, \bar p_\delta]$.
Letting $k=1$ if $\gamma\in \Succ^{L_{\bar \mu}}$ and $k=2$ if $\gamma$ is inaccessible in ${L_{\bar \mu}}$, for each $\nu \in [\gamma,\gamma^{+L_{\bar \mu}})$ we claim
\begin{equation*}
\begin{gathered}
\bar p_\gamma(\nu)=1 \iff \\
\{ \xi < \gamma \setdef \xi \in b^{\bar p_\gamma\res\nu}\wedge \bar p((\xi)_k) = 1\}\text{ is unbounded below }\gamma,
\end{gathered}
\end{equation*}
where by $b^{\bar p_\gamma\res\nu}$ we mean  $(b^{\bar p_\gamma\res\nu})^{L_{\bar \mu}[\bar A \cap \gamma, \bar p_\gamma\res\nu]}$.

\item \label{local_mahlo_coding}
For $\gamma=\bar \kappa$, the Mahlo of $L_{\bar \mu}$, and $\nu \in [\bar \kappa,\bar \kappa^{+L_{\bar \mu}})$, suppose  we have already constructed $\bar A \cap \bar \kappa$  inside $L_{\bar \mu}[\bar A\cap \delta, \bar p_\delta]$.
Let $\bar p_{<\bar \kappa} = [\bar A]_1$ and $\bar A_0 = [\bar A]_0$.
For each $\nu \in [\bar \kappa,\bar \kappa^{+L_{\bar \mu}})$ we claim
\begin{equation*}
\begin{gathered}
\bar p_{\bar \kappa}(\nu)=1 \iff \\
\{ \xi < \bar \kappa \setdef \xi \in b^{\bar p_{\bar \kappa}\res\nu}\wedge \bar p_{<\bar\kappa}((\xi)_3) = 1\}\text{ is unbounded below }\gamma,
\end{gathered}
\end{equation*}
where by $b^{\bar p_{\bar \kappa}\res\nu}$ we mean  $(b^{\bar p_{\bar \kappa}\res\nu})^{L_{\bar \mu}[\bar A \cap \bar \kappa, \bar p_{\bar \kappa}\res\nu]}$.

\item \label{local_A_coding}
Suppose for $\gamma \in \Card^{L_{\bar \mu}}$ below the Mahlo of $L_{\bar \mu}$ (if there is one), we have already constructed $\bar A \cap \gamma$ and $\bar p\res\gamma$ inside $L_{\bar \mu}[\bar A\cap \delta, \bar p_\delta]$.
We claim:
\begin{gather*}
\nu \in \bar A \cap [\gamma, \gamma^{+L_{\bar \mu}})\iff \\
\{ \xi \in [\gamma, \gamma^{+L_{\bar \mu}}) \setdef \bar p_\gamma ((\langle \xi, \nu \rangle)_0)=1 \}\text{ is unbounded below }\gamma^{+L_{\bar \mu}}.
\end{gather*}
The same holds for $\bar A_0$ if $\gamma = \bar \kappa$, the Mahlo of $L_{\bar \mu}$.
\end{enumerate}
From these claims, the lemma follows by induction.
In a sense the proof of these claims is again trivial, by definition of $\D$, i.e., by strategicity or \eqref{lambda_xi}.
The claim in \eqref{local_sing_coding} is a straightforward consequence of elementarity:
since $\bar p_\gamma\res\nu \in \bar X$ for $\nu < | \bar p_\gamma |$,  
$j(f^{\bar p_\gamma\res\nu})=f^{p^\xi_{j(\gamma)} }\res j(\nu)$ for large enough $\xi < \rho$,
and $X \vDash$``$p^\xi \res j(\gamma)$ exactly codes $p^\xi_{j(\gamma)}$.
By the obvious continuity of singular limit coding, the claim in \eqref{local_sing_coding} follows.

Now for the claim in \eqref{local_reg_coding}.
We restrict our attention to the exemplary case that $\gamma < \beta$ and is inaccessible in $L_{\bar \mu}$.
Observe for the claim we can assume $\delta < \gamma$, but note for later that the present proof also works if $\delta = \gamma$ and the second alternative of \eqref{delta_relevant} obtains.

Now suppose $\nu \in [\gamma, \gamma^{+L_{\bar\mu}})$ and show the claim in \eqref{local_reg_coding}.
We have seen for $\gamma \in U^M(p^\xi, \{\bar w\res\xi+1, \vec{x}\})$ and 
$j(\nu) \in H_{j(\gamma)} (p^\xi,  \{\bar w\res\xi+1, \vec{x}\})$ 
for large enough $\xi < \rho$ 
(we assume for now that $\nu > \delta$, but note for later reference 
the present proof also goes through if the second alternative of \eqref{delta_relevant} obtains).
Thus by \eqref{restraints}, $b^{p_{j(\gamma)}^\xi \res j(\nu)} = b^{p_{j(\gamma)}^\rho \res j(\nu)} \in (p^\xi)^*_{j(\gamma)}$ for large enough $\xi < \rho$,
restraining extensions of $p^\xi$, i.e., making sure $p^\rho_{<j(\gamma)} ((\zeta)_1) = 0$ for a tail of $\zeta \in b^{p_{j(\gamma)}^\rho \res j(\nu)}$ if $p^\rho_{j(\gamma)}(j(\nu))=0$.
Moreover, \eqref{inacc_coding} makes sure 
$p^\rho_{<j(\gamma)} ((\zeta)_1) = 1$ for a tail of $\zeta \in b^{p_{j(\gamma)}^\xi \res j(\nu)}$  if $p^\rho_{j(\gamma)}(j(\nu))=1$.
Together, for large enough $\xi < \rho$ we have
\begin{gather*}
p^\xi_{j(\gamma)}(j(\nu))=1 \iff\\  X \models\text{``}p^\xi_{<j(\gamma)} ((\zeta)_1) = 1 
\text{ for unboundedly many }\zeta \in b^{p_{j(\gamma)}^\xi \res j(\nu)}\text{''}
\end{gather*}
and since
\[
\sigma( b^{p_{j(\gamma)}^\xi \res j(\nu)} ) = (b^{\bar p_\gamma \res \nu})^{L_{\bar \mu}[\bar A\cap \gamma, \bar p \res \nu]},
\]
applying $\sigma$ proves the claim.

The other case of the claim in \eqref{local_reg_coding} is entirely analogous, substituting the use of \eqref{inacc_coding} by \eqref{succ_coding} 
if $\gamma < \beta$ and using \eqref{dense_top} if $\gamma = \beta$.
Claim \eqref{local_mahlo_coding} is proved in the same manner, once more using \eqref{inacc_coding}.
The claim in  \eqref{local_A_coding} is proved using using \eqref{A_coding}.  \renewcommand{\qedsymbol}{{\tiny  Lemma~\ref{local_coding}~}$\Box$}
\end{proof}
\begin{lem}\label{catch_structure}
Provided $\delta$ satisfies \eqref{delta_relevant},
$\bar X \in \mathcal{A}_0^{\bar p_\delta}$
\end{lem}
\begin{proof}[Proof of Lemma~\ref{catch_structure}]
Observe that 
$
L_{\bar \mu}[A\cap\delta]\vDash |\bar p_\delta| =\pi^{-1}(\delta^+)\in \Card.
$
Thus by the definition of steering ordinal, $\bar \mu  < \mu^{\bar p_\delta}$.
Since by Lemma~\ref{local_coding}, $\bar A$ and (in case $\beta<\infty$) $\bar s$ are definable over $L_{\bar \mu}[A\cap\delta, \bar p_\delta] \in L_{\mu^{\bar p_\delta}}[A\cap\delta, \bar p_\delta] = \mathcal{A}_0^{\bar p_\delta}$,
we have $X = L_{\bar \mu}[\bar A, \bar s]\in \mathcal{A}_0^{\bar p_\delta}$.  \renewcommand{\qedsymbol}{{\tiny  Lemma~\ref{catch_structure}~}$\Box$}
\end{proof}
\begin{lem}\label{collapsed_string}
If  $\delta = \bigcup_{\xi<\rho} \sup(\supp(p^\xi)\cap\delta)$, 
\begin{equation}
p\res\delta \in \mathcal{A}^{\bar p_\delta}_0\label{growth_collapse}.
\end{equation}
\end{lem}
\begin{proof}[Proof of Lemma~\ref{collapsed_string}]
As $\pi(p^\xi)\res\delta = p^\xi\res\delta$, Lemmas~\ref{def} and  \ref{catch_structure} together yield
$p^\rho \res\delta \in \mathcal{A}^{\bar p_\delta}_0$.  \renewcommand{\qedsymbol}{{\tiny  Lemma~\ref{collapsed_string}~}$\Box$}
\end{proof}

With the help of the preceding lemmas, we now show \eqref{is_string} and \eqref{growth}.
Let $\tilde\delta = \min \On\cap X \setminus \delta$, so that $\pi(\tilde\delta)=\delta$.

\medskip

First,  suppose $\delta < \tilde\delta$.
Then we must have $X \models \tilde\delta \in \Inacc$ ($\delta$ is the critical point of $\pi$).
Observe that  for each $\xi < \rho$, by Easton support and since $\sup(\supp(p^\xi)\cap\tilde\delta)\in X$ we have $\sup(\supp(p^\xi)\cap\tilde\delta) < \delta$.
Thus \eqref{is_string} is trivially satisfied as $p^\rho_\delta = \emptyset_\delta$.
We may assume without loss of generality that
\begin{equation}\label{relevant_support_unbounded}
\delta = \sup_{\xi<\rho} (\supp(p^\xi)\cap\tilde\delta)
\end{equation}
hold, for otherwise \eqref{growth} is trivially satisfied and we are done.
This means $\delta \in \Sing$ ($\rho\leq\lambda$ and we must have $\lambda < \delta$ as the conditions grow below $\delta$).

We show that $\bar p_\delta$ is coded using virtual inaccessible coding\index{virtual inaccessible coding}\index{inaccessible coding!virtual}, i.e., 
\begin{equation}\label{fake_inaccessible_coding}
s(\bar p_{<\delta})=\bar p_\delta
\end{equation}
(of course $\bar p_{<\delta} = p_{<\delta}$).
We adapt the proof of claim \eqref{local_reg_coding}, Lemma~\ref{local_coding}.
Because of \eqref{fake_inacc_coding}, $p_{<\delta}$ does not immediately use singular coding. 
As we have noted there, the proof of the aforementioned claim also goes through if $\nu=\gamma= \delta$ because \eqref{relevant_support_unbounded} holds.
This means that for $\nu \in [\delta, \delta^{+L_{\bar\mu}})$ we have
\[
\bar p_\delta(\nu) = 1 \iff \bar p_{<\delta}((\zeta)_1)\text{ for a tail of } \zeta \in (b^{\bar p_\delta\res\nu})^{\bar X}.
\]
By \eqref{growth_collapse} $b^{\bar p_\delta}$ is eventually disjoint from $\bar p_{<\delta}$.
Moreover $p\res\delta\in \mathcal{A}_0^{\bar p_\delta}$ so the coding stops there.
It remains to show that %
\begin{equation}\label{restraints_in_dom}
\forall \nu \in [\delta, |\bar p_\delta| ) \quad b^{\bar p \res\nu}\subseteq^* \dom(\bar p_{<\delta}).
\end{equation}
Fixing $\nu$ as above, in the proof of Lemma~\ref{local_coding} we've seen $j(\delta) \in U^M(p^\xi, \{\bar w\res\xi+1,\vec{x}\})$ and $j(\nu) \in H_{<\delta} ^M(p^\xi, \{ \bar w\res\xi+1, \vec{x})$ for large enough $\xi <\rho$, as the second alternative of \eqref{delta_relevant} obtains.
So by \eqref{coding_areas_in_supp} (which we haven't used up to now) and \eqref{relevant_support_unbounded} it follows that
\[
b^{\bar p_{j(\delta)}^\rho \res j(\nu)} \cap [\eta, \delta)  \subseteq \dom(p^\rho_{<\delta}) 
\]
for some $\eta <\delta$.
Now use elementarity and apply $\sigma$ to get \eqref{restraints_in_dom}.
By \eqref{restraints_in_dom} and \eqref{fake_inaccessible_coding}, we have $\mathcal{A}_0^{\bar p_\delta} \subseteq \mathcal{B}(p^\rho_{<\delta})^{\emptyset_\delta}$, so \eqref{growth} follows from \eqref{growth_collapse} and  we are done.

\medskip

Now suppose $\delta \in X^\rho_\delta$.
Since $\bar p_\delta = p^\rho_{\tilde\delta}$ and $\mathcal{A}_0^{p_\delta} \subseteq \mathcal{B}(p_{<\delta})^{p_\delta}$ if $\delta \in \Sing$, \eqref{growth} immediately follows from \eqref{growth_collapse}. It remains to show \eqref{is_string}, i.e., that $p_\delta \in S^*_\delta$.

We only need to check the requirement given in Definition~\ref{strings_dfs} for $\zeta = |p^\rho_\delta|$, for when $\zeta < |p^\rho_\delta|$, the requirement is met as 
$\zeta < |p^\xi_\delta|$ for some $\xi<\rho$ and
$p^\xi_\delta \in S^*_\delta$.
So let $N =L_\eta[A\cap\delta, p^\rho_\delta] $ be a test model\index{test model}, i.e., let $\eta, \bar \kappa$ be such that
\begin{gather}
N\vDash \ZF^-\text{and }|p^\rho_\delta| = \delta^+, \label{still_card} \\
N\vDash \Card=\Card^{L[A\cap\omega]}\wedge \Card\setminus\omega_1=\Card^{L}\setminus\omega_1 \label{cards_r_L_cards}\\\
L_\eta\models \text{$\bar \kappa$ is the least Mahlo, }\bar\kappa^{++} \text{ exists.}\label{mahlo}
\end{gather}
We must show that $L_\eta[\bar A^*] \models r$ is coded by branches, where we let 
\[
\bar A^* = A^*(A\cap\delta, p_\delta)^N,
\]
i.e., the set obtained when running the decoding procedure (see \ref{decoding_def}) relative to $N$.
By Lemmas~\ref{def} and \ref{local_coding} and as $\pi(p^\xi_\delta)=p^\xi_\delta$ for $\xi < \rho$ the sequence $(|p^\xi_\delta|)_{\xi < \rho}$ is definable over $L_{\bar \mu}[A\cap\delta, p^\rho_\delta]$.
Thus \eqref{still_card} implies $\eta < \bar \mu = \bar X\cap \On$.
Let $\alpha = \sup(\Card^{L_{\bar \mu}} \cap\eta)$. 
Clearly $\eta \in [\alpha, \alpha^{+\bar X})$.

\medskip

\textbf{Case 1:}
$\alpha \geq \bar\kappa^{+\bar X}$.
Observe $\eta >  \bar \kappa^{+N} = \bar \kappa^{+\bar X}$, where the first inequality holds by \eqref{mahlo}.
Thus also $\bar X \vDash \bar \kappa$ is the least Mahlo and  by elementarity of $j\colon \bar X \rightarrow M$, $\beta=\kappa$ or $\beta=\infty$:
 If $\beta<\kappa$, $M=\mathcal{A}^{q_{\beta^+}}\vDash$ there is no Mahlo cardinal. Likewise, $j(\bar\kappa)=\kappa$.

By Lemma~\ref {local_coding} and \eqref{cards_r_L_cards}, $\bar A^* = \bar A_0 \cap\bar \kappa^{++N}$, where $\bar A_0=\pi[A_0]$.
By elementarity, $L_{\bar\mu}[\bar A_0]\models \Psi_0(r)$ and also $L_\eta[\bar A_0\cap\bar\kappa^{++N}]\models \Psi_0(r)$ since $\bar T$ is locally semidecidable.\footnote{Cf.\ the end of the proof of Fact~\ref{s.f.david.qc}.}

\medskip

\textbf{Case 2:} Otherwise, assume $\alpha \leq\bar \kappa$ and if $\alpha = \bar \kappa$, then $\bar \kappa$ is Mahlo in $\bar X$.
Write $\bar \zeta$ for $\alpha^{+N}=\alpha^{+L_\eta}$.
Observe that as in the previous case, by the proof of Lemma~\ref{local_coding} and by \eqref{cards_r_L_cards} the decoding procedure starting from $\bar p_\delta\res\bar\zeta$ run in $L_\eta[A\cap\delta]$ up to $\alpha$ yields $\bar A \cap\bar\zeta$ and $\bar p\res\bar\zeta$.
In the case $\alpha = \bar \kappa$, this uses the assumption that $\bar \kappa$ is Mahlo in $\bar X$ so that $N$ contains enough of the coding apparatus of $\bar X$ to run the proof of Lemma~\ref{local_coding} inside $N$.

As $\bar \zeta< \alpha^{+\bar X} = \sup_{\xi<\rho} |p^\xi_\alpha| $, 
we can pick $\xi<\rho$ such that $\bar\zeta < |p^\xi_\alpha|$.
By elementarity, 
\[
\tilde N= L_{j(\eta)}[A \cap j(\alpha), p^\xi_{j(\alpha)}\res j(\bar\zeta)],
\]
is a valid test-model for $\tilde s = p^\xi_{j(\alpha)}\res j(\bar\zeta)$ and as $\tilde s \in  S^*_{j(\alpha)}$, this string codes a predicate $A^*$ over $L_{j(\eta)}[A\cap j(\alpha)]$ such that $L_{j(\eta)}[A^*]\models \Psi_0(r)$.
By the way the coding works and using elementarity of $\pi$, $\bar p_\delta$ also codes $\bar A^* = \pi(A^*)$ and  $L_{\eta}[\bar A^*]\models \Psi_0(r)$.

\medskip

\textbf{Case 3:} 
It remains to deal with the case $\alpha = \bar \kappa$ when $\bar \kappa$ is not Mahlo in $\bar X$.
We assume for simplicity that $\kappa \in X$, or equivalently, $\beta \geq \kappa$ (and supply a few paranthetical remarks which should suffice for the case $\beta < \kappa$).
Again using the proof of Lemma~\ref{local_coding} and \eqref{cards_r_L_cards}, the decoding procedure starting from $p_\delta$ run in $L_\eta[A\cap\delta,p_\delta]$ up to
$\bar\kappa$ yields $A_0 \cap \bar \kappa$ and $\bar p_{<\pi(\kappa)} \res\bar \kappa$ (or if $\beta < \kappa$, the appropriate collapse of $p^0_{<\kappa} \res X$).
As $\bar \kappa< \pi(\kappa)$, we can find $\xi<\rho$ such that $|p^\xi_{<\kappa}| > j(\bar\kappa)$ (if $\beta<\kappa$, this holds for $\xi=0$). %
By elementarity, 
\[
\tilde N= L_{j(\eta)}[A \cap j(\bar\kappa), p^\xi_{<\kappa}\res j(\bar\kappa)],
\]
is a valid test-model for $\tilde s = p^\xi_{<\kappa}\res j(\bar\kappa)$ and as $\tilde s \in  S^*_{<\kappa}$, this string codes a predicate $A^*$ over $L_{j(\eta)}[A \cap j(\bar\kappa)]$ such that $L_{j(\eta)}[A^*]\models \Psi_0(r)$.
By the way the coding works and using elementarity of $\pi$, $\bar p_\delta$ also codes $\bar A^* = \pi(A^*)$ and  $L_{\eta}[\bar A^*]\models \Psi_0(r)$.
(Analogously for $\beta < \kappa$.)

\medskip

Finally, we prove $p^\rho_{<\kappa} \in S^*_{<\kappa}$, i.e., the requirement in Definition~\ref{strings_dfs} is met
(to avoid trivialities, we assume $\lambda < \kappa$).
So let $N = L_\eta[A_0\cap \bar \kappa, p^{\rho}_{<\kappa}\res\bar \kappa]$ be a $<\kappa$-test model with $\bar \kappa$ the least Mahlo in $N$.
We assume again $\bar \kappa = |p^\rho_{<\kappa}|$ as otherwise there is nothing to prove.
Let 
$$X = \bigcup_{\xi<\rho} H_{|p^\xi_{<\kappa}|},$$
let 
$\pi\colon X \rightarrow \bar X$ be the the collapsing map, $j=\pi^{-1}$, and $\bar\mu = \bar X \cap \On$.
By \eqref{grow}, we have $X = H_{\bar \kappa}$ and $j(\bar \kappa)=\kappa$ (so $\crit(j)=\bar\kappa$).

Since the sequence $(|p^\xi_{<\kappa}| )_{\xi <\rho}$ is definable over $\bar X$, we can argue as previously that $\eta < \mu$.
Also, as before, $\bar p_{\bar\kappa} = \bigcup_{\xi < \rho} \pi(p^\rho_{\kappa})$
is decoded from $p^\rho_{<\kappa}$ relative to $\bar X$,
and $\bar p_{\bar\kappa}\res {\bar\kappa}^{+N}$ is decoded relative to $N$.
Writing $\tilde \eta = j(\eta)$ %
it follows that the model $\tilde N= L_{\tilde\eta}[A\cap \kappa, p^\rho_\kappa \res j({\bar\kappa}^{+N})]$
is a valid test-model for $\kappa$ and so we finish the argument once more using that $p^\rho_\kappa \in S^*_\kappa$.  \renewcommand{\qedsymbol}{{\tiny  Theorem~\ref{jensen_qc_main}~}$\Box$}
\end{proof}

\section{Density}\label{sec:D}
We now show \eqref{qc:redundant}, \emph{Density} for the quasi-closure of $P$, i.e., that $\D$ is dense.
In a sort of bootstrapping process using Theorem~\ref{jensen_qc_main} in its local version, we shall show by induction on $\gamma$ that
$\D^M_{[\lambda,\gamma)}(p, \vec{x})\neq \emptyset$.
To be able to meet \eqref{coding_areas_in_supp}, we have to strengthen our inductive hypothesis and assume that densely, any Easton set can be ``covered'' by the domain of a condition.
\begin{lem}\label{lem:D}
Let $\gamma\in \Card$. For any $q$, $\beta$ and $M$ %
as in Assumption~\ref{M} such that  $\beta\geq\gamma$ and
any  $\lambda \in \Reg$ and $p \in P$, we have:
\begin{enumerate}[(i)]
\item\label{D_dense_ih}  for any $\vec{x} \in M$, there is $p' \in P$ such that $p'\leqlol p$ and $p' \in \D^M_{[\lambda,\gamma)}(p, \vec{x})$ 
\item\label{ext2_ind} for any Easton set $B \subseteq [\lambda,\gamma)$ there is $p' \in P$ such that $p' \leqlol p$ and $B \subseteq \dom(p'_{<\gamma})$.
\end{enumerate}
\end{lem}

Before we give the proof, note the following corollary, which complements Theorem~\ref{jensen_qc_main} in its global form (Corollary~\ref{jensen_qc_main_cor}):

\begin{cor}\label{cor_qc}\index{Easton supported Jensen coding!is quasi-closed}\index{quasi-closure}
$P$ is quasi-closed.
\end{cor}
\begin{proof}[Proof of the corollary]
It remains to show $\D$ is dense. 
By the previous lemma for $\beta= \infty$, $M=L[A]$ and $\gamma = \kappa^{+3}$,
for any $p \in P$, $\lambda \in \Reg$ and $\vec{x} \in L[A]$, there is $q\in \D(\lambda, p, \vec{x})$ such that $q \leqlol p$.
\end{proof}

In the following proof of Lemma~\ref{lem:D} you will notice a somewhat unexpected role-reversal of $\gamma$ and $M$:
the induction is over $\gamma$ while $q$, $\beta$ and $M$ vary freely.
Also, note how $\D^M_{[\lambda,\gamma)}(p, \vec{x})$ has been decoupled from $R$; we talk about density in $P$ and $q$ is merely a parameter.\todo{explain why $q$ is now merely a parameter}

To make life a little easier, we could prove from Item~\ref{D_dense_ih} in the Lemma the seemingly stronger fact that for any $\bar \gamma \geq \gamma$ and $r\in P$ such that $r \leq p$,
$\D^M_{[\lambda,\gamma)}(p, \vec{x})$ is $\leqlol$-dense in $P(r)^{\bar\gamma^+}$  (but we only use it in this proof, so we leave it implicit).

\begin{proof}[Proof of Lemma~\ref{lem:D}.]
The proof is by induction  on $\gamma$, so suppose both statements are true for all cardinals $\gamma' <\gamma \in \Card$.
Fix $q$, $M$, $\lambda\in \Reg$, $\vec{x}\in M$ and an Easton set $B$ as in the hypothesis,
and let $p \in P$ be arbitrary.
We shall find $p' \leqlol p$ satisfying both $p' \in \D^M_{[\lambda,\gamma)}(p, \vec{x})$
and $B \subseteq \dom(p'_{<\gamma})$.

For the successor case, assume $\gamma=\delta^+$, for $\delta \in \Card$.
We need to take care of \eqref{restraints_start_high} and \eqref{grow} to \eqref{coding_areas_in_supp} (in Definition~\ref{D}, p.~\pageref{D}) for $\delta$, and then we can finish this case quickly by induction.
So let $\iota_0$ be least above %
$|p_\delta| + \delta$ such that
any pair $b, b' \in p^*_{\delta^+}$ is disjoint above $\iota_0$ and let
\[
\iota = \sup (H_\delta \cap \On ) \cup \sup_{b \in p^*_{\delta^+}} (\min (b\setminus\iota_0))_1 \cup \sup B
\]
Note that $\iota < \delta^+$ and let
\[
p^1_\delta = (p^0_\delta  \conc E) \cup Z 
\]
where $E \subseteq \delta$ codes the relation $\in\res\iota\times\iota$ in some recursive way and where
\begin{gather*}
Z \colon [ |p^0_\delta| + \delta, \iota) \rightarrow 2\\
Z(\zeta)=\begin{cases}
1 & \text{ if $\zeta = (\langle \xi, |p^0_\delta|\rangle )_0$ for some $\xi \in [\delta, |p^0_\delta|)$,}\\
1 & \text{ if $\zeta = (\xi)_1$ for some $\xi \in b\setminus \iota_0$ and $b \in p^*_{\delta+}$,}\\
0 &\text{ otherwise.}\\
\end{cases}
\end{gather*}
Moreover, let $\eta = \rho(p^*_\delta) \cup \sup(U\cap\delta)\cup \lVert p^*_\delta \rVert\cup\lambda$ and let
\[
(p^1)^*_\delta = (p^0)^*_\delta \cup \{ b^{p_\delta \res\nu} \setminus \eta \setdef \nu \in H_{<\delta}\} \cup \lambda.
\]

If $\delta\in \Sing$, we must find $p^1\res\delta$ so that it exactly codes $p^1_\delta$,
using the extendibility Lemma~\ref{ext1}.
Otherwise we can set $p^1\res\delta = p^0\res\delta$.
It is clear that $p^1$ satisfies \eqref{restraints} and also \eqref{restraints_start_high} in definition of $\D$.
Also, by the choice of $Z$ and since $\iota$ was chosen large enough, $p^1$ satisfies \eqref{grow}, \eqref{A_coding} and \eqref{succ_coding}.
It is easy to see that $p^1_\delta \in S^*_\delta$:
 by choice of $E$,  
$p^1_\delta$ collapses the size of $\iota$ to $\delta$ over every $\ZF^-$ model.\footnote{Of course, much less than $\ZF^-$ is needed here.}
Thus any test model for $p^1$ is a test model for $p^0$ and we are done.

We must also arrange \eqref{coding_areas_in_supp} for $\delta$.
Here we use the inductive assumption that Item~\ref{ext2_ind} (from Lemma~\ref{lem:D}) holds below $\gamma$.
Let $\gamma'= \sup(\supp(p)\cap\delta)$ and let 
\[
B'= \bigcup \{ b \cap \gamma' \setdef b\in  p^*_{\delta}\setminus \On\}.
\]
By Item~\ref{ext2_ind} for $\gamma'$ we can find $p^2 \leqlol p^1$ such that
$B' \cup (B \cap \delta) \subseteq \dom(p^2_{<\delta})$. 

By induction hypothesis, we can find $p' \leqlol p^{2}$ such that 
$$p' \in \D^M_{[\lambda,\delta)}(p^2,\vec{x})\subseteq \D^M_{[\lambda,\delta)}(p,\vec{x}).$$
It follows that $p' \in \D^M_{[\lambda,\gamma)}(p,\vec{x})$ and $B \subseteq \dom(p'_{<\gamma})$, finishing the successor case.

\medskip

Now suppose $\gamma$ is a limit ordinal.
We can assume $\gamma \in \Sing$, for otherwise $\supp(p)\cap\gamma$ and $B$ are both  bounded below $\gamma$ and we can simply use the induction hypothesis.
So let $ (\beta'_\xi)_{\xi<\rho}$ be the increasing enumeration of a club in  $\beta' = \gamma$, $\rho = \cof(\beta')$ and $\beta'_0 > \rho$.

We can assume without loss of generality that $|p_{<\kappa}| \geq \beta^+$ (by \ref{ext_mahlo}, extendibility for the Mahlo coding).
Let $R'= P(p)^{\beta'^+}$, $M' = \mathcal{A}^{p_{\beta'^+}}=\mathcal{A}(A_p)^{p_{\beta'^+}}$
and let $p^0=p\res\beta'^++1$, observing $p^0 \in R'$.
For $\xi < \rho$, let 
$$D_\xi = \D^{M}_{[\lambda, \beta'_\xi)}(p,\vec{x})\cap L_{\beta'^+_\xi}[A]$$ 
and let 
$\vec{x}' = \{ \lambda, p^0, (\beta'_\xi)_{\xi<\rho}, (D_\xi)_{\xi<\rho}, B\}$, noting that $\vec{x}' \in M'$.
The point here is that for any $p' \in P$,
\begin{equation*}
p' \in \D^{M}_{[\lambda, \beta'_\xi)}(p,\vec{x}) \iff p'\res\beta'_\xi \in D_\xi,
\end{equation*}
so that we can talk about $\D^{M}_{[\lambda, \beta'_\xi)}(p,\vec{x})$ inside $M'$.

We shall now use Theorem~\ref{jensen_qc_main} for $M'$ and $R'$ to construct a sequence whose greatest lower bound will be the desired condition.
Let $X^0$ be least such that $X^0 \prec_1 M'$, $x'\in X^0$ and $X^0 \in M'$;
we can find such $X^0$ by the fact that $\cof(M'\cap\On) = \beta'^+$. %

We now find a sequence $\bar w = (w^\xi)_{\xi < \rho}$ such that for each $\xi < \rho$,
$w^\xi = (X^\xi, p^\xi, \bar w\res\xi)$ 
and
 \begin{gather*}
w^\xi \in M'\\ 
X_{\xi} = h^{M'}_{\Sigma_1}(\{ p^\xi, \bar w\res\xi\}),
\end{gather*}
and if $\xi$ is limit, then $p^\xi$ is a greatest lower bound of $\bar p\res\xi+1$ in $R'$.
Note that this implies $(X^\nu)_{\nu<\xi}\in X_{\xi}$ (also note the sequence $(X_\xi)_{\xi<\rho}$ is not continuos).
To complete the definition of $\bar w$, we must specify how to construct $p^{\xi+1}$ in the successor case.

By induction, assume we have built such a sequence $\bar w\res\xi+1$.
Now let $p^{\xi+1} \in R'$ be least such that 
\begin{itemize}
\item $p^{\xi+1} \res \beta'_\xi \in D_{\xi}$---equivalently, $p^{\xi+1}\in \D^{M}_{[\lambda, \beta'_\xi)}(p,\vec{x})$,
\item $p^{\xi+1} \in \D^{M'}_{[\lambda, {\beta'_\xi})}(p^{\xi},\{\vec{x}', \bar w \res \xi+1\})$ (note this is $\Pi^T_1(\{\vec{x}', \bar w \xi+1, p^\xi, \xi \})$ in the model $M'$)
\item $B \cap \beta'_\xi\subseteq \dom(p^{\xi+1})_{<\beta'^+_\xi}$,
\item the clauses in \eqref{dense_top} of Theorem~\ref{jensen_qc_main}
hold for $p^{\xi+1}$ with $\beta$, $\beta_\xi$, $M$, $\vec{x}$ in \eqref{dense_top} replaced by $\beta'$, $\beta'_\xi$, $M'$, and $\vec{x}'$ (note this is $\Pi^T_1(\{\vec{x}', \bar w \xi+1, p^\xi, \xi \})$ in $M'$ as well).
\end{itemize}
The last point can be arranged as in the first part of the argument in the successor case (using extendibility,  \ref{ext1}).
The first point can be arranged as the induction hypothesis implies that $D_\xi$ is dense in $R'$.\footnote{Literally, the induction hypothesis talks about $P$, not $R'$. In fact, we have that $\D^{M}_{[\lambda, \gamma)}(p,\vec{x})$ is dense in $P$ below $p$ if and only if it has non-empty intersection with $P(q)^{\bar\gamma^+}$ for any $q \leq p$ and any $\bar\gamma \geq \gamma$.} 
If in doubt, here is a pedestrian proof: since $p^\xi \in R'=P(p)^{\beta'^+}$, we know $p^\xi\cdot p\neq0$; so there is $p' \leqlol p^\xi\cdot p$ such that
$p' \in \D^{M}_{[\lambda, \beta'_\xi)}(p^\xi\cdot p,\vec{x})$;
But then
since trivially $\D^{M}_{[\lambda, \beta'_\xi)}(p^\xi\cdot p,\vec{x}) \subseteq \D^{M}_{[\lambda, \beta'_\xi)}(p,\vec{x})$, 
$$p' \res\beta'^+_\xi \cup p^\xi \res [\beta'^+_\xi, \beta') \in \D^{M}_{[\lambda, \beta'_\xi)}(p,\vec{x}) \cap R'.$$
Similar arguments work for the remaining two points.
Note that the conjunction of theses items can be expressed by  a $\Pi^T_1(\vec{x}'\cup\{ p^\xi, \xi \})$ formula (say, $\phi(p^{\xi+1}, \vec{x}',\bar w\res\xi+1)$) inside $M'$.

The sequence $\bar w$ is well defined since 
by the following fact and Theorem~\ref{jensen_qc_main},
the sequence $\bar p \res \rho'$ has a greatest lower bound $p^{\rho'}$ when  $\rho' \leq \rho$ is a limit ordinal.

Finally, we have that  $p^\rho \in \bigcap_{\xi < \rho} D_\xi = \D^M_{[\lambda,\gamma)}(p,\vec{x})$ and $B \subseteq \dom(p^\rho_{<\gamma})$.
Since also $p^\rho \in R'= P(p)^{\gamma^+}$, 
$p' = (p^\rho \cdot p) $ is the desired condition.
It remains to prove:
\begin{fct}\label{f.sequence.is.canonical}
The sequence $\bar w = (w^\xi)_{\xi<\rho'} = (X^\xi, p^\xi)_{\xi<\rho'}$ is a $(\lambda, x')$-canonical witness for $\bar p \res\rho' = (p^\xi)_{\xi<\rho'}$ inside $M'$
\end{fct}
\begin{proof}[Proof of Fact~\ref{f.sequence.is.canonical}]
Note that $w^\xi$ is the unique $x=(X,p,\bar w')$ such that
\begin{equation}\label{canonical_witness_D}
\begin{gathered}
X = h^{M'}_{\Sigma_1}(\{ p, \bar w^*\})\text{ and}\\
X\models p \text{ is $\leq_{M'}$-least such that }\phi(p,\vec{x}',\bar w)
\end{gathered}
\end{equation}
where $\phi$ comes from the inductive definition of $\bar p$ discussed above,
and $\bar w^* = \bar w \res \xi$.
We leave it to the reader to check that \eqref{canonical_witness_D} can be expressed by a $\Pi^T_1(\vec{x}')$-formula $\Psi(X,p,\bar w')$.

Now it is easy to see the fact holds:
$x = w^\xi$ if and only if $x=(X,p,\bar w^*)$, $\Psi(X,p,\bar w^*)$ holds and $\bar w^*$ is a sequence of length $\xi$ such that for each $\nu < \xi$, we have $\Psi( \bar w^*(\nu)_0 ,\bar w^*(\nu)_1, \bar w^*\res\nu)$.
This is obviously $\Pi^T_1(\vec{x}')$, as $\Psi$ is.
See Fact~\ref{density reduction}, especially Lemma~\ref{adequate} for a similar (but harder) argument. \renewcommand{\qedsymbol}{$\Box$ \tiny  Fact~\ref{f.sequence.is.canonical}.}
\end{proof}\renewcommand{\qedsymbol}{$\Box$ \tiny  Lemma~\ref{lem:D}.}\end{proof}

\section{Stratification}\label{sec_coding_strat}\index{Easton supported Jensen coding!is stratified}\index{stratification}

We now define the stratification system for $P(A_0)$.
Remember $\leqlol$ and $\D$ were defined in Definitions~\ref{d.lessthan} and \ref{D}.
\begin{dfn}
Let $\lambda \in \Reg \cap \kappa^{++}$.
\begin{enumerate}
\item Let $q \lequpl p$\index[notation]{lessthanu  lambda@$\lequpl$} just if $q\res [\lambda^+, \infty) \leq p \res [\lambda^+, \infty)$ in $P(A_0)$ and $q_\lambda \leq p_\lambda$.
\item\index[notation]{C lambda@$\Cl$} For any $p \in P(A_0)$, writing $\rho$ for $\rho(p^*_\lambda) \cup \sup\dom(p_{<\lambda})$, let
$\Cl(p)=\{ (p\res\lambda, \rho, B^{\rho}_p) \}$.
\end{enumerate}
\end{dfn}
We remind the reader that we take the singleton on the right-hand side in the above equation only to satisfy the abstract definition of stratification, which allows for a ``multifunction'' $\Cl$, as this is necessary for iterations.

A more straightforward definition would be $\Cl(p)=\{ p_{<\lambda} \}$
for all $p \in P(A_0)$.
This would would show quasi-closure in the setting of \myplacecite{4.1}{friedman:codingbook}.

The problem is that we had to introduce $\rho(p^*_\lambda)$ for inaccessible $\lambda$.
Recall that $\leq$ was defined so that ``new restraints'' (those in $q^*_\lambda$ for $q\leq p$) may only differ from old ones (those already in $p^*_\lambda$) \emph{above} $\rho(p^*_\lambda)$.
This means that with the more straightforward definition, in \eqref{linking}, when $\Cl(p)\cap\Cl(q)\neq\emptyset$ and $p \lequpl q$, and we want to form $q\cdot p$, the natural candidate (the point-wise union) might not actually lie below $p$ and $q$: not if any ``new'' restraints disagree with old restraints below $\rho(p^*_\lambda)$ (or $\rho(q^*_\lambda)$, respectively).
Thus, it is natural to make $\rho(p^*_\lambda)$ and the restraints below it part of $\Cl$.

Having checked quasi-closure, we invite the reader to check the rest of stratification, all of which is outright trivial.

\chapter{Extension and Iteration}\label{sec:ext}

The proof of the main result makes it necessary to consider iterations $\bar Q^\theta$ such that each initial segment $P_\iota$ is stratified on an interval $[\lambda_\iota,\kappa]$, but it is not forced that $\dot Q_\iota$ be stratified for all $\iota <\theta$.

A paradigmatic example are products: Supposing $P$ and $Q$ are stratified, $P \times Q$ is stratified even though forcing with $P$ might destroy that $Q$ is stratified. 
This is captured by the notion that $P\times Q$ is a \emph{stratified extension} of $P$,
or that $(P, P \times Q)$ is a stratified extension.
The notion of stratified extension is introduced below in Definition~\ref{def:s:ext}.

Most importantly for us, amalgamation produces stratified extensions.
Similarly, to ensure that $P(A)$ ``codes no unwanted branches'' we use shall coding forcings from an inner model, and these are not quasi-closed in the model where we force with them.

In our iteration, it will always hold that $(P_\iota,P_{\iota+1})$ is a stratified extension on some interval $[\lambda_\iota, \kappa]$ 
of regular cardinals;  $\lambda_\iota$ will not be same fixed cardinal throughout the iteration. 
With diagonal support, this suffices to conclude that $P_\theta$ is stratified (on some interval).
This coherency provided by extension is vital: e.g., an iteration whose proper initial segments are all $\sigma$-strategically closed can add a real (see \cite{souslinproblem}, or
 \cite{kellner:sh:sacks} for a particularly extreme example).
 From this one can show that quasi-closure for a particular preclosure system can, without coherence, be lost in the limit.

We treat quasi-closed and stratified extension separately (Sections~\ref{sec:qc:ext} and  \ref{sec:s:ext}). Each axiom of stratified (or quasi-closed) extension corresponds to an axiom of stratification (or quasi-closure)---in fact, interestingly, $P$ is stratified if and only if $(\{ 1_P\}, P)$ is a stratified extension.
To prove the iteration theorem, we also have to add some additional axioms concerning the interplay of the prestratification (preclosure) systems on $P_\iota$ and $P_{\iota+1}$; see Definitions~\ref{pcs:is} and \ref{pss:is}.

In Section~\ref{sec:prod} we show products of stratified forcings are stratified extensions.
Finally, we introduce \emph{the stable meet operator} in Section~\ref{sec:stm} and \emph{remote suborders} in Section~\ref{sec:remote}. See the beginning of those sections for a motivating discussion.

\section{Quasi-Closed Extension and Iteration}\label{sec:qc:ext}
In this section, we start by giving the definition of quasi-closed extension. 
We show that composition of quasi-closed forcing is a special case of quasi-closed extension.
We give a sufficient condition which makes sure that if $(P_0,P_1)$ is a quasi-closed extension, then $P_1$ is quasi-closed.
We prove that the relation of being a quasi-closed extension is transitive.
Finally, we formulate and prove an iteration theorem for quasi-closed forcing.

\medskip

Let $P_0$ be a complete suborder of $P_1$ and let $\pi\colon P_1 \rightarrow P_0$ be a strong projection.
Moreover, assume we have a system $\pcs_i=(\D_i, \param_i, \leqlol_i)_{\lambda\in I}$ for $i \in \{0,1\}$ such that $\D_i \subseteq I \times V \times (P_i)^2$ is a class definable with parameter $\param_i$ 
and for every $\lambda\in I$, $\leqlol_i$ is a  binary relation on $P_i$.
\begin{dfn}\label{pcs:is}
We write $\pcs_0 \is \pcs_1$\index[notation]{lessthan triangle@$\is$|textbf}\index[notation]{  lessthan triangle@$\is$|textbf}
to mean that for every $\lambda \in I$,
\begin{enumerate}[label=(\pcsIsRef), ref=\pcsIsRef]
\item For all $p,q \in P_0$, $p\leqlol_0 q \Rightarrow p \leqlol_1 q$. \label{ext}
\item For all $p,q \in P_1$, $p \leqlol_1 q \Rightarrow \pi(p) \leqlol_0 \pi(q)$. \label{pi:mon}
\item For all $p,q \in P_1$ and for every $x$, $p \in \D_1(\lambda,x,q) \Rightarrow \pi(p) \in \D_0(\lambda,x,\pi(q))$. \label{F:coh}\label{D:coh}

\end{enumerate}
\end{dfn}

Observe that if $\pcs_0 \is \pcs_1$ we can drop the subscripts on $\leqlol_0$, $\leqlol_1$ and just write $\leqlol$ without causing confusion.
\begin{dfn}\label{qc:ext}
We say the pair $(P_0,P_1)$ is a \emph{quasi-closed extension\index{quasi-closed extension} on $I$, as witnessed by $(\pcs_0,\pcs_1)$} 
if and only if $\pcs_0$ witnesses that $P_0$ is quasi-closed on $I$,  $\pcs_1$ is a preclosure system on $P_1$, $\pcs_0 \is \pcs_1$ and for any pair $\lambda,\bar \lambda \in I$ such that
$\lambda \leq \bar \lambda$, the following conditions hold:
\begin{enumerate}[label=(\qcExtRef), ref=\qcExtRef]
\item 
If $p\in P_1$ and $q\in P_0$ is such that $q \leqlol \pi(p)$ and $q\in \D(\lambda,x,\pi(p))$, there is $r \leqlol p$
such that $r \in \D(\lambda,x,p)$, $r \leqlol p$ and $\pi(r)=q$.
Moreover we can ask of $r$ that for any $\lambda' \in I$ such that $p \leqlo^{\lambda'}\pi(p)$ we also have $r \leqlo^{\lambda'} \pi(r)$.\label{ext:redundant}\label{ext:D:dense}
\item If $\bar p = (p_\xi)_{\xi<\rho}$ is a sequence of conditions in $P_1$ such that for some $q \in P_0$ and some $\bar w$
\begin{enumerate}
\item $q$ is a greatest lower bound of the sequence $(\pi(p_\xi))_{\xi<\rho}$ and for all $\xi <\rho$, $q \leqlol \pi(p_\xi)$,\label{ext:pi:bound}
\item  $\bar w$ is a $(\lambda,x)$-strategic guide and a $(\bar \lambda, x)$-canonical witness for $\bar p$, \label{ext:glb:strat:def}
\item  either $\lambda = \bar \lambda$ or $p_\xi \leqlo^{\bar\lambda}_1 \pi(p_\xi)$ for each $\xi<\rho$,\label{ext:glb:tail}
\end{enumerate}
then $\bar p$ has a greatest lower bound $p$ in $P_1$ such that for each $\xi < \rho$, $p \leqlol p_\xi$ and $\pi(p)=q$.
Moreover, if $p_\xi \leqlo^{\bar\lambda}_1 \pi(p_\xi)$ for each $\xi<\rho$, then also $p \leqlo_1^{\bar\lambda} \pi(p)$.\label{ext:glb}
\end{enumerate}
As before, if we say $(P_0,P_1)$ is a quasi-closed extension and don't mention either of $\pcs_0$, $\pcs_1$ or $I$, 
that entity is either clear from the context 
or we are claiming that one can find such an entity.
\end{dfn}

We will grow tired of repeating all the conditions $\bar p$ has to satisfy in (\ref{ext:glb}), so we issue the following definition:
\begin{dfn}\label{adequate2}
We say $\bar p$ is $(\lambda,\bar\lambda,x)$-adequate\index[notation]{lambda,lambda,x-adequate@$(\lambda,\bar\lambda,x)$-adequate}\index[notation]{ lambda,lambda,x-adequate@$(\lambda,\bar\lambda,x)$-adequate}\index{adequate sequence} if and only if $\lambda, \bar \lambda \in I$, $\lambda\leq\bar\lambda$ and $\bar p$ satisfies (\ref{ext:glb:strat:def}) and (\ref{ext:glb:tail}). We say $q$ is a $\pi$-bound\index[notation]{pi-bound@$\pi$-bound} if and only if (\ref{ext:pi:bound}) holds.
\end{dfn}

Of course, the obvious example for quasi-closed extension is provided by composition of forcing notions:\index{composition (forcing)!quasi-closed extension}
\begin{lem}\label{qc:comp:implies:ext}
If $P$ is quasi-closed on $I$ and $\forces_P \dot Q$ is quasi-closed on $I$, then $(P, P*\dot Q)$ is a quasi-closed extension on $I$. 

To be more precise, let $\pcs_0$ denote the preclosure system witnessing that $P$ is quasi-closed and let $\pcs_1=(\bar\D,\bar\param, \bleqlol)_{\lambda\in I}$ be the preclosure system constructed as in the proof of \ref{stratified:composition}, where we showed that $P*\dot Q$ is stratified. Then $(\pcs_0,\pcs_1)$ witnesses that $(P, P*\dot Q)$ is a quasi-closed extension on $I$.
\end{lem}

We give the proof after we prove the following simple lemma, which will be useful in several contexts.

\begin{lem}\label{bar:lambda:adeq}
Say $R$ carries a preclosure system $\pcs$ on $I$ and $\bar p= (p_\xi)_{\xi<\rho}$ has a $(\lambda,x)$-strategic guide $\bar w$ which is also a  $(\bar \lambda,x)$-canonical witness.
If for all $\xi < \rho$, $p_\xi \leqlo^{\bar\lambda} 1_R$, then $\bar p$ is in fact $(\bar \lambda,x)$-adequate.
\end{lem}
\begin{proof}[Proof of Lemma~\ref{bar:lambda:adeq}.]
For arbitrary $\xi<\bar\xi<\rho$,
by \ref{def:pcs}(\ref{er}), as $p_{\bar \xi} \leq p_\xi \leq 1_R$ and $p_{\bar \xi} \leqlo^{\bar\lambda} 1_R$, we have $p_{\bar\xi}\leqlo^{\bar\lambda} p_\xi$. Thus $\bar w$ is in fact  a $(\bar \lambda,x)$-strategic guide for $\bar p$ and so $\bar p$ is $(\bar\lambda,x)$-adequate.
\end{proof}

\begin{proof}[Proof of Lemma~\ref{qc:comp:implies:ext}.]
  
Just by looking at the definition of $\bar\leqlol$ and $\bar\D$ as given in the proof of \ref{stratified:composition} (see~p.~\pageref{stratified:composition}), it is immediate that $\pcs_0 \is \pcs_1$ and that $\pcs_1$ is a preclosure system.
To check that $\pcs_1$ is a preclosure system, observe \ref{def:pcs}(\ref{er}) has already been checked in the proof of Theorem~\ref{stratified:composition}. The other conditions we leave to the reader.

Condition~\ref{qc:ext}(\ref{ext:redundant}) holds since $P$ forces \ref{def:qc}(\ref{qc:redundant}) for $\dot Q$:
say we have $(p,\dot p) \in P*\dot Q$ and $q \in P$ such that $q \in \D(\lambda,x,p)$.
By fullness we may find $\dot q$ such that $p \forces \dot q \in \D(\lambda,x,\dot q)$ and $p\forces \dot q \leqlol \dot p$. 
We have $(q, \dot q) \leqlol (p,\dot p)$ and $(q,\dot q)\in \bar \D(\lambda,x,(p,\dot p))$.
Moreover, by the last clause of \ref{def:qc}(\ref{qc:redundant}), we can demand that
$p\forces$  for any $\lambda' \in I$, $\dot p \leqlo^{\lambda'} 1_{\dot Q} \Rightarrow \dot q \leqlo^{\lambda'} 1_{\dot Q}$. Thus, for any $\lambda' \in I$ such that $(p, \dot p) \bleqlo^{\lambda'} (p, 1_{\dot Q})$ we also have $(q, \dot q) \bleqlo^{\lambda'} (q, 1_{\dot Q})$, which proves the last statement of \ref{qc:ext}(\ref{ext:redundant}).

Now to the main point, that is \ref{qc:ext}(\ref{ext:glb}):\label{proof:of:def:qc:in:s:comp:impl:ext}
Say $\bar p = (p_\xi,\dot p_\xi)_{\xi<\rho}$ is sequence of conditions in $P_1$ which is $(\lambda,\bar\lambda,x)$-adequate. That is, we may fix $\bar w$ which is a $(\lambda,x)$-strategic guide and a $(\bar\lambda,x)$-canonical witness. We may also fix $q \in P_0$ which is a $\pi$-bound---i.e., (\ref{ext:pi:bound}) holds.

We have already seen that $q$ forces that 
$\bar w$ is a $(\bar\lambda,x)$-canonical witness for $(\dot p_\xi)_{\xi<\rho}$ (this is the same argument as in the proof of Theorem~\ref{stratified:composition}).

It is easy to check that $q$ also forces that $\bar w$ is a $(\lambda,x)$-strategic guide for $(\dot p_\xi)_{\xi<\rho}$.

If $\lambda=\bar \lambda$, we conclude that $q$ forces that $(\dot p_\xi)_{\xi<\rho}$ is $\lambda$-adequate in $\dot Q$. 
Thus we may pick a $P$-name $\dot q$ such that $q$ forces $\dot q \in \dot Q$ is the greatest lower bound of $(\dot p_\xi)_{\xi<\rho}$ in $\dot Q$.
Then $(q,\dot q)$ is the greatest lower bound of $\bar p$ and we are done with the proof of \ref{qc:ext}(\ref{ext:glb}) in this case.

If on the other hand, $\lambda<\bar\lambda$, we have that for all $\xi<\rho$, $(p_\xi,\dot p_\xi) \leqlo^{\bar\lambda}_1 (p_\xi,1_{\dot Q})$.
Thus by Lemma~\ref{bar:lambda:adeq} we have that
$q$ forces $(\dot p_\xi)_{\xi<\rho}$ is $\bar \lambda$-adequate.
Thus $q$ also forces that this sequence has a lower bound,
for which we may fix a name $\dot q$.
By quasi-closure for $\dot Q$ in the extension and since 
for all $\xi<\rho$ we have 
\[ q\forces\dot p_\xi \dot\leqlo^{\bar\lambda}1_{\dot Q},\]
we conclude that for any $\xi < \rho$ we have 
\[ q \forces \dot q \dot \leqlo^{\bar \lambda} \dot q_\xi.\]
Thus $(q,\dot q)$ is a greatest lower bound of $\bar p$ and
\[ (q,\dot q)\bar\leqlo^{\bar\lambda}(q,1_{\dot Q}). \]
\renewcommand{\qedsymbol}{{\tiny  Lemma~\ref{qc:comp:implies:ext}~}$\Box$}\end{proof}

We now embark on a series of lemmas culminating in the insight that the second forcing of a quasi-closed extension $(P_0,P_1)$
is itself quasi-closed. Thus, we obtain a second proof that $P*\dot Q$ is quasi-closed (under the assumptions of the previous lemma).
This makes use of the fact that the projection map $\pi_0\colon P*\dot Q \rightarrow P$ is definable.
In general, we shall see that we have to assume that the strong projection map from $P_1$ to $P_0$ is sufficiently definable.  
\begin{lem}\label{lem:strat:proj}
Assume for $i\in\{0,1\}$,  $P_i$ carries a preclosure system $\pcs_i$ on $I$ and $\pcs_0 \is \pcs_1$.  If $\bar p=(p_\xi)_{\xi<\rho}$ is a  sequence of conditions in $P_1$ and $\bar w$ is a $(\lambda,x)$-strategic guide with respect to $\pcs_1$, then $\bar w$ is also a $(\lambda,x)$-strategic guide for $(\pi(p_\xi))_{\xi<\rho}$ with respect to $\pcs_0$.
\end{lem}
\begin{proof}
Suppose we are given $\bar p$ and $\bar w$ as in the hypothesis. 
If $\xi < \bar \xi < \rho$, since $p_\xi \leqlol_1 p_{\bar \xi}$, by \ref{pcs:is}(\ref{pi:mon}), $\pi(p_\xi)\leqlol_0\pi(p_{\bar\xi})$.
Let $\xi<\rho$ be arbitrary. Fix a regular $\lambda'$ such that
$p_{\xi+1} \leqlo^{\lambda'}_1 p_\xi$
and
$p_{\xi+1} \in \D_1(\lambda', (x, \bar w\res\xi+1),  p_\xi)$
By \ref{pcs:is}(\ref{pi:mon}), 
$\pi(p_{\xi+1})\leqlo^{\lambda'}_0 \pi(p_\xi)$
 and by \ref{pcs:is}(\ref{F:coh}),
$\pi(p_{\xi+1}) \in \D_1(\lambda', (x, \bar w\res\xi+1),  ,\pi(p_\xi))$,
finishing the proof.
\end{proof}

\begin{lem}\label{lem:adeq:proj}
Assume $(P_i, \pcs_i)$, $i\in\{0,1\}$ are as in Lemma~\ref{lem:strat:proj}. 
Further, assume that the strong projection map $\pi\colon P_1\rightarrow  P_0$ is $\qcdefG(\lambda\cup\{x\})$\index[notation]{SigmaT1(z), SigmaT1(X)@$\Sigma^T_1(z)$, $\Sigma^T_1(X)$}. %
If $\bar p=(p_\xi)_{\xi<\rho}$ is a  sequence of conditions in $P_1$ which is $(\lambda,x)$-adequate with respect to $\pcs_1$, then $(\pi(p_\xi))_{\xi<\rho}$ is $(\lambda,x)$-adequate with respect to $\pcs_0$.
\end{lem}
\begin{proof}
Fix $\bar w$ which is both a $(\lambda,x)$-strategic guide and a $(\lambda,x)$-canonical witness for $\bar p$.
By the previous lemma, $\bar w$ is a $(\lambda,x)$-strategic guide for $(\pi(p_\xi))_{\xi<\rho}$.
We may find a $\qcdefG(\lambda\cup\{x\})$ function $G$ such that $p_\xi = G(\bar w\res\xi+1)$. 
As $\pi \circ G$ is also $\qcdefG(\lambda\cup\{x\})$, $\bar w$ is also a 
$(\lambda,x)$-canonical witness for $(\pi(p_\xi))_{\xi<\rho}$.
\end{proof}

The following is useful, e.g., when we show a condition has legal support. Here lies one of the reasons for asking (\ref{er}).
\begin{lem}\label{lem:support:simple}
Assume $(P_i,\pcs_i)$, for $i\in\{0,1\}$ are as in Lemma~\ref{lem:strat:proj}. For any $p\in P_1$ and any regular $\lambda\in I$ we have:
\begin{equation}\label{support:simple}
(\exists q \in P_0 \quad p \leqlol_1 q) \iff p \leqlol_1 \pi(p)
\end{equation}
\end{lem}
\begin{proof}
One direction is clear, so say $p \leqlol_1 q$ for some $q \in P_0$.
Apply \ref{def:pcs}(\ref{er}): 
As $\pi$ is a strong projection, $p \leq \pi(p) \leq q$ and so $p \leqlol_1 \pi(p)$.
\end{proof}

The intuition behind Definition~\ref{qc:ext} is that $P_0$ and $P_1$ are both quasi-closed, not independently of each other, but in a very coherent way.
That $P_1$ is quasi-closed is almost implicit in Definition~\ref{qc:ext}---it depends on a further assumption about the definability of $\pi$ (this is responsible for the distinct flavor of quasi-closure, setting it apart from the other axioms of stratification): 
\begin{lem}\label{lem:qc:succ}
If $(P_0,P_1)$ is a quasi-closed extension on $I$ and $\pi$ is $\qcdefG(\min I \cup\{\param_1 \})$\index[notation]{SigmaT1(z), SigmaT1(X)@$\Sigma^T_1(z)$, $\Sigma^T_1(X)$}\index[notation]{c bb(parameter in preclosure system) bar@$\bar c$ (parameter in preclosure system)}, then $P_1$ is quasi-closed on $I$.
\end{lem}

Before we give the proof, note that this assumption on $\pi$ is not entirely trivial: in an iteration, the canonical projection $\pi\colon P_\theta \rightarrow P_\iota$ is $\Delta_0$ in the parameter $\iota$; it is not in general $\qcdefG(\min I)$.
Also we would like to note in passing that in fact $P$ is quasi-closed exactly if $(\{ 1_P\}, P)$ is a quasi-closed extension; the same will be true for stratified forcing.
\begin{proof}
First check \ref{def:qc}(\ref{qc:redundant}): 
Say $p \in P_1$, $\lambda\in I$ and $x$ are given. 
Use (\ref{qc:redundant}) for $P_0$ to find $q \in P_0$ such that $q \leqlol \pi(p)$ and $q \in \D_0(\lambda,x,\pi(p))$.
Now apply \eqref{ext:D:dense} to get $p' \leqlol  p$ such that $p' \in \D_0(\lambda,x,p)$ and $\pi(p')=q$. 

For the last clause of \ref{def:qc}(\ref{qc:redundant}), we can assume that $q$ has been chosen so that for any $\lambda'\in I$, if $\pi(p) \leqlo^{\lambda'}_0 1_{P_0}$, then $q \leqlo^{\lambda'}_0 1_{P_0}$.
We can also assume that $p'$ has been chosen so that for any $\lambda'\in I$, if $p \leqlo^{\lambda'}_1  \pi(p)$, then $p' \leqlo^{\lambda'}_1 \pi(p')$.
Thus, $p' \leqlo^{\lambda'}_1 \pi(p) \leqlo^{\lambda'}_1 1_{P_1}$.

It remains to check \ref{def:qc}(\ref{qc:glb}), so say $\bar w$ witnesses that $\bar p =(p_\xi)_{\xi<\rho}$ is $(\lambda,x)$-adequate for $P_1$. 
By definition of adequate sequence $c_1$ is $\qcdefG(x)$ and 
by assumption, $\pi$ is $\qcdefG(\min I \cup\{x,c_1\})$ and hence also $\qcdefG(\lambda\cup\{x\})$.  
So by Lemma~\ref{lem:adeq:proj}, $(\pi(p_\xi))_{\xi<\rho}$ is  $(\lambda,x)$-adequate. 
Since $P_0$ is quasi-closed, $(\pi(p_\xi))_{\xi<\rho}$ has a greatest lower bound $q$.
Thus, applying \ref{def:qc}(\ref{qc:glb}) for $\bar\lambda=\lambda$, we conclude that $\bar p$ has a greatest lower bound.
\end{proof}

The next lemma will be used in \ref{thm:it:qc} when we show that if the initial segments of an iterations form a chain of quasi-closed extensions, then the limit is itself a quasi-closed extension. It says that the relation of being a quasi-closed extension is transitive.
Let $P_0$, $P_1$ and $P_2$ be preorders such that for $i \in \{0,1\}$, $P_i$ is a strong suborder of $P_{i+1}$.
\begin{lem}\label{qc:ext:transitive}\index{quasi-closed extension!is transitive}\index{transitivity!of quasi-closed extension}
Say $\pi_1\colon P_2 \rightarrow P_1$ and $\pi_0\colon P_2 \rightarrow P_0$ are strong projection maps and $\pi_1$ is $\qcdefG(\min I \cup\{ c_2\})$.
If both $(P_0,P_1)$ and $(P_1,P_2)$ are quasi-closed extensions on $I$, then $(P_0,P_2)$ is also a quasi-closed extension on $I$.
\end{lem}
\begin{proof}
Let $(\pcs_0,\pcs_1)$ and $(\pcs_1,\pcs_2)$ witness that $(P_0,P_1)$ and $(P_1,P_2)$ are quasi-closed extensions.

We now check all the conditions of \ref{qc:ext} for $(P_0,P_2)$ and $(\pcs_0,\pcs_2)$.
That $\pcs_2$ is a preclosure system holds by assumption, and that $\pcs_0 \is \pcs_2$ is obvious.

Observe that by \ref{pcs:is}(\ref{ext}) and the remark following Definition~\ref{qc:ext}, we don't need to distinguish between $\leqlol_0$, $\leqlol_1$ and $\leqlol_2$ and therefore we drop the subscripts in what follows.

We check \ref{qc:ext}(\ref{ext:redundant}):
Say $p \in P_2$ and $q \in \D_0(\lambda,x,\pi_0(p))$.
As in the previous proof, we can find $p' \in P_1$ such that $\pi_0(p')=q$, $p'\leqlol_1 \pi_0(p)$ and $p' \in \D_1(\lambda,x,\pi_1(p))$ and then $r$ such that $r\in\D_2(\lambda,x,p)$, with $\pi_1(r)=p'$ and $r\leqlol_2 p$.

Now if $\lambda'\in I$ and $p \leqlo^{\lambda'} \pi_0(p)$, we also have $p \leqlo^{\lambda'} \pi_1(p)$ by Lemma~\ref{lem:support:simple}.
This means we have both $p' \leqlo^{\lambda'} \pi_0(p')$ and $r \leqlo^{\lambda'} \pi_1(r)$.
As $\pi_1(r) = p'$, it follows that $r\leqlo^{\lambda'} \pi_0(r)$.

It remains to check \ref{qc:ext}(\ref{ext:glb}). So let $\bar p = (p_\xi)_{\xi<\rho}$ be a $(\lambda,\bar\lambda,x)$-adequate sequence of conditions in $P_2$ and let $q_0$ be a greatest lower bound of $(\pi_0(p_\xi))_{\xi<\rho}$ as in the hypothesis. 
Fix $\bar w$ which is a $(\lambda,x)$-strategic guide and a $(\bar\lambda,x)$-canonical witness for $\bar p$.
Show exactly as in the proof of Lemma~\ref{lem:qc:succ} that $\bar w$ is both a $(\bar\lambda,x)$-canonical witness and a $(\lambda,x)$-strategic guide for $(\pi_1(p_\xi))_{\xi<\rho}$ in $P_1$.
Denote this sequence by $\bar q$.
Moreover, if it is the case that $\bar\lambda>\lambda$, then 
\begin{equation}\label{glb:supp:long}
\forall \xi <\rho \quad p_\xi \leqlo^{\bar\lambda} \pi_0(p_\xi).
\end{equation}
By \ref{def:pcs}(\ref{pi:mon}), we have that 
\begin{equation}\label{glb:supp:latter}
\forall \xi <\rho \quad\pi_1(p_\xi) \leqlo^{\bar\lambda} \pi_0(p_\xi).
\end{equation}
Thus $\bar q$ satisfies the hypothesis of \ref{qc:ext}(\ref{ext:glb}) for $(P_0,P_1)$ and we may find a greatest lower bound $q_1 \in P_1$ with $\pi_0(q_1)=q$.
Now use \ref{qc:ext}(\ref{ext:glb}) for $(P_1, P_2)$ to find a greatest lower bound $q$ of $\bar p$ such that $\pi_1(q)=q_1$ and so $\pi_0(q)=q_0$.
If $\lambda<\bar\lambda$, Lemma~\ref{lem:support:simple} and \eqref{glb:supp:long} yield
\begin{equation}
\forall\xi < \rho \quad p_\xi \leqlo^{\bar\lambda} \pi_1(p_\xi).\label{glb:supp:former} 
\end{equation}
So finally, as $q \leqlo^{\lambda} \pi_1(q)$ by \eqref{glb:supp:former},
and $\pi_1(q)= q_1 \leqlo^{\bar\lambda} \pi_0(q_1)$ by \eqref{glb:supp:latter},
we conclude
$q \leqlo^{\bar\lambda} \pi_0(q)$.
\end{proof}

\begin{dfn}\label{def:diagsupp}
Say $\theta$ is a limit ordinal, and $\bar Q^\theta$ is an iteration such that for each $\iota <\theta$, $P_\iota$  carries a preclosure system $\pcs_\iota$ on $\Reg\cap[\lambda_\iota,\kappa^*)$, where the sequence $\bar \lambda=(\lambda_\iota)_{\iota<\theta}$ is a non-decreasing sequence of regulars and $\kappa^* \in \Card$ is arbitrary.
All of the following definitions are relative to these preclosure systems and to $\bar \lambda$.
\begin{enumerate}
\item For a thread\index{thread (forcing)} (in the sense of Definition~\ref{it:terminology}) $p$ through $\bar Q^\theta$, let 
\[ 
\supp^\lambda(p)\index[notation]{supp lambda(p)@$\supp^\lambda(p)$|textbf}\index[notation]{lambda-support@$\lambda$-support|textbf} = \{ \iota < \theta \setdef \lambda_\iota\leq\lambda \text{ and } \pi_{\iota+1} \not \leqlol_{\iota+1} \pi_\iota(p) \},
\]
and let $\sigma^\lambda(p)$ be the least ordinal $\sigma$ such that $\supp^\lambda(p)\subseteq \sigma$.

\item\label{d.diagonal.support} Let $\lambda^*$ be regular such that $\lambda^*\geq\lambda_\iota$ for all $\iota<\theta$. We say $P_\theta$ is the \emph{$\lambda^*$-diagonal support limit
\index{diagonal support|textbf}
\index{support!diagonal|textbf}
of $\bar Q^\theta$} if and only if $P_\theta$ consists of all threads $p$ through $\bar Q^\theta$ such that
for each regular $\lambda\geq\lambda^*$, $\supp^\lambda(p)$ has size less than $\lambda$ and $\sigma^{\lambda^*}(p)<\theta$.

\item\label{def:diagsupp:natural} We define the \emph{natural system of relations on $P_\theta$}\index{natural system of relations} to be given as follows 
\begin{enumerate}
\item $p \leqlol_{\theta} q \iff \forall\iota < \theta \quad \pi_\iota(p) \leqlol_{\iota} \pi_\iota(q)$;
\item $p \in \D(\lambda,x,q) \iff \forall \iota<\theta \pi_\iota(q)\in\D(\lambda,x,\pi_\iota(q)$.
\item The parameter $\param_\theta$ has as the form
\[
\param_\theta= (\bar c, L_\mu[A], \bar \pi )
\]
where $\bar c = ( c_\iota)_{\iota<\theta}$\index[notation]{c bb(parameter in preclosure system) bar@$\bar c$ (parameter in preclosure system)} and for each $\iota<\theta$, $c_\iota$ is the parameter from $\pcs_\iota$;
$L_\mu[A]$ is large enough so that $\theta \in L_\mu[A]$; and
$\bar \pi = ( \pi_\iota)_{\iota<\theta}$ where $\pi_\iota$ is the strong projection from $P_\theta$ to $P_\iota$.
\end{enumerate}
We shall see in the proof of Theorem~\ref{thm:it:qc} that under natural assumptions the natural system of relations is a preclosure system. 
\item We say $\bar Q^\theta$ is a \emph{$\bar\lambda$-diagonal support iteration}\index{iteration (forcing)!with diagonal support}\index[notation]{lambdaz-diagonal support iteration@$\bar\lambda$-diagonal support iteration}\index{diagonal support|textbf} if and only if for any limit $\iota<\theta$, $P_\iota$ is the $\lambda_\iota$-diagonal support limit of $\bar Q^\iota$.
\end{enumerate}
\end{dfn}
\begin{thm}\label{thm:it:qc}\index{quasi-closed extension}\index{iteration (forcing)!of quasi-closed extensions}
Let $\bar Q^\theta$ be an iteration such that for each $\iota < \theta$, $P_\iota$ carries a preclosure system $\pcs_\iota$ on $[\lambda_\iota,\kappa^*)$, where the sequence $\bar \lambda=(\lambda_\iota)_{\iota<\theta}$\index[notation]{lambdaz, (lambda xi) xi<kappa@$\bar \lambda$, $(\lambda_\xi)_{\xi\leq\kappa}$} is non-decreasing and $\kappa^*$ is regular.
Moreover, let $\lambda_\theta=\min(\Reg\setminus \sup_{\iota<\theta}\lambda_\iota)$ and assume
\begin{enumerate}
\item For all $\iota<\theta$, $(P_\iota,P_{\iota+1})$ is a quasi-closed extension on $[\lambda_\iota,\kappa^*)$.
\item If $\bar \iota < \theta$ is limit, $\pcs_{\bar\iota}$ is the natural system of relations on $P_{\bar\iota}$ on $[\lambda_\iota,\kappa^*)$ and $\bar Q^\iota$ is a $\bar \lambda\res\iota$-diagonal support iteration.
\item We have $\cof(\theta)<\kappa^*$.
\end{enumerate}
Let $P_\theta$ be the $\lambda_\theta$-diagonal support limit of $\bar Q^\theta$. Then $P_\theta$ is quasi-closed on $I_\theta= [\lambda_\theta,\kappa^*)$, as witnessed by the natural system of relations $\pcs_\theta$.
\end{thm}

In the proof of the theorem, we need the following Lemmas~\ref{supp:union}--\ref{supp:threads}, showing that the notion of $\lambda$-support
behaves as we expect.
So fix an iteration $\bar Q^{\theta+1}$ and preclosure systems as in the hypothesis of the theorem.
These lemmas are somewhat technical but straightforward to show.
\begin{lem}\label{supp:union}
For each $\lambda\in[\lambda_\theta,\kappa^*)\cap\Reg$ and $p\in P_\theta$, 
\[ \supp^\lambda(p)=\bigcup_{\iota<\theta} \supp^\lambda(\pi_\iota(p)).\]
\end{lem}
\begin{proof}
First, prove $\supseteq$:
Say $\xi$ is a member of the set on the right. Thus there is some $\iota<\theta$ such that 
\begin{equation}\label{xi:not:in:support}
\pi_{\xi+1}(\pi_\iota(p))\not\leqlol \pi_\xi(\pi_\iota(p)).
\end{equation}
We consider two cases: first, assume $\iota \leq \xi$. Then $\pi_{\xi+1}(\pi_\iota(p))=\pi_\iota(p)=\pi_{\xi}(\pi_\iota(p))$,
and so as $\leqlol$ is a preorder, \eqref{xi:not:in:support} is false. Thus this case never occurs, and we can assume $\iota > \xi$.
Then $\pi_{\xi+1}(\pi_\iota(p))=\pi_{\xi+1}(p)$ and $\pi_{\xi}(\pi_\iota(p))=\pi_\xi(p)$, so \eqref{xi:not:in:support} is equivalent to 
$\pi_{\xi+1}(p) \not \leqlol \pi_\xi(p)$. We infer that $\xi \in  \supp^\lambda(p)$.
All of the above inferences can be reversed, so $\subseteq$ holds as well.
\end{proof}
\begin{lem}\label{supp:mon}
If $\lambda,\bar\lambda$ are regular such that $\lambda_\theta \leq \lambda \leq \bar\lambda < \kappa^*$ and $p \leqlo^{\bar\lambda}q$, then $\supp^\lambda(p)\subseteq \supp^\lambda(q)$.
\end{lem}
\begin{proof}
Left to the reader.
\end{proof}
\noindent
Observe though that $\supseteq$ does not necessarily hold: in a two-step iteration $P*\dot Q$, we could have and $(p,1_{\dot Q}) \bleqlol (q,\dot q)$ but $q \not \forces_P \dot q \dleqlol 1_{\dot Q}$ (supposing, e.g., $p \forces \dot q=1_{\dot Q}$).
In this example we have  $\suppl(p,1_{\dot Q})=\{0\}\not \supseteq\suppl(q,\dot q)=\{0,1\}$.
\begin{lem}\label{supp:initial}
Fix $\bar \iota<\theta$.
If $p \in P_\theta$ and $\lambda,\bar\lambda$ are regular such that $\lambda_\theta \leq \lambda \leq \bar\lambda< \kappa^*$ and $p \leqlo^{\bar\lambda} \pi_{\bar\iota}(p)$, then
\[\supp^\lambda(p)=\supp^\lambda(\pi_{\bar\iota}(p)). \]
\end{lem}
\begin{proof}
A short proof: $\supseteq$ holds by Lemma~\ref{supp:union} and $\subseteq$ is a consequence of Lemma~\ref{supp:mon}.
We also give a direct proof:
If $\iota < \bar \iota$, $\pi_\iota(\pi_{\bar\iota}(p))=\pi_\iota(p)$, so for such $\iota$, 
\[ \iota \in  \supp^\lambda(p) \iff \iota \in \supp^\lambda(\pi_{\bar\iota}(p)). \]
If $\iota \geq \bar \iota$, we have $\pi_{\iota+1}(p)\leq \pi_\iota(p)\leq \pi_{\bar\iota}(p)$. By assumption and by \ref{pcs:is}(\ref{pi:mon}) for $(P_{\iota+1},P_\theta)$, we have $\pi_{\iota+1}(p)\leqlo^{\bar \lambda} \pi_{\bar\iota}(p)$. 
So by Lemma~\ref{lem:support:simple}, $\pi_{\iota+1}(p)\leqlol\pi_\iota(p)$. We conclude that for $\iota \geq \bar \iota$, we have $\iota \not\in\supp^\lambda(p)$.
\end{proof}

\begin{lem}\label{supp:threads}
Let $\lambda_1 \in [\lambda_\theta,\kappa^*)\cap\Reg$ such that $\lambda_1 \geq \cof(\theta)$.
Say $q = (q^\iota)_{\iota<\theta}$ is a thread\index{thread (forcing)} through $\bar Q^\theta$ (see Definition~\ref{it:terminology}, p.~\pageref{it:terminology} for this terminology) and say there is $w \in P_\theta$ such that for all $\iota<\theta$, $q^\iota \leqlo^{\lambda_1}w$. Then $q$ satisfies the requirement of diagonal support (see Definition~\ref{def:diagsupp}, Item~\ref{d.diagonal.support}), i.e., $q \in P_\theta$.
\end{lem}
\begin{proof}
Let $\lambda$ be regular.
First consider the case $\lambda_\theta \leq \lambda \leq \lambda_1$. As $q \leqlo^{\lambda_1}w$, by Lemma~\ref{supp:mon},
$\supp^\lambda(q)\subseteq\supp^\lambda(w)$, which satisfies the requirement of diagonal support by assumption.
Now say $\lambda>\lambda_1$ and fix a sequence $(\theta(\zeta))_{\zeta<\cof(\theta)}$ which is cofinal in $\theta$.
By Lemma~\ref{supp:union}
\[ \supp^\lambda(q)= \bigcup_{\zeta < \cof(\theta)} \supp^\lambda(q^{\theta(\zeta)}), \]
and by assumption the right hand side is a union over bounded subsets of $\lambda$. Thus $\supp^\lambda(q)$ is a bounded subset of $\lambda$.
\end{proof}

With these lemmas at our disposal, we can give the proof of the theorem.

\begin{proof}[Proof of Theorem~\ref{thm:it:qc}] \renewcommand{\qedsymbol}{{\tiny  Lemma~\ref{thm:it:qc}~}$\Box$}
Now is the time that we will use that $\param_\theta$ includes (by the definition of the natural system of relations, see Definition~\ref{def:diagsupp} Item~\ref{def:diagsupp:natural}, p.~\pageref{def:diagsupp}) the parameters $L_\mu[A]$  and $(\pi_\iota)_{\iota\leq\theta}$, where $\mu$ is a cardinal and $\theta \in L_\mu[A]$.
This ensures that the least sequence $(\theta(\zeta))_{\zeta<\cof(\theta)}$ witnessing  the cofinality of $\theta$ is $\qcdefG(x)$\index[notation]{SigmaT1(z), SigmaT1(X)@$\Sigma^T_1(z)$, $\Sigma^T_1(X)$}.
Alternatively, one could include this sequence itself in $\param_\theta$\index[notation]{c aa(parameter in preclosure system)@$c$ (parameter in preclosure system)}; by recursion $\param_\theta$ would then include an entire ladder system. 
(What we really need below is that $\pi_{\theta(\sigma)}$ is $\qcdefG(w^\xi)$ for each $\zeta < \cof(\theta)$ and $\xi < \rho$. 
In practice, we can replace the parameters $L_\mu[A]$, and $(\pi_\iota)_{\iota\leq\theta}$ by just $\mu$ since the whole iteration up to $\theta$ and $(\pi_\iota)_{\iota\leq\theta}$ can be assumed to be definable from $\theta$ inside $L_\mu[A]$ (if $\mu$ was chosen large enough), and $\theta$ is simply definable from any condition $p \in P_\theta$ and hence is $\qcdefG(w^\xi)$ for each $\xi < \rho$.)

We will show by induction on $\theta$ that for each pair $\iota < \bar \iota \leq \theta$, $(P_\iota, P_{\bar \iota})$ is a quasi-closed extension. Thus $(P_0,P_\theta)$ is a quasi-closed extension and so by Lemma~\ref{lem:qc:succ}, $P_\theta$ is quasi-closed.
The inductive hypothesis is that for each pair $\iota < \bar \iota <\theta$, $(P_\iota, P_{\bar \iota})$ is a quasi-closed extension
as witnessed by $(\pcs_\iota,\pcs_{\bar\iota})$. 
We may assume $\theta$ is limit: 
For if $\theta$ is  a successor ordinal, $\pi^{\theta}_{\theta-1}$ is a $\Delta_1$-definable function and thus by induction hypothesis and Lemma~\ref{qc:ext:transitive}, for any $\iota < \theta$, $(P_\iota, P_\theta)$ is a quasi-closed extension. 

So assume $\theta$ is limit and let $\pcs_\theta$ be the natural system of relations on the diagonal support limit $P_\theta$.
Fix an arbitrary $\iota^*<\theta$. We show that $(P_{\iota^*},P_\theta)$ is a quasi-closed extension witnessed by $(\pcs_{\iota^*},\pcs_\theta)$.
By definition of $\pcs_\theta$, we have $\pcs_{\iota^*} \is \pcs_\theta$.
It is straightforward to show that $\pcs_\theta$ is a preclosure system (as defined in \ref{def:pcs}, p.~\pageref{def:pcs}). 
It is obvious that $\D_\theta$ is $\qcdefF(c_\theta)$:
Letting $\Phi$ be a universal $\Pi_1^A$ formula, by definition of the natural system of relations we can assume the first component $\bar \param = ( \param_\iota)_{\iota<\theta}$ of $\param_\theta$ is such that for each $\iota<\theta$ and each $p, q \in P_\iota$,
\[ q\in \D_\iota(\lambda,x,p)\iff\Phi(c_\iota,q,\lambda,x,p).\]
Thus $\Phi$ and $\bar c$ witness that $\D_\theta$ is $\qcdefD(\{c_\theta\})$: for $q \in\D_\theta(\lambda,x,p)$ is equivalent to
\[ \forall \iota \in\dom(p) \quad\Phi(c_\iota,\pi_\iota(q),\lambda,x,\pi_\iota(p)).\]

We finish the proof that $\pcs_\theta$ is a preclosure system  by proving \ref{def:pcs}(\ref{er}), as the remaining conditions have similar proofs:
Say $p \leq_\theta q \leq_\theta r$ and $p \leqlol_\theta r$. 
Fixing an arbitrary $\iota<\theta$, we have
$\pi_\iota(p) \leq_\iota \pi_\iota(q) \leq_\iota \pi_\iota(r)$ and $\pi_\iota(p) \leqlol_\iota \pi_\iota(r)$. 
Thus, by \ref{def:pcs}(\ref{er}) for $P_\iota$, $\pi_\iota(p)\leqlol_\iota \pi_\iota(q)$. As $\iota<\theta$ was arbitrary, $p \leqlol_\theta q$ holds.
So as mentioned earlier, the natural system of relations is a preclosure system.

We check \ref{qc:ext}(\ref{ext:redundant}).
Say $p \in P_\theta$ and $q^* \in P_{\iota^*}$ are such that $q^*\leqlol_{\iota^*} \pi_{\iota^*}(p)$ and $q^*\in \D_{\iota^*}(\lambda,x, \pi_{\iota^*}(p))$.
Assume first that $\cof(\theta)\in I$ and $\cof(\theta) > \lambda$.
Let $\sigma = \sigma^{\cof(\theta)}(p)$ and observe that $\sigma < \cof(\theta) \leq \theta$ by diagonal support.
If $\sigma \leq \iota^*$, set $r_0 = q^*$ and let $(\theta(\zeta))_{\zeta<\cof(\theta)}$ be the least normal sequence cofinal in $\theta$ such that $\theta(0)=\iota^*$.
Otherwise,
let $(\theta(\zeta))_{\zeta<\cof(\theta)}$ be the least normal sequence cofinal in $\theta$ such that $\theta(0)=\sigma$.
Use \ref{qc:ext}(\ref{ext:redundant}) for $(P_{\iota^*},P_\sigma)$ to get $r_0 \in P_\sigma$ such that $r_0 \in \D(\lambda, x, \pi_{\sigma}(p))$, $r_0 \leqlol_\sigma \pi_{\sigma}(p)$ and $\pi_{\iota^*}(r_0)=q^*$.
Observe that in either case, $p \leqlo^{\cof(\theta)} \pi_{\theta(0)} (p)$.

Now construct by induction on $\zeta < \cof(\theta)$ a thread $(r_\zeta)_{\zeta < \cof(\theta)}$: having $r_\zeta \in P_\zeta$, use \ref{qc:ext}(\ref{ext:redundant}) for $(P_{\theta(\zeta)},P_{\theta(\zeta+1)})$ to get $r_{\zeta+1} \in\D(\lambda, x, \pi_{\theta(\zeta+1)}(p))$ such that $r_{\zeta+1} \leqlol_{\theta(\zeta+1)} \pi_{\theta(\zeta+1)}(p))$,
$\pi_{\theta(\zeta)}(r_{\zeta+1})=r_\zeta$ and in addition, $r_{\zeta+1}\leqlo^{\cof(\theta)}_{\theta(\zeta+1)} r_\zeta$.
Let $r$ be the unique condition in $P_\theta$ defined by the thread $(r_\zeta)_{\zeta < \cof(\theta)}$.
Since $r \leqlo^{\cof(\theta)} r_0$, Lemma~\ref{supp:threads} allows us to conclude that $r$ has legal support. By construction, $r \in \D_\theta(\lambda,x,p)$, $r \leqlol_\theta p$ and $\pi_{\iota^*}(r)=q^*$.

If  $\cof(\theta) \leq \lambda$, we can skip the first step in the above: we just let $r_0=q^*$ and proceed as in the previous case. 
At the end, we use $r \leqlol p$ and Lemma~\ref{supp:threads} allows us to conclude that $r$ has legal support.
We leave the rest of \ref{qc:ext}(\ref{ext:redundant}) to the reader, as it is similar to previous arguments.

\medskip

\textbf{Taking lower bounds of adequate sequences:}
Now check \ref{qc:ext}(\ref{ext:glb}), the existence of greatest lower bounds.
Say  $\bar p = (p_\xi)_{\xi<\rho}$ is a $(\lambda,\bar\lambda,x)$-adequate sequence of conditions in $P_\theta$; then $\lambda$ and $\bar \lambda$ are both regular, $\lambda_\theta\leq\lambda\leq\bar\lambda< \kappa^*$. 
We may fix $\bar w$ which is both a $(\lambda,x)$-strategic guide and a $(\bar\lambda,x)$-canonical witness for $\bar p$.
Moreover, let $q^*$ be a greatest lower bound of the sequence $(\pi_{\iota^*}(p_\xi))_{\xi<\rho}$.

The construction of $q$ is by induction on $\zeta < \cof(\theta)$, and we shall use a sequence $(\theta(\zeta))_{\zeta<\cof(\theta)}$.
The construction is split in cases for the following reasons:
When $\cof(\theta) > \bar \lambda$, the definability of each $\pi_{\theta(\zeta)}$ poses a problem,  and so we first have to find a $\pi_{\cof(\theta)}$-bound.
Moreover, the argument that supports are legal goes differently depending on whether $\cof(\theta) \leq\lambda$ or not.
There are several possibilities for distinguishing cases as the trick for dealing with $\cof(\theta)>\bar \lambda$ can be applied whenever $\cof(\theta) > \lambda_\theta$. 
We give the version of the argument that seemed clearest.

First, assume that $\cof(\theta)\leq \bar\lambda$.
Let $\bar \theta = (\theta(\zeta))_{\zeta\leq\cof(\theta)}$ be the least (i.e., least in the sense of the canonical well-order of $L[S]$) normal sequence such that $\theta(0)=\iota^*$ and $\theta(\cof(\zeta))=\theta$.
By induction on $\zeta$, we now construct  a lower bound $q^{\zeta}\in P_{\theta(\zeta)}$ of the sequence $(\pi_{\theta(\zeta)}(p_\xi))_{\xi<\rho}$ for each $\zeta \leq \cof(\theta)$. 
Set $q^{0}=q^*$. 
Now assume we have $q^{\zeta}$ and show how to find $q^{\zeta+1}$.

The decisive point regarding definability is that $(\pi_{\theta(\zeta+1)}(p_\xi))_{\xi<\rho}$ is $(\lambda,\bar\lambda,x)$-adequate in $(P_{\theta(\zeta+1)},P_{\theta(\zeta)})$.
This holds because $\zeta < \bar\lambda$ and because 
 $\bar \theta$ is easily obtained from $\param_\theta$.
By \ref{qc:ext}(\ref{ext:glb}) we obtain a greatest lower bound $q^{\theta(\zeta+1)}\in P_{\theta(\zeta)}$. 

Now let $\zeta\leq\cof(\theta)$ be limit.
By construction and by (\ref{ext:glb}), the $q^{\zeta'}$, for $\zeta'<\zeta$ form a  thread. Define $q^{\theta(\zeta)}$ to be this thread. 
To show it has legal support, first assume that $\cof(\theta)\leq\lambda$.
By construction and by (\ref{ext:glb}), for each $\zeta$, we have $q^{\zeta}\leqlo^{\lambda} p_0$. 
As $\lambda$ is greater than the maximum of $\lambda_\theta$ and $\cof(\theta)$, Lemma~\ref{supp:threads} allows us to infer that $q^{\zeta}$ has legal support, and thus is a condition in $P_{\theta(\zeta)}$ and a $\pi_{\theta(\zeta)}$-bound of $\bar p$ (in the sense of Definition~\ref{adequate2}).
The final condition $q^{\cof(\theta)}$ is a greatest lower bound of $(p_\xi)_{\xi<\rho}$ and for all $\xi<\rho$, $q^{\cof(\theta)} \leqlol p_\xi$.

To finish the case where $\cof(\theta) \leq\bar\lambda$, we have to consider the subcase
where $\cof(\theta) \leq\lambda$ fails, i.e., we assume $\lambda < \bar \lambda$ and $\cof(\theta) \in (\lambda, \bar\lambda]$.
In this case we have  $p_\xi\leqlo^{\bar\lambda} \pi_{\iota^*}(p_\xi)$ for all $\xi<\rho$.
By \ref{qc:ext}(\ref{ext:glb}) and by induction, we have $q^{\zeta} \leqlo^{\bar\lambda} q^{\iota^*}$ for each $\zeta<\cof(\theta)$, and so $q \leqlo^{\bar\lambda} q^{\iota^*}$.
Moreover, as $\bar\lambda$ is greater than both $\cof(\theta)$ and $\lambda_\theta$, 
$q^\zeta$ has legal support by Lemma~\ref{supp:threads}.
Thus we are finished with the case $\cof(\theta) \leq \bar\lambda$.

Now assume $\cof(\theta) > \bar\lambda$.
We will now find $\iota'$ and $q'$ such that $q'$ is a $\pi_{\iota'}$-bound of $\bar p$ and for each $\xi < \rho$, 
\begin{equation}\label{better_iota}
p_\xi \leqlo^{\cof(\theta)}\pi_{\iota'}(p_\xi).
\end{equation}
As $\rho\leq\lambda<\cof(\theta)$ and $\sigma^{\cof(\theta)}(p_\xi)<\cof(\theta)$ for each $\xi<\rho$, letting
\[ \iota' = \sup_{\xi<\rho} \sigma^{\cof(\theta)}(p_\xi),\] 
we have $\iota' <\cof(\theta)$
and so $\iota' < \theta$.

Fix a $\nu < \rho$ and let $\sigma(\nu) = \sigma^{\cof(\theta)}(p_\nu)$.
Let $\bar q^\nu$ be the sequence defined by $q^\nu_\xi = \pi_{\sigma(\nu)}(p_\xi)$ for $\xi \geq \nu$. For $\xi < \nu$, the value of  $q^\nu_\xi$ is arbitrary as long as $\bar q^\nu$ has $\bar w$ as a canonical witness (e.g. set $q^\nu_\xi = \pi_{1}(p_0)$).
It is straightforward to check that $\bar w$ is a $(\bar\lambda,x)$-canonical witness and a $(\lambda,x)$-strategic guide for each sequence $\bar q^\nu$, for $\nu < \xi$ --- this is why we only \label{tail} make demands on a tail in the definition of a strategic guide. 
It is also clear that we can build  a thread $(q_\nu)_{\nu<\rho}$ using (\ref{ext:glb}), where $q_\nu$ is a greatest lower bound of $\bar q^\nu$.
Let $q'$ be this thread.
We invite the reader to check that $q'$ is a greatest lower bound of the sequence $(\pi_{\iota'}(p_\xi))_{\xi<\rho}$ and that $q'\leqlol_{\iota'} \pi_{\iota'}(p_0)$.
Thus, $q'$ has legal support as $\lambda \geq \cof(\rho)$ and by Lemma~\ref{supp:threads}.
By choice of $\iota'$, \eqref{better_iota} holds.
Observe that in the case $\lambda<\bar \lambda$, (\ref{ext:glb}) also entails $q' \leqlol \pi_{\iota^*}(q')=q^*$.

Now we argue exactly as in the case where $\cof(\theta)\leq\lambda$, but this time
setting $q^0=q'$ and $\theta(0)=\iota'$.  Again we build a sequence $(q^\zeta)_{\zeta<\cof(\theta)}$ by induction.
At each successor step, $\{ \pi_{\theta(\zeta+1)}(p_\xi) \setdef \xi<\rho \}$ is $(\lambda,\cof(\theta),x)$-adequate in $(P_{\theta(\zeta)},P_{\theta(\zeta+1)})$.
By \ref{qc:ext}(\ref{ext:glb}) we obtain a greatest lower bound $q^{\zeta+1}\in P_{\theta(\zeta)}$ such that $q^{\zeta+1} \leqlo^{\cof(\theta)} q^{\zeta}$. 
By induction, this entails $q^{\iota'}\leqlo^{\cof(\theta)} q^0 = q'$.
Thus at each limit stage $\zeta\leq\cof(\theta)$, $q^\zeta\leqlo^{\cof(\theta)} q'$ 
and Lemma~\ref{supp:threads} allows us to conclude that $q^{\zeta}$ has legal support.
Lastly, if $\lambda <\bar\lambda$, as $q^\zeta\leqlo^{\cof(\theta)} q' \leqlol q^*$, we also have
$q^\zeta\leqlol q^*=\pi_{\iota^*}(q^\zeta)$. \renewcommand{\qedsymbol}{{\tiny  Theorem~\ref{thm:it:qc}~}$\Box$}
\end{proof}

We conclude this section with an observation about the support of a greatest lower bound of an adequate sequence.
\begin{lem}\label{supp:glb}\index[notation]{supp lambda(p)@$\supp^\lambda(p)$}\index[notation]{lambda-support@$\lambda$-support}
Say $\bar p =(p_\xi)_{\xi<\rho}$ is a $(\lambda,x)$-adequate sequence with greatest lower bound $p$. Then for any regular $\bar \lambda$,
\[ \supp^{\bar\lambda}(p) \subseteq \bigcup_{\xi<\rho}\supp^{\bar\lambda}(p_\xi).\]
\end{lem}
\begin{proof}
Assume $\iota<\theta$ and $\iota \not \in \bigcup_{\xi<\rho}\supp^{\bar\lambda}(p_\xi)$. We may assume $\iota < \bar\lambda$ (since $p$ has diagonal support).
Then as $\pi_{\iota+1}$ is $\qcdefSeq(\bar\lambda)$, the sequence $(\pi_{\iota+1}(p_\xi))_{\xi<\rho}$ is $\qcdefSeq(\bar\lambda\cup\{x\})$-definable,
and for all $\xi<\rho$, $\pi_{\iota+1}(p_\xi)\leqlo^{\bar\lambda}\pi_{\iota}(p_\xi)$. Therefore we can apply \ref{qc:ext}(\ref{ext:glb})  for $(P_\iota,P_{\iota+1})$ (see  \ref{qc:ext}, p.~\pageref{qc:ext}).
We conclude that $\pi_{\iota+1}(p)\leqlo^{\bar\lambda}\pi_\iota(p)$ and so $\iota \not\in \supp^{\bar\lambda}(p)$.
\end{proof}

\section{Stratified Extension and Iteration}\label{sec:s:ext}

In this section, we begin by given the definition of stratified extension. 
We show that composition of stratified forcing is a special case of stratified extension.
We show that the second forcing in a stratified extension is stratified.
Finally we prove an iteration theorem for stratified forcing.

\medskip

Let $P_0$ be a complete suborder of $P_1$ and let $\pi\colon P_1 \rightarrow P_0$ be a strong projection and let $I$ be an interval of regular cardinals.
Moreover, assume for $i \in \{0,1\}$, we have a system 
\[\pss_i=(\D_i, \param_i, \leqlol_i, \lequpl_i, \Cl_i)_{\lambda\in I}
\]
such that $\D\subseteq I \times V \times (P_i)^2 $ is a class which is definable with parameter $\param_i$, and for every $\lambda\in I$,
$\leqlol_i$ and $\lequpl_i$ are binary relations on $P_i$ and $\Cl_i \subseteq P_i\times\lambda$.
\begin{dfn}\label{pss:is}
We write $\pss_0 \is \pss_1$\index[notation]{lessthan triangle@$\is$|textbf}\index[notation]{  lessthan triangle@$\is$|textbf} if and only if in addition to (\ref{ext}), (\ref{pi:mon}) and (\ref{F:coh}) (see \ref{pcs:is}, p.~\pageref{pcs:is}), the following hold:
\begin{enumerate}[label=(\pssIsRef), ref=\pssIsRef] %
\item If $q,q' \in P_0$ and $p,p' \in P_1$ are such that $q' \leqlol_0 q \leq \pi(p')$ and $p' \leqlol_1 p$, then $q'\cdot p' \leqlol_1 q\cdot p$. \label{pss:is:cdot:lo}
\item For all $p,q \in P_0$, $p\lequpl_0 q \Rightarrow p \lequpl_1 q$. \label{pss:is:ext}
\item For all $p,q \in P_1$, $p \lequpl_1 q \Rightarrow \pi(p) \lequpl_0 \pi(q)$. \label{pss:is:pi:mon}
\item If $w \leq \pi(d), \pi(r)$ and $d \lequpl r$ then $w \cdot d \lequpl w \cdot r$. \label{pss:is:cdot:up}
\item If $\Cl_1(p)\cap\Cl_1(q)\neq 0$ then $\Cl_0(\pi(p))\cap\Cl_0(\pi(q))\neq 0$. \label{pss:is:c}
\end{enumerate}
\end{dfn}
Observe that if $\pss_0 \is \pss_1$, we can drop the subscripts on $\lequpl_0$, $\lequpl_1$ and just write $\lequpl$ without causing confusion.
Observe also that by Corollary~\ref{lambda:is:big}, we can assume that $r \leqlo^{\card{P_1}} p$ holds exactly if $p=r$.
This implies\footnote{Interestingly, \eqref{suborder:and:leqlo} also follows just from the assumption that for any $r,p\in P_1$, $r \lequp^{\card{P_1}} p$, together with (\ref{s:ext:exp}) \emph{Coherent Expansion}}
\begin{equation}\label{suborder:and:leqlo}
  \forall p\in P_1 \big( p \leqlo^{\card{P_1}}\pi(p) \iff p \in P_0 \big).
\end{equation}
We could assume that $\Cl_0 = \Cl_1 \cap P_0\times \lambda$.
For if not, simply replace $\Cl_1$ by the following relation $\Cl_*$:
$s \in \Cl_*(p)$ if and only if $s \in {}^{\leq 2}\lambda$ such that $s(0) \in \Cl_0(\pi(p))$ and if $p \not \in P_0$ then $1 \in \dom(s)$ and $s(1) \in \Cl_1(p)$ (now in fact we get $\Cl_0(p) = \{ s\res 1 \setdef s \in \Cl_1(p) \}$ for $p \in P_0$).
To sum up, we could in principle completely eliminate any mention of $\pss_0$ from the definition of stratified extension.

Replacing (\ref{pss:is:cdot:lo})by the following two conditions yields an equivalent version of the above definition:
\begin{enumerate}[label=(\pssIsRe\Alph*), ref=\pssIsRe\Alph*]
\item $w \leqlol_0 \pi(p) \Rightarrow w \cdot p \leqlol_1 p$. \label{pss:is:lo:init}
\item If $w \leq \pi(p')$ and $p' \leqlol p$ then $w\cdot p' \leqlol w \cdot p$. \label{pss:is:lo:tail!}
\end{enumerate} 
Sometimes it is more convenient to check both of these rather than (\ref{pss:is:cdot:lo}), which is concise but cumbersome to show.
Further notice that (\ref{pss:is:lo:tail!}) implies
\begin{enumerate}[label=(\pssIsRe b), ref=\pssIsRe b]
\item If $w \leq \pi(p)$ and $p \leqlol \pi(p)$ then $w\cdot p \leqlol w$. \label{pss:is:lo:tail}
\end{enumerate}
This is weaker than (\ref{pss:is:lo:tail!}). We note in passing that we could do entirely with (\ref{pss:is:lo:init}) and (\ref{pss:is:lo:tail}) and without (\ref{pss:is:lo:tail!}).
Neither condition (\ref{pss:is:cdot:lo}) nor any of its variants were included in \ref{def:pcs}, the definition of a preclosure system simply because they are not needed to preserve quasi-closure in iterations---rather we need (\ref{pss:is:lo:tail}) to preserve coherent linking, and (\ref{pss:is:lo:init}) helps to preserve density at limits; see below.

We fix some convenient notation: If $d\leq r$, we say \emph{$p$ $\lambda$-interpolates $d$ and $r$} to mean that $p \lequpl d$ and $p\leqlol r$.
We say $p \leqlo^{<\lambda}q $ to mean that for all $\lambda'\in I \cap \lambda$, 
we have $p \leqlo^{\lambda'} q$.
\begin{dfn}\label{def:s:ext}
Let $I$ be an interval of regular cardinals.
We say the pair $(P_0,P_1)$ is a \emph{stratified extension\index{stratified extension|textbf} on $I$, as witnessed by $(\pss_0,\pss_1)$} 
if and only if $\pss_0$ witnesses that $P_0$ is stratified on $I$, $\pss_1$ is a prestratification system on $P_1$ and $\pss_0 \is \pss_1$; 
Moreover, for all $\lambda \in I$ we have that (\ref{ext:redundant}), (\ref{ext:glb}) and all of the following conditions hold:
\begin{enumerate}[label=(\sExtRef), ref=\sExtRef]

\item \emph{Coherent Expansion}\index{Coherent!Expansion}: For $p,d \in P_1$, if $p \lequpl d$, $d \leqlol \pi(d)$ and $\pi(p)\leq\pi(d)$, we have that $p \leq d$.\label{s:ext:exp}

\item \emph{Coherent Interpolation}:\index{Coherent!Interpolation} Given $d,r\in P_1$ such that $d \leq r$ and $p_0 \in P_0$ such that $p_0$ $\lambda$-interpolates $\pi(d)$ and $\pi(r)$ we can find $p \in P_1$ which $\lambda$-interpolates $d$ and $r$ such that $\pi(p)=p_0$. If moreover $d \leqlo^{<\bar\lambda}\pi(d)$, we can in addition assume $p \leqlo^{<\bar\lambda} \pi(p)\cdot r$.\label{s:ext:int}

\item \emph{Coherent Linking}:\index{Coherent!Linking} Say $d,p \in P_1$, $d \lequpl p$ and $\Cl_1(d)\cap \Cl_1(p) \neq\emptyset$. 
Given $w_0 \in P_0$ such that both $w_0 \leqlo^{<\lambda} \pi(d)$ and $w_0 \leqlo^{<\lambda} \pi(p)$, we can find $w \in P_1$ such that $w\leqlo^{<\lambda}p,d$ and $\pi(w)=w_0$.\label{s:ext:cent}

\end{enumerate}
We find it relieving to notice that $P$ is stratified exactly if $(\{1_P\}, P)$ is a stratified extension.
Again, if we don't mention $\pss_0$, $\pss_1$ or $I$ we are either claiming that they can be appropriately defined or they can be inferred from the context.
\end{dfn}

\begin{lem}\label{stratified:comp:implies:ext}\index{composition (forcing)!stratified extension}
If $P$ is stratified on $I$ and $\forces_P \dot Q$ is stratified on $I$, then $(P, P*\dot Q)$ \
is a stratified extension on $I$.

To be more precise, let $\pss_0$ denote the prestratification system witnessing that $P$ is stratified and let $\pss_1=(\bar\D,\bar\param,\bleqlol,\blequpl,\bCl)_{\lambda\in I}$ be the prestratification system constructed as in the proof of \ref{stratified:composition}, where we showed that $P*\dot Q$ is stratified. Then $(\pss_0,\pss_1)$ witnesses that $(P, P*\dot Q)$ is a stratified extension on $I$.
\end{lem}

\begin{proof}
We have already checked (\ref{ext}), (\ref{pi:mon}), (\ref{F:coh}), (\ref{ext:redundant}) and (\ref{ext:glb})---i.e., that $(P_, P*\dot Q)$ is a \emph{quasi-closed extension}---in Lemma~\ref{qc:comp:implies:ext}.
We showed that $\pss_1$ is a prestratification system when we proved Theorem~\ref{stratified:composition}.
It's technical but straightforward to check that $\pss_0 \is \pss_1$ (see Definition~\ref{pss:is}, p.~\pageref{pss:is}):

Fix $\bar p=(p,\dot p) \in P*\dot Q$ and $ w \in P$, $w \leq p$.
For (\ref{pss:is:lo:init}), say $w \leqlol p$. Then as $p \forces \dot p \dleqlol p$, we have $(w,p)\bleqlol \bar p$.
For (\ref{pss:is:lo:tail!}), fix another condition $\bar q=(q,\dot q) \in P*\dot Q$ such that $\bar p \bleqlol \bar q$. Then $p \forces \dot p \dleqlol \dot q$, whence $w\forces \dot p \dleqlol \dot q$ and so $(w,p) \bleqlol (w, 1_{\dot Q})$, done.
For (\ref{s:ext:exp}) and (\ref{pss:is:cdot:up}), let $\bar r=(r,\dot r) \in P*\dot Q$ and say $\bar p \blequpl \bar r$, i.e., $p \lequpl r$ and if $p\cdot r > 0$ then $p \cdot q \forces \dot p \dlequpl \dot r$.
To check (\ref{s:ext:exp}) \emph{Coherent Expansion},  assume $\bar r \bleqlol (r,1_{\dot Q})$ and $p \leq r$. Then $p \forces  \dot p\dlequpl \dot r$. As $P$ forces \emph{expansion} for $\dot Q$, $p \forces \dot p \leq \dot q$ and we are done with (\ref{s:ext:exp}).
To check (\ref{pss:is:cdot:up}), say $w \leq p$. Then $w \cdot r \leq p \cdot r$, and so if $w\cdot r > 0$, it forces $\dot p \dlequpl \dot r$.
Since $w \lequpl w$, we infer that $(w,\dot p) \blequpl \bar r$.
The remaining (\ref{pss:is:ext}), (\ref{pss:is:pi:mon}) and (\ref{pss:is:c}) are immediate by the definition.

Now we check the conditions of \ref{def:s:ext} (see p.~\pageref{def:s:ext}).
For (\ref{s:ext:int}) \emph{Coherent Interpolation}, just look at how we found an interpolant in the proof of Theorem~\ref{stratified:composition}.
Do the same for (\ref{s:ext:cent}) \emph{Coherent Linking}.
\end{proof}

The following is the analogue of \ref{lem:qc:succ} for quasi-closed extension:
\begin{lem}\label{lem:s:succ}\index[notation]{SigmaT1(z), SigmaT1(X)@$\Sigma^T_1(z)$, $\Sigma^T_1(X)$}
If $(P_0,P_1)$ is a stratified extension on $I$ and $\pi$ is $\qcdefG(\min I \cup\{\param_1 \})$, then $P_1$ is stratified on $I$.
\end{lem}
\begin{proof}
This follows straightforwardly from the definitions and Lemma~\ref{lem:qc:succ}. We leave details to the reader.
\end{proof}

\begin{dfn}\label{natural:pss}\index{natural system of relations}
Say $\bar Q^\theta$ is an iteration such that each initial segment $P_\iota$ carries a prestratification system $\pss_\iota$ on $I_\iota = \Reg\cap[\lambda_\iota,\kappa)$, where the sequence $\bar \lambda=(\lambda_\iota)_{\iota\leq\theta}$ is a non-decreasing sequence of regulars.  
Let $P_\theta$ be its $\lambda_\theta$-diagonal support limit.
We now add to the definition of \emph{natural systems of relations on $P_\theta$}.
Let $\lambda\in I_\theta$, where we set $I_\theta= [\lambda_\theta,\kappa)\cap \Reg$.
The relations $\leqlol$ and $\D$ are defined as in \ref{def:diagsupp}, p.~\pageref{def:diagsupp}. Let
\begin{enumerate}%
\item $p \lequpl_{\theta} q \iff \forall\iota < \theta \quad \pi_\iota(p) \lequpl_{\iota} \pi_\iota(q)$;
\item $p \in \dom(\Cl)$ if and only if for all $\iota <\sigma^\lambda(p)$, $\pi_\iota(p)\in\dom(\Cl_\iota)$;
\item $s \in \Cl_\theta(p)$ if and only if $s \colon \sigma^\lambda(p)\rightarrow\lambda$ and for all $\iota < \dom(s)$, we have $s(\iota)\in\Cl_\iota(\pi_\iota(p))$.
\end{enumerate}
\end{dfn}

As before, the above yields a prestratification system under natural assumptions, as we shall see in the proof of Theorem~\ref{thm:it:strat}.
\begin{thm}\label{thm:it:strat}\index{iteration (forcing)!of stratified extensions}
Let $\bar Q^\theta$ be an iteration such that for each $\iota < \theta$, $P_\iota$ carries a prestratification system $\pss_\iota$ on $I_\iota$, where $I_\iota = \Reg\cap[\lambda_\iota,\kappa)$ and the sequence $\bar \lambda=(\lambda_\iota)_{\iota<\theta}$ is a non-decreasing sequence of regulars.
Moreover, let $\lambda_\theta=\min(\Reg\setminus \bigcup_{\iota<\theta}\lambda_\iota)$, let $I_\theta= [\lambda_\theta, \kappa)$ and assume
\begin{enumerate}
\item For all $\iota<\theta$, $(P_\iota,P_{\iota+1})$ is a stratified extension on $I_\iota$.
\item If $\bar \iota < \theta$ is limit, $\pss_{\bar\iota}$ is the natural system of relations on $P_{\bar\iota}$
and $P_{\bar \iota}$ is the $\lambda_{\bar\iota}$-diagonal support limit
\index{diagonal support}
\index{support!diagonal}
of $\bar Q^{\bar \iota}$.
\item For each regular $\lambda\in[\lambda_\theta,\kappa)$ there is $\iota<\lambda^+$ such that for all $p \in P_\theta$ we have $\supp^\lambda(p) \subseteq \iota$.
\end{enumerate}
Let $P_\theta$ be the $\lambda_\theta$-diagonal support limit of $\bar Q^\theta$. Then $P_\theta$ is stratified on $I_\theta$.
\end{thm}
\begin{rem}
In our particular application we will have that for each regular $\lambda<\kappa$, there is $\iota<\lambda^+$ such that $\lambda<\lambda_\iota$.
Observe that by the definition of $\supp^\lambda(p)$, this implies that the last clause of the above is satisfied.
\end{rem}

Of course, the following proof can be easily adapted to show that under the same hypothesis, for every $\iota<\theta$, $(P_\iota,P_\theta)$ is a stratified extension on $I_\theta$; while this approach facilitated the inductive proof in the case of quasi-closure, it would serve no purpose in the present context.
\begin{proof}[Proof of Theorem~\ref{thm:it:strat}.]
By Lemma~\ref{lem:s:succ}, we may assume $\theta$ is limit.
That $P_\theta$ is stratified on $I_\theta$ is witnessed by the natural system of relations $\pss_\theta$, as defined in \ref{natural:pss}.
The proof of the following lemma is a straightforward induction, which we leave to the reader:
\begin{lem}
For any $\iota <\bar\iota\leq\theta$, $\pss_\iota \is \pss_{\bar\iota}$.
\end{lem}

Next, we check that $\pss_\theta$ is a prestratification system (see \ref{def:pss} p.~\pageref{def:pss}):
Conditions~(\ref{up:extra}), (\ref{up}) and (\ref{s:lequp:vert}) are immediate by the definition of $\lequpl_\theta$ and the fact that for each $\iota<\theta$, $\pss_\iota$ is a prestratification system. The proofs resemble that of (\ref{exp}), see below.

The non-trivial condition is  \ref{def:pss}(\ref{density}), \emph{Density}.\label{C:dense}
First we must check that $\ran(\Cl_\theta)$ has size at most $\lambda$: this is because by the last assumption of the theorem and by diagonal support, 
$\supp^\lambda(p)\in[\iota]^{<\lambda}$ for some $\iota<\lambda^+$.

For the more interesting part of the argument, we use density and \emph{Continuity} for the initial segments $P_\iota$, $\iota<\theta$ together with quasi-closure.
Observe that by Theorem~\ref{thm:it:qc}, for any $\iota < \bar\iota \leq \theta$, $(P_\iota,P_\theta)$ is a quasi-closed extension on $I_\theta$.
Say we are given $p \in P_\theta$. 
Let $\sigma = \sigma^\lambda(p)$. We may assume that $\sigma=\theta$, for otherwise we can use induction and \emph{Density} for $P_\sigma$ and are done. 
Thus we can assume both $\lambda_\theta<\lambda$ and $\cof(\theta)<\lambda$, for otherwise, since $P_\theta$ is a diagonal support limit, $\supp^{\lambda}(p)$ is bounded below $\theta$.

So say we are given $\lambda' \in [\lambda_\theta,\lambda)$. We must find $q \leqlo^{\lambda'} p$ such that $q \in \dom(\Cl)$.
Let $\delta=\cof(\theta)$ and assume without loss of generality $\lambda'\geq\delta$ (otherwise we may increase $\lambda'$). 
Fix a normal sequence $(\sigma(\xi))_{\xi\leq\delta}$ such that $\sigma(\delta)=\theta$. 
We inductively construct a $\delta$-adequate sequence $(p_\xi)_{\xi<\delta}$ such that $p_0=p$ and for any $\nu,\xi$ such that $\nu < \xi < \delta$, 
\begin{equation}\label{inDomainC}
 \pi_{\sigma(\nu)}(p_{\xi}) \in \dom(\Cl_{\sigma(\nu)}).
\end{equation}
Fix appropriate $x$ so that the following sequence can be built in a $(\delta,x)$-adequate fashion and such that $(\sigma(\xi))_{\xi\leq\delta}$ is a component of $x$. 
That is, put any set which is mentioned below and would otherwise push up the complexity of the definition into $x$.
We follow the recipe described in Lemma~\ref{adequate} to construct a a $(\delta,x)$-adequate sequence.
Let $p_0 = p$. 
Assuming we have $p_\xi$ and $\bar w \res \xi+1$, find $p_{\xi+1}$ as follows.

We may find $q \in P_\theta$ such that $q \in \D_\theta(\lambda', (x,\bar w \res \xi+1) , p_\xi)$ and $q\leqlo^{\lambda'}p_\xi$. 
Also, there is $q' \leqlo^{\lambda'} \pi_{\sigma(\xi)}(q)$ such that $q \in \dom(\Cl_{\sigma(\xi)})$.
Let $p_{\xi+1}=q'\cdot q$.
Since $\pss_{\sigma(\xi)}\is \pss_\theta$, by \ref{pss:is}(\ref{pss:is:lo:init}), $p_{\xi+1} \in \D_\theta(\lambda',x,p_\xi)$, and also $p_{\xi+1}\leqlo^{\lambda'} p_\xi$. 
Moreover, by \ref{pss:is}(\ref{pss:is:c}), for any $\nu\leq\xi$, 
\[
\pi_{\sigma(\nu)}(p_{\xi+1})\in\dom(\Cl_{\sigma(\nu)}).
\]
Apply the usual trick  to find $w_{\xi+1}$ and $p_{\xi+1}$ as above so as to obtain a $(\delta,x)$-adequate sequence:
see again Lemma~\ref{adequate} (there is no reason to repeat the argument here).

At limit stages $\bar \xi \leq \delta$,  $p_{\bar \xi}$ is a greatest lower bound in $P_\theta$ of the sequence constructed so far. 
It exists by quasi-closure for $P_\theta$.
We show
\begin{equation}\label{inDomainC:limit}
\pi_{\sigma(\bar\xi)}(p_{\bar \xi})\in\dom(\Cl_{\sigma(\bar\xi)}).
\end{equation}
Let $\nu < \bar \xi$ be arbitrary.
As $(P_{\sigma(\nu)},P_\theta)$ satisfies (\ref{qc:glb}), 
\begin{equation}
\pi_{\sigma(\nu)}(p_{\bar \xi}) = \prod_{\xi\in(\nu,\bar\xi)} \pi_{\sigma(\nu)}(p_\xi).
\end{equation}
We want to apply (\ref{continuous}) for $P_{\sigma(\nu)}$.
By choice of $x$, $(\pi_{\sigma(\nu)}(p_\xi))_{\xi\in(\nu,\bar\xi)}$ is a $\lambda'$-adequate sequence; 
and so we may use \emph{Continuity} for $P_{\sigma(\nu)}$. 
Now by induction hypothesis, \eqref{inDomainC} holds for all $\xi \in (\nu,\bar \xi)$ and so by (\ref{continuous}) we have $\pi_{\sigma(\nu)}(p_{\bar \xi}) \in \dom(\Cl_{\sigma(\nu)})$.
As $\nu<\bar \xi$ was arbitrary and by definition of $\Cl_{\sigma(\bar \xi)}$ we conclude that \eqref{inDomainC:limit} holds.
In particular, for the last stage of our construction, we set $\bar\xi=\delta$ in \eqref{inDomainC:limit} and conclude $p_\delta\in\dom(\Cl_{\theta})$, finishing the proof of \emph{Density}. So $\pss_\theta$ is a prestratification system.

\emph{Quasi-closure} was shown in Lemma~\ref{thm:it:qc}.
First we check conditions (\ref{exp})--(\ref{linking}) of \ref{stratified:main}, stratification (see p.~\pageref{stratified:main}). 
\emph{Expansion} (\ref{exp}) is trivial:
If $d \lequpl_\theta r$ and $r \leqlol_\theta 1$, then for all $\iota<\theta$, $\pi_\iota(d) \lequpl_\iota \pi_\iota(r)$ and $\pi_\iota(r) \leqlol_\iota 1$.  By induction, we may assume \emph{expansion} holds for each $P_\iota$, $\iota<\theta$. Thus $d \leq r$.

We show \emph{Interpolation} (\ref{interpolation}) holds.
So fix $d,r \in P_\theta$ such that $d \leq r$ holds. We construct the interpolant $p$ by induction on its initial segments $\pi_\iota(p)$, for $\iota<\theta$. 
Say we have already constructed $\pi_\iota(p)$. 
Use \emph{Coherent Interpolation} for $(P_\iota,P_{\iota+1})$ to obtain $\pi_{\iota+1}(p)$ interpolating $\pi_{\iota+1}(d)$ and $\pi_{\iota+1}(r)$: 
demand that 
\begin{equation}\label{int:above:supp:construction}
\pi_{\iota+1}(p) \leqlo^{<\bar\lambda(\iota)} \pi_{\iota+1}(r)\cdot \pi_\iota(p),
\end{equation}
 where $\bar \lambda(\iota)$ is the maximal $\bar \lambda$ with the property that $\pi_{\iota+1}(d)\leqlo^{<\bar\lambda} \pi_\iota(d)$.\footnote{actually, it would suffice to demand this whenever $\iota \geq \sigma^\lambda(d)$}
We claim that for any $\gamma \in I_\theta$,
\begin{equation}\label{int:above:supp}
 \iota \not \in \supp^\gamma(d)\cup\supp^\gamma(r)\Rightarrow \pi_{\iota+1}(p) \leqlo^\gamma \pi_\iota(p).
\end{equation}
So fix $\gamma\in\Reg$ and assume the hypothesis of \eqref{int:above:supp}. As $d\res\iota+1 \leqlo^\gamma d\res\iota$, by \ref{def:pcs}(\ref{qc:leqlo:vert}) and by definition of $\bar\lambda(\iota)$, we have $\gamma < \bar\lambda(\iota)$. Thus, \eqref{int:above:supp:construction} yields
\begin{equation}\label{int:above:supp:p:pr}
\pi_{\iota+1}(p) \leqlo^{\gamma} \pi_{\iota+1}(r)\cdot \pi_\iota(p).
\end{equation}
Since $r\res\iota+1 \leqlo^\gamma r\res\iota$, by \ref{pss:is}(\ref{pss:is:lo:tail}) we infer 
\begin{equation}\label{int:above:supp:pr:r}
\pi_{\iota+1} (r)\cdot \pi_\iota(p) \leqlo^{\gamma} \pi_\iota(p).
\end{equation}
From (\ref{int:above:supp:p:pr}) and (\ref{int:above:supp:pr:r}) we get
$\pi_{\iota+1}(p) \leqlo^\gamma \pi_\iota(p)$.

At limit stages $\bar\iota \leq \theta$ of the construction of the interpolant $p$, \eqref{int:above:supp} holds for all $\iota<\bar\iota$, and so $\pi_{\bar\iota}(p)$ satisfies the support requirement.
This completes the proof of interpolation.

Now for \emph{Linking} (\ref{linking}).
Say $p \lequpl d$ and fix $s \in \Cl_\theta (p) \cap \Cl_\theta(d)$.
Write $\sigma$ for $\dom(s)$. 
By definition of $\Cl_\theta$, $\sigma = \sigma^\lambda(p)=\sigma^\lambda(d)$.
First, assume $\sigma=\theta$. In this case, we have $\lambda > \lambda_\theta$ by definition of diagonal support.
We construct $w$ by induction on its initial segments $w\res\iota$, for $\iota <\sigma$. 
To start, use \emph{Linking} for $P_1$ to obtain $w\res1$.
Assume we have $w\res\iota$; just use \emph{Coherent Linking} for $(P_\iota,P_{\iota+1})$ to obtain $w\res\iota+1$. At limits $\iota \leq \sigma$, use Lemma~\ref{supp:threads} and the fact that $\cof(\sigma)<\lambda$ and so $w_0 \res \iota \leqlo^{\cof(\sigma)} \pi_\iota(d)$.

Secondly, if $\sigma<\theta$, we can use \emph{Linking} for $P_\sigma$ to obtain a lower bound $w_0$ of $\pi_{\sigma}(p)$ and $\pi_{\sigma}(d)$ with the desired properties.
We claim that $w=w_0 \cdot d$ is the desired condition, i.e., $w \leqlo^{<\lambda} p,d$.
The proof is of course by induction on $\iota \leq \theta$. For limit $\iota$, just use the induction hypothesis and the definition of $\leqlo^{<\gamma}_\iota$.
For the successor case, write
\begin{gather*}
d^*= \pi_{\iota+1}(d),\\
p^*= \pi_{\iota+1}(p),\\
w_0^*=w_0\cdot \pi_\iota(d) \\
\end{gather*} 
and let $\pi$ denote $\pi_\iota$.
We may assume by induction that $w_0^*  \leqlo^{<\lambda} \pi(d^*), \pi(p^*)$. In the following, use that $\pss_{\iota+1}$ is a prestratification system, $\pss_\iota \is \pss_{\iota+1}$ and \ref{pss:is}(\ref{s:ext:exp}) \emph{Coherent Expansion}.

Firstly, since $d^* \leqlol \pi(d^*)$ and $w^*_0 \leq \pi(d^*)$, by \ref{pss:is}(\ref{pss:is:lo:tail}), 
we have
\begin{equation}\label{coh:cent:wd:small:supp}
w^*_0 \cdot d^* \leqlol w^*_0.
\end{equation}
In the same way, we can argue that 
\begin{equation}\label{coh:centlower:triv:p}
w^*_0 \cdot p^* \leqlol w^*_0.
\end{equation}
Equation \eqref{coh:cent:wd:small:supp} and $w^*_0 \leqlo^{<\lambda} \pi(d^*)$ give us
$w^*_0 \cdot d^* \leqlo^{<\lambda} \pi(d^*)$,
and together with $w^*_0 \cdot d^* \leq d^* \leq \pi(d^*)$ and \ref{def:pcs}(\ref{er}) we infer that $w^*_0 \cdot d^* \leqlo^{<\lambda} d^*$.

Since $d^* \lequpl p^*$ and $w^*_0 \leq \pi(d^*), \pi(p^*)$, we may conclude by $\pss_\iota \is \pss_{\iota+1}$ and \ref{pss:is}(\ref{pss:is:cdot:up}) that
\[ w^*_0 \cdot d^* \lequpl w^*_0 \cdot p^*.\]
This together with \eqref{coh:centlower:triv:p}, by \emph{Coherent Expansion} \ref{pss:is}(\ref{s:ext:exp}) yields 
\[ w^*_0 \cdot d^* \leq w^*_0 \cdot p^*.\]
Thus $w^*_0 \cdot d^* \leq p^* \leq \pi(p^*)$ while at the same time $w^*_0 \cdot d^* \leqlol w^*_0 \leqlo^{<\lambda} \pi(p^*)$. 
Another application of \ref{def:pcs}(\ref{er})
yields $w^*_0\cdot d^* \leqlo^{<\lambda} p^*$.
This ends the successor step of the inductive proof that $w \leqlo^{<\lambda} p,d$, and we are done with coherent linking.

Finally, check (\ref{continuous}) \emph{Continuity}: Fix $\lambda^*, \lambda\in I_\theta$ such that $\lambda^*<\lambda$.
Say $\bar p$ and $\bar q$ are $(\lambda^*,x)$-adequate sequences of length $\rho$ with greatest lower bound $p$ and $q$ respectively, and for each $\xi<\rho$, $\Cl_\theta(p_\xi) \cap \Cl_\theta(q_\xi)\neq\emptyset$. We show that $p, q \in \dom(\Cl_\theta)$ and $\Cl_\theta(p)\cap\Cl_\theta(q)\neq\emptyset$.

For $\nu < \rho$, let $\sigma(\nu)= \sigma^\lambda(p_\nu)$.
Look at the sequence $\bar p^\nu = (p^\nu_\xi)_{\xi<\rho}$ of conditions in $P_{\sigma(\nu)}$, defined by
$p^\nu_\xi = 1_{P_{\sigma(\nu)}}$ for $\xi < \nu$ and $p^\nu_\xi = \pi_{\sigma(\nu)}(p_\xi)$ for $\xi\in [\nu,\rho)$.
As in the proof of Theorem~\ref{thm:it:qc}, it is easy to see $p^\nu_\xi = G(\bar w\res\xi+1)$ for some $\qcdefG(\lambda^*\cup\{x\})$ function $G$, where $\bar w$ is a canonical witness and strategic guide for $\bar p$.
Also as in the proof of Theorem~\ref{thm:it:qc}, $\bar p^\nu$ is $(\lambda^*,x)$-adequate. Its greatest lower bound is $\pi_{\sigma(\nu)}(p)$.
Thus by \emph{Continuity} for $P_{\sigma(\nu)}$, $\pi_{\sigma(\nu)}(p) \in \dom\Cl_{\sigma(\nu)}$.

Observe that by definition of $\Cl_\theta$, we have $\sigma^\lambda(p_\xi)=\sigma^\lambda(q_\xi)$ for all $\xi < \rho$.
Analogously, define adequate sequences $\bar q^\nu$ in $P_{\sigma(\nu)}$ with greatest lower bound $\pi_{\sigma(\nu)}(q)$.
By \emph{Continuity} for $P_{\sigma(\nu)}$, we infer
$\Cl_{\sigma(\nu)}(\pi_{\sigma(\nu)}(p))\cap\Cl_{\sigma(\nu)}(\pi_{\sigma(\nu)}(q))\neq\emptyset$ for each $\nu < \rho$.
Letting $\sigma = \sup_{\nu<\rho}\sigma(\nu)$, we infer
$\Cl_{\sigma}(\pi_{\sigma}(p))\cap\Cl_{\sigma}(\pi_{\sigma}(q))\neq\emptyset$, by the definition of $\Cl_{\sigma}$.
As $\sigma^\lambda(p), \sigma^\lambda(q) \leq \sigma$ by Lemma~\ref{supp:glb}, this means
$\Cl_{\theta}(p)\cap\Cl_{\theta}(q)\neq\emptyset$, by the definition of $\Cl_{\theta}$.
We are done with the proof of \emph{Continuity}.
\end{proof}

\begin{cor}\label{cor:it:strat:comp}
Theorem~\ref{thm:it:strat:comp} holds, that is, iterations with stratified components and diagonal support are stratified.
\end{cor}
\begin{proof}
By Lemma~\ref{stratified:comp:implies:ext}, composition is an example of stratified extension.
By Theorem~\ref{thm:it:strat}, since the iteration has diagonal support and its initial segments form a sequence of stratified extensions, the whole iteration is stratified. 
\end{proof}

\section{Products}\label{sec:prod}

So far, stratified extension has only given us an overly complicated proof that iterations with stratified components are stratified. Here is a first non-trivial application: as a consequence of the next lemma, one can mix composition and products of stratified forcing freely in iterations with diagonal support, and the resulting iteration will be stratified.

\begin{lem}\label{products:ext}\index{product (forcing)!is a stratified extension}
If $P$ and $Q$ are stratified on $I$, $(P, P \times Q)$ is a stratified extension (on $I$).
\end{lem}
\begin{proof}
The proof is entirely as you expect.
Fix prestratification systems $\pss_P = (\D_P, \param_P, \leqlol_P,\lequpl_P,\Cl_P)_{\lambda\in I}$ and $\pss_Q= (\D_Q, \param_Q, \leqlol_Q,\lequpl_Q,\Cl_Q)_{\lambda\in I}$.
We now define a stratification system $\bar \pss = (\bar\D, \bar\param,\bleqlol,\blequpl,\bCl)_{\lambda\in I}$ on $P\times Q$ in the most natural way:
let $\bar \param = (\param_P, \param_Q)$, 
$\bar \D(\lambda,x,(p,q))= \D_P(\lambda,x,p) \times \D_Q(\lambda,x,q)$ and let
\begin{align*}
(p,q) \bleqlol (\bar p,\bar q) \iff &p \leqlol_P \bar p \text{ and } q \leqlol_P \bar q \\
(p,q) \blequpl (\bar p,\bar q) \iff &p \lequpl_P \bar p \text{ and } q \lequpl_P \bar q \\
s \in \bCl(p,q) \iff & \Big[ s \in \Cl_P(p) \text{ and } q \leqlol 1_Q \Big] \text{ or } \Big[ s=(\chi,\zeta)\\
                     &  \text{ where }  \chi \in \Cl_P(p) \text{ and } \zeta \in \Cl_Q(q) \Big].\\
\end{align*}
That $\bar\pss$ is a prestratification system requires but a glance at the definitions (see \ref{def:pcs}, p.~\pageref{def:pcs} and \ref{def:pss}, p.~\pageref{def:pss}). 
For example, \emph{Continuity}, (\ref{continuous}) is a straightforward application of \emph{Continuity} for both $P$ and $Q$.    

The same holds for (\ref{ext}), (\ref{pi:mon}) and (\ref{F:coh}) (see p.~\pageref{pss:is} for the definition of $\pss_P \is \bar \pss$, and see p.~\pageref{pcs:is} for (\ref{ext}), (\ref{pi:mon}) and (\ref{F:coh}).
For the following, let $(p,q)\in P\times Q$, $w\in P$.
For your entertainment, we check \ref{pss:is}(\ref{pss:is:lo:init}) (see page~\pageref{pss:is:lo:init}).
Say $w \leqlol_p p$. Then clearly $(w,q) \bleqlol (p,q)$, done.
Now \ref{pss:is}(\ref{pss:is:lo:tail}):
say $w \leq p$ and $(p,q)\bleqlol(p,1_Q)$. This means $q \leqlol_Q 1_Q$ and so $(w,q)\bleqlol (w,1_Q)$, which is what we wanted to prove.

For the next two conditions, let $\bar d= (d,d^*) , \bar r=(r,r^*) \in P \times Q$ satisfy $\bar d \blequpl \bar r$.
We jump ahead and check (\ref{s:ext:exp}) of \ref{def:s:ext} (see p.~\pageref{def:s:ext}): say $d \leq r$ and $\bar r \bleqlol (r,1_Q)$. Then $r^* \leqlol_Q 1_Q$ and $d^* \lequpl_Q r^*$ by assumption, so by \ref{def:pss}(\ref{exp}) for $Q$, $d^* \leq r^*$ and thus $\bar d \leq \bar r$.

Let's check (\ref{pss:is:cdot:up})). Say $w \leq d$ and $w \leq r$. 
By \ref{def:pss}(\ref{up:extra}) for $P$, $w \lequpl_P w$ and so $(w,d^*)\blequpl (w,r^*)$.
We omit the rest of \ref{pss:is} and conclude that $\pss_P \is \bar\pss$.

The most interesting part of the present proof is that of \emph{quasi-closed extension}\index{product (forcing)!is a quasi-closed extension} (Definition~\ref{qc:ext}, see p.~\pageref{qc:glb}), of which we check (\ref{ext:glb}) , leaving (\ref{ext:redundant}) to the reader.
So say $(p_\xi,q_\xi)_{\xi<\rho}$ is $(\lambda,\bar\lambda,x)$-adequate and $(p_\xi)_{\xi<\rho}$ has a greatest lower bound $p$.
Since the projection from $P\times Q$ to $P$ is obviously very simply definable, $\bar q= (q_\xi)_{\xi<\rho}$ is $(\lambda,\bar\lambda,x)$-adequate as well (by Lemma~\ref{lem:adeq:proj}). If $\lambda=\bar\lambda$, we are done as $\bar q$ is $\lambda$-adequate and $Q$ is quasi-closed.
If on the other hand, $\lambda<\bar\lambda$, we have that for all $\xi<\rho$, $q_\xi\leqlo^{\bar\lambda}_Q 1_Q$. By Lemma~\ref{bar:lambda:adeq},
$\bar q$ is $\bar\lambda$-adequate. Moreover, if $q$ is a greatest lower bound of $\bar q$, by quasi-closure for $Q$, we have $q \leqlo^{\bar\lambda}_Q 1_Q$. So $(p,q) \bar\leqlo^{\bar\lambda}(p,1_Q)$ and we are done.

To conclude that $(P, P\times Q)$ is a stratified extension, we check the remaining conditions of \ref{def:s:ext} (see p.~\pageref{def:s:ext}).
\emph{Coherent Interpolation}, \ref{def:s:ext}(\ref{s:ext:int}) and \emph{Coherent Linking}, \ref{def:s:ext}(\ref{s:ext:cent}) are identical to interpolation and linking for $Q$ in this context.
\end{proof}

\section{Stable Meets for Strong Suborders}\label{sec:stm}

In the next section, we introduce the operation of amalgamation and show that the amalgamation of a stratified forcing $P$ is a stratified extension of $P$.
In that proof, we must show that a certain dense subset of $P$ is closed under taking meets with conditions from an ``initial segment'' (or rather, a strong suborder) $Q$ (see Lemma~\ref{Q:cdot:D}, p.~\pageref{Q:cdot:D}). This will be facilitated by  the so-called $Q$-stable meet operation\index{stable meet}\index[notation]{Q-stable meet@$Q$-stable meet}\index[notation]{wedge (stable meet)@$\stm$ (stable meet)}\index[notation]{  wedge (stable meet)@$\stm$ (stable meet)} $p \stm_Q r$, which we introduce in the present section.
In a standard iteration
this is a simple operation: $p \stm_Q r$ ``starts like $p$ on $Q$'' and then continues ``like $r$'' (as closely as possible) on $\quot{P}{Q}$. 
What makes it useful is the following: if $r, p \in P$, $r \leq p$ and $r$ is moreover ``a direct extension on the tail \quot{P}{Q}'' of $p$, then $p\stm_Q r$ is a \emph{de-iure} direct extension of $p$, and moreover $r$ can be obtained straightforwardly from $p\stm_Q r$.

We now give a formal definition of such an operation, and then show that we can always define an operation $\stm$ on products and compositions. 
Then we show how to define $\stm$ for infinite iterations. In the next section we shall see we also have a stable meet operator for amalgamation.
Of course amalgamation necessarily introduces an element of recursion into the definition of this operation;
thus we are forced into this formal, inductive approach (rather than defining $\stm$ directly on the iteration used in the main theorem) by our choice to present a general preservation theorem (rather than simply showing the requisite properties hold of the specific iteration used to prove Theorem~\ref{p.t.main}).

Let $Q$ be a strong suborder of $P$, and let
$\pi\colon P \rightarrow Q$ be the strong projection. 
Say $\pss = (\hdots, \leqlol, \hdots)_{\lambda\in I}$ is a prestratification system on $P$.
\begin{dfn}\label{schnitzel}
We call $\stm$ a \emph{$Q$-stable meet operator on $P$} with respect to $\pss$ or a \emph{stable meet on $(Q,P)$} if and only if 
\begin{enumerate}
\item $\stm\colon (p,r) \mapsto p \stm r$ is a function with $\dom(\stm)\subseteq P^2$ and $\ran(\stm)\subseteq P$.
\item $\dom(\stm)$ is the set of pairs $(p, r) \in P^2$ such that $r \leq p$ and 
\begin{equation}\label{schnitzel:dom}
\exists \lambda \in I \quad  r \leqlol \pi(r)\cdot p
\end{equation}
\item Whenever $r\leq p$ and $r \leqlol \pi(r)\cdot p$, the following hold:
\begin{gather}
p \stm r \leqlol p \label{schnitzel:leqlol}\\
\pi(p \stm r)=\pi(p)\label{schnitzel:proj}\\
\pi(r)\cdot (p \stm r) \approx r \label{schnitzel:meet}
\end{gather}
\end{enumerate}
As usual, we don't mention $\pss$ when context permits.
\end{dfn}

A few remarks are in order to clarify this definition.
\begin{itemize}
\item We certainly don't have $p \stm r = r \stm p$.
\item The gist of \eqref{schnitzel:dom} is that we try to express that $\pi(r)$ forces that in $\quot{P}{Q}$,  the ``tail'' of $r$ is a direct extension (in the sense of $\leqlol$) of $p$; \eqref{schnitzel:dom} captures the essence of this even when $\quot{P}{Q}$ is not stratified. 
\item Observe that $r \leq p$ implies $\pi(r)\leq\pi(p)$ and so $\pi(r)\cdot p\in P$; thus \eqref{schnitzel:dom} makes sense.
\item By $\pi(r)\cdot (p \stm r) \approx r $ we mean that $\pi(r)\cdot (p \stm r) \leq r $ and $\pi(r)\cdot (p \stm r) \geq r $.
Admittedly, we are very careful here. %
\item Observe that there could be more than one map $\stm$ satisfying the definition.
Intuitively, this is because \eqref{schnitzel:leqlol} is not strong enough to fully determine $p \stm r$ on $\pi(p)-\pi(r)$.
If we add to the above the requirement that $-\pi(r) \cdot (p \stm r) = p$ hold in $\ro(P)$, this uniquely determines $\stm$. In fact this entails 
\begin{equation}\label{schnitzel:var}
p \stm r= r + (p-\pi(r))
\end{equation}
in $\ro(P)$.\footnote{In all the applications we have in mind, the natural definition of $\stm$ satisfies \eqref{schnitzel:var}---provided we work with the \emph{separative quotient} of $P$.}
For our purposes, this point is moot.
\end{itemize} 

That we have chosen to define stable meet operators abstractly for arbitrary iterations may give the concept an unnecessarily difficult appearance.
We give some examples showing how simple it really is.
\begin{lem}\label{stable:meet:products}\index{product (forcing)!stable meet}
Say $\bar P = Q_0 \times Q_1$, and say for each $\lambda\in I$, $\bleqlol$ is obtained from $\leqlol_0$ and $\leqlol_1$ as in the proof of \ref{products:ext} (where of course $\leqlol_i \subseteq (Q_i)^2$).
Then there is a stable meet operator on $(Q_0,\bar P)$ with respect to $\bleqlol$.
\end{lem}
\begin{proof}
Let $\pi$ denote the projection to the first coordinate. 
Define $\dom(\stm)$ to be the set of pairs prescribed in Definition~\ref{schnitzel}. 
Say $r=(r_0,r_1) \in Q_0\times Q_1$ and $p=(p_0,p_1)\in Q_0\times Q_1$ are such that $(r,p)\in\dom(\stm)$. 
Define 
\[(p_0,p_1) \stm (r_0,r_1) = (p_0,r_1).\] 
As $(r,p) \in \dom(\stm)$, we can fix $\lambda$ such that  $(r_0,r_1) \bleqlol \pi(r)\cdot p=(r_0,p_1)$, and so $r_1 \leqlo_1 p_1$. 
Thus $(p_0, r_1) \bleqlol (p_0,p_1)$. 
To check the other properties is left to the reader.
\end{proof}
\begin{lem}\label{stable:meet:comp}\index{composition (forcing)!stable meet}
Say $\bar P = Q * \dot R$, and say for each $\lambda\in I$, $\bleqlol$ is obtained from $\leqlol$ and $\dleqlol$ as in the proof of \ref{stratified:composition}.
Then there is a stable meet $\stm$ on $(Q_0,P)$ with respect to $\bleqlol$.
\end{lem}
\begin{proof}
Let $\pi$ denote the projection to the first coordinate. 
Again, define $\dom(\stm)$ to be the set of pairs prescribed in Definition~\ref{schnitzel}. 
Say $\bar r=(r,\dot r)$ and $\bar p=(p,\dot p)$ are such that $(\bar r,\bar p)\in\dom(\stm)$. 
Define $\bar p \stm \bar r = (p,\dot r^*)$, where $\dot r^*$ is such that $r \forces \dot r^*=\dot r$ and $-r\forces \dot r^*=\dot p$. 
Fixing a $\lambda$ witnessing that $(\bar r,\bar p)\in\dom(\stm)$, so that we have $(r,\dot r) \bleqlol \pi(\bar r)\cdot \bar p=(r,\dot p)$, and so $r\forces \dot r \dleqlol \dot p$.
Then $(p,\dot r^*) \bleqlol (p,\dot p)$, since $r \forces \dot r^*=\dot r \dleqlol \dot p$ and $p-r\forces \dot r^*= \dot p \dleqlol \dot p$. 
To check the other properties is left to the reader.
\end{proof}

The stable meet operator behaves very nicely in iterations:
\begin{lem}
Let $\bar Q^{\theta+1}$ be an iteration\index{iteration (forcing)!stable meet} with diagonal support and say for each $\iota < \theta$, $P_\iota$ carries a prestratification system $\pss_\iota$ on $I$ and
\begin{enumerate}
\item For all $\iota<\theta$, we have $\pss_\iota \is \pss_{\iota+1}$.
\item If $\bar \iota \leq \theta$ is limit, $\pss_{\bar\iota}$ is the natural system of relations on $P_{\bar\iota}$.
\end{enumerate}
Moreover, say for each $\iota<\theta$, there is a stable meet operator $\stm^{\iota+1}_\iota$ on $(P_\iota,P_{\iota+1})$ with respect to $\pss_{\iota+1}$.
Then for each $\iota < \theta$ such that $\iota>0$ there is a $P_\iota$-stable meet operator on $P_\theta$.
\end{lem}
\begin{proof}

By induction on $\theta$, we show that for each pair $\iota, \eta$ such that $0<\iota<\eta \leq \theta$, there is a stable meet operator $\stm^\eta_\iota$ for $(P_{\iota},P_{\eta})$.
For $\iota, \eta$ as above and for $p$, $r \in P$ such that $r \leq p$ and \eqref{schnitzel:dom} hold, define
\begin{equation}\label{schnitzel:it}
p \stm^\eta_\iota r = \prod_{\iota\leq\nu<\eta} \pi_{\nu+1}(p) \stm^{\nu+1}_\nu \pi_{\nu+1}(r).
\end{equation}
We prove by induction on $\theta$ that 
\begin{enumerate}
\item For $\iota$, $\eta$ such that $0<\iota<\eta \leq \theta$ and for $(p,r)\in\dom(\stm^\theta_\iota)$, 
\begin{equation}\label{schnitzel:comm}
\pi_\eta(p \stm^\theta_\iota r)=\pi_\eta(p) \stm^\eta_\iota \pi_\eta(r).
\end{equation}
The sequence of $\pi_\eta(p) \stm^\eta_\iota \pi_\nu(r)$, for $\eta \in (\iota,\theta]$ determines a thread in $P_\theta$, in the sense of Definition~\ref{it:terminology}.
\item For $\iota$ and $\eta$ as above,  $\stm^\eta_\iota$ is a stable meet operator on $(P_\iota,P_\eta)$.
\end{enumerate}
Fix $\iota<\theta$.
Let $(p,r)\in\dom(\stm^\theta_\iota)$ be arbitrary and let $\lambda$ be an arbitrary witness to \eqref{schnitzel:dom}.
For the rest of the proof let $t^\eta_\iota$ denote $\pi_\eta(p) \stm^\eta_\iota \pi_\eta(r)$, for $0<\iota < \eta \leq \theta$.

First assume $\theta$ is limit. By induction hypothesis, $(t^\eta_\iota)_{\eta\in(\iota,\theta)}$ is a thread through $\bar Q^\theta$; by definition (\ref{schnitzel:it}), this thread is $p \stm^\theta_\iota r = t^\theta_\iota$. 
We must show that $t^\theta_\iota$ has legal support.
It suffices to show that for each $\gamma \in I$, $\supp^\gamma (t^\theta_\iota) \subseteq \supp^\gamma (p) \cup \supp^\gamma (r)$.
So fix such a $\gamma$ and a $\xi <\theta$ such that we have 
\begin{gather}
\pi_{\xi+1}(p)\leqlo^{\gamma} \pi_\xi(p),\label{xi:not:in:supp:p}\\
\pi_{\xi+1}(r)\leqlo^{\gamma} \pi_\xi(r).\label{xi:not:in:supp:r}
\end{gather}
We have $\pi_{\xi+1}(r)\leq \pi_\xi(r)\cdot \pi_{\xi+1}(p)\leq \pi_\xi(r)$ (simply because $r\leq p$) and so by  (\ref{er}) and (\ref{xi:not:in:supp:r}), we have $\pi_{\xi+1}(r) \leqlo^\lambda \pi_\xi(r)\cdot \pi_{\xi+1}(p)$.
Since $\stm^{\xi+1}_\xi$ is a stable meet operator, and by (\ref{xi:not:in:supp:p}) we have
\[
\pi_{\xi+1}(p) \stm^{\xi+1}_\xi \pi_{\xi+1}(r)\leqlo^{\gamma}\pi_{\xi}(p). 
\]
In other words, 
$t^{\xi+1}_\xi \leqlo^\gamma \pi_\xi(p)$ and thus, taking the Boolean meet with $t^\xi_\iota$ on both sides,
\[
t^{\xi+1}_\iota = t^{\xi+1}_\xi \cdot t^\xi_\iota  \leqlo^\gamma \pi_\xi(p)\cdot t^\xi_\iota = t^\xi_\iota,
\]
where the last equation holds since $\stm^\xi_\iota$ is a stable meet operator by induction.
So we have $t^{\xi+1}_\iota \leqlo^\gamma t^\xi_\iota \in P_\xi$.
We conclude by Lemma~\ref{lem:support:simple} that $\xi \not \in \supp^\gamma (t^\theta_\iota)$, finishing the proof that  $t^\theta_\iota$ has legal support.

It is straightforward to prove equations (\ref{schnitzel:leqlol}), (\ref{schnitzel:proj}) and (\ref{schnitzel:meet}) for $t^\theta_\iota=p \stm^\theta_\iota r$, assuming by induction that for each $\eta<\theta$, $\stm^\eta_\iota$ is a stable meet operator ($t^\theta_\iota$ is a thread whose initial segments satisfy these equations). We leave this to the reader.

Now let $\theta=\eta+1$. To see that $(t^\nu_\iota)_{\nu\leq\theta}$ is a thread, it suffices to show that $\pi_\eta(t^\theta_\iota)=t^\eta_\iota$. In order to show this, observe
\begin{equation*}
\pi_\eta(t^\theta_\iota)=t^\eta_\iota\cdot\pi_\eta(t^\theta_\eta)=t^\eta_\iota\cdot \pi_\eta(p)=t^\eta_\iota,
\end{equation*}
where the last equation holds since by induction, $t^\eta_\iota \leq \pi_\eta(p)$.
It follows by the induction hypothesis that $(t^\nu_\iota)_{\nu\leq\theta}$ is a thread.

It remains to show that $\stm^\theta_\iota$ is a $P_\iota$-stable meet on $P_\theta$, i.e we must show (\ref{schnitzel:leqlol}), (\ref{schnitzel:proj}) and (\ref{schnitzel:meet}).
Firstly, by induction, 
\[ t^\eta_\iota \leqlol \pi_\eta(p),\]
and as $\stm^\theta_\eta$ is a $P_\eta$-stable meet on $P_\theta$, 
\[\pi_\eta(p) = \pi_\eta(t^{\eta+1}_\eta).\]
By \ref{pss:is}(\ref{pss:is:lo:init}), this entails 
\[ t^\eta_\iota \cdot t^{\eta+1}_\eta  \leqlol t^{\eta+1}_\eta.\]
As $\stm^\theta_\eta$ is a $P_\eta$-stable meet on $P_\theta$, we have  $t^{\eta+1}_\eta \leqlol p$, whence $t^\theta_\iota =t^\eta_\iota \cdot t^{\eta+1}_\eta \leqlol p$, proving (\ref{schnitzel:leqlol}).
Secondly, 
\begin{equation*}\label{show:schnitzel:proj}
\pi_\iota(t^\theta_\iota) = \pi_\iota(t^\eta_\iota \cdot \pi_\eta(t^{\eta+1}_\eta))=\pi_\iota(t^\eta_\iota)=\pi_\iota(p).
\end{equation*}
 The first equality here is trivial. The second holds since $t^\eta_\iota \leq \pi_\eta(p)$ by induction hypothesis and since by the assumption that $\stm^{\eta+1}_\eta$ is a $P_\eta$-stable meet, we have $\pi_\eta(p) = \pi_\eta(t^{\eta+1}_\eta)$.
The last equality of holds by induction.
Finally, we prove (\ref{schnitzel:meet}). We have
\[ \pi_\iota(r)\cdot t^\theta_\iota = \pi_\iota(r) \cdot t^\eta_\iota \cdot t^{\eta+1}_\eta = \pi_\eta(r) \cdot t^{\eta+1}_\eta = r,\]
where the first equation holds by definition, the second by induction hypothesis, and the last one since $\stm^{\eta+1}_\eta$ is a $P_\eta$-stable meet.
We are done with the successor case of the induction, and thus with the inductive proof of the lemma.
\end{proof}

By the lemma, if $\bar Q^\theta$ is an iteration as in the hypothesis of the lemma and $\iota<\eta < \theta$, the map $\stm^\eta_\iota$
is the same as $\stm^\theta_\iota\res (P_\eta)^2$. So as we do for strong projections, we just write $\stm_\iota$ and we speak of the $P_\iota$-stable meet operator (without specifying the domain).
Moreover, we can formally set $p \stm_0 r = r$ and $p \stm_\iota r = p$ for $\iota \geq \theta$.

\section[Remoteness]{Remoteness: Preserving Strong Suborders}\label{sec:remote}
Let $C, Q$ be complete suborders of $P$, and say $\pi_C\colon P \rightarrow C$ and 
$\pi_Q\colon P \rightarrow Q$ are strong projections. 
We want to find a sufficient condition to ensure that $C$ is a complete suborder of $\quot{P}{Q}$, after forcing with $Q$.
In our application $C$ will just be $\kappa$-Cohen forcing of $L$, for $\kappa$ the least Mahlo.
Our iteration will be of the form $P= Q * (\dot Q_0 \times C) * \dot Q_1$, so after forcing with $Q$, $C$ is a complete suborder of $\quot{P}{Q} = (\dot Q_0 \times C) * \dot Q_1$.
We want the same to hold for $\Phi[C]$ (where $\Phi$ is a member of a particular family of automorphisms of $P$ which we construct using the technique of amalgamation); this helps to ensure ``coding areas'' don't get mixed up by the automorphisms, see Lemma~\ref{index:sequ} and Lemma~\ref{coding:survives}. So we have to introduce a property sufficient for $C$ to be a complete suborder of $\quot{P}{Q}$, in such a way that this condition is inherited by $\Phi[C]$. 
For this, we use of course the stratification of $P$.
This is necessary since  forming $\quot{P}{Q}$ will not only ``take away an initial segment'' and leave $\Phi[C]$ in the tail in same obvious fashion as for $C$; instead forming $\quot{P}{Q}$ will also ``take away'' a small subalgebra of $P$ (a copy of the random algebra).

Fix a preorder $P$ which is stratified on $I$. The following definition is, as usual, relative to a particular prestratification system.
\begin{dfn}\label{remote}
We say $C$ is \emph{remote in $P$ over $Q$\index{remote|textbf} (up to height $\kappa$)} if and only if $C$ and $Q$ are strong suborders of $P$ with strong projections $\pi_C$ and $\pi_Q$, and for all  $c \in C$ and $p \in P$ such that $c \leq \pi_C(p)$, we have
\begin{enumerate}
\item \label{remote:leqlo} 
$p\cdot c \leqlol p$ for every $\lambda \in I\cap\kappa$;
\item  $\pi_Q(p \cdot c) = \pi_Q(p)$.

\end{enumerate}
Observe that if we drop the first clause, this just says that $C$ is independent in $P$ over $Q$ (see Definition~\ref{indie}).

For a $P$-name $\dot C$, we say  $\dot C$ is remote in $P$ over $Q$ if and only if it is a name for a generic of a remote complete suborder of $P$; i.e., there is a complete suborder $R_C$ of $P$ (with a strong projection $\pi_C\colon P \rightarrow R_C$) such that $R_C$ is dense in $\gen{\dot C}^{\ro(P)}$ and $R_C$ is remote in $P$ over $Q$.
\end{dfn}
\begin{lem}\label{remote:lemma:not:in}
If $\dot C$ is a $P$-name which is remote over $Q$, then $\dot C$ is not in $V^Q$.
\end{lem}
\begin{proof}
An immediate consequence of Lemma~\ref{indie:lemma:not:in}
\end{proof}

\chapter{Amalgamation}\label{sec:amalgamation}\index{amalgamation|textbf}

Amalgamation is a technique to build iterations with many automorphisms.
We need two types of amalgamations: Using type-1 amalgamation, we make sure a stage of our iteration has an automorphism extending an isomorphism of two \emph{small} complete subalgebras $B_0$, $B_1$ of the previous stage of the iteration. Using type-2 amalgamation, we take care that we can extend automorphisms of initial segments (in our case, those created by type-1 amalgamation).

The technique presented here differs substantially from that of \cite{shelah:amalgamation} (described also in \cite{jr:amalgamation}): It has a ``mixed support'' flavor rather than a ``finite support'' flavour;
moreover, some fine tuning is needed to preserve stratification.

\medskip
 
In \ref{modf}, we define \emph{basic amalgamation}, i.e., the forcing $P^\Int_f$, which will be put to use when we define either type of amalgamation.
Before defining these two types of amalgamation, we pause to analyze $P^\Int_f$ and find that it can be decomposed as a product after forcing with $B_0$ (Section \ref{factor}; this will be put to use in Lemmas~\ref{newreal} and \ref{index:sequ} to show we can close off $\Gamma^0$ under automorphisms).

In Section~\ref{am} we define type-1 amalgamation (denoted by $\am$) and show it is a stratified extension (by finding a dense set where ``Boolean values are stable''; two small details here are the use of ``reduced pairs of random reals'' in Lemma~\ref{Q:in:D}, and the stable meet operator in Lemma~\ref{q:cdot:p:in:D} to preserve that $Q$ is a strong suborder),
and in  Section~\ref{simpleram} we do the same for the simpler type-2 amalgamation (denoted by $\simpleram$).

In the last section, we construct a stable meet operator for amalgamation and discuss remote suborders. 
This completes the inductive proof that there is a stable meet for each stage of the iteration for the proof of the main theorem.
Also, we prove Lemma~\ref{remote:lemma} which helps to ensure ``coding areas'' don't get mixed up by the automorphisms (it will be put to use in Lemmas~\ref{index:sequ} and \ref{coding:survives}).

\section{Basic Amalgamation}\label{modf}

Let $P$ be a forcing, $Q$ a complete suborder of $P$ such that $\pi\colon P \rightarrow Q$ is a strong projection (see \ref{stron:proj:equiv}, p.~\pageref{stron:proj:equiv} and the preceding discussion).
For $i \in \{0,1\}$, let $\dot B_i$ be a $Q$-name such that $\forces_Q \dot B_i$ is a complete subalgebra of $\quot{P}{Q}$.
Moreover, say we have a $Q$-name $\dot f$ such that $\forces_Q \dot f\colon \dot B_0 \rightarrow \dot B_1$ is an isomorphism of Boolean algebras.

\medskip

Our task is to find $P'$ containing $P$ as a complete suborder, carrying an automorphism $\Phi \colon P' \rightarrow P'$ which extends the isomorphism of $\dot B_0$ and $\dot B_1$ (in the extension by $Q$) and which is trivial on $Q$.
Moreover, we want to preserve stratification: 
Suppose $(Q,P)$ is a stratified extension on $I = [\lambda_0, \kappa^*)\cap \Reg$, where we allow $\kappa^* = \infty$. 
We want $\lambda_1<\kappa^*$ such that $(P',P)$ is a stratified extension above $\lambda_1$  ($\lambda_1$ may be strictly greater than $\lambda_0$). 

\medskip

We first make some observations:
Let $\ro(Q) * \dot B_i$ be denoted by $B_i$. This is a complete subalgebra of $B=\ro(P)$, consisting of the $Q$-names (or if you prefer, $\ro(Q)$-names) $\dot b$ such that $1_Q\forces_Q\dot b\in\dot B_i$. 

Keep in mind that we can canonically identify the partial order $Q * (\dot B_i\setminus\{ 0 \})$
with the set of $b \in B_0$ such that $\pi_Q(b) \in Q$. Also, don't confuse this with the set of $b \in B_0$ such that $\pi_Q(b) = 1$---or, equivalently, $1_Q\forces_Q [b]_{\dot G} > 0$, which is called the term-forcing, usually denoted by $(\dot B_i\setminus\{ 0 \})^{Q}$.

Let $\pi_i$ denote the canonical projection from $P$ to $B_i$.
Then $\pi_i$ coincides with $\pi$ on $Q$ (by \ref{strong:vs:canonical:proj}). 
Moreover, $\dot f$ can be viewed as an isomorphism $f$ of $B_0$ and $B_1$ (mapping names to names). We have
\begin{equation}\label{f:and:pi:commute}
\pi \circ f = f \circ \pi = \pi.
\end{equation}
\todo{In fact, for any pair of subalgebras $B_0$, $B_1$ of $\ro(P)$ such that $Q\subseteq B_0\cap B_1$ and an isomorphism $f\colon B_0 \rightarrow B_1$, Equation~(\ref{f:and:pi:commute}) holds if and only if $f$ generates an isomorphism of the pair $\quot{B_i}{Q}$, $i\in\{0,1\}$ in any $Q$-generic extension. 
Thus instead of starting with $\dot f$ and $\dot B_0$, $\dot B_1$ as in the first paragraph, we could also have started with $f$, $B_0$ and $B_1$ as above, satisfying (\ref{f:and:pi:commute}).}

\medskip

In a first step, we define $P^\Int_f$, the amalgamation of $P$ over $f$. $P^\Int_f$ contains $P$ as a complete suborder and has an automorphism $\Phi$ extending $f$.
\begin{rem}\label{am:remark}
If we require $P'$ to be stratified, we have to be more careful: we must carefully pick a dense subset $\Dam$ of $P$, such that $P'={\Dam}^\Int_f$ is stratified. The partial order ${\Dam}^\Int_f$ is in general not equivalent to $P^\Int_f$. 
Finally, we will define a forcing $\am$ which is equivalent to ${\Dam}^\Int_f$, and moreover $(P, \am)$ is a stratified extension, to make sure that the stratification of $P'$ is witnessed in a way that is coherent enough to survive (infinite) iteration. Let's postpone these complications, and first look at $P^\Int_f$.
\end{rem}
Amalgamation is not at all a well-behaved construction: If $D$ is a dense subset of $P$, we cannot infer that $D^\Int_f$ is dense in $P^\Int_f$.
Even the weaker statement fails: if $\ro(P)=\ro(R)$, we cannot conclude $\ro(P^\Int_f)=\ro(R^\Int_f)$.
This in combination with the fact that stratification is similarly ill-behaved (e.g., one dense subset of $\ro(P)$ may be stratified while another dense subset is not) is the main obstacle in this proof.

\medskip

Amalgamation is so ill-behaved that in general we cannot even preclude $P^\Int_f=\{ 1_P \}$, although this pathology does not arise if we ask $B_0\cup B_1 \subseteq P$.
On the other hand,  we cannot simply work with $\ro(P)$; for although $\ro(P)$ has a dense stratified subset (namely $P$),
this doesn't mean that $\ro(P)^\Int_f$ will have a dense stratified subset. 
With the intuition that we want to stick as closely to $P$ as possible (for the sake of stratification) but still have $B_0, B_1 \subseteq P$, we define a ``hybrid'':
\begin{dfn}\label{def:blowup}\index[notation]{P(Q,f)@$\blowup{P}(Q,f)$}
Consider the set $P \times  B_0 \times B_1$, i.e., the set of triples $(p,\dot b^0,\dot b^1)$ where $p \in P$ and $\forces_Q \dot b^i \in \dot B_i$ for $i \leq 2$. 
Order this set by $(p,\dot b^0,\dot b^1)\leq (q,\dot d^0,\dot d^0)$ \emph{if and only if} $p \leq q$ and $p \cdot \dot b^0 \cdot \dot b^1 \leq  q \cdot \dot d^0 \cdot \dot d^1$ in $\ro(P)$. 
This makes sense since we can canonically identify $\dot b^j, \dot d^j$ with elements of $B_j$.
We call $\blowup{P}=\blowup{P}(Q,f)$ the set of $(p,\dot b^0,\dot b^1) \in P \times B_0 \times B_1$ such that
\begin{equation}
\pi(p)\forces p \cdot \dot b^0 \cdot \dot b^1 \neq 0,
\end{equation}
or equivalently, 
\begin{equation}\label{eq:blowup:def:alt}
\pi(p \cdot \dot b^0 \cdot \dot b^1)=\pi(p).
\end{equation}
For $\hat p \in \blowup{P}$, when we refer to the components of $\hat p$, we use the notation $\hat p = (\hat p^P, \hat p^0, \hat p^1)$.
When appropriate, we identify $\hat p$ with $\hat p^P\cdot\hat p^0\cdot\hat p^1$, i.e., the meet of the components in $\ro(P)$. In particular,
if $g$ is a function such that $\dom(g)=\ro(P)$, we write $g(\hat p)$ for $g(\hat p^P\cdot\hat p^0\cdot\hat p^1)$.
\end{dfn}
Clearly, $P$ is isomorphic to the subset of $\blowup{P}$ where the two latter components are equal to $1_{\ro(P)}$, and this set is in turn dense in $\blowup{P}$. So $P$ can be considered a dense subset of $\blowup{P}$. 
Thus, the separative quotient of $\blowup{P}$ is the completion under $\cdot$ of $P\cup B_0 \cup B_1$ in $\ro(P)$ (leaving aside the $0$ element). 
Observe, moreover, that if $D \subseteq P$ is dense in $P$, then $\{\hat p \in \blowup{P} \setdef \hat p^P \in D \}$ is the same as $\blowup{D}$, and
we shall often use this fact tacitly.
Lastly, observe that 
\begin{equation}\label{blowup:order}
\hat p \leq \hat q \iff \big[ \text{ $\hat p^P \leq \hat q^P$ and $\pi_j(\hat p)\leq \pi_j(\hat q)$ for $j\in\{0,1\}$ }\big]
\end{equation}
and $\hat p \approx (\hat p^P,\pi_0(\hat p), \pi_1(\hat p))$.\footnote{We may regard $(\hat p^P,\pi_0(\hat p), \pi_1(\hat p))$ the canonical representative of $\hat p$ if $P$ is separative.} These two observations together would make for an equivalent, more strict definition of $\blowup{P}$, yielding separative $\blowup{P}$ provided $P$ is separative. Notwithstanding, we find the current definition more convenient---if less elegant. 
In the following, we identify $P$ with $\{ \hat p \in \blowup{P} \setdef \hat p^0=\hat p^1=1 \}$.

\medskip

Much of the following would work if we replace (\ref{eq:blowup:def:alt}) by the weaker $p \cdot \dot b_0 \cdot \dot b_1 \neq 0$.
The advantage of asking (\ref{eq:blowup:def:alt}) is that it makes the projection $\bar \pi\colon P^\Int_f \rightarrow P$ take a simple form.

\begin{dfn}\index[notation]{P T f@$P^\Int_f$}\index{basic amalgamation}\index{amalgamation!basic}
We define $P^\Int_f$ to consist of all sequences $\bar p: \Int \rightarrow \blowup{P}$ such that
for all but finitely many $k$ we have $\bar p(k)^P \leqlo^{\lambda_0} \pi(\bar p(k)^P)$ and for all $k$ we have
\begin{equation*}
f(\pi_0(\bar p(k+1)^P\cdot \bar p(k+1)^0 \cdot \bar p(k+1)^1))=\pi_1(\bar p(k)^P\cdot \bar p(k)^0 \cdot \bar p(k)^1),
\end{equation*}
or, simply
\begin{equation}\label{thinout}
f(\pi_0(\bar p(k+1)))=\pi_1(\bar p(k)).
\end{equation}
The ordering on  $P^\Int_f$ is given by $\bar r \leq \bar p$ if and only if for all $k$, $\bar r(k) \leq \bar p(k)$ in $\blowup{P}$.
We define a map \index[notation]{Phi (amalgamation)@$\Phi$ (amalgamation)}\index{automorphism}$\Phi \colon P^\Int_f \rightarrow P^\Int_f$ by:
\begin{equation*}
 \Phi(\bar p)(i)= \bar p(i+1) \text{ for }i \in \Int.
\end{equation*}
Obviously, $\Phi$ is one-to-one and onto, and $\Phi(\bar p)\leq \Phi(\bar q) \iff \bar p \leq \bar q$.
\end{dfn}

Observe that (\ref{f:and:pi:commute}) together with (\ref{thinout}) and (\ref{eq:blowup:def:alt}) imply that for all $i \in \Int$,
\begin{equation}\label{proj:am:to:Q}
\pi(\bar p(i))=\pi(\bar p(0))=\pi(\bar p(0)^P).
\end{equation}

\medskip

Let \index[notation]{F (amalgamation)@$F$ (amalgamation)}$F \colon \blowup{P} \rightarrow B_1$ be defined by $F(x)=f(\pi_0(x))$ and let \index[notation]{G (amalgamation)@$G$ (amalgamation)}$G \colon \blowup{P} \rightarrow B_0$ be defined by $G(x)=f^{-1}(\pi_1(x))$. 

\medskip

It may seem more natural to replace (\ref{thinout}) by the weaker requirement that $f(\pi_0(p(k+1)))$ and $\pi_1(p(k))$ be compatible;
however, if we define $P^\Int_f$ in this manner, $P$ will \emph{not} in general be a complete suborder. 
Moreover, that $P$ is simply a strong suborder with a simple projection satisfying (\ref{proj:am:to:Q}) simplifies the proof that $\am$ is stratified, as it allows to build conditions in $\D$ (for the preclosure system) in a coherent way, by induction on the length of conditions, without having to revisit an initial segment of the condition at amalgamation stages (see \eqref{ext:D:dense}). 

\medskip

We now define a complete embedding $e\colon\blowup{P} \rightarrow P^\Int_f$\index{strong suborder!basic amalgamation}\index{strong projection!basic amalgamation} 
and a \index[notation]{pi bar@$\bar \pi$}strong projection $\bar \pi \colon P^\Int_f \rightarrow\blowup{P}$.
For $\hat u \in \blowup{P}$
define $e(\hat u)\colon \Int \rightarrow \blowup{P}$ by  
\begin{equation*}
e(\hat u) (i) = \begin{cases}
(\pi(\hat u^P), G^i(\hat u),1) & \text{ for $i>0$,}\\
\hat u &  \text{ for $i=0$,}\\
(\pi(\hat u^P),1,F^i(\hat u)) &  \text{ for $i<0$.}
\end{cases}
\end{equation*}
\noindent
For $\bar p \in  P^\Int_f$,   define $\bar \pi(\bar p)\in \blowup{P}$ by  $\bar \pi(\bar p) = \bar p(0)$. 

\begin{lem}\label{strong:proj}
The map $\bar \pi$ is a strong projection, that is:
if $\hat w \leq \bar \pi(\bar q)$ in $\blowup{P}$, we may find $e(\hat w) \cdot \bar q \in P^\Int_f$.
\end{lem}
\begin{proof}
Let $\hat w \leq \bar \pi(\bar p)$. We define $\bar w$ by induction,  as follows:
\[ \bar w(0)=\hat w \]
Assume $\bar w(i) \in \blowup{P}$ has already been defined. We know $\pi (\bar w(i))=\pi(\bar w(i)^P)$.
Assume by induction that $\pi(\bar w(i)^P)=\pi({\hat w}^P) $.
Also, assume by induction that $\bar w(i) \leq \bar p(i)$ and $\bar w(i) \leq e(\hat w)(i)$ in $\blowup{P}$.
To inductively define $\bar w$ on the positive integers, assume $i \geq 0$ and define:
\[ \bar w(i+1) = (\pi(\hat w)\cdot\bar p(i+1)^P,\bar p(i+1)^0,\bar p(i+1)^1 \cdot F(\bar w(i))). \]

The definition of $\bar w$ on the negative integers is also by induction. Assuming $i\leq 0$, we set: 
\[ \bar w(i-1) = (\pi(\hat w)\cdot\bar p(i-1)^P,\bar p(i-1)^0 \cdot G(\bar w(i)),\bar p(i-1)^1) \]
For $i \geq 0$, as $\bar w(i) \leq \bar p(i)$, we have
\begin{multline*}
f(\pi_0(\bar w(i)))=F(\bar w(i)) \cdot f(\pi_0( \bar p(i))) \\ 
=F(\bar w(i)) \cdot \pi_1( \bar p (i+1))\\
=\pi_1( \pi({\hat w}^P)\cdot \bar p (i+1) \cdot F (\bar w(i))\ %
\end{multline*}
where the second equation holds as (\ref{thinout}) holds for $\bar p$, 
and the last equation follows from $F(\bar w(i)) \leq \pi(\bar w(i))=\pi({\hat w}^P)$.
We conclude, by definition of $\bar w (i+1)$, that
\begin {equation}\label{e:complete:main}
f(\pi_0(\bar w(i)))=\pi_1( \bar w(i+1)).
\end{equation}
\noindent Applying $\pi$ to (\ref{e:complete:main}), we see $\pi(\bar w(i+1))=\pi(\bar w(i))$, and so
\begin{align*}
\pi(\bar w(i+1)) &= \pi({\hat w}^P) =\\%\label{e:complete:ih}
\pi (\pi({\hat w}^P)\cdot\bar p(i+1)^P) &= \pi( \bar w(i+1)^P),%
\end{align*} 
where the first equation follows from the induction hypothesis and the second follows from 
\[ \pi({\hat w}^P)\leq \pi(\bar p(0)^P)=\pi(\bar p(i+1)^P). \]
Thus, $\bar w(i+1) \in \blowup{P}$, $\pi(\bar w(i))=\pi({\hat w}^P)$ and by construction, both
$\bar w(i+1) \leq \bar p(i+1)$ and $\bar w(i+1) \leq e({\hat w}^P)(i+1)$ hold.

Replacing $F$ by $G$ in the above, we obtain a similar argument for the inductive step from $i\leq 0$ to $i-1$; we leave the details to the reader. 
Finally we have that $\bar w(i) \in\blowup{P}$ and (\ref{e:complete:main}) holds for all $i \in \Int$, whence $\bar w \in P^\Int_f$. We have already shown $\bar w \leq \bar p$ and $\bar w \leq e(\hat w)$.

\medskip

We now show $\bar w \geq e(\hat w)\cdot\bar p$:
Say $\bar r \in P^\Int_f$ such that $\bar r \leq e(\hat w)\cdot \bar p$. Clearly $\bar r(0) \leq \bar w(0)=w$.
Now assume by induction that $\bar r(i) \leq \bar w(i)$. 
Then by (\ref{thinout}), 
\[
 \bar r(i+1) \leq \pi_1(\bar r(i+1)) \leq F(\bar w(i))%
\]
so as $\bar r(i+1) \leq \bar p(i+1)$, we have $\bar r(i+1) \leq \bar w(i+1)$.

A similar argument shows $\bar r(i-1)\leq\bar w(i-1)$, so we we've shown by induction that $\bar r \leq \bar w$. 
So finally, $\bar w = e(\hat w)\cdot\bar p$.
\end{proof}
For $i \in \Int$, we write $e_i$ for $\Phi^i \circ e$ and $\bar \pi_i$ for $\bar \pi \circ \Phi^i$.
\begin{cor}\label{strong:proj:cor}
For each $i \in \Int$,
the map $e_i$ is a complete embedding of $\blowup{P}$ into $P^\Int_f$. It is well-defined and injective on the separative quotient of $\blowup{P}$. The map $\bar \pi_i\colon P^\Int_f\rightarrow \blowup{P}$ is a strong projection. The map $e_i \res P$ is a complete embedding of $P$ into $P^\Int_f$. 
Letting $R = \{ \bar p \in P^\Int_f \setdef \bar p(i)^0=\bar p(i)^1=1 \}$, $R$ is dense in $P^\Int_f$, we have $e_i[P]\subseteq R$ and $\bar \pi_i\res R\colon R \rightarrow P$ is a strong projection. 
\end{cor}
\begin{proof}
The first claim is an obvious corollary of the lemma. The rest follows straightforwardly from elementary properties of $e$ and $\bar \pi$.
\end{proof}
From now on, we identify $\blowup{P}$ with $e[\blowup{P}]$ and accordingly $P$ with $\{ e(p,1,1) \setdef p \in P \}$.
\begin{cor}\index[notation]{Phi (amalgamation)@$\Phi$ (amalgamation)}
$\Phi$ is an automorphism of $P^\Int_f$ extending $f$.
\end{cor}
\begin{proof}
Let $b \in B_0$. We may assume $\pi(b)\in Q$ (this holds for a dense set of conditions in $B_0$). Thus $b \in \blowup{P}$ (to be precise, we should write $(\pi(b),b,1)$ instead of $b$).
Now as $F^n(f(b))=F^{n+1}(b)$ and $G^{n+1}(f(b))=G^n(b)$, 
\begin{align*}
 \Phi(e(b)) = \Phi((\hdots, G^2(b), G(b),\stackrel{\stackrel{0}{\downarrow }}{  b}, f(b), F^2(b)&, \hdots))=\\
                     (\hdots, G^2(b), G(b), b, \stackrel{\stackrel{0}{\downarrow }}{f( b)}, F^2(b), \hdots&) = e( f(b))
\end{align*}
So since $\Phi$ and $f$ agree on a dense set of conditions in $B_0$, they are equal on $B_0$.
\end{proof}

\section{Factoring the Amalgamation}\label{factor}

Interestingly, we can factor the amalgamation over a generic for $B_0$. We will put this to use when we investigate the tail $\am \colon P$. 
In particular, it enables us to show that if $\dot r$ is a $P$-name which is unbounded over $V^Q$, $\Phi(\dot r)$ will be unbounded not just over $V^Q$ but over $V^P$.
This will play a crucial role in the proof of the main theorem, ensuring that when we make the set without the Baire property definable, the coding (ensuring its definability) doesn't conflict with the homogeneity afforded by the automorphisms. The main point of the present section is Lemma~\ref{newreal}; it is used in Section~\ref{sec:reals:areas} on p.~\pageref{sec:reals:areas}, to prove Lemma~\ref{index:sequ}. This is in turn used in Section~\ref{sec:preserving:coding} to prove the crucial Lemma~\ref{coding:survives}.

\medskip

For an interval $I \subseteq \Int$,  let $\restam{P}{I}$\index[notation]{P I f@$\restam{P}{I}$} be the set of $\bar p\colon I \rightarrow \blowup{P}$ such that whenever both $k\in I$ and $k+1 \in I$, (\ref{thinout}) holds.
In other words  
\begin{gather*}
\restam{P}{I} =\{ \bar p \res I \setdef \bar p \in P^\Int_f \}.
\end{gather*}
It is clear that for each $k\in I$, the map $e^I_k\colon \blowup{P} \rightarrow\restam{P}{I}$, defined by $e^I_k(p)=e_k(p)\res I$ is a complete embedding. Similarly, there is a strong projection $\pi^I_k \colon\restam{P}{I}\rightarrow \blowup{P}$. 
\begin{lem}\label{factor:am}
Let $G_0=G_Q * H_0$ be $Q* \dot{B_0}$-generic. Then in $V[G_0]$, there is a dense embedding of $P^\Int_f:G_0$ into 
\[\big[\restam{P}{(-\infty,0]} : {e_0[G_Q * H_0]}\big] \times \big[\restam{P}{[1,\infty)} : {e_1[G_Q * f[H_0]]}\big]\]
and another one into  
\[ \big[\restam{P}{(-\infty,-1]} : {e_{-1}[G_Q * H_0]}\big] \times \big[\restam{P}{[0,\infty)} : {e_{0}[G_Q * f[H_0]]}\big].\]
\end{lem}
\begin{proof}
We only show how to construct the first embedding; the second part of the proof is only different in notation.
Let $R_0$ denote $\loweram{P}{0}$ and $R_1$ denote $\upperam{P}{1}$, let $H_1 = f[H_0]$ and $G_1= G_Q*H_1$.
In $V$, let $S$ denote the obvious map $S\colon P^\Int_f \rightarrow R_0 \times R_1$: $S(\bar p)_0= \bar p \res (-\infty,0]$ and $S(\bar p)_1=   \bar p\res[1,\infty) $.

Let $S^* = S \res ( P^\Int_f:G_0 )$. We show that the range of $S^*$ is dense in $(R_0:G_0) \times (R_1 : G_1)$. Since $S(\bar p) \leq S(\bar q) \iff \bar p \leq \bar q$, this implies that $S^*$ is injective on the separative quotient of its domain and thus is a dense embedding.

To show that $\ran(S^*)$ is dense, let $\bar p_0$, $\bar p_1$ be given such that $\bar p_i \in R_i:G_i$, for $i\in\{0,1\}$. 
Fix $i\in\{0,1\}$ for the moment.
Without loss of generality, $\bar p_i(i) \in P$ (and not just in $\blowup{P}$).
Let $b_i = \pi_i(\bar p_i(i)) \in \dot B_i^{G_Q}$. Then as $\bar p_i \in R_i:G_i$, $b_i \in H_i$.
Find $q \in G_Q$ and a $Q$-name $\dot b$ such that $q \leq \pi_Q(\bar p_0),\pi_Q(\bar p_1)$ and $q$ forces that 
\begin{gather}\label{dot b links p0 and p1}
\dot b = b_0\cdot f^{-1}(b_1) > 0 \quad\text{in $(\dot B_0)^{G_Q}$.} 
\end{gather}
We have that $q \forces \dot b \cdot \bar p_0(0) \neq 0$ and $f(\dot b)\cdot \bar p_1(1) \neq 0$,
or in other words, $q \cdot \dot b  \in \blowup{P}$, 
$q \cdot \dot b \leq \bar \pi ( \bar p_0)$ and $q \cdot f(\dot b) \leq \bar \pi ( \bar p_1)$. 
So we can define
$\bar p^*_0=e(q \cdot \dot b)\cdot \bar p_0$ and  $\bar p^*_1=e_1(q \cdot f(\dot b))\cdot \bar p_1$.
As 
\begin{equation}\label{in G:0}
q \in G_Q\text{ and }\dot b^{G_Q} \in H_0,
\end{equation}
we have $p^*_0 \in R_0 : G_0$ and  $p^*_1 \in R_1 : G_1$. Define $\bar p^*$:
\begin{equation*}%
\bar p^* = ( \hdots, \bar p^*_0(-1),\bar p^*_0(0),\bar p^*_1(1),\bar p^*_1(2),   \hdots )
\end{equation*}
Then $\pi_Q(\bar p^*)=q$ and by (\ref{dot b links p0 and p1}), $q$ forces that the following hold in $(\dot B_0)^{G_Q}$:
\begin{gather*}
f(\pi_0(\bar p^*_0(0)))=f(\pi_0(q\cdot\bar p_0(0)\cdot \dot b))=f(\dot b) \\
\pi_1(\bar p^*_1(1))=\pi_1(q\cdot\bar p_1(1)\cdot f(\dot b))=f(\dot b). 
\end{gather*}
Thus $\bar p^* \in D^\Int_f$, and again by (\ref{in G:0}), $\bar p^* \in  D^\Int_f:G_0$. 
As $S^*(\bar p^*)=(\bar p_0^*, \bar p_1^*) \leq (\bar p_0,\bar p_1)$, we are done.
\end{proof}
 
Let $\dot{R_i}$ be a $Q*\dot{B_0}$-name for $R_i$, for each $i \in\{0,1\}$.
We just showed that $Q*\dot{B_0}$ forces that there is a dense embedding from $\quot{P^\Int_f}{G_0}$ into $\dot R_0 \times \dot R_1$.
So there is a dense embedding of $P^\Int_f$ into $Q*\dot{B_0} * (\dot R_0 \times \dot R_1)$. Since the latter is equivalent to $\restam{P}{(-\infty,0]} * \dot R$ for some $\dot R$, we find that $\restam{P}{(-\infty,0]}$ is a complete suborder of $P^\Int_f$. The same is true for 
$\restam{P}{[0,\infty)}$ (or more generally, for $\restam{P}{I}$, where $I$ is any interval in $\Int$). 
In fact, it's easy to show that the natural embedding and projection witness this.

\medskip

The previous lemma yields an important insight concerning the action of the automorphism $\Phi$:
It enables us to show that if $\dot x$ is a $P$-name which is not in $V^{B_0}$ (and hence also not in $V^{B_1}$), then
for all $i \in  \Int\setminus\{0\}$, $\Phi(\dot x) \not \in V^{P}$. In fact, for the proof of the main theorem, we shall need something a bit more specific: 
\begin{lem}\label{newreal}\index{random real!over V Q@over $V^Q$}\index{unbounded (over V Q)@unbounded (over $V^Q$)}
Assume that $\dot r_0, \dot r_1$ be $P$-names for reals random over $V^Q$, and 
assume  $\forces_Q \dot B_i = \gen{ \dot r_i}^{\quot{P}{Q}}$ (as is the case in our application).
If $\dot r$ is a $P$-name for a real such that $\dot r$ is unbounded over $V^Q$, then for any $i \in  \Int\setminus\{0\}$, $\Phi^i(\dot r)$ unbounded over $ V^{P}$.
\end{lem}
\begin{proof}
Firstly, $\dot r$ is unbounded over $V^{B_i}$, for each $i\in\{0,1\}$, since the random algebra does not add unbounded reals.

For a start, let's assume $i = 1$.
Let $G_1=G_Q * f[H_0]$ be $Q* \dot{B_1}$-generic and work in $W=V[G_1]$. 
We have that $\dot r$ is a $\quot{P}{ G_1 }$ name for a real which is unbounded over $W$ in the sense of Definition~\ref{generic:reals}---in any $\quot{P}{G_1}$-generic extension of $V[G_1]$, the interpretation of $\dot r$ will be unbounded over $V[G_1]$.
Let $R_0$, $R_1$ be defined as in the previous proof, i.e.,
\begin{align*}
R_1&=\quot{ \restam{P}{[1,\infty)} }{G_1 },\\
R_0&=\restam{P}{(-\infty,0]} : {(G_Q * H_0)}, 
\end{align*}
let $\dot{R_i}$ be a $Q*\dot{B_0}$-name for $R_i$, for each $i \in\{0,1\}$, and let $I=[1,\infty)$. 
As $\quot{P}{ G_1 }$ is a complete suborder of $R_1$,
$e^I_1(\dot r)$ is an $R_1$-name which is unbounded over $W$. 
By Lemma~\ref{unbounded}, viewing $e^I_1(\dot r)$ as a $R_0\times R_1$-name, it is unbounded over $W^{R_0}$.
As $G_1$ was arbitrary, $e^I_1(\dot r)$ is a $Q * \dot{B_0} * (\dot{R_1}\times \dot{R_0})$-name unbounded over $V^{Q * \dot{B_0} * \dot{R_0}}$.
By the previous theorem this means that $e_1(\dot r)$ is a $P^\Int_f$-name unbounded over $V^{Q * \dot{B_0} * \dot{R_0}}=V^{\restam{P}{(-\infty,0]}}$ and hence over $V^P$,
since $S \circ e_0=e^I_0$ shows that $P$ is a complete suborder of $Q * \dot{B_0} * \dot{R_0}$.

\medskip

For arbitrary $i \in \Int$ such that $i > 0$: 
We just showed that $e_1(\dot r)$ is a $P^\Int_f$-name unbounded over $V^{\restam{P}{(-\infty,0]}}$.
Since $e^I_{-i+1}[P]$ is a complete suborder of $\restam{P}{(-\infty,0]}$, we know $e_1(\dot r)$ is unbounded over $V^{e_{-i+1}[P]}$.
Apply $\Phi^{i-1}$ to see $\Phi^i(\dot r)$ is unbounded over $V^{e_0[P]}$, as $e_0= \Phi^{i-1} \circ e_{-i+1}$.
For $i < 0$, argue exactly as above but use the second dense embedding mentioned in Lemma~\ref{factor:am}. 
\end{proof}

\section{Stratified Type-1 Amalgamation}\label{am}

We now turn to the matter of stratification.
Assume $(Q,P)$ is a stratified extension on $I=[\lambda_0,\kappa^*)\cap \Reg$, as witnessed by 
$$\pss_P=(\D_P, \param_P,\leqlol,\lequpl,\Cl)_{\lambda\in I}$$ 
and $\pss_Q=(\D_Q,\param_Q, \leqlol_Q,\lequpl_Q,\Cl_Q)_{\lambda\in I}$.
As mentioned, we allow $\kappa^*$ to be $\infty$, as well.
We never need to mention $\leqlol_Q$, $\lequpl_Q$, $\Cl_Q$ and $\D_Q$ as we can always use the corresponding relation from $\pss_P$ (see the remark following Definition~\ref{pss:is}, p.~\pageref{pss:is}). 
Moreover, assume $\forces_Q \card{\dot B_0} \leq \lambda_0$.

\medskip

The main problem with stratification and amalgamation is quasi-closure:
Consider two sequences $(p_\xi)_{\xi<\rho}$ and $(q_\xi)_{\xi<\rho}$ such that
$p_\xi$ and $q_\xi$ are compatible for every $\xi<\rho$, with greatest lower bounds $p$ and $q$ respectively. In general, $p$ and $q$ don't have to be compatible.  A similar problem occurs with regard to the defining equation (\ref{thinout}) of amalgamation: say we have a sequence of conditions $\bar p_\xi \in \am$ and for each $i \in \Int$, $\bar p(i)$ is a greatest lower bound of $(\bar p_\xi(i))_\xi$. Even though (\ref{thinout}) holds for every $\bar p_\xi$, it could fail for $\bar p$.

\medskip

The solution to this problem is to thin out to a dense subset of $P$ where $\pi_i$ is stable with respect to ``direct extension''\index{direct extension}, before we amalgamate.
That is, on this dense subset, $\pi_i$  doesn't change (in a strong sense) when conditions are extended in the sense of $\leqlol$, for $\lambda\in I$.
\begin{dfn}\label{am:gen:D}
Let $\Dam=\Dam(Q,P,f,\lambda_0)$\index[notation]{D(Q,P,f,lambda0)@$\Dam(Q,P,f,\lambda_0)$} be the set of $p \in P$ such that %
for all $q \in P$, if $q \leqlo^{\lambda_0} p$ we have
\begin{equation}\label{am:lower:part:frozen:nice}
\forall (b_0,b_1) \in B_0 \times B_1 \quad \big(\pi(q)\cdot p \cdot b_0 \cdot b_1 \neq 0\big) \Rightarrow \big(q\cdot b_0 \cdot b_1 \neq 0\big)
\end{equation}
\end{dfn}
 
Observe that (\ref{am:lower:part:frozen:nice}) is equivalent to:
\begin{equation}
\forall j\in\{0,1\} \quad \pi(q)\forces_Q \forall b \in \dot B_{1-j} \quad \pi_j(q\cdot b)=\pi_j(p \cdot b),\label{lower:part:frozen}\\
\end{equation}
and also to the following:
\begin{equation}
\forall j\in\{0,1\}\quad \forall b \in B_{1-j} \quad \pi_j(q\cdot b) = \pi(q) \cdot  \pi_j(p \cdot b).\label{lower:part:frozen:alt}
\end{equation}
\begin{lem}\label{lem:shrink}
$\Dam$ is open dense in $\langle P, \leqlo^{\lambda_0}\rangle$.
\end{lem}
\begin{proof}
Let $p_0$ be given. We inductively construct an adequate sequence of $p_\xi$, $0<\xi\leq{\lambda_0}$ with $p_{\lambda_0} \in \Dam$. 
First fix $x$ such that the following definition is $\qcdefSeq$ in parameters from $x$. Fix $Q$-names $\dot b_j$ such that $\forces_Q \dot b_j\colon \lambda_0 \rightarrow \dot B_j$ is onto, for $j=0,1$, 
and let $\xi \mapsto ( \alpha_\xi, \beta_\xi, \zeta_\xi)$ be a surjection from $\lambda_0$ onto $(\lambda_0)^3$. 

\medskip

For limit $\xi$, let $p_\xi$ be the greatest lower bound of the sequence constructed so far.
Say we have constructed $p_\xi$, we shall define $p_{\xi+1}$. 
Let's first assume there are $p^*, \bar p$ such that 
$\bar p \leqlo^{\lambda_0} p_\xi$, $\bar p\in\D(\lambda_0,x,p_\xi)$, $p^* \leq \bar p$ and
\begin{enumerate}
\item\label{descend} $\pi(p^*)\forces \bar p \cdot \dot b_0(\alpha_\xi)\cdot \dot b_1(\beta_\xi) = 0$,
\item\label{colour} $\zeta_\xi \in \Clink^{\lambda_0}(p^*)$.
\end{enumerate}
In this case pick $p_{\xi+1}$ such that $p_{\xi+1} \leqlo^{\lambda_0} \bar p$ and $p_{\xi+1}\lequp^{\lambda_0} p^*$ (using interpolation).
If, on the other hand, no such $\bar p, p^*$ exist, just pick $p_{\xi+1} \in\D(\lambda_0,x,p_\xi)$.

\medskip

We now show (\ref{lower:part:frozen}) holds for the final condition $p_{\lambda_0}$: say, to the contrary, we can find $j \in \{0,1\}$ and $\dot b \in B_{1-j}$ together with $\bar q \leqlo^{\lambda_0} p_{\lambda_0}$ such that
\[ \pi(\bar q )\not\forces_Q \pi_j(\bar q\cdot \dot b)=\pi_j(p_{\lambda_0}\cdot \dot b). \]
Without loss of generality say $j=0$. We can find $q^* \leq \bar q$ such that for some $\alpha, \beta < \lambda_0$
\begin{enumerate}[(i)]
\item \label{compatible} $\pi(q^*)\forces \pi_0(p_{\lambda_0}\cdot \dot b) - \pi_0(\bar q\cdot \dot b) = \dot b_0(\alpha)\neq 0$,
\item $\pi(q^*)\forces \dot b = \dot b_1(\beta)$,\label{fix_b}
\item $q^* \in \dom(\Clink^{\lambda_0})$.
\end{enumerate}
Find $\xi<{\lambda_0}$ so that $\alpha=\alpha_\xi$, $\beta=\beta_\xi$ and $\zeta_\xi \in \Clink^{\lambda_0}(q^*)$. 
By construction, at stage $\xi$ of our construction we had $\bar p$ and $p^*$ satisfying (\ref{descend}) and (\ref{colour}).
As $\Clink^{\lambda_0}(p^*)\cap \Clink^{\lambda_0}(q^*)\neq 0$ and $q^*\leq p_{\xi+1}\lequp^{\lambda_0} p^*$, we can find $w \leq p^*, q^*$.
By \ref{compatible} and \ref{fix_b}, $\pi(w)\forces p_{\lambda_0} \cdot \dot b_0 (\alpha) \cdot \dot b_1(\beta) \neq 0$.
But since $w \leq p^*$, 
$\pi(w)\forces \bar p \cdot \dot b_0 (\alpha) \cdot \dot b_1(\beta) = 0$ and so also
$\pi(w)\forces p_{\lambda_0} \cdot \dot b_0 (\alpha) \cdot \dot b_1(\beta) = 0$, contradiction.

\medskip

Now we show $\Dam$ is open: 
For any $r \leqlo^{\lambda_0} q$, $j=0,1$ and $\dot b \in B_{1-j}$, since $r \leqlo^{\lambda_0} p$, we have $\pi(r)\forces \pi_j (r \cdot \dot b) = \pi_j (p \cdot \dot b)$.
Since $\pi(q)\forces \pi_j (q \cdot \dot b) = \pi_j (p \cdot \dot b)$ and $r \leq q$,
$\pi(r)\forces \pi_j (r \cdot \dot b) = \pi_j (q \cdot \dot b)$. So $q \in \Dam$.
\end{proof}

Having $Q \subseteq \Dam$ helps in many circumstances, in particular we like to have $1_P \in \Dam$. 
To this end we introduce the notion of $B_0, B_1$ being $\lambda_0$-reduced.
\begin{dfn}
We say the pair $B_0, B_1$ is $\lambda$-reduced\index{reduced pair|textbf}\index[notation]{lambda-reduced@$\lambda$-reduced pair|textbf} over $Q$ if and only if whenever $p \in P$, $p \leqlo^{\lambda} q$ for some $q\in Q$ and $b\in B_j$ for $j=0$ or $j=1$, we have
\[ \pi_{1-j}(p\cdot b)=\pi(p)\cdot\pi(b). \]
\end{dfn}

Henceforth assume $B_0$, $B_1$ is a $\lambda_0$-reduced pair. We will later see that this is a very mild assumption, see Lemmas~\ref{reduced:equ} and \ref{reduce:a:pair}.
\begin{lem}\label{Q:in:D}
If $p \leqlo^{\lambda_0} q$ for some $q \in Q$ and $j\in\{0,1\}$ we have 
\[ 
\pi(p)\forces \forall b \in \dot B_{1-j}\setminus \{ 0\} \quad\pi_j(p\cdot b)=1,
\]
and moreover,  $p \in \Dam$. In particular, we have $Q \subseteq \Dam$.
\end{lem}
\begin{proof}
Fix $p$ as in the hypothesis. Say $r \in Q$, $r \leq \pi(p)$ and $b \in B_0$ such that $r \forces b \in \dot B_0\setminus\{0\}$. Then $r \leq \pi( b)$.
So as $B_0$, $B_1$ is $\lambda_0$-reduced, $r \leq \pi_1(p\cdot  b)$, whence $r \forces \pi_1(p\cdot b)=1$. This proves the first statement for $j=1$, and in the other case the proof is the same.

\medskip

We now show $p \in \Dam$: Say $p' \leqlo^{\lambda_0} p$. Since also $p' \leqlo^{\lambda_0} q$, we have
\[
\pi(p')\forces \forall b \in \dot B_{1-j}\setminus\{0\}\quad \pi_j(p'\cdot b) = 1 =\pi_j(p\cdot b),
\]
and thus $p \in \Dam$.
\end{proof}

In fact, the first statement of Lemma~\ref{Q:in:D} is equivalent to $B_0$, $B_1$ being a reduced pair (this is really just a slight variation of Lemma~\ref{reduced:equ}). 

\medskip

The following lemma provides a hint as to how we can assume that $B_0$, $B_1$ is $\lambda_0$-reduced.
For the lemma, recall from the beginning of this chapter that $B_i = \ro(Q) * \dot B_i$.
\begin{lem}\label{reduced:equ}
Assume that $\dot r_0, \dot r_1$ are $P$-names for reals random over $V^Q$, and 
assume  $\forces_Q \dot B_i = \gen{ \dot r_i}^{\quot{P}{Q}}$ (as will be the case in our application).
Say $j=0$ or $j=1$. The following are equivalent (interestingly, in (\ref{reduced:alt}), there is no mention of $j$):
\begin{enumerate}
\item\label{reduced:orig} Whenever $p \in P$, $p \leqlo^{\lambda} q$ for some $q\in Q$ and $b\in B_j$, we have 
\[ \pi_{1-j}(p\cdot b)=\pi(p)\cdot\pi(b).\]
\item\label{reduced:alt}
Whenever  $p \in P$, $p \leqlo^{\lambda} q$ for some $q\in Q$ and $b_0$, $b_1$ are $Q$-names for Borel sets such that for some $w \leq \pi(p)$,
$w \forces_Q$``both $b_0$ and $b_1$ are not null'', there is $p' \leq p$ such that $p'\forces_P \dot r_0 \in b_0$ and $\dot r_1 \in b_1$.
\end{enumerate}
\end{lem}
\begin{proof}
First, assume Item \ref{reduced:alt}. We carry out the proof for $j=0$ (the other case is exactly the same). Let $p \in P$ such that for some $q \in Q$, $p \leqlo^{\lambda_0} q$ and let $b_0 \in B_0$. As for any $r \in \ro(P)$, $r\leq\pi_j(r)\leq\pi(r)$ holds, we have 
$\pi_j(p\cdot b_0) \leq \pi(p) \cdot \pi(b_0)$. 
We now show $\pi_j(p\cdot b_0)\geq\pi(p)\cdot\pi(b_0)$. It suffices to show that whenever $b_1 \in B_1$ is compatible with $\pi(p)\cdot\pi(b_0)$, it is compatible with $p\cdot b_0$. 
So fix $b_1\in B_1$. 
We have $\pi(b_1)\cdot \pi(b_0)\cdot \pi(p) \neq 0$, so we may pick $w \leq \pi(b_1)\cdot \pi(b_0)\cdot \pi(p)$.
For $j=0,1$, let $\dot b_j$ be a $Q$-name for a Borel set such that $b_j = \bv{ \dot r_j \in \dot b_j}^{\ro(P)}$.
The last inequality means $w \forces \dot b_0$ and $\dot b_1$ are not null. So by assumption, we can find $p'$ forcing $\dot r_j \in \dot b_j$ for both $j=0,1$. 
In other words, $p' \leq p \cdot b_0 \cdot b_1$, whence $b_1$ is compatible with $p\cdot b_0$.

\medskip

For the other direction, assume Item \ref{reduced:orig} and again assume $j=0$, fix $p$ as above, and say $\dot b_0, \dot b_1$ are $Q$-names such that 
$w \forces \dot b_0, \dot b_1 \in \Borelplus$ for some $w\leq \pi(p)$. 
Let $b_j = \bv{ \dot r_j \in \dot b_j}^{\ro(P)}$. As $\pi(b_0)\cdot \pi(b_1)\cdot \pi(p)\neq0$,
$b_1$ is compatible with $\pi(b_0)\cdot \pi(p) = \pi_1(p\cdot b_0)$. 
Thus $b_1$ is compatible with $p\cdot b_1$. So we may pick $p' \in P$, $p' \leq p\cdot b_0 \cdot b_1$. 
\end{proof}
\begin{dfn}\label{d:reduced:pair:alt}\index{reduced pair}\index[notation]{lambda-reduced@$\lambda$-reduced pair}
Under the assumptions of the previous lemma, we also say \emph{the pair $\dot r_0$, $\dot r_1$ is $\lambda$-reduced} to express that the two (equivalent) conditions from said lemma hold.
\end{dfn}

We shall need the next lemma to show that $P$ completely embeds into $\am$ (see \ref{strong:proj:am}). Observe that the next lemma does not make the assumption that $B_0,B_1$ is a
$\lambda_0$-reduced pair obsolete, i.e., by itself the lemma does not imply $Q \subseteq \Dam$.
\begin{lem}\label{Q:cdot:D}\index{stable meet}\index[notation]{Q-stable meet@$Q$-stable meet}\index[notation]{wedge (stable meet)@$\stm$ (stable meet)}
Assume that there exists a $Q$-stable meet operator $\qstm$ on $P$ with respect to $\pss$.
Then if $p \in \Dam$ and $q \in Q$ are such that $q \leq \pi(p)$, we have $q \cdot p \in \Dam$.
Moreover, if $(p,r)\in\dom(\qstm)$ and $p \in \Dam$, for any $j\in\{0,1\}$ and $b \in B_{1-j}$ we have $\pi_j((p\qstm r)\cdot b)=\pi_j(p\cdot b)$.
\end{lem}
\begin{proof}
Let $p \in P$, $q \in Q$ and $q\leq\pi(p)$. We check that $q \cdot p \in \Dam$.
So let 
\begin{equation}\label{schnitzel:dom:var}
r \leqlo^{\lambda_0} q \cdot p,
\end{equation}
and fix $j \in \{0,1\}$ and $b \in B_{1-j}$.
To prove that $q \cdot p \in \Dam$, it suffices to show
\begin{equation}\label{q:cdot:p:in:D}
\pi_j(r\cdot b)=\pi(r)\cdot \pi_j(q\cdot p\cdot b).
\end{equation}
Observe that (\ref{schnitzel:dom:var}) implies that $r \leqlo^{\lambda_0} \pi(r) \cdot p$---for by \ref{pcs:is}(\ref{pi:mon}), $\pi(r)\leqlol \pi(q \cdot r) = q$; now use \ref{pss:is}(\ref{pss:is:lo:init}). Thus $(p,r) \in \dom(\qstm)$ and $p\qstm r \leqlo^{\lambda_0} p$.
Thus $p\qstm r \in \Dam$ and 
\[\pi_j((p\qstm r )\cdot b)=\pi(p \qstm r) \cdot \pi_j(p\cdot b) = \pi_j(p\cdot b),\]
where the last equation holds because $\pi(p \qstm r)=\pi(p) \geq \pi_j(p\cdot b)$. Note in passing that this proves of the ``moreover'' clause of the lemma. We continue with the proof of the remaining part of the lemma.
By the previous, as $r=\pi(r)\cdot (p\qstm r)$,
\[ \pi_j(r \cdot b) = \pi(r)\cdot \pi_j( (p\qstm r) \cdot b)= \pi(r) \cdot \pi_j(p \cdot b)= \pi(r) \cdot \pi_j(q\cdot p \cdot b). \]
The last equation holds as $\pi(r)\leq q$.
This finishes the proof of the lemma.
\end{proof}

From now on, assume we have a $Q$-stable meet $\qstm$ on $P$.

\medskip

While it is true that $(\blowup{\Dam}, \Dam^\Int_f)$ is a stratified extension, this is not quite the partial order we use in the main theorem:
For replacing $P_\xi$ by a dense set $\Dam_\xi$ at cofinally many stages of our iteration leads to obvious complications at limit stages. Instead, we have a much simpler solution.
\begin{dfn}[Type-1 amalgamation]\index{type-1 amalgamation}\index[notation]{Am1(Q,P,f,lambda 0)@$\am(Q,P,f,\lambda_0)$}\index{amalgamation!type-1}
Let $\am=\am(Q,P,f,\lambda_0)$ be the set of $\bar p\colon\Int\rightarrow \blowup{P}$ such that the following conditions are met.
\begin{enumerate}
\item For all $i \in \Int$, $\pi(\bar p(i)^P) = \pi(\bar p(0)^P)$.
\item For all but finitely many $i\in\Int$, $\bar p(k)^P \leqlo^{\lambda_0} \pi(\bar p(k)^P)$.
\item For all $i \in \Int\setminus\{-1,0\}$, $f(\pi_0(\bar p(i)))=\pi_1(\bar p(i+1))$ --- that is, (\ref{thinout}) holds.
\item $\bar p(0) \in P$, i.e., $\bar p^0(0)=\bar p^1(0)=1$ and 
\begin{align}
f(\pi_0(\bar p(-1)))& \geq  \pi_1(\bar p(0)),\label{middle:l}\\
f(\pi_0(\bar p(0)) & \leq \pi_1(\bar p(1))\label{middle:r}.
\end{align}
\item For $i \in \Int\setminus\{0\}$, $\bar p(i)^P \in \Dam(Q,P,f,\lambda)$.
\end{enumerate}
\end{dfn}

Observe we can replace (\ref{middle:l}) and (\ref{middle:r}) by
\begin{equation}\label{middle:both}
\bar p(0)^P \leq f(\pi_0(\bar p(-1))) \cdot f^{-1}(\pi_1(\bar p(1))).
\end{equation}
and obtain an equivalent definition. 
Thus, $\bar p \in \am$ if and only if the following conditions are met:
\begin{enumerate}
\item $\bar p(0) \in P$, 
\item $\bar p\res [1,\infty) \in \Dam^{[1,\infty)}_f$ and $\bar p\res (-\infty,-1] \in \Dam^{(-\infty,-1]}_f$
\item for both $j\in\{-1,1\}$ we have $\pi(\bar p(0))=\pi(\bar p(j))$ and (\ref{middle:both}) holds.
\end{enumerate}
Let $a\colon P \rightarrow \am$ be defined by $a(p)(0)=(p,1,1)$ and $a(p)(i)=(\pi(p),1,1)$ for all $i\in\Int\setminus\{0\}$. 
As before, let \index[notation]{pi bar@$\bar \pi$}$\bar\pi(\bar p)=\bar p(0)^P$ (we see no problem in using the same designation as for the projection from $\Dam^\Int_f$ to $\Dam$---see the remark after the next lemma). 
\begin{lem}\label{strong:proj:am}
The map $a\colon P\rightarrow \am$ is a complete embedding and 
\[\bar \pi \colon \am \rightarrow P
\] is a strong projection.
\end{lem}
\begin{proof}
Let $\bar p \in \am$, $w \in P$, and $w \leq \bar p(0)$. Define $\bar p'$ by
\[ \bar p' (i)= \begin{cases} w &\text{for $i=0$,}\\
                              (\pi(w)\cdot \bar p(i)^P,\bar p(i)^0,\bar p(i)^1) &\text{for $i \in\Int\setminus\{0\}$.} \end{cases} \]
Clearly $\bar p' \in \am$, $\bar p' \leq a(w)$ and $\bar p' \leq \bar p$. 
Moreover, for arbitrary $\bar q \in \am$, if $\bar q\leq a(w)$ and $\bar q \leq \bar p$, clearly $\bar q \leq \bar p'$; so $\bar p' = a(w)\cdot \bar p$.
This shows that $\bar \pi$ is a strong projection and accordingly, $a$ is a complete embedding.
\end{proof}

In what follows, we identify $P$ and $a[P]$---except when we feel this would hide the point of the argument.
Next we show that in fact, $\am$ and $\Dam^\Int_f$ are presentations of the same forcing.
\begin{lem}\label{common:dense:set}
The set $D^*=\{ \bar p \in \Dam^\Int_f \setdef \bar p(0)^0=\bar p(0)^1=1 \}$ is dense in both $\Dam^\Int_f$ and $\am$.
\end{lem}
\begin{proof}
First, we notice that $D^* \subseteq \am$ and that the ordering of $\Dam^\Int_f$ and that of $\am$ coincide on $D^*$.
Given $\bar p \in \Dam^\Int_f$, find $d \in \Dam$ such that $d \leq \bar p(0)^P\cdot p(0)^0\cdot p(0)^1$; clearly, $d \cdot \bar p \in D^*$.

\medskip

Now let $\bar p \in \am$. We find $\bar w \leq \bar p$, such that $\bar w \in D^*$.
Find $d \in \Dam$ such that $d \leq \bar p(0)$.
First let $I=(-\infty,0]$ and construct $\bar w^{-}=\bar w\res I$.
Let $b_0=f(\pi_0(\bar p(-1)))$ and define $\bar p^{-} \in \Dam^{I}_f$ by
\[ \bar p^{-}=( \hdots, \bar p(i), \hdots, \bar p(-1), b_0 ), \]
where of course we identify $b_0$ and $(1_P, \pi(b_0),b_0)\in \blowup{\Dam}$.
Since $d \leq \bar p(0) \leq b_0$ and $b_0 = \bar \pi^I_0(\bar p^{-})$,
we can let $\bar w^{-}= d \cdot \bar p^{-} \in \Dam^{I}_f$.
Observe that $\bar \pi^I_0(\bar w^{-})=d$. %

Let $I=[0,\infty)$. In an analogous fashion, define $\bar w^{+} \in \Dam^I_f$ such that $\bar w^{+} \leq \bar p\res I$ and $\pi^I_0(\bar w^{+})=d$. Letting 
\[ \bar w(i)=\begin{cases} \bar w^{-}(i) &\text{for $i < 0$,}\\
                           \bar w^{+}(i)  &\text{for $i \geq 0$,} \end{cases}
\]
we conclude $\bar w \in \am$.
Moreover, $\bar\pi(\bar w)=d \in \Dam$ whence $\bar w \in D^*$, and $\bar w \leq \bar p$ in $\am$.
\end{proof}
Thus, although $\Phi$\index[notation]{Phi (amalgamation)@$\Phi$ (amalgamation)} is not an automorphism of $\am$, since it is an automorphism of $\Dam^\Int_f$, it gives rise to an automorphism of the associated Boolean algebra.
We call $\Phi$ the \emph{automorphism resulting from the amalgamation,}\index{automorphism resulting from amalgamation} and we refer to $Q$ as the \emph{base of the amalgamation}\label{base}\index{base of amalgamation, or of Phi@base of amalgamation, \emph{or} of $\Phi$} or, interchangeably, \emph{the base of $\Phi$}. 

\medskip

That $\ro(\am) = \ro(\Dam^\Int_f)$ justifies that we use the same notation for the strong projections $\bar \pi \colon \am \rightarrow P$ and $\bar \pi \colon \Dam^\Int_f \rightarrow \Dam$---as we know a strong projection agrees on its domain with the canonical projection.
The next lemma clarifies the role of $\Dam$.
\begin{lem}\label{stay:in:am}
Let $\bar p \in \am$ and say $\bar q \colon \Int \rightarrow P\times B_0\times B_1$ satisfies the following conditions:
\begin{enumerate}
\item \label{stay:pi} for each $i \in \Int$, $\pi(\bar q(i)^P)=\pi(\bar q(0)^P)$.
\item $\bar q(0)^0=\bar q(0)^1=1$.
\item $\forall i \in \Int\setminus\{0\}\quad \bar q(i)^P \leqlo^{\lambda_0} \bar p(i)^P$.
\item \label{stay:pi:j} $\forall i \in \Int\setminus\{0\}\quad \pi_j(\bar q(i)) = \pi(\bar q(i)^P)\cdot \pi_j(\bar p(i))$
\end{enumerate}
Then $\bar q \in \am$.
\end{lem}
\begin{proof}
First, let $I=[1,\infty)$ and show $\bar q \res I \in \Dam^I_f$.
Let $i\in I$ be arbitrary. 
By Item~\ref{stay:pi:j} above, we have
\begin{equation}\label{q:p:proj}
\pi_j(\bar q(i))=\pi(\bar q(i)^P)\cdot \pi_j(\bar p(i))
\end{equation}
for $j \in \{0,1\}$.
Since by \ref{stay:pi} we have $\pi(\bar q(i)^P) = \pi(\bar q(0)^P) \leq \pi(\bar p(0)^P) = \pi(\bar p(i))$,
applying $\pi$ to (\ref{q:p:proj}) yields
\begin{equation}
\pi(\bar q(i)) = \pi(\bar q(i)^P)\cdot\pi(\bar p(i))= \pi(\bar q(i)^P),
\end{equation}
which means 
\begin{equation}\label{in:blowup}
\bar q(i) \in \blowup{P}.
\end{equation} 
Since $\bar p \in \am$ and since (\ref{q:p:proj}) holds, we have 
\begin{equation*}
f(\pi_0(\bar q(i)))=\pi(\bar q(0))\cdot f(\pi_0(\bar p(i)))=\pi(\bar q(0)) \cdot \pi_1(\bar p(i+1)) =\pi_1(\bar q(i))
\end{equation*}
Thus $\bar q\res I \in  \Dam^I_f$.
Repeat the argument above to show $\bar p\res (-\infty,-1] \in \Dam^{(-\infty,-1]}_f$.
As $\bar q(0)^0=\bar q(0)^1=1$ by assumption, (\ref{in:blowup}) holds for $i=0$.
Let $b = f(\pi_0(\bar p(-1))) \cdot f^{-1}(\pi_1(\bar p(1)))$.
As $\bar q(0) \leq \bar p(0) \leq b$, clearly
\[ \bar q(0) \leq \bar\pi(\bar q(0)) \cdot b = f(\pi_0(\bar q(-1))) \cdot f^{-1}(\pi_1(\bar q(1))). \]
Thus, finally $\bar q \in \am$.
\end{proof}

\medskip

Finally, we are ready to state and prove the main theorem of this section:
\begin{thm}\label{thm:am:s:ext}\index{amalgamation!is a stratified extension}
$(P,\am)$ is a stratified extension on $J=I \setminus (\lambda_0)^+$.
\end{thm}
\begin{proof}
We proceed to define a stratification of $\am$. $\am$ is going to be stratified above $(\lambda_0)^+$,
but in general not above $\lambda_0$, which comes from the fact that possibly $B_0$ and $B_1$ conspire to yield antichains of size $(\lambda_0)^+$.
\footnote{This is the case if we amalgamate, e.g., over copies of the Cohen algebra, as was shown by Shelah which is why he invented sweetness---see \cite{shelah:amalgamation}. Perhaps the same occurs for Random algebras, although we are currently unable mention a concrete example.}

For notational convenience, we define $\bar q \alol \bar p$\index[notation]{lessthanl bar lambda@$\alol$} for arbitrary $\Int$-sequences $\bar q,\bar p \in {}^\Int(P\times B_0\times B_1)$ and for
$\lambda \geq \lambda_0$: $\bar q \alol \bar p$ exactly if for every $i\in\Int$, $\bar q(i)^P \leqlo^\lambda \bar p(i)^P$ and for every $i\in\Int\setminus\{0\}$ we have $\pi(\bar q(i)^P)\forces_Q \pi_j(\bar q(i)) = \pi_j(\bar p(i))$---or equivalently, 
\begin{equation}\label{alol:bool}
\pi_j(\bar q(i)) = \pi(\bar q(i)^P)\cdot \pi_j(\bar p(i))
\end{equation} 
for both $j\in\{0,1\}$.
\begin{cor}\label{stay:in:am:cor}
Using this notation we can state Lemma~\ref{stay:in:am} in the following way:
If for some regular $\lambda\geq\lambda_0$, $\bar p \in \am$ and $\bar q \colon \Int \rightarrow P\times B_0 \times B_1$ satisfy $\bar q \alol \bar p$ and moreover $\bar q(0) \in P$ and for all $i \in\Int$, $\pi(\bar q(i)^P)=\pi(\bar q(0)^P)$ holds, then $\bar q \in \am$.
\end{cor}
\begin{lem}\label{alo}
Observe that if $\bar q \colon \Int \rightarrow P\times B_0 \times B_1$ and $\bar p \in \am$ satisfy $\bar q(i)^P\leqlo^{\lambda} \bar p(i)^P$ for all $i \in \Int$ and $\bar q(i)^j=\bar p(i)^j$ for all $i \in\Int\setminus\{0\}$ and $j\in\{0,1\}$, then $\bar q \alol \bar p$.
\end{lem}
\begin{proof}
For $i \in \Int\setminus\{0\}$ and $j\in\{0,1\}$, we have 
\begin{multline*}
\pi_j(\bar q(i))=\bar p(i)^j \cdot\pi_j(\bar q(i)^P\cdot \bar p(i)^{1-j}) \\
 =\bar p(i)^j\cdot \pi(\bar q(i)^P)\cdot\pi_j(\bar p(i)^P\cdot \bar p(i)^{1-j})=\pi(\bar q(i)^P)\cdot\pi_j(\bar p(i)).
\end{multline*}
where the second line is equal to the first as $\bar p(i)^P \in \Dam$ and $\bar q(i)^P \leqlol \bar p(i)^P$.
Thus, $\bar q \alol \bar p$.
\renewcommand{\qedsymbol}{{\tiny  Lemma~\ref{alo}~}$\Box$}
\end{proof}

\medskip

Now let $\bar p, \bar q \in \am$ and say $\lambda \in I$ and $\lambda>\lambda_0$. 
Define 
\[\bar q \in \aD(\lambda,x,\bar p)\index[notation]{D bar(lambda,x,p)@$\aD(\lambda,x,\bar p)$}\iff \forall i \in \Int \quad\bar q(i)^P \in \D(\lambda,x,\bar p^P(i)).\] 
We say \index[notation]{lessthanu bar lambda@$\aupl$}$\bar q \aupl \bar p$ exactly if
\[ \forall i \in \Int\quad \bar q(i)^P \lequp^\lambda \bar p(i)^P.\]
Next we define \index[notation]{C lambdaz@$\aCl$}$\aCl$.
Fix a name $\bij$ such that 
\[ \forces_P \bij \colon \dot B_0 \cup \dot B_1\rightarrow \check{\lambda_0}\text{ is a bijection.}\]
Let $\dom(\aCl)$ be the set of all $\bar p \in \am$ such that for each $i \in \Int$, we have $\bar p(i)^P \in \dom(\Cl)$ and if $i\neq 0$, there is $\lambda'\in\lambda\cap I$ such that for $j \in\{0,1\}$ we have that $\bij(\pi_j(\bar p(i)))$ is $\lambda'$-chromatic below $\pi(\bar p(i)^P)$.
If $\bar p\in \dom(\aCl)$, we define $\aCl(\bar p)$ to be the set of all $(c(i),\lambda'(i),H^0(i),H^1(i))_{i\in\Int}$ such that
for all $i \in \Int$, $c(i)\in \Cl(\bar p(i))$ and for all $i \in \Int\setminus\{0\}$ and $j\in\{0,1\}$, $H^j(i)$ is a $\lambda'(i)$-spectrum of $\bij(\pi_j(\bar p(i)))$ below $\pi(\bar p(0)^P)$. Observe that $\lambda'(0)$, $H^0(0)$ and $H^1(0)$ can be chosen arbitrarily---they merely serve as place-holders to facilitate notation.
This finishes the definition of the stratification of $ \Dam^\Int_f$.

\medskip

First we check that $\aD$ and $(\alol)_{\lambda\in I}$ give us a preclosure system, see \ref{def:pcs}, p.~\pageref{def:pcs}.
That $\aD$ is $\qcdefF$ is immediate (without any further assumptions on the parameter $x$).
For the following, let $\bar p$, $\bar q$, $\bar r \in \am$, $\lambda \in I$ and $x$ be arbitrary.

It is clear that \eqref{qc:D} holds, for if $\bar q \leq \bar p \in \aD(\lambda, x, \bar r)$,
then $\bar q(i)^P \leq \bar p(i)^P \in \aD(\lambda, x, \bar r(i)^P)$ for each $i\in \Int$.
Thus by \eqref{qc:D} for $P$, $\bar q(i)^P  \in \aD(\lambda, x, \bar r(i)^P)$ for each $i\in \Int$ and we are done.
For (\ref{qc:preorder}), we must prove transitivity, so say $\bar p \alol \bar q \alol \bar r$ and show $\bar p \alol \bar r$. 
Fix $i\in \Int$ and $j\in\{0,1\}$. 
Clearly, $\bar p(i)^P \leqlol \bar r(i)^P$. 
As $\pi(\bar p(i)^P)\forces_Q\pi_j(\bar p(i))=\pi_j(\bar q(i))$ and $\pi_j(\bar q(i))=\pi_j(\bar r(i))$, we get $\pi(\bar p(i)^P)\forces_Q \pi_j(\bar p(i))=\pi_j(\bar r(i))$ and so as $i$, $j$ were arbitrary, $\bar p\alol \bar r$.
It remains to show that $\bar p \alol \bar q \Rightarrow \bar p \leq \bar q$.
So assume $\bar p \alol \bar q$ and fix $i\in\Int$. Firstly, $\bar p(i)^P\leq\bar q(i)^P$; moreover, (\ref{alol:bool}) implies $\pi_j(\bar p(i))\leq \pi_j(\bar q(i))$ for $j \in\{0,1\}$, and so as $i\in\Int$ was arbitrary and by (\ref{blowup:order}), we infer $\bar p\leq \bar q$.

(\ref{er}): Say $\bar p \leq \bar q \leq \bar r$ and $\bar p \alol \bar r$. Let $i \in \Int$ be arbitrary; clearly $\bar p(i)^P\leqlol\bar q(i)^P$.
Let $j \in \{0,1\}$ be arbitrary; as
\[ \pi_j(\bar p(i))\leq \pi(\bar p(i)^P)\cdot \pi_j(\bar q(i))\leq \pi(\bar p(i)^P)\cdot \pi_j(\bar r(i)) \]
and the terms on the sides of the equation are equal, we conclude $\bar p \alol \bar q$.
Condition (\ref{qc:leqlo:vert}) is trivial.

\medskip

We continue by checking the conditions of \ref{def:pss}, i.e., that we have a prestratification system on $\am$.
The conditions (\ref{up:extra}), (\ref{up}) and (\ref{s:lequp:vert}) are immediate by definition.
We prove (\ref{density}) \emph{Density}:
\begin{lem}\label{am:dom(C):dense}
Density holds; i.e
for $\lambda \in J$, $\bar p \in \am$ and $\lambda' \in [\lambda_0, \lambda)$ there is $\bar q \in \am$ such that 
$\bar q \in \dom(\aCl)$ and $\bar q \alo^{\lambda'} \bar p$.
\end{lem}
\begin{proof}
First, look through the following definition and find a set of parameters $x$ such that it is $\qcdefSeq$ in parameters from $x$.
We define conditions $p^n_i \in P$ for $n \in \nat$ and $i\in\Int$ and $q^n \in Q$ for $n\in\nat$. We do so by induction on $n$, in each step using induction on  $i$.
First, as $Q$ is stratified we can find $q^0 \in Q$ such that $q^0 \leqlo^{\lambda'} \pi(\bar p(0)^P)$ and for all $i \in \Int$ and both $j\in\{0,1\}$, $\pi_j(\bar p (i))$ is $\lambda'$-chromatic below $q^0$.

Set $p^0_i=\bar p(i)^P$, for $i\in\Int$.

Now say we have already defined a $\Int$-sequence $\bar p ^n = (p^n_i)_{i\in\Int}$ of conditions in $P$ and $q^n \in Q$. 
We will define a stronger
$\Int$-sequence $\bar p^{n+1}= (p^{n+1}_i)_{i\in\Int}$ of conditions in $P$ and a $q^{n+1} \in Q$.
We first define $\bar p^{n+1}$ on the positive integers by induction, then on the negative ones. 
At the end we find $q^{n+1} \in Q$.

So find $p^{n+1}_i \in P$ for $i\geq 0$, by induction on $i$.
Find $p^{n+1}_0 \leqlo^{\lambda'} q^n \cdot\bar p^n_0$ such that $p^{n+1}_0 \in \D(\lambda',x,q^n \cdot\bar p^n_0)$ and $p^{n+1}_0 \in \dom(\Cl)$.
Assume by induction that for all $i\in\Int$, $q^n \leqlo^{\lambda'} \pi(p^n_i)$, whence also $\pi(p^{n+1}_0) \leqlo^{\lambda'} q_n \leqlo^{\lambda'} \pi(p^n_i)$.
Continue by induction, choosing, for each $i \in \nat\setminus \{0\}$, a condition $p^{n+1}_i$ such that 
\begin{equation}
\begin{aligned}\label{domC:dense:F:spaced}
p^{n+1}_i \leqlo^{\lambda'}  &\pi(p^{n+1}_{i-1})\cdot p^n_i\\
p^{n+1}_i \in \D(\lambda', x, & \pi(p^{n+1}_{i-1})\cdot p^n_i)
\end{aligned}
\end{equation}
and $p^{n+1}_i \in \dom(\Cl)$.
By induction hypothesis, $\pi(p^{n+1}_{i-1}) \leqlo^{\lambda'} q_n \leqlo^{\lambda'} \pi(p^n_i)$, so $\pi(p^{n+1}_{i-1})\cdot p^n_i$ is a well defined condition in $P$ and $\pi(p^{n+1}_{i-1})\cdot p^n_i \leqlo^{\lambda'}p^n_i$.
Thus we have defined $p^{n+1}_i$ for $i \geq 0$. 
Before we consider the case $i<0$, observe that for any $i \in \nat\setminus\{0\}$, by (\ref{domC:dense:F:spaced}), \ref{pcs:is}(\ref{pi:mon}) and (\ref{F:coh}), we have
\begin{equation*}
\begin{aligned}
\pi(p^{n+1}_i)\leqlo^{\lambda'} &\pi(p^{n+1}_{i-1})\\
\pi(p^{n+1}_i)\in \D(\lambda',x, &\pi(p^{n+1}_{i-1})).
\end{aligned}
\end{equation*}
We also use that by construction, $\pi(p^{n+1}_{i-1}) \leq q^{n+1}\leq \pi(p^n_i)$.
Thus,  $(\pi(p^{n+1}_i))_{i\in \nat}$ is $(\lambda',x)$-strategic and we may assume by choice of $x$ and by Lemma~\ref{adequate} that it is $(\lambda',x)$-adequate.
Let $q^*$  be a greatest lower bound for $(\pi(p^{n+1}_i))_{i\in \nat}$. 

Now we define $p^{n+1}_i$, by induction on $i$ for $i < 0$:
Find $p^{n+1}_{-1} \in P$ such that
\begin{align*}
p^{n+1}_{-1} \leqlo q^* \cdot p^n_{-1}\\
p^{n+1}_{-1} \in \D(\lambda',x,q^* \cdot p^n_{-1})
\end{align*}
and such that $p^{n+1}_{-1} \in \dom(\Cl)$.
Again, continue choosing for each $i\in \nat$, $i>1$ a condition $p^{n+1}_{-i} \in \dom(\Cl)$ such that the sequence  $(\pi(p^{n+1}_{-i}))_{i\in \nat}$ is $(\lambda',x)$-adequate.
Finally, let $q^{n+1}$ be a greatest lower bound of $(\pi(p^{n+1}_{-i}))_{i\in \nat}$.
This finishes the inductive definition of $\bar p^{n}$.

For each $i \in \Int$, $(p^n_i)_{n\in\nat}$ is a $\lambda'$-adequate sequence and thus has a greatest lower bound which we call $\bar q(i)^P$.
By Lemma~\ref{lem:adeq:proj} and by choice of $x$, $\{ \pi(p^n_i)\}_{n\in\nat}$ is a $\lambda'$-adequate sequence in $Q$. 
Since $(Q,P)$ is a quasi-closed extension, by \eqref{qc:glb} $\pi(\bar q(i)^P)$ is a greatest lower bound of this sequence.
As for each $n \in \nat$, $q^{n+1}\leqlo^{\lambda'}\pi(p^n_i)\leqlo^{\lambda'} q^n$, $(q_n)_{n \in \nat}$ also has greatest lower bound $\pi(\bar q(i)^P)$, whence for all $i \in \Int$, $\pi(\bar q(i)^P) = \pi(\bar q(0)^P)$.
Set $\bar q(i)^j=\bar p(i)^j$ for $j=0,1$ and observe that $\bar q \bar \leqlo^{\lambda'} \bar p$.
Thus as $\lambda' \geq \lambda_0$,  we see $\bar q$ satisfies the hypothesis of Lemma~\ref{stay:in:am} and thus $\bar q \in \am$.
Lastly, as $\bar q(i)^P$ is a greatest lower bound of $\{ p^n_i \}_{n \in \nat}$, we conclude $\bar q(i)^P \in \dom(\Cl)$.
For each $i\in\Int$, fix $c(i) \in \Cl(\bar q(i)^P)$.

Fix $i \in \Int\setminus\{0\}$ and $j\in\{0\}$. 
At the beginning, we chose $q_0$ such that $\pi_j(\bar p (i))$ is $\lambda'$-chromatic below $q_0$.
So we may fix  a $\lambda'$-spectrum $H^j(i)$ of $\pi_j(\bar p (i))$ below $q_0$---and hence also below $\pi(\bar q(0)^P)\leq q_0$.
As $\bar q \alo^{\lambda'} \bar p$ we have $\pi_j(\bar q(i))=\pi(\bar q(i)^P) \cdot \pi_j(\bar p(i))$.
Thus, as $i\in\Int\setminus\{0\}$ and $j\in\{0,1\}$ were arbitrary, \renewcommand{\qedsymbol}{{\tiny  Lemma~\ref{am:dom(C):dense}~}$\Box$}
\[(c(i), \lambda', H^0(i),H^1(i))_{i\in\Int}  \in \aCl(\bar q). \qedhere\]
\end{proof}

This finishes the proof of  \emph{Density}. We postpone the proof of (\ref{continuous}) \emph{Continuity} (the last remaining clause of Definition~\ref{def:pss}, verifying that we have a prestratification system on $\am$).

Now we check that the prestratification system on $\am$ extends that of $P$. 
Conditions (\ref{ext}), (\ref{pi:mon}) and (\ref{D:coh}) are immediate.
For (\ref{pss:is:cdot:lo}), it suffices to check (\ref{pss:is:lo:init}) and (\ref{pss:is:lo:tail!}).

\medskip

(\ref{pss:is:lo:init}): Say $q \in P$ and $\bar p$ are such that $q \leqlol \bar \pi(\bar p)$. Let $i \in \Int\setminus\{0\}$.
By (\ref{pss:is:lo:init}) for $(Q,P)$ we have $\pi(q)\cdot \bar p(i)^P\leqlol \bar p(i)^P$.
Moreover, $\pi_j( (\bar p \cdot q) (i))=\pi(q)\cdot\pi_j( \bar p(i))$,
so as $(\bar p \cdot q)(0)=q\leqlol \bar p(0)^P$, we conclude $\bar p \cdot q \alol \bar p$.

\medskip

(\ref{pss:is:lo:tail!}):
Say $q \in P$ and $\bar p$, $\bar r \in \am$ are such that $q \leq \bar \pi(\bar r)$ and
$\bar r \alol \bar p$. Let $i \in \Int\setminus\{0\}$. By (\ref{pss:is:lo:tail!}) for $(Q,P)$ we have
$\pi(q)\cdot \bar r(i)^P\leqlol \pi(q)\cdot \bar p(i)^P$. Moreover,
\begin{multline*}
\pi_j((\bar r \cdot q)(i))=\pi(q)\cdot\pi_j( \bar r(i))
\\=\pi(q)\cdot\pi(\bar p(0)^P)\cdot\pi_j( \bar p(i))=\pi(q)\cdot\pi_j( \bar p(i))=\pi_j((\bar p \cdot q)(i)),
\end{multline*}
so as $\bar r \cdot q(0)=q=\bar p \cdot q(0)$, we conclude $\bar r \cdot q \alol \bar p \cdot q$.
Conditions (\ref{pss:is:ext}), (\ref{pss:is:pi:mon}) and (\ref{pss:is:c}) are left to the reader.
Being cautious, we check (\ref{pss:is:cdot:up}). Say $w \in P$, $\bar d$, $\bar r \in \am$ and $w \leq \bar\pi(\bar d)$ while $\bar d \alol \bar r$.
By (\ref{pss:is:cdot:up}) for $(Q,P)$, we have $w \cdot \bar d(i)^P \lequpl w \cdot \bar r(i)^P$. As $\bar d \cdot w (0)=w=\bar r \cdot w(0)$, we conclude that $w \cdot \bar d \lequpl w \cdot \bar r$.

\medskip

We now check that $(P,\am)$ is a quasi-closed extension (Definition~\ref{qc:ext}).
Start with \eqref{ext:D:dense}, that $\aD$ is coherently dense.
Observe that to find $\bar q \in \aD(\lambda, x, \bar p)$, with a given $\bar q(0)^P$, we need only make a direct extension $\bar q(i)^P$ of $\bar p(i)^P$ for every $i\in\Int\setminus\{0\}$
such that $\bar q(i)^P=\pi(\bar q(0)^P)$.
This is possible by \eqref{qc:ext} for $(Q,P)$.
We obtain $\bar q$ which satisfies all the requirements of \ref{stay:in:am}. 
Thus we can find such $\bar q \in \aD(\lambda, x, \bar p)$ without falling out of $\am$
and without changing $\bar q(0)^P$.

\medskip

(\ref{ext:glb}):  
So say $\bar\lambda<\lambda$ and $\bar{\bar p}=(\bar p)_{\xi<\rho}$ is a $(\lambda^*,\bar\lambda,x)$-adequate sequence in $\am$  with a $\bar \pi$-bound $p \in P$.
Towards finding a greatest lower bound $\bar p$, set $\bar p(0) = p$.
Fix $i \in \Int\setminus \{0\}$.
By definition of $\aD$ and $\alol$, the sequence $\{\bar p_\xi(i)^P\}_{\xi<\rho}$ is $(\lambda^*,\bar\lambda,x)$-adequate in $P$.  
Since $\{\pi(\bar p_\xi(0)^P)\}_{\xi<\rho}$ is the same as $\{\pi(\bar p_\xi(i)^P)\}_{\xi<\rho}$, the condition $\pi(\bar p(0))=\pi(p(i)^P) \in Q$ is a $\pi$-bound of $\{\pi(\bar p_\xi(i)^P)\}_{\xi<\rho}$.
Thus by (\ref{ext:glb}) for $(Q,P)$, the sequence $\{\bar p(i)^P\}_{\xi<\rho}$ has a greatest lower bound $p_i  \in P$ such that for all $\xi <\rho$,
$p_i \leqlo^{\lambda^*} \bar p_\xi(i)^P$ and $\pi(p_i)=\pi(\bar p(0)^P)$. 
Moreover, if $\lambda^*<\bar \lambda$, we have $p_i \leqlo^{\bar\lambda} \pi(p_i)$.
For each $i \in \Int$, let $\bar p(i)^P=p_i$ and for $j\in\{0,1\}$ let $\bar p(i)^j=\bar p_0(i)^j$. 
By Corollary~\ref{stay:in:am:cor}, $\bar p \in \am$ and $\bar p \alo^{\lambda^*} \bar p_0$.
We must check that for all $\xi <\rho$, $\bar p \alo^{\lambda^*} \bar p_\xi$.
This is clear as  for every $i\in \Int$ we have $\bar p(i)^P \leqlo^{\lambda^*} \bar p_\xi(i)^P$ by construction, and for every $i \in\Int\setminus\{0\}$ and $j\in\{0,1\}$ we have 
\[ \pi_j(\bar p(i))=\pi(\bar p(i)^P)\cdot \pi_j(\bar p_0(i)) = \pi(\bar p(i)^P)\cdot \pi_j(\bar p_\xi(i)), \]
where the first equation holds since $\bar p \alo^{\lambda^*} \bar p_0$ second equation holds since
\[ \pi(\bar p(i)^P)=\pi(p)\leq \pi(\bar p_\xi(i)^P) \]
and $\bar p_\xi \alo^{\lambda^*} \bar p_0$ gives us
\[ \pi_j(\bar p_\xi(i))=\pi(\bar p_\xi(i)^P)\cdot \pi_j(\bar p_0(i)).\]

\medskip

This is the natural point to check \emph{Continuity}, the the last remaining clause of Definition~\ref{def:pss}, verifying that we have a prestratification system on $\am$:

\medskip

(\ref{continuous})  \emph{Continuity}:
Say $\bar{\bar p}=(\bar p)_{\xi<\rho}$ and $\bar{\bar q}=(\bar q)_{\xi<\rho}$ are both $(\lambda^*,\bar\lambda,x)$-adequate for $\bar\lambda<\lambda$,
such that
\begin{equation}\label{sequence:colors}
\forall \xi<\rho \quad \aCl(\bar p_\xi)\cap \aCl(\bar q_\xi) \neq \emptyset.
\end{equation}
Say the sequence $\bar{\bar p}=(\bar p)_{\xi<\rho}$ has a greatest lower bound $\bar p$, the sequence $\bar{\bar q}=(\bar q)_{\xi<\rho}$ has a greatest lower bound $\bar q$.
We show 
\begin{equation}\label{color}
\aCl(\bar p)\cap \aCl(\bar q) \neq \emptyset.
\end{equation}
First, observe that in each $P$-component we obtain a common color for $\bar p$ and $\bar q$:
For each $i\in\Int$, as in the previous proof, $\{ \bar p_\xi(i)^P \}_{\xi<\rho}$ and  $\{ \bar q_\xi(i)^P \}_{\xi<\rho}$ are $(\lambda^*,\bar\lambda,x)$-adequate and so by (\ref{continuous}) for $P$ we can find $c(i) \in \Cl(\bar p(i)^P)\cap\Cl(\bar q(i)^P)$.
To obtain the colours fo the boolean values, we can use the spectra of the first condition on either sequence, say $\bar q_0$:
fix $(c_0(i),\lambda_0'(i),H^0_0(i),H^1_0(i))_{i\in\Int} \in \aCl(\bar p_0)\cap \aCl(\bar q_0)$.
We shall now check that 
\begin{equation}\label{color:am}
 (c(i),\lambda_0'(i),H^0_0(i),H^1_0(i))_{i\in\Int} \in \aCl(\bar p)\cap \aCl(\bar q).
 \end{equation}
This is clear by definition: fix $i \in \Int\setminus\{0\}$ and $j \in \{0,1\}$. Firstly, $\bar p \alol \bar p_0$ and so 
\[ \pi_j(\bar p(i)) = \pi(\bar p(i)^P)\cdot \pi_j(\bar p_0(i)).\]
Moreover, by choice of $\lambda_0'(i)$ and $H^j_0(i)$ there is $b^j \in B_j$ such that we have
\[ \pi_j(\bar p_0(i)) = \pi(\bar p_0(i)^P)\cdot b^j. \]
and $H^j_0$ is a $\lambda'_0(i)$-spectrum for $b^j$ below $\pi(\bar p(i)^P)$. 
The last two equations together yield
\[ \pi_j(\bar p(i))=\pi(\bar p(i)^P)\cdot b^j,\]
and so (\ref{color:am}) holds. This finishes the proof of (\ref{continuous}).
\medskip

We now verify the conditions of Definition~\ref{def:s:ext}, showing that $(P,\am)$ is a stratified extension on $J$.

\medskip

(\ref{s:ext:exp}) \emph{Coherent Expansion}: Assume $\bar q \aupl \bar p$ and $\bar p \alol \bar a(\bar p(0))$. Moreover, assume $\bar q(0) \leq \bar p(0)$. We show $\bar q \leq \bar p$. Let $i\in\Int\setminus\{0\}$ be arbitrary. As $\bar q(i)^P \lequpl \bar p(i)^P$ and
$\bar p(i)^P \leqlol a(\bar p(0)^P)(i)^P= \pi(p(i)^P)$, and as $\pi(\bar q(0)^P)\leq \pi(\bar p(0))$, we have $\bar p(i)^P\leq \bar q(i)^P$ by (\ref{s:ext:exp}) for $(Q,P)$.
Say $i\neq 0$ and $j \in \{0,1\}$. Then
\[ \pi_j(\bar p(i)) \leq \pi(\bar p(i)) = \pi(\bar p(0)^P) \leq \pi(\bar q(0)^P)=\pi(\bar q(i))=\pi_j(\bar q(i)), \]
where the last equality holds as $\bar q(i)^P \leqlol \pi(\bar q(0)^P) \in Q \subseteq \Dam$.
By (\ref{blowup:order}),
$\bar q \leq \bar p$, and we are done.
 
\medskip

(\ref{s:ext:int}) \emph{Coherent Interpolation}:
Let $\bar d, \bar r \in \am$ be such that $\bar d \alol \bar r$, and say $p \in P$ interpolates $\bar\pi(\bar d)$ and $\bar \pi(\bar r)$. 
We find $\bar p \in \am$ such that $\bar p \alol \bar r$ and $\bar p \aupl \bar d$ and moreover $\bar \pi(\bar p)=p$.
For $i \in \Int\setminus\{0\}$, use \emph{Coherent Interpolation} for $(Q,P)$ to find $p_i \in P$
such that $p_i \leqlol \bar r(i)^P$ and $p_i \lequpl \bar d(i)^P$ and moreover $\pi(p_i)=\pi(p)$.
Now we define a sequence $\bar p \colon \Int \rightarrow P\times B_0 \times B_1$.
Set $\bar p(0) = (p,1,1)$ and set $\bar p(i)=(p_i, \bar r(i)^0, \bar r(i)^1)$ for $i \in\Int\setminus\{0\}$.
Clearly, $\bar p \alol \bar r$ and so $\bar p \in \am$. By construction, $\bar p \aupl \bar d$ and $\bar \pi(\bar p)=\bar p(0)^P=p$.

\medskip

It remains to demonstrate (\ref{s:ext:cent}) \emph{Coherent Linking}:
\begin{lem}\label{l.a.coherent.linking}
\emph{Coherent Linking} holds: Say $\lambda \in J$, $\bar p \aupl \bar d$ and $\bar \Cl (\bar p) \cap \bar \Cl(\bar d)\neq\emptyset$.
Say further we have $w_0\in P$ such that $w_0 \leqlo^{<\lambda}\bar p(0)^P$ and  $w_0 \leqlo^{<\lambda}\bar d(0)^P$. 
Then there is $\bar w \in \am$ such that $\bar \pi(\bar w)=w_0$ and both $\bar w \alo^{<\lambda} \bar p$ and $\bar w  \alo^{<\lambda} \bar d$.
\end{lem}
\begin{proof}
Fix $\bar p, \bar d$ and $w_0$ as in the hypothesis. Fix $i \in \Int\setminus\{0\}$ for the moment. 
Observe we have 
Since $\Cl(\bar p(i)^P)\cap \Cl(\bar d(i)^P)\neq\emptyset$, by \emph{Coherent Linking} for $(Q,P)$ we can find $\bar w_i \in P$ such that $\pi(w_i)=\pi(w_0)$. If the additional assumption at the end of the lemma holds, we may assume $w_i \leqlo^{<\lambda} \bar p(i)^P$ and $w_i \leqlo^{<\lambda} \bar d(i)^P$. 
For $i \in \Int$, set 
\[ \bar w(i) = (w_i, \bar p(i)^0,\bar p(i)^1). \] 
Since $\bar w \alo^{\lambda_0} \bar p$ and $\pi(\bar w(i)^P)=\pi(w_0)$ for each $i\in\Int$,  by Lemma~\ref{stay:in:am}, $\bar w \in \am$. 

Now say the additional assumption holds. 
By construction, $\bar w(i)^P \leqlo^{<\lambda} \bar p$ for each $\lambda'\in[\lambda_0,\lambda)$.
Fix $i \in \Int$. 
Since $\aCl(\bar p)\cap \aCl(\bar d)\neq\emptyset$, $\bar p(i)^j$ and $\bar d(i)^j$ have a common $\lambda$-spectrum below $\pi(w_0)$,  and so
\begin{equation}\label{forces:same:bv}
\pi(w_0)\forces_Q \bar p(i)^j=\bar d(i)^j.
\end{equation}
Thus for each $i \in \Int$, 
\[\bar w(i) = w(i)^P \cdot d(i)^0 \cdot d(i)^1 \leq \bar d(i)\]
whence $\bar w \leq \bar d$. 
In fact, as $\bar w(i)^P \leqlo^{\lambda'} \bar d(i)^P$ and (\ref{forces:same:bv}) holds, $\bar w \alo^{\lambda'} \bar d$  for each $\lambda'\in[\lambda_0,\lambda)$.
\renewcommand{\qedsymbol}{{\tiny  Lemma~\ref{l.a.coherent.linking}~}$\Box$}
\end{proof}
\renewcommand{\qedsymbol}{{\tiny  Theorem~\ref{thm:am:s:ext}~}$\Box$}
\end{proof}

\section{Stratified Type-2 Amalgamation}\label{simpleram}

We now consider the simpler case when we want to extend an automorphism already defined on an initial segment of the iteration.
Let $P$ be a forcing, $Q$ a complete suborder, $f$ an automorphism of $Q$ and $\pi\colon P \rightarrow Q$ a strong projection. 
Assume $\lambda_0$ is regular and $(Q,P)$ is a stratified extension on $I= [\lambda_0, \kappa^*)$. 
We denote the stratification on $Q$ by $\leqlol_Q$, $\lequpl_Q, \hdots$ and write $\leqlol$, $\lequpl, \hdots$ for the stratification of $P$.

Further we assume that for each regular $\lambda\in I$ and $q,r \in Q$, 
\begin{enumerate}
\item $q\leqlol_Q r \iff f(q)\leqlol_Q f(r)$;
\item $q \lequpl_Q r \iff f(q)\lequpl_Q f(r)$;
\item $p\in \D_Q (\lambda,x,q))\iff f(p) \in \D(\lambda,x, f(q))$;
\item $\Cl_Q(q)\cap \Cl_Q(r) \neq \emptyset \iff \Cl_Q(f(q))\cap \Cl_Q(f(r)) \neq \emptyset$.
\end{enumerate}
We define the type-2 amalgamation \index{amalgamation!type-2}\index{type-2 amalgamation}\index[notation]{Am1(Q,P,f,lambda 0)@$\simpleram(Q,P,f,\lambda_0)$}$\simpleram(Q,P,f,\lambda_0)$ (or just $\simpleram$ where the context allows) as the set of all $\bar p \colon \Int \rightarrow P$ such that for all but finitely many $i\in\Int$ we have $\bar p(i) \leqlo^{\lambda_0} \pi(\bar p(i))$ and for all $i \in \Int$,
\begin{equation}\label{s:thinout}
f(\pi(\bar p(i)))=\pi(\bar p(i+1)).
\end{equation}
The ordering is $\bar p \leq \bar q$ if and only if for each $i\in\Int$, $\bar p(i) \leq \bar q(i)$.
To make notation more uniform, we also write $\bar p(i)^P$ to mean just $\bar p(i)$, for $i\in\Int$, in the case of type-2 amalgamation.

\medskip

Define \index[notation]{pi bar@$\bar \pi$}$\bar \pi\colon \simpleram\rightarrow P$ by $\bar \pi(\bar p)=\bar p(0)$. The map $\bar e\colon P\rightarrow\simpleram$ is defined by
$\bar e(p)(0)= p$ and $\bar e(p)(i)=\pi(p)$ for all $i \in\Int$, $i\neq 0$.

It is straightforward to check that $\bar e$ is a complete embedding and $\bar \pi$ is the restriction of the canonical projection
from $\ro(\simpleram)$ to $\ro(\bar e[P])$. 
Moreover, if $q \in P$, $\bar p \in \simpleram$ and $q \leq \bar \pi(\bar p)$, then $q\cdot\bar p \in \simpleram$.

\medskip

We now define the stratification of $\simpleram$, consisting of $\aCl$, $\aD$, $\alol$, $\aupl$ for each regular $\lambda \in I$.
We say \index[notation]{lessthanl bar lambda@$\alol$}$\bar q \alol \bar p$ exactly if for every $i \in \Int$, $\bar q(i) \leqlol \bar p(i)$, and
\index[notation]{lessthanu bar lambda@$\aupl$}$\bar q \aupl \bar p$ exactly if for every $i \in \Int$, $\bar q(i) \lequpl \bar p(i)$.
Similarly for \index[notation]{D bar(lambda,x,p)@$\aD(\lambda,x,\bar p)$}$\aD(\lambda,x,\bar p)$. 
For $\bar p$ such that for each $i \in \Int$, $\bar p(i) \in \dom(\Cl)$, we define \index[notation]{C lambdaz@$\aCl$}$\aCl(\bar p)$ to be the set of all $c \colon \Int \rightarrow \lambda$ such that for each $i\in\Int$, $c(i) \in \Cl(\bar p(i))$. 

\begin{thm}\label{l:simpleram:strat}
$(P,\simpleram)$ is a stratified extension on $I$.%
\end{thm}
\begin{proof}
The proof is a slight modification of the argument for type-1 amalgamation.
Therefore, we only touch the main points, and leave the rest to the reader.
\begin{lem}\label{l:simpleram:qc}
$(P,\simpleram)$ is a quasi-closed extension on $I$.
\end{lem} 
\begin{proof}
Let $\lambda \in I$.
Let $(\bar p_\xi)_{\xi<\theta}$ be $\lambda$-adequate. Fix $i \in \Int$ and let $\bar p(i)$ be the greatest lower bound of the $\lambda$-adequate sequence $(\bar p_\xi(i))_{\xi<\theta}$. By coherency, $(\pi(\bar p_\xi(i)))_{\xi<\theta}$ is also adequate and its greatest
lower bound is $\pi(\bar p(i))$. 
As $f$ is an automorphism, for each $i\in\Int$, $f(\pi(\bar p(i)))$ is a greatest lower bound of $(q_\xi(i))_{\xi<\theta}$, where $\bar q_\xi(i) = f(\pi(\bar p_\xi(i)))$. As the latter is equal to $\pi(\bar p_\xi(i-1))$, we obtain (\ref{s:thinout}) for $\bar p$. 
So $\bar p \in\simpleram$; it is straightforward to check it is a greatest lower bound of $(\bar p_\xi)_{\xi <\theta}$.
\renewcommand{\qedsymbol}{{\tiny  Lemma~\ref{l:simpleram:qc}~}$\Box$}
\end{proof}

\begin{lem}\label{l:simpleram:interp}
\emph{Coherent Interpolation} holds, i.e whenever $\bar r,\bar d \in \simpleram$, $\bar d \leq \bar r$ and $p_0 \in P$ such that $p_0 \leqlol \bar \pi(\bar r)$ and  $p_0 \lequpl \bar \pi(\bar d)$,
there is $\bar p \in \simpleram$ such that $\bar p \alol \bar r$, $\bar p \aupl \bar d$ and $\bar \pi(\bar p)=p_0$.
\end{lem}
\begin{proof}
Let $\lambda \in I$.
Given $\bar r, \bar d$ and $p_0$ as above, first set $\bar p(0)=p_0$. 
As $\pi(\bar r(i))=\pi(\bar r(0))$ and $\pi(\bar d(i))=\pi(\bar d(0))$, $p_0 \lequpl \pi(\bar d(i))$ and $p_0 \leqlol \pi(\bar r(i))$, for all $i\in\Int$. \emph{Coherent Interpolation} for $(Q,P,\pi)$ allows us to find, for each $i\in\Int$, $i\neq 0$ a condition $\bar p(i) \in P$ such that $\pi(\bar p (i))=\pi(p_0)$, $p_0 \lequpl \bar d(i)$ and $p_0 \leqlol\bar r(i)$. As for each $i\in\Int$, $\pi(\bar p(i))=\pi(p_0)$, $\bar p \in \simpleram$.
\renewcommand{\qedsymbol}{{\tiny  Lemma~\ref{l:simpleram:interp}~}$\Box$}
\end{proof}

The proof of the next lemma is obvious and we leave it to the reader.
\begin{lem}
\emph{Coherent Linking} holds. That is: Say $\bar p \aupl \bar d$ and either of the following holds: $\bar \Cl (\bar p) \cap \bar \Cl(\bar d)\neq\emptyset$ or for some $q \in Q$, $\bar p \alol q$ or $\bar d \alol q$.
Say further $w_0\in\Dam$ such that for each regular $\lambda' \in [\lambda_0, \lambda)$, $w_0 \leqlo^{\lambda'} \bar \pi (\bar p)$ and  $w_0  \leqlo^{\lambda'} \bar \pi (\bar d)$.
Then there is $\bar w \in \simpleram$ such that for each regular $\lambda' \in [\lambda_0, \lambda)$, $\bar w \alo^{\lambda'} \bar p$, $\bar w  \alo^{\lambda'} \bar d$ and $\bar \pi(\bar w)=w_0$.
\end{lem}

\begin{lem}\label{l:simpleram:density}
If $\bar p \in \simpleram$ and $\lambda' , \lambda \in I$ and $\lambda' < \lambda$, there is $\bar q \in \simpleram$ such that 
$\bar q \in \dom(\aCl)$ and $\bar q \alo^{\lambda'} \bar p$.
\end{lem}
\begin{proof}
We define a sequence $(q_i)_{i\in\Int}$ of conditions in $P$, by induction on $i$. As usual, read through the following definition and pick $x$ such that it is $\qcdefSeq$ with parameters in $x$.
First, find $q_0 \leqlo^{\lambda'} \bar p(0)^P$ such that $q_0 \in \dom(\Cl)$.
Continue by induction, choosing, for each $n \in \nat\setminus \{0\}$, a condition $q_n \leqlo^{\lambda'} \bar p(n)\cdot \pi(q_{n-1})$ such that 
$\pi(q_n) \in \D_Q(\lambda,x,\pi(q_{n-1}))$
$q_n \in \dom(\Cl)$.
Let $q^*_1$  be a greatest lower bound for $(\pi(q_k))_{k\in \nat}$; it exists by quasi-closure for $Q$.
Find $q_1 \leqlo^{\lambda'} q^*_1 \cdot\bar p(1)$ such that $q_1 \in \dom(\Cl)$.
Again, continue by induction, choosing for each $n\in \nat\setminus \{0,1\}$, a condition  $q_{-n} \in \dom(\Cl)$ ensuring that $(\pi(q_{-k}))_{k\in \nat}$ form an adequate sequence. 
Finally, let $q$ be a greatest lower bound of $(\pi(q_{-k}))_{k\in \nat}$.
For each $i \in \Int$, $q \leqlo^{\lambda'} \pi(q_i)$, so $q\cdot q_i \leqlo^{\lambda'} q_i\leqlo^{\lambda'}\bar p(i)$.
Observe that by coherent stratification, $q\cdot q_i \in \dom(\Cl)$ for each $i\in\Int$. 
Setting $\bar q(i) = q\cdot q_i$, we have $\pi(\bar q(i))=q$, for all $i\in\Int$.
Thus $\bar q \in \simpleram$, $\bar q \alo^{\lambda'} \bar p$ and $\bar q \in \dom(\aCl)$.
\renewcommand{\qedsymbol}{{\tiny  Lemma~\ref{l:simpleram:density}~}$\Box$}
\end{proof}
With all remaining details left to the reader, this concludes the proof that $(P,\simpleram)$ is a stratified extension on $I$.
\renewcommand{\qedsymbol}{{\tiny  Theorem~\ref{l:simpleram:strat}~}$\Box$}
\end{proof}

\section{Remoteness and Stable Meets}\label{am:remote:stm}

The following lemma helps to ensure ``coding areas'' don't get mixed up by the automorphisms, as we shall see in Lemmas~\ref{index:sequ} and \ref{coding:survives}. 
Also see the discussion at the beginning of Section~\ref{sec:remote}.
\begin{lem}\label{remote:lemma}\index{remote}
Say $C$ is remote in $P$ over $Q$ (up to some height $\kappa'$, where  $\kappa'\leq\kappa^*$). Then $\Phi^k[C]$ is remote in $\am$ over $P$ (up to the same height) for any $k \in \Int\setminus\{0\}$.
\end{lem}
\begin{proof}
Let $D^* = \bar\pi^{-1}[\Dam] \subseteq \am\cap \Dam^\Int_f$ as in Lemma~\ref{common:dense:set}.
Let $j \in \{0,1\}$ arbitrary. If $\hat p \in \blowup{\Dam}$, $c \in C$ and $c \leq \pi_C(\hat p^P)$, by the definition of $\Dam$, 
\[ \pi_Q(\hat p\cdot c)\forces\pi_j(\hat p\cdot c)=\pi_j(\hat p), \]
that is, $\pi_j(\hat p\cdot c)=\pi_j(\hat p) \cdot \pi_Q(\hat p^P\cdot c)$, so as $C$ is independent over $Q$ and thus $\pi_Q(\hat p^P\cdot c) = \pi_Q(\hat p^P)$, we have $\pi_j(\hat p\cdot c)=\pi_j(\hat p)$.
In fact, if we have $\hat p^P \in \Dam$, we have $c\cdot \hat p \in \blowup{\Dam}$. 
Observe further that for any $c \in C$, $\pi_j(c)=1$, and moreover, $C\subseteq \Dam$. Thence, $C \subseteq D^* \subseteq \dom(\Phi^k)$.
Moreover, $\Phi^k(c) (0)= e_k(c) (0)=(1,1,1)$ and so $\Phi^k(c) \in D^* \subseteq \am$.

\medskip

We now show $\Phi^k[C]=e_k[C]$ is independent in $\am$ over $P$:
Let $c\in C$, $\bar p \in \am$ and say $c \leq (\bar \pi_k\circ \pi_C)(\bar p)=\pi_C(\bar p(k))$. 

Since $\pi_j (c \cdot \bar p(k))=\pi_j(\bar p(k))$,  for every $i \in \Int$, 
\begin{equation}\label{remote:equality:on:sides}
i \neq k \Rightarrow e_k(c) \cdot \bar p (i) = \bar p(i).
\end{equation}
Thus $e_k(c) \cdot \bar p \in \am$. This firstly shows that $\bar \pi_k\circ \pi_C$ is a strong projection from $\am$ to $C$.
Moreover $\bar \pi(\bar p \cdot e_k(c)) = \bar p(0) =\bar \pi (\bar p)$, and we are done with the proof that $\Phi^k[C]=e_k[C]$ is independent in $\am$ over $P$.

It follows that $\Phi^k[C]$ is remote in $\am$ over $P$, by definition of $\alol$: let $\lambda \in [\lambda_0, \kappa')$. Say $c \leq \bar \pi_k (\bar p)$.
Then $e_k(c) \cdot \bar p (k) = (\bar p(k)^P \cdot c, \bar p(k)^0,\bar p(k)^1)$ and $\bar p(k)^P \cdot c \leqlol p(k)^P$, by Clause \ref{remote:leqlo} of Definition~\ref{remote}. So by (\ref{remote:equality:on:sides}), $e_k(c) \cdot \bar p \alol \bar p$, and we are done.
\end{proof}

The last lemma of this section is the counterpart of Lemmas~\ref{stable:meet:products} and \ref{stable:meet:comp}. Together these lemmas make sure that in the iteration used in our application, we have stable meet operators for every initial segment.
We assume $P$ is stratified on $J$.
\begin{lem}\index{stable meet}\index[notation]{Q-stable meet@$Q$-stable meet}\index[notation]{wedge (stable meet)@$\stm$ (stable meet)}
There is a $P$-stable meet operator $\astm$ on $\am$. 
\end{lem}
\begin{proof}
Of course we set 
\[ \dom(\astm)=\{ (\bar p,\bar r) \setdef \exists \lambda\in J \quad\bar r \alol \bar \pi(\bar r) \cdot \bar p \}. \]
Say we have $\bar p$, $\bar r \in \am$ such that $(\bar p, \bar r) \in \dom(\astm)$.
This means we can fix a regular $\lambda \in J$ such that for each $i \in\Int\setminus\{0\}$, $\bar r(i)^P \leqlol \pi(\bar r(i)^P)\cdot\bar p(i)^P$.
Let $w_i =  \bar p(i)^P \stm \bar r(i)^P$ for $i \in \Int\setminus\{0\}$ and
set 
\[ 
    \bar p \astm \bar r = \begin{cases} (w_i, \bar p(i)^0, \bar p(i)^1) &\text{ for $i \in \Int\setminus\{0\}$}\\
                                       \bar p(0)                       &\text{ for $i=0$.}
                                       \end{cases}
\] 
Let $\bar p \astm \bar r$ be denoted by $\bar w$.
By LemmaLemma~\ref{Q:cdot:D}, for $i\in\Int\setminus\{0\}$ and $j\in\{0,1\}$ we have 
\begin{equation}\label{am:stm:proj:equal}
\begin{split}
\pi_j(\bar w(i))&=\bar p(i)^j\cdot \pi_j(w_i \cdot \bar p(i)^{1-j})\\
                 &=\bar p(i)^j\cdot \pi_j(\bar p(i)^P \cdot \bar p(i)^{1-j})=\pi_j(\bar p(i)). \\
\end{split}
\end{equation}
In particular, as $w_i \leqlol \bar p(i)^P$ and $i$ was arbitrary, %
we have 
\begin{equation}\label{astm:leqlol}
\bar p \astm \bar r \alol \bar p.
\end{equation} 
Moreover, $\pi(w_i)= \pi(\bar p(i)^P)=\pi(\bar p(0))=\pi(\bar w(0))$. 
So $\bar w$ satisfies the hypothesis of Lemmas~\ref{stay:in:am} and therefore $\bar w \in \am$. 
Clearly, $\bar \pi(\bar p \astm \bar r)= \bar p(0)$. 
It remains to see that $\bar \pi(\bar r) \cdot (\bar p \astm \bar r) \approx \bar r$;
we have
\[ \bar \pi(\bar r)\cdot \bar w = \begin{cases} (\pi(\bar r(0)^P)\cdot w_i, \bar p(i)^0, \bar p(i)^1) &\text{ for $i \in \Int\setminus\{0\}$}\\
                                       \bar r(0)                       &\text{ for $i=0$.}
                                       \end{cases}
 \]
Write $\bar u = \bar \pi(\bar r)\cdot \bar w$ and write $\bar v = \bar\pi(\bar r)\cdot \bar p$. For arbitrary $i \in \Int \setminus\{0\}$ and $j \in \{0,1\}$ we have
\begin{align*}
\pi_j(\bar u(i))&=\pi(\bar r(0)^P)\cdot \pi_j(\bar w(i))= \pi(\bar r(0)^P)\cdot \pi_j(\bar p(i)) &&\text{by (\ref{am:stm:proj:equal}),}\\
&=\pi(\bar r(0)^P)\cdot \pi_j(\bar v(i)) = \pi_j(\bar r(i)) &&\text{as $\bar r \leqlol \bar v$.}
\end{align*}
Thus by (\ref{blowup:order}), $\bar u(i)\approx \bar r(i)$. 
As $i \in \Int \setminus\{0\}$ was arbitrary and as $\bar u(0)= \bar r(0)$,
we conclude $\bar u \approx \bar r$, finishing the proof that $\astm$ is a $P$-stable meet on $\am$.
\end{proof}

\chapter{Proof of the Main Theorem}\label{sec:main}

In this chapter we prove Theorem~\ref{p.t.main} (resp.\ Theorem~\ref{p.t.main2}) building on the work of the previous chapters.
We begin with the assumption $V=L$ and fix $\kappa$ which we assume to be the only Mahlo.
The first step is to force with with the product of our canonical sequence\index{canonical!sequence of trees}\index{Suslin trees!canonical sequence of} (see Definition~\ref{p.n.d.canonical}) of independent $\kappa^{++}$-Suslin trees $\bar T =\prod_{\xi<\kappa} T(\xi)$\index[notation]{T bar, Pi(T(xi)) xi<kappa@$\bar T$, $\prod_{\xi<\kappa} T(\xi)$}\index[notation]{Pixi(T(xi)) xi<kappa, T bar@$\prod_{\xi<\kappa} T(\xi)$, $\bar T$}. %
Recall that $\bar T$ is $\kappa^+$-closed and that by independent we mean that $\bar T$ has the $\kappa^{++}$-cc.

Recall also from Section~\ref{p.s.trees} that for a dense set of $\bar t \in \bar T$, all components $\bar t(\nu)$, for $\nu < \kappa$, are at the same height---but in Section~\ref{sec:preserving:coding} it will be important that different heights are allowed, so we don't restrict ourselves to said dense subset.

Forcing with $\bar T$ adds a sequence of branches $\bar B = (B(\xi))_{\xi<\kappa}$\index[notation]{B bar, (B(xi)) xi<kappa@$\bar B$, $(B(\xi))_{\xi<\kappa}$}, %
where $B(\xi)$ denotes the branch through $T(\xi)$.
As a notational convenience, we assume the sequence of trees (resp. branches) is indexed by elements of 
$$J={}^{<\kappa}2 \times \omega\times \omega \times 2\index[notation]{J@$J$}$$ %
rather than by ordinals in $\kappa$, that is as $B(s,n,i,j)$ and $T(s,n,i,j)$ for $s \in{}^{<\kappa}2$,  $n,i \in \nat$ and $j\in\{0,1\}$.
Note that $\bar T$  $\kappa^{++}$-distributive (as it satisfies the $\kappa^{++}$-cc) whence $({}^{\kappa^+} \On)^{L[\bar B]} \subseteq L$, $\Card^{L[\bar B]} = \Card^L$ and the $\GCH$ still holds in $L[\bar B]$. 

\medskip

We now define $(P_\xi)_{\xi\leq\kappa}$, by induction on $\xi$.
We start with $P_0 = \bar T$ %
and identify $\bar B$ with $G_0$. %
We shall call $G_\xi$ the $P_\xi$ generic over $L$ %
for each $\xi \leq \kappa$, and we also write $G$ for $G_\kappa$. %

We construct this iteration to deal with the following tasks:
\begin{description}
\item[Task 1]\label{task:cohens} Add a set of reals $\Gamma^0$ such that $P_\kappa$ forces that the Baire-property fails for $\Gamma^0$
(see Definition~\ref{d.gamma}). 
\item[Task 2]\label{task:code} For each real $r$ added by $P_\kappa$, make sure that $P_\kappa$ forces $r \in \Gamma^0 \iff \Psi(r,0) \iff \neg \Psi(r,1)$,
where $\Psi(x,y)$ is a $\Sigma^1_3$ formula. 
The formula $\Psi(x,y)$ will be given explicitly at the beginning of Section~\ref{sec:preserving:coding}.
\item[Task 3]\label{task:amalgamation} Make sure every projective set of reals is Lebesgue-measurable in the extension by $P_\kappa$. This is achieved using the two types of amalgamation.
\item[Task 4]\label{task:collapse} To make the construction more uniform, we force with a Levy-collapse at certain stages.
\end{description}
We force with the Levy-collapse for the following reasons: firstly, when we amalgamate, and at limits, whether we collapse the continuum depends on factors beyond our control. So we always make sure we collapse the continuum at stages following a limit or an amalgamation stage, making the construction more uniform and the cardinal structure in intermediate extensions more independent of happenstance. 
Secondly, we want to make sure $\CH$ holds before we force with Jensen coding (for Task 2). 

Task 2 requires the sophisticated technique of Jensen coding with localization which was discussed in Section~\ref{sec:overview:coding} and Chapter~\ref{sec:coding}.
We will make the real $r$ (along with information about its membership in $\Gamma$) definable by coding a subset of our set of branches $\bar B$ by a real $u$, where $u$ is generic for Jensen coding.
Say we have iterated for $\xi$ steps and are in $L[G_\xi]$.
For now, let's call the set of branches we ``code'' at the $\xi+1$-th step $\bar B^- =\{ B(\nu) \setdef \nu \in I \}$, where $I\in[\kappa]^\kappa$.

We have explained in Section~\ref{p.s.making.def} why we use generic ``coding areas'' of size $\kappa$. 
The point here is that the automorphisms that arise from amalgamation (Task 3) threaten to make any simpler coding ``unreadable''. 
See Section~\ref{sec:preserving:coding} for the proof that the coding as we now define it is not made ``unreadable'' by automorphisms.

We shall pick a set $A_\xi \subseteq \kappa^+$ such that $(\Hhier_\alpha)^{L[A_\xi]}=L_\alpha[A_\xi]$ for every cardinal $\alpha \leq \kappa$ and $\{ B(\nu) \setdef \nu \in I\}$ is definable in some simple recursive fashion from $A_\xi$ (see below). 
Note $L[A_\xi]$ will be a proper submodel of $L[G_\xi]$ (since we don't want to code \emph{all of} $\bar B$). 
We shall then force over $L[G_\xi]$ with $P(A_\xi)$ of Chapter~\ref{sec:coding} \emph{as defined in the model $L[A_\xi]$} to obtain a generic real $u$ such that $A_\xi \in L[u]$ and moreover the following is true:
\begin{multline}\label{david}
\text{for all $\alpha, \beta<\kappa$ \emph{ if } $L_\beta[u]$ is a model of $\ZF^-$ and of
``$\alpha$ is the least }\\
\text{Mahlo and $\alpha^{++}$ exists'' \emph{ then:}}\\
\text{$I\cap\alpha \in L_\beta[u]$ and $L_\beta[u] \models$``$\forall \nu \in I \cap \alpha$ $T^\alpha(\nu)$ has a branch,''}
\end{multline}
where $ T^\alpha(\nu)$ denotes the outcome of the construction of $T(\nu)$ carried out in $L_\beta$ (cf.\ Definition~\ref{p.n.d.canonical} and Lemma~\ref{p.l.canonical}).
Note that as we are using a forcing from $L[A_\xi]$, a proper submodel of $L[G_\xi]$, we shall have to argue that this step represents a stratified extension.

\medskip

There are $4$ types of forcing involved, so we fix a simple and convenient partition $E^0, \dots, E^3$ of $\kappa$: let $E^n$\index[notation]{E n@$E^n$}, for $0\leq n\leq 3$, denote the set of ordinals $\xi<\kappa$ such that for some limit ordinal $\eta$ and $k \in \omega$, $\xi=\eta+k$ and $k \equiv n \pmod 4$. For an ordinal $\xi<\kappa$, let $E^n(\xi)$ denote the $\xi$-th element of $E^n$. 
Also fix, for each $\rho<\kappa$, an increasing sequence $\bar \alpha_\rho = (\alpha^\zeta_\rho)_{\zeta<\kappa}$\index[notation]{alpha bar rho, (alpha zeta rho) zeta < kappa@$\bar \alpha_\rho$, $(\alpha^\zeta_\rho)_{\zeta<\kappa}$} of ordinals $>\rho$ cofinal in $\kappa$: we let $\alpha^\zeta_\rho = \langle \zeta,\rho\rangle$, where $\langle .\,, .\rangle$ is the G\"{o}del pairing function.%

As we have to tackle certain tasks for every real of the extension, our definition will make use of two bookkeeping devices, 
$\bar s=(\dot  s_\xi)_{\xi<\kappa}$\index[notation]{s bar, (s xi)@$\bar s$, $(\dot  s_\xi)_{\xi<\kappa}$  (bookkeeping)}\index[notation]{ s xi@$(\dot  s_\xi)_{\xi<\kappa}$  (bookkeeping)} and $\bar r=(\bar\iota(\xi),\dot r^0_\xi, \dot r^1_\xi)_{\xi<\kappa}$.\index[notation]{r bar, (iota(xi), r dot 0 xi, r dot 1 xi)@$\bar r$, $(\bar\iota(\xi),\dot r^0_\xi, \dot r^1_\xi)_{\xi<\kappa}$ {\tiny (bookkeeping)}}
\index[notation]{iota bar, (iota(xi)) xi<kappa@$\bar \iota$, $(\bar\iota(\xi))_{\xi<\kappa}$ (bookkeeping)}
\index[notation]{ iota(xi), r dot 0 xi, r dot 1 xi@$(\bar\iota(\xi),\dot r^0_\xi, \dot r^1_\xi)_{\xi<\kappa}$ {\tiny (bookkeeping)}}
We define $\bar s$ to list all reals which end up in the complement of $\Gamma^0$, in order to handle Task 2 for each of these. To make sure all projective sets of reals are measurable (Task 3) we ask that for each $\iota<\kappa$, the set of $\dot r^0_\xi, \dot r^1_\xi $ such that $\bar \iota(\xi)=\iota$ list all the pairs of reals in $L[G_\kappa]$ which are random over $L^{P_{\iota}}$.
We also ask that each pair $\dot r^0_\xi, \dot r^1_\xi $ be $\lambda$-reduced over $P_{\bar \iota(\xi)}$ for some large enough $\lambda$.
We shall first proceed with the definition of the iteration, and after that argue that a bookkeeping with the requisite properties can be defined at the same time.  

We define a sequence $\bar \lambda= (\lambda_\xi)_{\xi\leq\kappa}$\index[notation]{lambdaz, (lambda xi) xi<kappa@$\bar \lambda$, $(\lambda_\xi)_{\xi\leq\kappa}$|textbf} by induction, so that for each $\xi \leq \kappa$, $P_\xi$ will be stratified on $[\lambda_\xi,\kappa]$.
We let $\lambda_0 = \omega$. For limit $\xi$, let $\lambda_\xi$ be the minimum of $\Reg \setminus \bigcup_{\nu<\xi}\lambda_\nu$. For successors, define
\begin{equation}\label{def:lambda:xi}
\lambda_{\xi+1} =\begin{cases} 
      
                  (\lambda_{\xi})^+& \text{if $\xi \in E^0$ or $\exists \rho$ s.t. $\xi=E^3(\alpha^0_\rho)$},\\
                  \lambda_{\xi}&\text{otherwise.}
\end{cases}
\end{equation}
For the readers orientation, be aware we will always have that $P_\xi$ collapses $\lambda_\xi$, except when $\xi$ is inaccessible or $P_\xi$ is an amalgamation of type-1. 
In those cases $P_\xi$ may or may not preserve $\lambda_\xi$.
The continuum of $L[G_\xi]$ is always at most $(\lambda_\xi)^+$ 
so $\CH$ always holds in $L[G_\xi]$, except when $P_\xi$ does not collapse $\lambda_\xi$, in which case we shall force with the collapse at the next step. 

\medskip

At all limit stages $\xi \leq \kappa$, we define $P_\xi$ to be the $\bar \lambda$-diagonal support limit of the iteration up to that point.
We generally write $B_\xi = \ro(P_\xi)$ for $\xi\leq \kappa$. In the inductive definition of the iteration, we also define 
\begin{enumerate}
\item\index[notation]{G o xi@$G^o_\xi$} A sequence of sets $G^o_\xi$, for $\xi<\kappa$; these arise because at coding stages, the next forcing is not taken from the natural model $L[G_\xi]$ but from a smaller model $L[\bar B^-][G^o_\xi]$ (of course, since we do not want to code \emph{all of }$\bar B$). 
The $G^o_\xi$ help to ``integrate out the $\bar T$-part'' and to show that each $(P_\xi,P_{\xi+1})$ is a stratified extension. 
Note that we will have $\pset(\omega)^{L[G_\xi]}= \pset(\omega)^{L[G^o_\xi]}$.%
\item A sequence $\bar c=(\dot c_\xi)_{\xi<\kappa}$ of names for reals where each $\dot c_\xi$ is Cohen over $L[G_{E^0(\xi)}]$.
\item A sequence $(C_\xi)_{\xi<\kappa}$ of so-called coding areas, where each $C_\xi \in {}^\kappa 2$ is generic over 
$L[G_{E^1(\xi)}]$ but has constructible initial segments.
\item\label{def.auto} Maps $\Phi^\zeta_\rho$, for $\rho,\zeta < \kappa$, where $\Phi^{\zeta}_\rho$ induces an automorphism of  some $B_\xi$---to be precise, $\xi = E^3(\alpha^\zeta_\rho)$---and $\Phi^{\bar \zeta}_\rho$ extends $\Phi^\zeta_\rho$ for $\zeta < \bar \zeta$.
\end{enumerate} 

\section{Definition of the Forcing Iteration}\label{def:it:succ:stage}
For the successor stage, assume by induction that we have already defined $P_\xi$ for $\nu<\xi$ and $\dot r^0_\nu, \dot r^1_\nu$, $\dot s_\nu$ for $\nu \leq \xi$. 
Fix $k$ and $\eta$ such that $\xi=E^k(\eta)$. 
Let $G_\xi$ denote a generic for $P_\xi$.
We may assume we have already defined $G^o_\xi$, letting $G^o_0=\emptyset$.
\begin{description}

\item[$k=0$] 
At this stage we collapse the continuum of $L[G_\xi]$ and make sure the $\GCH$ holds (Task 4). 
Let $P_{\xi+1}= P_\xi \times \Coll(\omega,\lambda_{\xi+1})$,
and let $G^\xi$ be the generic for $\Coll(\omega,\lambda_{\xi+1})$ over $L[G_\xi]$.

As $\lambda_{\xi+1} =(\lambda_\xi)^+\geq2^\omega$ in $L[G_\xi]$, we can pick a $\Coll(\omega,\lambda_{\xi+1})$-name such that the $P_{\xi+1}$-name it induces via the embedding into the product is fully Cohen over $P_\xi$. 
We define $\dot c_\eta$\index[notation]{c eta@$\dot c_\eta$, $c_\eta$} to be this name and $c_\eta$ to be its interpretation under under the generic $G^\xi$.
In $L[G_{\xi+1}]$, simply let $G^o_{\xi+1} = G^o_\xi \times G^\xi$.

\item[$k=1$] Let $P_{\xi+1}= P_\xi\times (\Add(\kappa))^L$. 
We denote by $G^\xi$ the generic and by $C_\eta$\index[notation]{C eta@$\dot C_\eta$, $C_\eta$ (coding area)}\index{coding area|textbf} the new subset of $\kappa$ it represents  (and let $\dot C_\eta$ denote its $P_{\xi+1}$-name). This will be the generic ``coding area'' used in the next step.
Again let $G^o_{\xi+1} = G^o_\xi \times  G^\xi$.

\item[$k=2$] We take care of Task 2, making sure $\Psi(c,j)$, holds for some real $c$ given to us by bookkeeping ($j\in \{0,1\}$ indicates whether $c \in \Gamma^0$; again, see Section~\ref{sec:preserving:coding} for the precise definition of $\Psi(c,j)$). 
If $\eta$ is a limit or $\eta=0$, let $c$ denote $\dot c_\eta$ (the Cohen real defined at stage $E^0(\eta)$), and let $j=0$.
If $\eta$ is a successor, let $c$ denote $\dot s_{\eta-1}^{G_\xi}$, and let $j=1$.
Note that $j$ indicates that $c \in \Gamma^j$. 
We wish to code a branch through $T(s,n,i,j)$ if and only if $\sigma$ is an initial segment of $C=C_\eta$ (the coding area from the previous step) and $c(n)=i$. That is, we let\label{def.B.B.minus.etc}
\[B(C,c,j)=\{ B(s,n,i,j) \setdef s \is C, c(n)=i \}\]
be the set of branches to code, and represent it in a $\Delta_1$ way as a subset of $\kappa^{++}$:
\[\index[notation]{B bar minus@$\bar B^-$}\bar B^-=\{\#(s,n,i,j,t) \setdef s \is C, c(n)=i, t \in B(s, n,i,j) \}\]
where $\#x$ denotes the order-type of $x$ in the well-ordering of $L[\bar B^-][G^o_\xi]$.

Working in $L[\bar B^-][G^o_\xi]$, define \index[notation]{A xi 2@$A_\xi$}$A_\xi \subseteq \kappa^{++}$ as in \ref{sec:full:setting} so that
$L[\bar B^-][G^o_\xi] = L[A_\xi]$ and let $P(A_\xi)$ be the forcing discussed in Chapter~\ref{sec:coding}.
 Finally, define 
\[ Q_\xi = P(A_\xi)^{L[A_\xi]},\]
and let $\dot Q_\xi$ be a $P_\xi$-name for this forcing.
Observe that it takes an argument to show we get a stratified extension; $Q_\xi$ isn't obviously stratified (on any interval) in $L[G_\xi]$.

Letting $G^\xi$ denote the $Q_\xi$-generic, let $G^o_{\xi+1} =G^o_{\xi} \times \{ p_{<\kappa^{++}} \setdef p \in G^\xi \}$ (if you prefer, replace $p_{<\kappa^{++}}$ by $p_0$).

\item[$k=3$]

Say $\eta=\alpha^\zeta_\rho$.
We first treat the case where $\zeta=0$:
By induction the bookkeeping device $\bar r$ gives us $\bar r(\rho)=(\bar\iota(\rho), \dot r^0_\rho, \dot r^1_\rho)$, where $\bar\iota(\rho)<\xi$ (in fact, $<\rho$)  and the pair of names reals $\dot r^0_\rho,\dot r^1_\rho$ is fully random over $L^{P_{\bar\iota(\rho)}}$ and $\lambda_\xi$-reduced over $P_{\bar\iota(\rho)}$.

\label{def.succ.auto} Let $f$ be the automorphism of the complete Boolean algebras generated by $\dot r^0_\rho$ and $\dot r^1_\rho$ in $B_\xi$ and let $P_{\xi+1}$ be the type-1 amalgamation of $P_\xi$ over $f$ and $P_{\bar\iota(\rho)}$:
\[ P_{\xi+1}= \am(P_{\bar\iota(\rho)},P_\xi,f,\lambda_\xi). \]
The resulting automorphism\index{automorphism} of $P_{\xi+1}$ we denote by %
$\Phi^0_\rho$. 

Observe that, in general, this automorphism need not extend to an automorphism of $\ro(P_\kappa)$.
Also observe that by induction and Theorem~\ref{thm:am:s:ext} $(P_\xi,\am(P_{\bar\iota(\rho)},P_\xi,f,\lambda_\xi))$ will be a stratified extension above $\lambda_{\xi+1}$.

In the second case, when $\eta=\alpha^\zeta_\rho$ and $\zeta > 0$, we make sure $\Phi^0_\rho$ is extended by an automorphism of $P_{\xi+1}$. So we let
\[ P_{\xi+1}= \simpleram(\dom(\Phi),P_\xi, \Phi), \]
where $\Phi$ is (an extension of) $\Phi^0_\rho$, constructed at an earlier stage of the iteration:

If $\zeta$ is a successor ordinal, at a previous stage $E^3(\alpha_\rho^{\zeta-1})$, we defined $\Phi_\rho^{\zeta-1}$ extending $\Phi^0_\rho$. Set $\Phi=\Phi_\rho^{\zeta-1}$.
 
If $\zeta$ is a limit, we have a sequence $(\Phi_\rho^\nu)_{\nu<\zeta}$, forming an increasing chain, and all extending $\Phi^0_\rho$. Letting $\delta=\bigcup_{\nu<\zeta}\alpha_\rho^\nu$, there is a unique automorphism of $P_\delta$, extending each of them (note $\delta < \xi$). Let $\Phi$ be this automorphism.

The resulting automorphism of $P_{\xi+1}$ we denote by \index[notation]{Phi rho zheta@$\Phi_\rho^\zeta$}$\Phi_\rho^\zeta$.
In both cases we say \emph{$\xi+1$ or $P_{\xi+1}$ is an amalgamation stage associated to $\Phi^\zeta_\rho$}\index{amalgamation!stage associated to Phi zeta rho@stage associated to $\Phi^\zeta_\rho$}.

$G^o_{\xi+1}$ is defined to be the $\Int$-sequence of the sets defined like $G^o_\xi$ in each of the $\Int$-many components of the amalgamation.
\end{description}

\medskip

We define the family of automorphisms which witness that $P_\kappa$ is sufficiently homogeneous (in the sense of Definition~\ref{p.d.hom}; we give a detailed proof in Lemma~\ref{auto}).
\begin{dfn}[The automorphisms]\label{d:auto}
The system of maps $\langle \Phi^\zeta_\rho \setdef \zeta < \kappa\rangle$ uniquely determines an automorphism\index{automorphism} \index[notation]{Phi rho@$\Phi_\rho$}$\Phi_\rho$ of $\ro(P_\kappa)$ such that $\Phi_\rho (\dot r^0_\rho)=\dot r^1_\rho$ and $\Phi_\rho \res P_{\bar\iota(\rho)}$ is the identity. For a more uniform notation, we also write $\Phi^\kappa_\rho$ for $\Phi_\rho$.
We call any stage of the iteration $P_{\xi+1}$ such that $\xi=E^3(\alpha^\zeta_\rho)$ for some $\zeta < \kappa$ and thus such that $\Phi^\kappa_\rho$ extends $\Phi^\zeta_\rho$, \emph{ associated to $\Phi_\rho$}.
\end{dfn}

Next, we collect the some important properties of the iteration.
\begin{lem}~\label{iteration:prop}
\begin{enumerate}
\item\label{P:strat:ext} For all $\xi<\kappa$, $(P_\xi, P_{\xi+1})$ is a stratified extension on $[\lambda_{\xi+1},\kappa]$.
\item  \label{it:strat} For all $\xi\leq\kappa$, $P_\xi$ is stratified on $[\lambda_\xi, \kappa]$.
\item \label{P:cont:small} For all $\xi<\kappa$, $ P_\xi \forces 2^\omega \leq (\lambda_\xi)^+$.
\item \label{P:coll} If $\xi<\kappa$ is not of the form $E^3(\alpha^0_\rho)$ for some $\rho<\kappa$, i.e., if $P_{\xi+1}$ is not a type-1 amalgamation, we have that $ P_{\xi+1} \forces \card{\lambda_{\xi+1}} = \omega$ and $P_{\xi+1}\forces\GCH$. \label{it:CH}
\item \label{P:cent} It is forced by $\bar T$ (i.e., by $P_0$) that for each $\xi \leq \kappa$, $P_\xi : P_0$ is $\kappa^+$-linked.
\end{enumerate}
\end{lem}
Observe it follows that $P_\xi$ preserves all cardinals greater than $\lambda_\xi$, for $\xi\leq\kappa$.
We will show in \ref{it:reals:are:caught} that it also preserves that $\kappa$ is Mahlo and that $P_\kappa$ preserves $\kappa$.
\begin{proof}
When $\xi \not \in E^2$, the first item holds by induction and by %
Lemma~\ref{products:ext} (products) and Theorem~\ref{thm:am:s:ext} (amalgamation); 
the prestratification systems are the ones stemming from the constructions in Lemmas~\ref{stratified:comp:implies:ext}, \ref{products:ext} and Theorem~\ref{thm:am:s:ext} of course.

It remains to show $(P_\xi, P_{\xi+1})$ is a a stratified extension\index{stratified extension} on $[\lambda_{\xi+1},\kappa]$ when $\xi$ is a coding stage.
We let $q \in \D_{\xi+1}(\lambda, p, x)$
if and only if $q\res\xi \in  \D_\xi(\lambda, \pi_\xi(p), x)$,
$q \res\xi \forces q(\xi) \in \dot \D_\xi(\lambda, p(\xi), x)^{L[A_\xi]}$---where $\dot \D_\xi$ witnesses quasi-closure of $P(A_\xi)$ in $L[A_\xi]$---
and both $p(\xi)_{<\kappa^+}$ and $p(\xi)^*_{\kappa^+}$ are $\kappa$-chromatic names below $\pi_\xi(p)$;
observe this makes sense since we may view $p(\xi)_{<\kappa^+}$ and $p(\xi)^*_{\kappa^+}$ as functions with domain $\kappa$, and since by induction $P_\xi$ is stratified at $\kappa$.
The proof of the density property of $\D$ goes through (the additional chromaticity requirement can be met since we only have to look at $\lambda\leq \kappa$).
Define $\Cl$, $\lequpl$, $\leqlol$ as in the proof for composition of stratified forcings \ref{stratified:composition}.

To see that this is a stratified extension, let $\bar p=(p^\nu)_{\nu<\rho}$ be a $(\lambda, \bar \lambda,x)$-adequate sequence in $P_{\xi+1}$, as witnessed by $\bar w$ and let $q$ be a greatest lower bound of $(p^\nu\res\xi)_{\nu<\rho}$.
It suffices to show that $q$ forces that $(p^\nu(\xi))_{\nu<\rho}$ is $(\lambda, \bar\lambda,x)$-adequate in $L[A_\xi]=L[\bar B^-, G^o_\xi]$.
That $\bar w$ is a strategic guide is clear by definition of $\D$ and the usual argument for composition; so we check that it is a canonical witness.
Firstly, $\bar w$ is $\qcdefSeq(\{ x\} \cup\bar\lambda)$ in $L[A_\xi]$ since it is $\qcdefSeq(\{ x\} \cup\bar\lambda)$ in $L$ (as in \ref{stratified:composition}).
Secondly, $p^\nu(\xi)^{G_\xi}$ can be obtained by applying a $\Sigma_1^{A_\xi}(\{ x\} \cup\bar\lambda)$-function to $w^\nu$ in $L[A_\xi]$:
since each $p^\nu(\xi)$ is a $\kappa$-reduced $P_\xi$-name, its interpretation can be found
inside $L[G^0_\xi]$ and hence in $L[A_\xi]$. 
In fact, the (partial) function assigning to $p^\nu(\xi)$ its interpretation is $\Sigma_1^{A_\xi}(\{ x\} \cup\bar\lambda)$, provided $\kappa^{+++}$ is among the parameters in $x$---in fact, the existential quantifier can be bounded by $L_{\kappa^{+++}}[A_\xi]$ since we only have to find a $\kappa$-spectrum witnessing the interpretation.
Thus $p^\nu(\xi)^{G_\xi}$ can be obtained by a $\Sigma^{A_\xi}_1(x)$ partial function from $p^\nu(\xi)$, which can in turn be obtained by a recursive function from $p^\nu$, which can be obtained by a ${\qcdefG(\{ x\} \cup\bar\lambda)}$ partial function from $w^\nu$.
This shows that $\bar w$ is a canonical witness.

The second item holds by Theorem~\ref{thm:it:strat}.
The third one is a corollary of the previous ones and Lemma~\ref{density reduction} and the next follows since we collapse $\lambda_\xi$ at the right stage.

Lastly, the $\kappa^+$-linkedness follows since stratification at $\kappa$ allows us to define $\Clink^{\kappa^+}$ in the usual way;
if $p, q \in P_\xi$ are such that $\pi_0(p)$ (recall $\pi_0$ is the strong projection to $\bar T$) and $\pi_0(q)$ are compatible in $\bar T$ and $\Clink^{\kappa^+}(p) \cap \Clink^{\kappa^+}(q) \neq \emptyset$ then clearly $p \cdot q \neq 0$.
Although this is not essential, note that we can now show cardinals above $\kappa$ are preserved:
This follows from the general fact about the composition of $\kappa^{++}$-distributive and $\kappa^+$-linked forcing.
\end{proof}

\begin{rem}\label{r.C.explicit}\index[notation]{C lambda@$\Cl$}
Note that we can give $\Cl$ the following explicit form.
The definition given earlier in the context of our iteration theorem is equivalent to the present one modulo a recursive translation.

Let $\lambda \leq\kappa$ be regular and let $\sigma <\lambda$. 
First, define $D^\sigma_{\lambda} \subseteq P_\kappa$, by induction on the length of a condition. 
For the successor step, say $p \in P_{\nu+1}$. We let $p \in D^\sigma_{\lambda}$ if and only if $\pi_\nu(p) \in D^\sigma_{\lambda}$ and the following hold:
\begin{enumerate}
\item in case $\nu \in E^0$ (i.e., $\dot Q_\nu$ is $\Coll(\omega,\lambda_\nu)$), we require that 
$\pi_\nu(p) \forces \ran(p(\nu))\subseteq \sigma$,
\item in case $\nu \in E^2$ (i.e., $\dot Q_\nu$ is Jensen coding of $L[A_\nu]$), we require that 
$\pi_\nu(p) \forces \sup\dom(p(\nu)_{<\lambda})\leq\sigma$ and $\rho(p(\nu)^*_\lambda)\leq\sigma$, if $\lambda\in\Inacc$.
\item in case $\nu \in E^3$ (i.e., $P_{\nu+1}$ is an amalgamation), we require that for all $i \in \Int\setminus\{0\}$ we have $p(i)^P \in D^\sigma_{\lambda}$ (it would be redundant to require this also for $i=0$).
\end{enumerate}
For $p \in P_\nu$ where $\nu \leq\kappa$ is a limit ordinal, let $p \in D^\sigma_{\lambda}$ if and only if for all $\nu'<\nu$, $\pi_{\nu'}(p)\in D^\sigma_{\lambda}$.
Finally, define $p \in D^\Sigma_\lambda$ if and only if there is $\sigma<\lambda$ such that $p \in D^\sigma_\lambda$.
Also, for any $p \in D^\Sigma_\alpha$, let $\sigma^2_\lambda(p)$ be the least $\sigma<\alpha$ such that $p \in  D^\sigma_{\lambda}$ and $\supp^\lambda(p)\cap\lambda\subseteq \sigma$.

The sets $\Cl$ can now be defined by a similar induction:
they are binary relations 
\[ \Cl \subseteq P_\kappa\times \{\text{ sequences of length }\leq \kappa \} . \]
Concentrating on the successor step, say we've already defined
\[  \Cl \cap ( P_\nu \times\{\text{ sequences of length }\leq \nu \}  ).\]
For $p \in P_{\nu+1}$, let $H \in \Cl(p)$ if and only if $H=(H(\xi))_{\xi<\nu+1}$ is a sequence of length $\nu+1$, $ H\res\nu \in \Cl(\pi_\nu(p))$ and either 
\begin{enumerate}
\item $p \leqlo^\lambda \pi_\nu(p)$, i.e., $p(\nu)$ is trivial below $\lambda$, and 
\item $H(\nu) \in \Hhier(\lambda)$ is arbitrary, 
\end{enumerate}
or else, the following hold:
\begin{enumerate}
\item $p \in D^\Sigma_\lambda$; we write $\sigma$ for $\sigma^2_\lambda(p)$ in the following,
\item in case $\nu \in E^0$, we require that $p(\nu)$ (a collapsing condition) is $\beta$-chromatic below $\pi_\nu(p)$, for some $\beta <\lambda$, with spectrum $H(\nu)$,  
\item in case $\nu \in E^2$, we require that $H(\nu)=(\sigma, H, H_B, H_\rho)$, where 
 $H$ is a $\lVert \sigma \rVert$-spectrum for $p(\nu)\res \lVert \sigma \rVert^+$ below $\pi_\nu(p)$ (note that we can identify $p\res \lVert \sigma \rVert^+$ in some convenient way with a function with domain $\lVert \sigma \rVert$),
and for inaccessible $\lambda$, 
$H_\rho$ is a $\lVert \sigma \rVert$-spectrum for 
 $\rho(p(\nu)^*_\lambda$ below $\pi_\nu(p)$ and $H_2$ is a $\lVert \sigma \rVert$-spectrum for 
 $B^{\rho(p(\nu)^*_\lambda}_{p(\nu)}$ below $\pi_\nu(p)$;
if $\lambda$ is not inaccessible, we can ask $\rho= h_1=\emptyset$;
\item in case $\nu \in E^3$, we require that $H(\nu)=(\bar H_i^P, \bar H^0_i, \bar H^1_i)_{i\in\Int\setminus\{0\}}$ and for all $i \in \Int\setminus\{0\}$, $ \bar H^P_i \in \Cl(p(i)^P)$ and for $j \in \{0,1\}$, $p(i)^j$ is $\lVert \sigma \rVert$-chromatic with spectrum $\bar H^j_i$ below $\pi_\iota(p(0)^P)$---where $\iota$ is chosen so that $P_\iota$ is the base of the amalgamation $P_{\nu+1}$ (see p.~\pageref{base} for the definition of \emph{base}).
\end{enumerate}
\end{rem}

\section{A Word about Bookkeeping}\label{bookkeeping}

We give a recipe for cooking up a definition of 
\index[notation]{r bar, (iota(xi), r dot 0 xi, r dot 1 xi)@$\bar r$, $(\bar\iota(\xi),\dot r^0_\xi, \dot r^1_\xi)_{\xi<\kappa}$ {\tiny (bookkeeping)}}
\index[notation]{ iota(xi), r dot 0 xi, r dot 1 xi@$(\bar\iota(\xi),\dot r^0_\xi, \dot r^1_\xi)_{\xi<\kappa}$ {\tiny (bookkeeping)}}
$\bar r=(\bar\iota(\rho),\dot r^0_\rho,\dot r^1_\rho)_{\rho<\kappa}$. %
The definition is given by induction ``on blocks''.
Assume $\bar r \res \xi$ has been defined.
We shall now define $\bar r$ and $\bar \iota$ on $[\xi,\lambda_\xi^+)$---the ``next block''.

Assume as induction hypothesis that $\xi \in E^0$ is a limit ordinal or $0$, so the last forcing of $P_{\xi+1}$ collapses $\lambda_{\xi+1}$ to $\omega$
and by definition $\lambda_{\xi+1}=\lambda_\xi^+$.
For the induction start, assume $\xi=\iota=0$, and for $\xi > 0$, consider for the moment some fixed $\iota <\xi$. 

Enumerate all reals in $L^{P_\xi}$ which are random over $L^{P_\iota}$ in order type $\beta \leq \lambda^+_\xi$ (it may be there are no such reals, in which case $\beta=0$). 
This is possible as $\lambda^+_\xi \geq 2^\omega$ in $L[G_\xi]$.
In other words, find names $P_\xi$-names $(\dot x^\iota_\nu)_{\nu<\beta}$ such that 
\[ P_\xi \forces \omega^\omega \setminus \dot N = \{ \dot x^\iota_\nu \}_{\nu<\beta},\]
where $\dot N$ is a name for the union of the Borel null sets with code in $L^{P_\iota}$.

For each $\nu,\nu' < \beta$, apply Lemma~\ref{reduce:a:pair} below to
obtain a set $Y=Y(\nu,\nu',\iota)$ of size $\lambda^+_\xi$ consisting of pairs which are $\lambda^+_\xi$-reduced over $P_\xi$.
If there are no reals in $L^{P_\xi}$ which are random over $L^{P_\iota}$ (i.e., if $\beta = 0$) let $Y$ be any set of pairs of random reals in $L^{P_{\xi+1}}$ which are $\lambda^+_\xi$-reduced over $P_\xi$. 
Such a set exists by the proof of Lemma~\ref{reduce:a:pair} below.

Now define $\bar r \res \lambda^+_\xi$ and $\bar\iota\res\lambda^+_\xi$ (using a bijection of $\lambda^+_\xi$ with $\xi \times (\lambda^+_\xi)^3$) in such a way that all pairs obtained in this way are listed, i.e., for each $\iota < \xi$, each pair  and $\nu,\nu' < \lambda^+_\xi$ and each $y \in Y(\nu,\nu',\iota)$ there is $\rho \in [\xi,\lambda^+_\xi)$ such that $\bar\iota(\rho)=\iota$ and $(\dot r^0_{\rho},\dot r^1_{\rho})=y$. 

Note that it will follow from Lemma~\ref{it:reals:are:caught} below that we ``catch our tail'' and $\bar r$ enumerates all the pairs of random reals of the final model $L[G_\kappa]$ (see Lemma~\ref{bookkeeping:catches:all}).
\begin{lem}\label{reduce:a:pair}\index{reduced pair}\index[notation]{lambda-reduced@$\lambda$-reduced pair}
Let $\iota<\xi$, where $\xi \in E^0$and say  $1_{P_\xi}$ forces $\dot x^0, \dot x^1$ are $P_\xi$-names random over $L^{P_\iota}$. 
Then there is a set $Y=\{ (\dot y^0_\nu,\dot y^1_\nu) \}_{\nu<\lambda^+_\xi}$ of $P_{\xi+1}$-names such that 
\[ 1\forces (\dot x^0,\dot x^1)\in \{ (\dot y^0_\nu,\dot y^1_\nu) \}_{\nu<\lambda^+_\xi} \]
 and each pair in $Y$ is $\lambda^+_\xi$-reduced over $P_\xi$.

\end{lem}
\begin{proof}
Find a maximal antichain $\{ q_\zeta\}_{\zeta < \lambda^+_\xi}$in $Q_\xi = \Coll(\omega,\lambda^+_\xi)$. 
Note that $\{ (1_{P_\xi}, q_\zeta)\}_{\zeta < \lambda^+_\xi}$ is maximal antichain in $P_{\xi+1}$.
Fix a map
\[ b\colon \zeta \mapsto (\dot b_0(\zeta),\dot b_1(\zeta))\]
such that $\forces_{P_\xi} b\colon \lambda^+_\xi \rightarrow (\Borelplus)^2$ is onto, where $\Borelplus$ denotes the set of Borel sets with positive measure coded in $L[G_\xi]$. 
For each $\zeta < \lambda^+_\xi$ and $j\in \{0,1\}$ pick $R^0_\zeta,R^1_\zeta$ such that $(1_{P_\xi}, \dot q_\zeta)$ forces $R^j_\zeta$ is random over $L^{P_\xi}$ and $R^j_\zeta \in \dot b_j(\zeta)$ for both $j\in \{0,1\}$.
This is possible since $\Coll(\omega,\lambda^+_\xi)$ collapses the continuum of $L[G_\xi]$.
Fix $\nu < \theta$ for the moment, and define $(\dot y^0_\nu,\dot y^1_\nu)$ as follows:
for each $j\in \{0,1\}$, pick $\dot y^j_\nu$ such that $(1_{P_\xi}, q_\nu) \forces \dot y^j_\nu=\dot x^j$ and for each $\zeta\in \lambda^+_\xi\setminus\{ \nu\}$ we have $(1_{P_\xi}, q_\zeta)\forces \dot y^j_\nu=\dot R^j_\zeta$.

As $\{ (1_{P_\xi}, q_\zeta)\}_{\zeta <\lambda^+_\xi}$ is maximal, $1_{P_\xi}$ forces $\dot r_j$ is random over $L^{P_\iota}$.
For each $\nu<\theta$, the pair $\dot y^0_\nu,\dot y^1_\nu$ is $\lambda^+_\xi$-reduced over $P_\xi$ (see Lemma~\ref{reduced:equ} Item~\ref{reduced:alt} or equivalently, Definition~\ref{d:reduced:pair:alt}): 
Let $p \leqlo^{\lambda^+_\xi} q \in P_\xi$, let $\dot b_0$, $\dot b_1$ be $P_\xi$-names and fix $w \leq \pi_\xi(p)$ such that $w\forces_\iota \dot b_0$ and $\dot b_1$ are codes for positive Borel sets.
Find $w' \in P_\xi$ and $\zeta<\theta$, such that $w'\leq w$, $\zeta\neq\nu$ and 
\begin{equation}\label{r:fix:b:j}
\forall j\in \{0,1\}\; w' \forces \dot b_j(\zeta) \subseteq \dot b_j.
\end{equation}
We can ask $\zeta \neq \nu$ because we are content with $\subseteq$ instead of $=$ in (\ref{r:fix:b:j}).
As $w'\forces p(\xi)\dot \leqlo^{\lambda^+_\xi}_\xi 1$, $w'\cdot p$ is compatible with $(1_{P_\xi}, q_\zeta)$.
If $p' \leq w'\cdot p$ and $p' \leq (1_{P_\xi}, q_\zeta)$, we have $p'\forces \dot y^j_\nu \in \dot b_j(\zeta)\subseteq\dot b_j$.
Clearly, the same proof shows that $\dot y^0_\nu,\dot y^1_\nu$ is $\lambda^+_\xi$-reduced over $P_\iota$: Any code for a positive Borel set in $L[G_\iota]$ is remains one in $L[G_\xi]$.

Lastly, as $\{ (1_{P_\xi}, q_\zeta)\}_{\zeta < \theta}$ is maximal and $(1_{P_\xi}, q_\nu)\forces \forall j \in \{0,1\}\; \dot x^j=\dot y^j_\nu$,
\[ 1_{P_{\xi+1}}\forces_{\xi+1} (\dot x^0,\dot x^1)\in \{ (\dot y^0_\nu,\dot y^1_\nu) \}_{\nu<\lambda}. \]
\end{proof}

We now define $\tilde\Gamma^0_\xi$\index[notation]{Gamma 0 xi@$\tilde\Gamma^0_\xi$}, the set of $P_\xi$-names whose interpretations will be in $\Gamma^0$ (i.e., in the $\Delta^1_3$ set without the Baire property, to be defined in the next section) at stage $\xi<\kappa$ of the iteration.
Let $\tilde\Gamma^0_\xi$ be the smallest superset of $\{ \dot c_\eta \setdef E^0(\eta)<\xi \text{ and $\eta$ is limit or }\eta=0\}$ (for limit $\eta$ of course $E^0(\eta)=\eta$) closed under all of the functions $F=\Phi^\zeta_\rho,(\Phi^\zeta_\rho)^{-1}$ such that $\dom F \subseteq P_\xi$, i.e., closed under functions in
\begin{equation*}
\{ \Phi^\zeta_\rho,(\Phi^\zeta_\rho)^{-1} \setdef E^3(\alpha^\zeta_\rho) \leq \xi \}.
\end{equation*}
(Recall that $\dot c_\eta$ is (a name for) a Cohen real defined at stage $E^0(\eta)$ and $\Phi^\zeta_\rho$ is an automorphism defined at stage $E^3(\alpha^\zeta_\rho)$; see Section~\ref{def:it:succ:stage}, Cases $k=0$ and $k=3$.)
Let 
\[\dot \Gamma^0_\xi = \{(1_{P_\xi},\dot c)\setdef \dot c \in \tilde\Gamma^0_\xi \}, \]
that is, $\dot \Gamma^0_\xi$ is the canonical choice for a name whose interpretation consists of the interpretations of the elements of $\tilde\Gamma^0_\xi$.

When defining $\bar s$,  we need to make sure that for every stage $\xi<\kappa$, all $P_\xi$-names for reals $\dot s$ 
such that
for any $\dot r\in \tilde\Gamma^0_\xi$, $\forces_{P_\xi} \dot r \neq \dot s$
are listed (in the course of the iteration) by $\bar s$.\index[notation]{s bar, (s xi)@$\bar s$, $(\dot  s_\xi)_{\xi<\kappa}$  (bookkeeping)}\index[notation]{ s xi@$(\dot  s_\xi)_{\xi<\kappa}$  (bookkeeping)}
We can easily make sure this is the case using arguments as above. 
As $P_\xi$ forces $\card{\omega^\omega}<\kappa$ (in fact $\leq \kappa$ would suffice), we can find $\dot f_\xi$ such that
\[ \forces_\xi \dot f_\xi\colon\kappa\rightarrow \omega^\omega\setminus \dot \Gamma^0_\xi\text{ is onto.} \]
We may assume (by induction hypothesis) we have such $\dot f_\nu$ for $\nu<\xi$. 
Pick $\dot s_\xi$ such that for $\xi=\langle\eta,\zeta\rangle$ (G\"odel pairing), $\forces_\xi \dot s_\xi=\dot f_{\eta}(\zeta)$.

Later (see Lemma~\ref{disjoint}), we show that $\bar s$ lists exactly the reals of the final model $L[G_\kappa]$ which are not in $\Gamma^0$ (which we are about to define).
This concludes the definition of $(P_\xi)_{\xi\leq\kappa}$, $\bar c$, $(C_\xi)_{\xi<\kappa}$, $\Phi^\zeta_\rho$ for ${\zeta\leq\kappa}$ and ${\rho<\kappa}$, $\bar r$, and $\bar s$, as well as that of $\tilde\Gamma^0_\xi$ and $\dot \Gamma^0_\xi$.

\section{Cohen Reals, Coding Areas and the Irregular Set}\label{sec:reals:areas}

We now define a set without the Baire property which, as we shall see in Section~\ref{sec:preserving:coding}, is $\Delta^1_3$.
\begin{dfn}[The irregular sets $\Gamma^0$ and $\Gamma^1$]\label{d.gamma}\index[notation]{Gamma 0, Gamma 1@$\Gamma^0$, $\Gamma^1$}
Let $\Gamma^0$ be the least superset of 
\[\{ c_\xi \setdef \xi<\kappa, \text{ $\xi$ limit ordinal or }\xi=0\}\]
closed under all functions $\Phi_\rho, (\Phi_\rho)^{-1}$, $\rho < \kappa$ (recall that $c_\eta$ is a Cohen real defined at stage $E^0(\eta)$ and $\Phi_\rho$ is automorphism of $\ro(P_\kappa)$; see Section~\ref{def:it:succ:stage}, Cases $k=0$, $k=3$ and Definition~\ref{d:auto}). 
Let $\dot \Gamma^0$ be a $P_\kappa$-name for this set. 
Moreover, let $\tilde \Gamma^0$\index[notation]{Gamma 0,z@$\tilde \Gamma^0$} be the least superset of 
\[\{ \dot c_\xi \setdef \xi<\kappa, \text{ $\xi$ limit ordinal or }\xi=0\}\]
closed under all functions $\Phi_\rho, (\Phi_\rho)^{-1}$, $\rho < \kappa$.
Recalling $\tilde\Gamma^0_\xi$ from the previous section (defined in the discussion of the coding device $\bar s$), note  that
\[
\bigcup_{\xi<\kappa} \tilde \Gamma^0_\xi =  \tilde\Gamma^0.
\] 
Let $\Gamma^1=\{ (\dot s_\xi)^G \setdef \xi<\kappa \}$ (coming from the bookkeeping device $\bar s$ that was discussed at the end of the previous section), with $P_\kappa$-name $\dot \Gamma^1$.
\end{dfn}
From Lemma~\ref{it:reals:are:caught} (proved in the next section) it follows $\Gamma^0$ does not have the Baire property\index{Baire property} in $L[G_\kappa]$ because for every $u\in\omega^\omega\cap L[G_\kappa]$, both $\Gamma^0$ and its complement will contain a dense set of reals which are Cohen over $L[u]$ (see Lemma~\ref{p.l.gamma} and \cite{jr:amalgamation}).

\medskip

By the next lemma, $\Gamma^0$ and $\Gamma^1$ are disjoint (that their union is $\omega^\omega$ will follow from Lemma~\ref{it:reals:are:caught} in the next section). 
Moreover, the lemma shows that the coding areas\index{coding area} $C_\nu$ behave in the same way as do the reals $c_\nu$, and this will be used in \ref{coding:survives} to show that the coding does not conflict with the automorphisms coming from amalgamation.

\medskip

We start by introducing the notion of an \emph{index sequence}.\index{index sequence}
\begin{dfn} Suppose $\dot y$ is of the following form:
\begin{equation}\label{e:index}
\dot y = (\Phi_{\xi_{m}})^{k_{m}} \circ \dots  \circ (\Phi_{\xi_1})^{k_1} (\dot x_{\xi_0})
\end{equation}
where $m \geq 0$, ${\xi_0}, \dots, \xi_{m} < \kappa$, $k_i \in \Int$ for $0<  i \leq m$ and $\dot x_{\xi_0}$ denotes either $\dot c_{\xi_0}$ or $\dot C_{\xi_0}$.
Observe that every $\dot y \in \Gamma^0$ can be written in the form above, for $\dot x_{\xi_0} = \dot c_{\xi_0}$,
where $m = 0$, we by convention we read \eqref{e:index} as $\dot y = \dot x_{\xi_0}$. 

We define the notion of an \emph{index sequence} of $\dot y$ as follows:
Write $\dot y_0$ for $\dot x_{\xi_0}$, and if $\not \forces_{P_\kappa} \dot y = \dot y_0$, for $0<i \leq m$ write
\[
\dot y_{i} = (\Phi_{\xi_i})^{k_i} \circ \dots  \circ (\Phi_{\xi_1})^{k_1} (\dot x_{\xi_0}),
\]
so that $\dot y_{m} = \dot y$.
We can trivially assume that $\xi_{i+1} \neq \xi_i$ for each $ i< m$ (otherwise contract the terms).
We can also assume $\not \forces_{P_\kappa} \dot y_{i+1}=\dot y_i$ for such $i$ (or we may simply leave out the $i$-th automorphism).  
Under these assumptions, we call ${\xi_0},  \dots, \xi_{m}, k_1, \dots, k_{m}$ an \emph{index sequence of $\dot y$}.
We say $m+1$ is the \emph{length} of this index sequence, so 
if $\forces_{P_\kappa} \dot y = \dot y_0 = \dot x_{\xi_0}$, the sequence of length $1$ consisting just of ${\xi_0}$ is an index sequence of $\dot y$. 
We also say $\xi_i$ and $k_i$ are not defined, for any $i >m$.
\end{dfn}

We now show that each $\dot y$ of the above form has only one index sequence, and is characterized uniquely by it:
\begin{lem}\label{index:sequ}
Given $\dot y$ with index sequence ${\xi_0},\xi_1, \dots, \xi_m, k_1, \dots, k_m$ (where of course $k_1, \dots, k_m$ are undefined in case $m=0$) there are $\rho_0, \dots, \rho_{m} < \kappa$ such that
\begin{enumerate}
\item the sequence $(\rho_i)_{i\leq m}$ is increasing, i.e., $\rho_0 < \dots < \rho_{m}$,
\item if $1  \leq   i \leq m$, $P_{\rho_i+1}$ is an amalgamation (of $P_{\rho_i}$) associated to $\Phi_{\xi_{i-1}}$,
\item\label{l.rho} if $0 \leq i \leq m$, $\dot y_i$ (using the notation from the previous definition) is a $P_{\rho_i+1}$-name not in $L^{P_{\rho_i}}$. Moreover, $\dot y_i$ is either unbounded over $L^{P_{\rho_i}}$ (if $\dot y_0 = \dot c_{\xi_0}$) or remote over $P_{\rho_i}$ up to height $\kappa$ (if $\dot y_0 = \dot C_{\xi_0}$).
\end{enumerate}

Moreover, for $\dot y, \dot y' \in \tilde\Gamma^0$, either $\forces_{P_\kappa} \dot y = \dot y'$ or $\forces_{P_\kappa} \dot y \neq \dot y'$.
The latter holds if and only if $\dot y$ and $\dot y'$ have different index sequences.
\end{lem}
\begin{proof}
By induction on $m$. For $m=0$, since $\dot y_0 = \dot x_{\xi_0}$, we pick $\rho_0$ least so that $\dot x_{\xi_0} \in P_{\rho_0 +1}$ (i.e., $\rho_0=E^0({\xi_0})$ if $\dot x_{\xi_0} = \dot c_{\xi_0}$; or $\rho_0=E^1({\xi_0})$ if $\dot x_{\xi_0} = \dot C_{\xi_0}$).
Then all of requirements from the lemma hold.

Now assume that $m > 0$ and by induction the above holds for $i \leq m-1$. 
Let $\rho_{m}$ be the least $\rho<\kappa$ such that $P_{\rho+1}$ is an amalgamation stage associated to $\Phi_{\xi_{m}}$, and $\dot y_{m-1}$ is a $P_{\rho}$-name (equivalently, $\dot y_{m}$ is a $P_{\rho+1}$-name).
Since by induction, $\dot y_{m-1}$ is not in $L^{P_{\rho_{m-1}}}$, $\rho_{m-1} < \rho_{m}$. 

We have that either $P_{\rho_{m}+1}= \am(P_\zeta, D, \dot r_0, \dot r_1)$ (for some $\zeta$, $D$ a dense subset of $P_{\rho_{m}}$, and some $\dot r_0$, $\dot r_1$), or $P_{\rho_{m}+1}= \simpleram(P_\zeta, P_{\rho_{m}}, \Phi)$ (for some $\Phi$ and $\zeta$).
We prove the lemma assuming the first holds; very similar arguments work for the other alternative, which we leave to the reader.  

Observe that $\dot y_{m-1}$ is not a $P_\zeta$-name, as otherwise, contrary to assumption, 
\[\forces_{P_\kappa} \dot y_{m}=\Phi^{k_m}_{\xi_m}(\dot y_{m-1})=\dot y_{m-1}.\] 
We have to consider two cases: If $\dot x_{\xi_0} = \dot c_{\xi_0}$, we can assume by induction that $\dot y_{m-1}$ is unbounded over $L^{P_\zeta}$ (as the induction hypothesis implies that for any initial segment, $\dot y_{m-1}$ is either in that initial segment or unbounded over it). Thus $\dot y_{m-1}$ is also unbounded over $L^{P_\zeta * \dot B (\dot r_i)}$ for $i=0,1$.
By Lemma~\ref{newreal} applied for $P = D$, $\dot y_{m}$ is unbounded over $L^{P_{\rho_{m}}}$ and we are done.
If on the other hand, $\dot x_{\xi_0} = \dot C_{\xi_0}$, we can assume by induction that $\dot y_{m-1}$ is remote over $P_\zeta$ up to height $\kappa$ and $\kappa > \lambda_{\xi_{m}}$. So by Lemma~\ref{remote:lemma}, $\dot y_{m}$ is remote over $P_{\rho_{m}}$. 

Lastly suppose ${\xi_0}, \dots, \xi_m$,$k_1,\dots,k_m$ is an index sequence of $\dot y$ with length $m+1$ and  
\[
{\xi'_0},\dots, \xi'_n, k'_1,\dots,k'_n
\]
is an index sequence of $\dot y'$ with length $n+1$.
Assume $\not \forces_{P_\kappa} \dot y = \dot y'$; we show $\forces_{P_\kappa} \dot y \neq \dot y$.
Let $\rho_0, \dots \rho_m$ and $\rho_0', \dots \rho_n'$ be obtained as above for $\dot y$ and $\dot y'$ respectively. 

Assume for now there is an $l$ such that both $\xi_l$ and $\xi'_{l}$ are defined and $\xi_l \neq \xi'_l$ or $k_l \neq k'_l$. 
Further, suppose $l$ is maximal with that property.
Then we can also assume $\rho_{l} = \rho'_{l}$, for otherwise
\[ \forces_{P_\kappa} \dot y_{l} \neq \dot y_{l'},\]
and we can apply $(\Phi_{\xi_m})^{k_m} \circ \dots (\Phi_{\xi_{l+1}})^{k_{l+1}}$ to this to obtain
\[\forces_{P_\kappa} \dot y \neq \dot y',\]
and we are done.
As by definition, $\rho_{l} = \rho'_{l}$ implies $\xi_l = \xi'_l$, it must be the case that $k_l \neq k'_l$. 
From now on write $\xi$ for $\xi_l$ (which is the same as $\xi'_{l}$) and $\rho$  for $\rho_{l}$ (which is the same as $\rho'_{l}$).

As in the previous argument, fix $\zeta$ such that $P_{\rho+1}= \am(P_\zeta, D, \dot r_0, \dot r_1)$ (for $D$ a dense subset of $P_{\rho}$, and some $\dot r_0$, $\dot r_1$), or $P_{\rho+1}= \simpleram(P_\zeta, P_{\rho}, \Phi)$ (for the appropriate $\Phi$).

By induction hypothesis $\dot y_{l-1}$ is either remote or unbounded over $V^{P_{\rho_{l-1}}}$, but $\dot y_{l-1} \in V^{P_{\rho_{l-1}+1}}$.
The latter makes $\zeta \geq \rho_{l-1}+1$ impossible, for otherwise again $\dot y_{l}=\Phi^{k_l}_\xi(\dot y_{l} ) = \dot y_{l-1}$.
Thus, $\zeta \leq \rho_{l-1}$, and again $\dot y_{l-1}$ is either remote or unbounded over $ V^{P_\zeta}$.

By  Lemma~\ref{newreal}, $\forces_{P_\kappa}  \dot y'_{l-1} \neq (\Phi_{\xi})^{-k'_l+k_{l-1}} (\dot y_{l})$. %
Apply $(\Phi_{\xi})^{k_l}$ to see $\forces_{P_\kappa}  \dot y_{l} \neq \dot y'_{l}$.
As above, apply $(\Phi_{\xi_m})^{k_m} \circ \dots (\Phi_{\xi_{l+1}})^{k_{l+1}}$ to this to obtain
\[\forces_{P_\kappa} \dot y \neq \dot y'.\]
This finishes the proof in the case that for some $l$, both $\xi_l$ and $\xi_{l'}$ are defined but not both $\xi_l = \xi'_l$ and $k_l = k'_l$. 

If no $l$ as above exists, the index sequences for $\dot y$ and $\dot y'$ agree where both are defined, except possibly at the first coordinate. 
If the sequences differ in length, we are done by Item~\ref{l.rho} of the lemma.
If they are of the same length, since $\forces_{P_\kappa} \dot y = \dot y'$ if ${\xi_0} = {\xi'_0}$, we have ${\xi_0}\neq{\xi'_0}$ and $\forces_{P_\kappa} \dot y_0 \neq \dot y'_0$.
Apply $(\Phi_{\xi_m})^{k_m} \circ \dots (\Phi_{\xi_{0}})^{k_{0}}$ and we're done.
\end{proof}

We prove the next lemma assuming Lemma~\ref{it:reals:are:caught}, and in particular, Corollary~\ref{cor:reals:are:caught} from the next section.
\begin{lem}\label{disjoint}\index[notation]{Gamma 0, Gamma 1@$\Gamma^0$, $\Gamma^1$!are disjoint}\index{disjointness of Gamma 0, Gamma 1@disjointness of $\Gamma^0$, $\Gamma^1$}
$\forces_{P_\kappa} \dot \Gamma^0=\omega^\omega \setminus \dot \Gamma^1$.
\end{lem}

\begin{proof}
First we show $\forces_{P_\kappa}\dot \Gamma^0\cup\dot \Gamma^1=\omega^\omega$. Let $r \not \in (\dot \Gamma^0)^G$. Find $\xi<\kappa$ such that $r\in L[G_\xi]$.
As $r \not \in \{\dot c^{G_\xi} \setdef \dot c \in \tilde\Gamma^0_\xi \}$, $\bar s$ was defined to list a name for $\dot r$, so $r \in (\dot \Gamma^1)^G$.

Now let $\dot c \in \Gamma^0$ and $\dot s \in \Gamma^1$, and show $\forces_{P_\kappa}\dot s \neq \dot c$. 
Fix $\xi <\kappa$ so that $\dot s$ is a $P_\xi$-name and $\forces_{P_\xi} \dot s\not \in \tilde\Gamma^0_\xi$. 
Let $v, \rho_1, \dots, \rho_n$ be obtained as in the previous lemma from an index sequence for $\dot c$ and write $\rho=\rho_n$.
By the last lemma $\dot c$ is a $P_{\rho+1}$ name not in $L^{P_\rho}$.
If $\rho+1 \leq \xi$, we are clearly done, for then $\dot c \in \tilde\Gamma^0_\xi$.
Otherwise, if $\rho \geq \xi$, $\dot c$ is not in $L^{P_\rho}\supseteq L^{P_\xi}$, so in any case, $\forces_{P_\kappa} \dot c \neq \dot s$.
\end{proof}

Alternatively, it is possible to show $\Gamma^0\cap\Gamma^1=\emptyset$ directly, without recourse to index sequences and Lemma~\ref{index:sequ}: Show by induction on the number of applications of automorphisms that for each 
$\dot c \in \tilde\Gamma^0 \setminus \tilde\Gamma^0_\xi$ there is $\rho \geq \xi$ such that $c\in L[G_{\rho+1}]$ but $c$ is unbounded over $L[G_\rho]$. 

\section{The Final Stage of the Iteration}\label{sec:stage kappa}

The next lemma shows that $\kappa$ remains a cardinal in the final model, $\kappa$ remains Mahlo at each earlier stage, and that all reals appear in some initial segment of the construction.
We abuse notation and write conditions in $P_\theta$ as $(t,p)$ in the following.
We do this solely in order to have a convenient and more suggestive way to refer to the $\bar T$-part (which more formally should be referred to as $\pi_0(p)$).
\begin{lem}\label{it:reals:are:caught}
Let $\theta \leq \kappa$, let $(\dot \alpha_\xi)_{\xi<\kappa}$ be a sequence of $P_\theta$-names for ordinals below $\kappa$ and let $(t,p) \in P_\theta$. Then for any $\beta_0<\kappa$ there is an inaccessible $\alpha \in (\beta_0,\kappa)$ and a condition $(t',p') \leqlo^{<\kappa} (t,p)$ such that for all $\xi <\alpha$,
$(t',p') \forces \dot \alpha_\xi <\alpha$. Moreover if $\theta = \kappa$, there is a sequence of $P_\alpha$-names $(\dot \alpha'_\xi)_{\xi<\alpha}$ such that for each $\xi<\alpha$,
$(t',p') \forces \dot \alpha_\xi = \dot \alpha'_\xi$.
\end{lem}

The ``moreover'' clause is of course meaningless if $\theta < \kappa$.
Before we prove the lemma, we draw a corollary.
\begin{cor}~\label{cor:reals:are:caught}\index{cardinal!Mahlo}\index{Mahlo cardinal}\index{cardinal!preservation}\index{preservation of kappa@preservation of $\kappa$}
\begin{enumerate}
\item If $\theta <\kappa$, $\kappa$ remains Mahlo in $L[G_\theta]$.
\item If $r \in L[G_\kappa]$ is a real, there is $\alpha<\kappa$ such that $r\in L[G_\alpha]$. In particular, $\kappa$ remains uncountable in $L[G_\kappa]$ (i.e., $\kappa=\omega_1$ in the final model).
\end{enumerate}
\end{cor}
\begin{proof}\renewcommand{\qedsymbol}{{\tiny  Corollary~\ref{cor:reals:are:caught}~}$\Box$}
For the first corollary, fix a real $r \in L[G_\kappa]$ and  let $\dot \alpha_n$ be a $P_\kappa$-name for $r(n)$, for each $n \in \nat$.
The lemma shows we can find $(t',p') \in \bar G_\kappa$, $\alpha<\kappa$ and a sequence of $P_\alpha$-names $\{\dot \alpha'_n \setdef n\in\nat\}$, such that for each $n \in \nat$, $(t',p')\forces \dot \alpha_n=\dot \alpha'_n$. Obviously, $r \in L[G_\alpha]$.

For the second, say $\theta < \kappa$ and fix a $P_\theta$-name $\dot C$ for a closed unbounded subset\index{closed unbounded set} of $\kappa$.
Let $\dot \alpha_\xi$ be a name for the least element of $\dot C$ above $\xi$.
By the lemma, we may find an inaccessible $\alpha \in (\lambda_\theta,\kappa)$ and $(t',p') \in G_\theta$ such that for each $\xi<\alpha$, $(t',p') \forces \dot \alpha_\xi<\alpha$.
Thus, $(t',p') \forces \check \alpha \in \dot C$. 
Now observe that as $P_\theta$ is stratified above $\lambda_\theta$, $\alpha$ is inaccessible in $L[G_\theta]$.
\end{proof}

Before we begin the proof we remind the reader of the following terminology.
Let $\eta < \theta$ a coding stage and work in $L[\bar B^-, G^o_\eta] = L[A_\eta]$, where $\dot B^-$, $G^o$, $A_\eta$ are as defined in Section~\ref{def:it:succ:stage} (see p.~\pageref{def:it:succ:stage}; what we call $\eta$ here is called $\xi$ there unfortunately) and $A_\eta$ is the set to be coded at this stage, as defined in Section~\ref{sec:full:setting}, see p.~\pageref{sec:full:setting}.
Let $H$ be any set and let $p\in P(A_\eta)$.

We say that $q \in P(A_\eta)$ is basic generic\index{basic generic} for $(H,p)$ at $\kappa^+$ if and only if
\begin{enumerate}
\item\label{grow_kappa_plus} $|q_\kappa| > H \cap \kappa^+$;

\item\label{restraints_kappa_plus} 
if $\nu \in H\cap[\kappa^+,\kappa^{++})$ then $b^{p_{\kappa^+}\res\nu} \in q^*_{\kappa^+}$;

\item \label{succ_coding_kappa_plus} if $\nu \in H\cap[\kappa^+,\kappa^{++})$ then there is $\zeta > |p_\kappa|$ such that $q_\kappa((\zeta)_1)=p_{\kappa^+}(\nu)$;

\item\label{A_coding_kappa_plus} if $\xi \in H \cap [\kappa,\kappa^+)$ there is $\nu > |p_\kappa|$ such that
$q_\kappa((\langle \xi, \nu \rangle)_0) = 1$ if $\xi\in A_\eta$ and $q_\kappa((\langle \xi, \nu \rangle)_0) = 0$ if $\xi\not\in A_\eta$;
\end{enumerate}
We say that $q \in P(A_\eta)$ is basic generic for $(H,p)$ at $\kappa$ if and only if
\begin{enumerate}
\item\label{grow_kappa}  $|q_{<\kappa}| \geq \sup (H\cap\kappa)$.

\item\label{restraints_kappa} if $\nu \in H\cap[\kappa,\kappa^{+})$ then $b^{p_{\kappa}\res\nu}\setminus \eta' \in q^*_\kappa$ for some $\eta'<\kappa$;

\item \label{mahlo_coding_kappa} if $\nu \in H\cap[\kappa,\kappa^{+})$ then
there is $\xi \in b^{p_\kappa\res\nu}\setminus\eta$ such that $\xi \geq  |p_{<\kappa}|$ and
$q_{<\kappa}(\{ \xi \}_2)=p_\kappa(\nu)$;
\end{enumerate}
We shall use these notions as a replacement for $\D$ in the following proof, since this approach simplifies the argument showing that the sequence we build is appropriately definable (or ``canonical'').
We need to do this because we build the canonical witness (as usual, a sequence of models, this time of two different types) in advance, before we build the associated sequence of conditions.

\begin{proof}[Proof of Lemma~\ref{it:reals:are:caught}]
The proof is by induction on $\theta$, so assume Lemma~\ref{it:reals:are:caught} holds for all $\theta' < \theta$.
Fix recursive (class) functions $\xi \mapsto \{ \xi \}_i$, for $i \in\{ 0,1\}$, from the class of ordinals to itself and such that whenever $\alpha \in \Card$ and $(\eta_0, \eta_1) \in \alpha\times\alpha$ then the set 
\[
\{ \xi < \alpha \setdef ( \{  \xi\}_0, \{ \xi\}_1 )= (\eta_0, \eta_1) \}
\]
has size $\alpha$.
Also, we denote by $h\colon \On \rightarrow L$ the map enumerating the constructible universe according to its canonical well-ordering.
We assume without loss of generality that $(t,p) \in \dom(\Clink^\kappa)\cap\dom(\Clink^{\kappa^+})$.

Start by inductively choosing continuous $\in$-chains $N^{\xi}$ and $M^{\xi}$ for $\xi<\kappa$ as follows. %
Set
\[
\vec{x}=\{\kappa^{+++}, (\dot\alpha)_{\xi<\kappa}, P_\theta, (t,p) \}.%
\]
Suppose $\xi < \kappa$ and if $\xi >0$, $M^\xi$ and $N^\xi$ are allready defined. 
Let $M^{\xi+1}\prec_{\Sigma_1} L$ be $\subseteq$-least such that %
$\vec{x} \in  M^{\xi+1}$, $M^{\xi+1} \cap \kappa \in \kappa$, and if $\xi >0$, $N_\xi \in M^{\xi+1}$. 
Let $N^{\xi+1}$ be $\subseteq$-least such that 
$N^{\xi+1}\prec_{\Sigma_1} L$, $N^{\xi+1} \cap\kappa^+ \in \kappa^+$, and 
$M^{\xi+1} \in N^{\xi+1}$.
For limit $\rho <\kappa$, let 
\begin{gather*}
M^\rho = \bigcup_{\xi<\rho} M^\xi\\
N^\rho = \bigcup_{\xi<\rho} N^\xi
\end{gather*}
and write $\kappa(\xi) = M^\xi \cap \On$. %
Observe that for each $\xi < \kappa$, $M^\xi$ has size less than $\kappa$, $N^\xi$ has size $\kappa$, both are closed under $h$ and $h^{-1}$, $M^\xi \subseteq N^\xi$ and each $\kappa(\xi)$ is a strong limit cardinal.

Let $\langle \kappa(\xi) \setdef \xi < \kappa\rangle$ be the strictly increasing enumeration of $C$ and find an inaccessible limit point $\alpha$ of $C$ such that $\alpha\geq \beta_0$ and
\begin{equation}\label{e.diamonduse}
\Diamond_\alpha = C\cap \alpha = \{\kappa(\xi)\setdef \xi<\alpha\}
\end{equation}
(remember we have chosen to denote by $(\Diamond_\xi)_\xi$ the canonical diamond sequence of $L$ concentrating on inaccessibles\index{Diamond Principle} below $\kappa$). 
Pick $C^* \subseteq \lim C\cap \Sing$ such that $C^*$ is club in $\alpha$ (where $\lim C$ of course denotes the set of limit point of $C$).
Now we construct a sequence of $(t^\xi, p^\xi) \in P_\theta$ by induction on $\xi\leq\alpha$, starting with $(t^0,p^0)=(t,p)$.

\medskip

\textbf{Successor step:}
At successor stages, assume we have constructed $(t^{\xi}, p^{\xi}) \in M^{\xi+1}$ and when $\xi$ is a limit, 
also assume that for each coding stage $\eta < \theta$,
$$(t^\xi, p^\xi\res\eta)\forces |p^\xi(\eta)_{<\kappa}|= \kappa(\xi).$$
Let $(t^{\xi+1},p^{\xi+1})$ be the $\leq_L$-least condition $(t', p') \in M^{ \xi+1}$ such that $(t', p') \mathbin{\leqlo^{<\kappa}} (t^{\xi}, p^{\xi})$, and meeting the following requirements: 
\begin{enumerate}
\item \label{Cstar}\label{l.first}
If $\xi$ is a limit ordinal, then for each coding stage $\eta < \theta$ we have
$$
(t', p'\res\eta)\forces p'(\eta)_{<\kappa} (\kappa(\xi))=i,
$$
where $i = 0$ if $\kappa(\xi)\in C^*$ and $i=1$ otherwise.

\item \label{kappa:work}
If there is a condition $(s_0,r_0) \leq (t^{\xi}, p^{\xi})$
with $h(\{\xi\}_0) \in \Clink^{\kappa(\xi)}(s_0,r_0)$, we demand that for some $(s_1, r_1)\leq (s_0,r_0)$ such that $(s_1, r_1)\in \dom(\Clink^\kappa)\cap\dom(\Clink^{\kappa^+})$ which decides $\dot \alpha_{\{\xi\}_1}$, we have that
$$ t' \leq s_1,$$
for every stage $\eta \in E^1$ were we force with $\kappa$-Cohen $p'(\eta)\leq r^1(\eta)$
and
for every coding stage $\eta < \theta$, $(t',p'\res\eta)$ forces both that
$$
(p'(\eta)_{<\kappa}, p'(\eta)^*_{\kappa}) \leq (r_1(\eta)_{<\kappa}), r_1(\eta)^*_{\kappa}) \text{ in } P^{p'(\eta)_\kappa}
$$
and that
$$
(p'(\eta)_{\kappa}, p'(\eta)^*_{\kappa^+}) \leq (r_1(\eta)_{\kappa}), r_1(\eta)^*_{\kappa^+}) \text{ in } P^{p'(\eta)_{\kappa^+}}.
$$
Moreover we demand $(t', p') \mathbin{\leqlo^{<\kappa}} (t^\xi, p^\xi)$.
Write $r^\xi_i = r_i$, where $i \in \{0,1\}$.
We also define $\alpha^\xi_\nu$ to be the ordinal such that $(s_1, r^\xi_1)\forces \dot \alpha_\nu = \check\alpha^\xi_\nu$.

\item\label{basic_generic}
For every $\eta < \theta$ which is a coding stage, we have that 
\begin{eqpar}\label{kappa:strategic}
$(t', p'\res\eta)\forces$``$p'(\eta)_{<\kappa}$ is basic generic for $(M^{\xi}, p^\xi(\eta))$ at $\kappa$ and
$p'(\eta)_{\kappa}$ is basic generic for 
$(N^{\xi}, p^\xi(\eta))$ at $\kappa^+$.'' 
\end{eqpar}

\item\label{generic}\label{l.last}
Notice that in the previous item we have ensured basic genericity over subsets of $L$, while we want to capture some information about $L[A_\eta]$; so we ask the following.
For every $\eta < \theta$ and every $P_\eta$-name for an ordinal $\dot \alpha \in M^{\xi}$  there is a set $a$ of size less than $\kappa$ such that
$(t', p'\res\eta) \forces \dot\alpha \in \check a$. 
Also, for every $\eta < \theta$ and every $P_\eta$-name for an ordinal $\dot \beta \in N^{\xi}$  there is a set $b$ of size $\kappa$ such that
$(t', p'\res\eta) \forces \dot\beta \in \check b$. 

\end{enumerate}
The conjunction of Requirements~\ref{l.first}--\ref{l.last} can be expressed by a formula which is $\Sigma_1$ in parameters from $\vec{x}\cup\{ M^{\xi}, N^\xi, (t^\xi, p^\xi) \}$,
so if $(t',p')$ satisfying these requirements can be found at all, then we can demand $(t',p') \in M^{\xi+1}$. 

Requirements~\ref{l.first} and \ref{kappa:work} are trivial.
Requirement~\ref{basic_generic} can be met by a density argument identical to the one showing the ``density'' of strategic class $\D$ in the limit of an iteration.
The difference is purely notational (we've treated the Mahlo coding as lower part at $\kappa$ and as upper part for $\lambda<\kappa$, now we are ``in-between'' these cases):
\begin{lem}\label{l.kappa.strategic.dense}
The set of all $(t', p'\res\eta)$ such that
for every $\eta < \theta$, \eqref{kappa:strategic} holds
is dense below $(t^\xi, p^\xi)$
\end{lem}
\begin{proof}
Say we are given $(t', q) \leq (t^\xi, p^\xi)$.
Exactly as in \ref{thm:it:qc}, construct $p'(\eta)$ by induction on $\eta$.
At successor stages, let $p'(\eta)$ be a name for a condition such that \eqref{kappa:strategic} is met.
This is possible as $(t', q\res\eta)$ forces such a condition to exist, by Corollary~\ref{cor_qc} (see p.~\pageref{cor_qc}).
Observe that if $(t', q\res\eta) \forces q(\eta)_{<\kappa} = \emptyset$, then we can choose $p'(\eta) \leqlo^{\kappa} q(\eta)$.
Thus, the resulting condition $(t', p')$ has legal support. \renewcommand{\qedsymbol}{{\tiny  Lemma~\ref{l.kappa.strategic.dense}~}$\Box$}\end{proof}

Lastly, we can satisfy Requirement~\ref{generic}: for the first line, use the induction hypothesis and $\lVert M^{\xi} \rVert < \kappa$, for the second use $\{\kappa\}$-stratification.
This shows we can find $(t',p')$ as above. 

\medskip

Note that by diagonal support, Requirement~\ref{generic} and since $M^{\xi+1}$ is $\Sigma_1$-elementary, for each $\eta<\theta$ we have
\begin{equation}\label{p_star_control}
(t^{\xi+1}, p^{\xi+1}\res\eta) \forces_{P_\eta} |p^{\xi+1}(\eta)_{<\kappa}| <   \kappa(\xi+1).
\end{equation}
(and analogously, a similar equation holds for $|p^{\xi+1}(\eta)_\kappa|$).
By induction, \eqref{p_star_control} and Requirement~\ref{l.first} will entail that for all $\xi < \kappa$ and all coding states $\eta < \theta$ we have
$$
(t^{\xi+1}, p^{\xi+1}\res\eta) \forces_{P_\eta} \{ \zeta \setdef p^{\xi}(\eta)_{<\kappa}(\zeta)= 1 \}\cap \lim C = C^* \cap \kappa(\xi+1).
$$

\medskip

\textbf{Limit step:}
At limit $\rho \leq \alpha$, we take $(t^\rho, p^\rho)$ to be the greatest lower bound of $(t^\xi,p^\xi)_{\xi<\rho}$. 
This is well-defined:
We show by induction on $\eta <\theta$ that $(t^\rho, p^\rho\res\eta)$ is a condition.
Since we are always taking $\leqlo^{<\kappa}$-direct extensions,
at amalgamation stages we can simply use induction to take the point-wise limit in each of the $\Int$-many components and the resulting $\Int$-sequence will be a condition in the amalgamation.
By $\kappa$-closure of $\bar T$ and $\kappa$-Cohen forcing, we only have to treat coding stages.

So fix a coding stage $\eta$ and assume $w=(t^\rho, \pi_\eta(p^\rho))$ is a condition.
We let $p^\rho(\eta)$ be the name for the obvious candidate for a greatest lower bound of the sequence $(p^\nu(\eta))_{\nu<\rho}$ (see \eqref{obvious_candidate}, p.~\pageref{obvious_candidate}).
\begin{claim}\label{claim:generic}
The condition $w$ is $M^\rho$-generic, in the sense that 
$$w \forces_{P_\eta} M^\rho[\dot G_\eta]\cap \On = M^\rho \cap \On.$$
It is $N^\rho$-generic in the same sense.
\end{claim}
\begin{proof}
This follows from Requirement~\ref{generic} in the construction of the sequence and by $\Sigma_1$-elementarity.
\end{proof}
Although we don't use this in what follows, note that modulo the transitive collapse, $M^\rho[\dot G_\eta]$ is a generic extension of $M^\rho$ by $P_\eta$; likewise for $N^\rho[\dot G_\eta]$.
\begin{claim}
We have that $w \forces p^\rho(\eta)_{<\kappa} \in S^*_{<\kappa}$.
\end{claim}
\begin{proof}
Let $G_\eta$ denote an arbitrary $P_\eta$-generic with $w \in G_\eta$ for the moment and write $L[\bar B^-, G^o_\eta] = L[A_\eta]$ (again, see Sections~\ref{sec:full:setting} and~\ref{def:it:succ:stage}).
Let 
\[
N = L_\gamma[A_\eta\cap\bar\kappa, p^\rho(\eta)_{<\kappa}\res\bar \kappa]
\] 
be a $<\kappa$-test model such that $N \vDash \bar \kappa$ is the least Mahlo.
Clearly, if $\bar \kappa < |p^\rho(\eta)_{<\kappa}|$, we are done since $\rho$ is a limit, $|p^\xi(\eta)_{<\kappa}|$ is strictly increasing in $\xi$ and $p^\xi(\eta)_{<\kappa} \in S^*_{<\kappa}$ for each $\xi< \rho$ by induction.

So assume $\bar \kappa = |p^\rho(\eta)_{<\kappa}|$ and note that the latter equals
$\kappa(\rho)$.
Let $\bar M^\xi$ be the transitive collapse of $M^\xi$ for $\xi \leq \rho$ .
Letting $\bar \mu =  \bar M^\rho \cap \On$, we show $\gamma < \bar \mu$:
For otherwise,
since $( \bar M^\xi)_{\xi<\rho}$ is definable over $\bar M^\rho = L_{\bar\mu}$ which is in  turn definable in $N$,
we find $C\cap\kappa(\rho) \in N$ and thus $\lim C^* \in N$, contradicting that $N \vDash \kappa(\rho)$ is Mahlo.

Thus indeed, $\gamma < \bar \mu$. Now we can quote the last part of the proof of Theorem~\ref{jensen_qc_main}, which easily shows that
by Requirement~\ref{basic_generic} and by Claim~\ref{claim:generic} in the construction of the sequence and by elementarity,
for $A^*= A^*(A_\eta\cap\bar\kappa,p^\rho_{<\kappa}(\eta))$,
$L_\gamma[A^*] \vDash r^\eta$ is coded by branches. \end{proof}
\begin{claim}
We have that $w \forces p^\xi(\eta)_\kappa \in S^*_\kappa$.
\end{claim}
\begin{proof}
This is completely analogous to the previous claim:
The proof of Theorem~\ref{jensen_qc_main} (see p.\ \pageref{jensen_qc_main}) carries over to the present situation almost verbatim (setting $\delta=\kappa$).
As in the previous claim, to show that the height of any test-model is less than that of the collapse of $N^\rho$, use that $(t^\xi,p^\xi)_{\xi<\rho}$ is appropriately definable over the transitive collapse of $N^\rho$ using $\vec{x}$ and $C^*$ as parameters.
The additional parameter is unproblematic since $C^* \in \Hhier(\kappa)$ and is therefore an element of $N^\rho$.
As in the previous claim, we may directly use basic genericity, i.e., Requirement~\ref{basic_generic} in the construction (instead of $\D$ as we did in the proof of Theorem~\ref{jensen_qc_main}).
\end{proof}

Observe when $\rho < \alpha$ is a limit ordinal, 
have $(t^\rho, p^\rho) \in M^{ \rho +1}$:
this is because $(t^{\xi},p^{\xi})_{\xi<\rho}$ is definable over
$\langle M^{\rho}, C^*\cap\kappa(\rho)\rangle$ and $C^* \cap\kappa(\rho) \in M^{\rho +1}$ since
$\kappa(\bar\rho+1)$ is a strong limit cardinal.
This ends the construction of the sequence $(t^\xi, p^\xi)$, for $\xi \leq \alpha$.

Finally, we extend $(t^\alpha, p^\alpha)$ to $(t^\alpha, \bar p^\alpha)$ by making sure $\alpha$ is in the support at each coding stage, i.e.,
\begin{equation}\label{alpha_in_supp}
\forall \eta \in \theta\cap \alpha\cap E^2\quad (t^\alpha, \bar p^\alpha \res\eta) \forces \alpha \in \supp(p^\alpha(\eta))
\end{equation}
We show $(t^\alpha, \bar p^\alpha) \forces \dot \alpha_\nu < \alpha$, for each $\nu < \alpha$.
In fact, letting $\dot \alpha'_\nu$ be the name such that whenever $\{ \xi \}_1=\nu$ and $r^\xi_1$ is defined, 
$r^\xi_1 \forces \dot \alpha'_\nu = \check \alpha^\xi_\nu$, we will show that
$(t^\alpha, \bar p^\alpha) \forces \dot \alpha_\nu =  \dot \alpha'_\nu$.
Observe that this is a $P_\alpha$-name below $(t^\alpha, \bar p^\alpha)$, so this proves the theorem.

So let $\nu < \alpha$, $(s,r) \leq (t^\alpha, \bar p^\alpha)$ and assume without loss of generality that $(s,r) \in \dom(\Clink^\alpha)$ and that $(s,r)$ decides $\dot \alpha_\nu$.
Pick $\xi$ such that $h(\{ \xi \}_0)\in \Clink^\alpha(s, r)$, $\{ \xi \}_1=\nu$ and also, 
$h(\{ \xi \}_0) \in \Hhier(\kappa(\xi))$.
The latter may be achieved by increasing $\xi$ if needed, without changing $\{ \xi \}_i$, $i\in \{0,1\}$.
Observe that it follows that $h(\{ \xi \}_0)\in \Clink^{\kappa(\xi)}(s, r)$, by definition of $\Clink$.

Since $\alpha$ is inaccessible, we have $\xi < \alpha$, and
$r$ witnesses that $r^\xi_1$ is defined.
We now show that $(s,r) \cdot (t^\alpha, r^\xi_1) \neq 0$.
This is clear for the $\bar T$-part, since $s \leq t^\alpha$.
We show that for each $\eta < \theta$, $(s,r \res \eta \cdot r^\xi_1\res\eta )\forces  r(\eta) \cdot r^\xi_1(\eta) \neq 0$.
The only non-trivial cases are amalgamation and coding stages.
At amalgamation stages, for each $i \in \Int$, we can takes point-wise meets by induction.
Observe that since the outcome is a $\leqlo^{<\kappa}$-direct extension of $r^\xi_1\res\eta+1$,
the resulting sequence is a condition in the amalgamation, i.e., in $P_{\eta+1}$.

Now assume $\eta$ is a coding stage ($\eta \in E^2$).
Letting $w = (s, r \res \eta \cdot r^\xi\res\eta )$, we have that 
$w$ forces that
$$(r(\eta)_{<\kappa}, r(\eta)^*_\kappa)  \leq (r^\xi_1(\eta)_{<\kappa}, r^\xi_1(\eta)^*_{\kappa}) \text{ in }P^{r(\eta)_\kappa}$$ and
$$
(r(\eta)_{\kappa}, r(\eta)^*_{\kappa^+})  \leq (r^\xi_1(\eta)_{\kappa}, r^\xi_1(\eta)^*_{\kappa^+}) \text{ in }P^{r(\eta)_{\kappa^+}}.
$$
Moreover, $w$ forces that $r^\xi_1(\eta)$ is an extension of a condition $r^\xi_0(\eta)$ which has ``lower part $h(\eta)$'', i.e., $h(\eta) \in \Clink^{\kappa(\xi)}(r^\xi_0(\eta))$.
Since $w\forces h(\eta) \in \Clink^{\kappa(\xi)}(r(\eta))$, $w$ forces that
$r^\xi_0(\eta) \res\kappa(\xi) = r(\eta)\res\alpha \geq r^\xi_1(\eta)\res\alpha(\xi)$.
It follows immediately that that $w \forces r(\eta)\res [\alpha, \kappa^+] \cup r^\xi_1\res\alpha$ is a condition, call it $w'$.
Moreover, note that $w$ forces that 
\begin{equation}\label{spaced_out}
\supp(r^\xi_1(\eta))\cap\kappa \subseteq \kappa(\xi+1).
\end{equation}

It remains to show that $w \forces w'=r(\eta) \cdot r^\xi_1(\eta) $.
It is clear that $w \forces w' \leq r^\xi_1$; in order to see $w \forces w' \leq r(\eta)$ we must check that making the extension from $r(\eta)\res\alpha$ to $r^\xi_1(\eta)$ below $\alpha$ did not violate any of the inaccessible restraints in $r(\eta)$ at and above $\alpha$.
There is nothing to show for restraints in $r(\eta)^*_\kappa$, as they only affect the coding into $\kappa$ and there, the extension from $r(\eta)$ to $w'$ is trivial: $w \forces r(\eta)_{<\kappa} \leq r^\xi_1(\eta)_{<\kappa}$ and so $w \forces r(\eta)_{<\kappa} = w'_{<\kappa}$.
Since the restraints in $r(\eta)^*_\alpha$ are spaced by $\Diamond_\alpha$ (see the discussion following Definition~\ref{reg_defs}, p.~\pageref{reg_defs}), and by \eqref{spaced_out} and \eqref{e.diamonduse},  
making this extension obeys all restraints in $r(\eta)^*_\alpha$.
Lastly, since  \eqref{alpha_in_supp} makes sure that $w \forces \alpha \in \supp(r(\eta))$, no
restraints from the interval $(\alpha, \kappa)$ are violated by this extension.
In other words, we have shown that $w \forces w'=r(\eta) \cdot r^\xi_1 \neq 0$.
This finishes the proof that $(t^\alpha, r^\xi_1) \cdot (s,r) \neq 0$ and so
since $(s,r)$ decides $\dot \alpha_\nu$, we must have
$(s,r) \forces \dot \alpha_\nu = \check \alpha^\xi_\nu$.
This concludes the proof of Lemma~\ref{it:reals:are:caught}
\end{proof}
It is crucial that by Lemma~\ref{it:reals:are:caught}, the bookkeeping devices $\bar r$ and $\bar s$ ``catch'' all the relevant reals in the
final extension by $P_\kappa$:

\begin{lem}\label{bookkeeping:catches:all}
If $\iota<\kappa$, $\dot r^0, \dot r^1$ are $P_\kappa$-names for reals and $p \in P_\kappa$ forces $\dot r^0, \dot r^1$ are random over $L^{P_\iota}$, there is $q \leq p$ and $\nu<\kappa$ such that $\bar\iota(\nu)=\iota$
\[ q \forces \dot r^j=\dot r^j_\nu. \]
If $\dot s$ is a $P_\kappa$-name for a real and $p \in P_\kappa$, there is $q\leq p$ and $\xi<\kappa$ such that either
$q \forces \dot s=\dot s_\xi$, or $q\forces \dot s \in \Gamma_\xi$. 
\end{lem}
\begin{proof}
By Lemma~\ref{it:reals:are:caught} there is $q\leq p$, $\xi<\kappa$ and $P_\xi$-names $\dot x^0, \dot x^1$ such that $q \forces \dot r^j=\dot x^j$.
As $P_\kappa$ collapses the continuum of any initial stage of the iteration, we may assume $1_{P_\kappa}$ forces $\dot x^j$ is random over $L^{P_\nu}$. 
Using the notation from Lemma~\ref{reduce:a:pair}, find $\nu, \nu'$ and $q'\leq q$ such that $q'\forces \dot x^0 = \dot x_\nu$ and $q'\forces \dot x^1 = \dot x_{\nu'}$, and find $q'' \leq q'$ and $y \in Y(\nu,\nu')$ such that $q'' \forces (\dot x^0, \dot x^1)=y$.
Thus, by construction of $\bar r$, we may find $\nu$ such that $\bar\iota(\nu)=\iota$ and $y=(\dot r^0_\nu, \dot r^1_\nu)$, and we have
$q'' \forces \dot r^0=\dot r^0_\nu$ and $\dot r^1=\dot r^1_\nu$.

The second claim follows immediately from Lemmas~\ref{it:reals:are:caught}, \ref{disjoint} and the definition of $\bar s$.
\end{proof}

\section{Projective Implies Measurable}\label{s.every-projective-set}\index{Lebesgue measurability}\index{projective (set, hierarchy)}

\begin{lem}\index{random real!measure one set of random reals}
For any $\nu<\kappa$, $\bigcup N^*_\nu$ is a null set, where
\[ N^*_\nu=\{N \in L[G_\nu] \setdef L[G_\nu]\models N \subseteq \omega^\omega \text{ has measure zero} \} \]
\end{lem}
\begin{proof}
Every null set $N \in L[G_\nu]$ is covered by a null Borel set whose Borel code is also in $L[G_\nu]$. The set $C^*$ of Borel codes for null sets in $L[G_\nu]$ is countable in $L[G]$, so $\bigcup N^*_\nu$, which is equal to the union of all the Borel sets with code in $C^*$, is a countable union of null sets in $L[G]$.
\end{proof}
The following, together with the last lemma, suffices to show that in the extension by $P_\kappa$, every projective set of reals is measurable.
\begin{lem}\label{auto}
Let $\nu<\kappa$. There is a name $\dot r_*$ which is fully random over $L^{P_\nu}$ such that the following hold: 
\begin{enumerate}
\item
Let $\dot B(\dot r_*)$ be a $P_\nu$-name for the complete subalgebra of $\ro(\quot{P_\kappa}{P_\nu})$ generated by $\dot r_*$ in $L[G_\nu]$ and let $B_0 = P_\nu * \dot B(\dot r_*)$.
For any $b \in \ro(P_\kappa) \setminus B_0$, there is an automorphism $\Phi$ of $\ro(P_\kappa)$ such that $\Phi(b)\neq b$ and $\Phi\res B_0 = \id$.
\label{lem:mix}
\item
For any $P_\kappa$-name $\dot r$ which is random over $L^{P_\nu}$ and any $p\in P_\kappa$ there is $q\leq p$ and an automorphism $\Phi$ of $\ro(P_\kappa)$ such that $q \forces \dot r =\Phi(r_*)$ and $\pi_\nu \circ \Phi = \Phi \circ \pi_\nu = \pi_\nu$.\label{lem:homogeneity:for:random:reals} 
\end{enumerate} 
\end{lem}

\begin{proof}
We show the first item of the lemma.
For $\dot r_*$ we may use any $\dot r^0_\eta$ (from our list $\bar r$) such that $\bar\iota(\eta)\geq \nu$ (i.e., it's fully random over $L^{P_\nu}$).
Note this includes the case where $\dot r = \dot r^*$.

Let $\pi_0$ be the canonical projection $\pi_0 \colon \ro(P_\kappa) \rightarrow B_0$, where $B_0$ is as in the hypothesis of Item~\ref{lem:mix} of the lemma.
Let $f$ be the identity $B_0 \to B_0$. %
Pick $\xi<\kappa$ such that 
\begin{enumerate}
\item $\pi_\xi(b) \not \in B_0$; this holds for large enough $\xi$ since $b \not\in B_0$;
\item $r^*$ is a $P_\xi$-name, i.e., $B_0$ is a complete subalgebra of $B_\xi$.
\item $P_{\xi+1}=\am(P_\iota,P_\xi,f,\lambda_\xi)$, where $\iota \geq \nu$.
\end{enumerate}

Let $b_0$ denote $\pi_\xi(b)$.
Clearly, there is $p \in P_\xi$, $p \leq b_0$ such that $\pi_0(p)\not\leq b_0$: for otherwise, the set
\[X =\{ d \in B_0 \setdef d\leq b_0\} \]
would be predense in $P_\xi$ below $b_0$, and thus $b_0=\sum^{\ro(P_\xi)} X\in B_0$, contradiction.

So pick $p$ as above and let $q \in P_\xi$, $q \leq \pi_0(p)$, whence $q \cdot b_0=0$. Let $b_1 = \pi_0(q)$.
Letting $\Dam=\Dam(P_{\nu},P_\xi,f,\lambda_\xi)$, consider a condition $\bar p \in (\Dam)^\Int_f$ such that $\bar p(-1)=q$, $\bar p (0)= b_1\cdot p$ and for $i \in \Int\setminus\{-1,0\}$,
$\bar p(i)=b_1$. Then we have $\bar p \leq p \leq b_0$. Letting $\Phi$ denote the automorphism of $\ro(P_\kappa)$ resulting from $P_{\xi+1}$, we have $\Phi(\bar p) \leq q$ whence $\Phi(\bar p) \cdot b_0 = 0$. So as $\bar p \leq \pi_\xi(b)$ and $\Phi(\bar p)\cdot b =0$, it follows that $\Phi(b)\neq b$; for otherwise since $\bar p \leq \pi_\xi(b)$, we have $\bar p \cdot b \neq 0$ but $\Phi(\bar p \cdot b) = \Phi(\bar p) \cdot b = 0$.

The second claim is clear from the construction, as $\Phi_\rho(\dot r^0_\rho)=\dot r^1_\rho$ for each $\rho<\kappa$.

\end{proof}

Finally, we show in $L[G]$:
\begin{lem}\label{l.meas}
Say $s \in [\On]^\omega$, $\phi$ a formula. If $X =\{r \in \omega^\omega \setdef \phi(r,s) \}$, $X$ is measurable.
\end{lem}
\begin{proof}
Let $X, s$ be as above, and say $s=\dot s^{G}$. Without loss of generality, $\dot s$ is a $P_\nu$-name, where $\nu<\kappa$ and $\forces_\nu \dot s \in [\On]^\omega$ (by Lemma~\ref{it:reals:are:caught}).
Fix $\dot r_*$ as in the previous lemma. Let $\dot B(\dot r_*)$ be a $P_\nu$-name for the complete subalgebra of $\ro(\quot{P_\kappa}{P_\nu})$ generated by $\dot r_*$ in $L[G_\nu]$.
\begin{claim}\label{claim.x}
$\bv{\phi(\dot r_*,\dot s)}^{\ro(P_\kappa)} \in \ro(P_\nu)*\bg{\dot r_*}$.
\end{claim}
\begin{proof}[Proof of Claim]
Write $B_0=P_\nu * \dot B(\dot r_*)$ and $b = \bv{\phi(\dot r_*,\dot s)}^{{\ro(P_\kappa)}}$. Towards a contradiction, assume $b \not\in B_0$. By Item~\ref{lem:mix} of Lemma~\ref{auto} there is an automorphism $\Phi$ of $\ro(P_\kappa)$ such that $\Phi(b)\neq b$ while $\Phi(\dot s)=\dot s$ and $\Phi(\dot r)=\dot r$.
This is a contradiction, as we infer
\[ b = \bv{ \phi (\dot r_*,\dot s ) }^{\ro(P_\kappa)}= \bv{ \phi (\Phi(\dot r_*),\Phi(\dot s) ) }^{\ro(P_\kappa)} = \Phi(b). \]
\renewcommand{\qedsymbol}{{\tiny  Claim~\ref{claim.x}~}$\Box$}
\end{proof}

Let $N^*$ denote 
\[ \bigcup \{N \in L[G_\nu] \setdef L[G_\nu]\models N \text{ has measure zero}\},\]
and let $\dot N^*$ be a $P_\nu$-name for this set. $N^*$ is null in $L[G]$.

We find a Borel set $B$ such that for arbitrary $r \not \in N^*$, we have $r \in X \iff r \in B$. Then $X\setminus N^*=B\setminus N^*$ is measurable, finishing the proof. We may regard $\bg{\dot r_*}$ as identical to the Random algebra in $L[G_\nu]$, so we may write $\bv{\phi(\dot r_*,\dot s)}^{\ro(\quot{P_\kappa}{P_\nu})} = \ecn{B}$ for a Borel set $B$.

To show $B$ is the Borel set we were looking for, let $r \not \in N^*$ be arbitrary.
Find $\dot r$ and $p \in G$ such that $\dot r^G=r$ and $p \forces \dot r \not \in \dot N^*$, i.e., $p$ forces $\dot r$ is random over $L^{P_\nu}$. By \ref{lem:homogeneity:for:random:reals} of the previous lemma, there is an automorphism $\Phi$ of $P_\kappa$ and $q \in G$ such that $q\forces \Phi(\dot r_*)=\dot r$, and thus $\Phi(\dot r_*)^G={\dot r_*}^{\Phi^{-1}[G]}=\dot r^G$.
We also have $\pi_\nu \circ \Phi = \Phi \circ \pi_\nu = \pi_\nu$ and so
$\Phi(\dot s)^G={\dot s}^{\Phi^{-1}[G]}=\dot s^G$.
We have
\begin{multline*}
\phi(\dot r^G, \dot s^G) \iff \bv{\phi(\dot r, \dot s)}\in G \iff \\
\iff \bar \Phi^{-1}(\bv{\phi(\dot r, \dot s)}) \in \bar \Phi^{-1}[G] \iff \\
\iff \bv{\phi(\bar \Phi^{-1}(\dot r), \dot s)} \in \bar \Phi^{-1}[G]\iff \\
\iff \bv{\phi(\dot r_*, \dot s)} \in \bar \Phi^{-1}[G]\cap\bg{\dot r_*}\iff \dot r_*^{\bar \Phi^{-1}[G]}\in B 
\end{multline*}
As $\dot r_*^{\bar \Phi^{-1}[G]}=\dot r^G$, we are done.
\renewcommand{\qedsymbol}{{\tiny  Lemma~\ref{l.meas}~}$\Box$}
\end{proof}

\section{A Projective Set Without the Baire Property}\label{sec:preserving:coding}

We now check that $\Gamma^0$ from Definition~\ref{d.gamma} is in fact $\Delta^1_3$. 
As we have seen, this is optimal in a model where all projective sets are measurable: 
On general grounds, if just all $\Sigma^1_2$ sets are Lebesgue measurable, all $\Sigma^1_2$ sets have the property of Baire (by \cite{bartoszynski:additivity}; or see Theorem~\ref{t.barto}). 
That all $\Sigma^1_2$ sets have the property of Baire can also be seen more easily from the fact $\omega_1$ is  inaccessible to reals in our model (using \cite{solovay}, see Theorem~\ref{p.t.sigma-1-2})---as again, it must be on general grounds, if just all $\mathbf{\Sigma}^1_3$ set are measurable (by \cite{shelah:amalgamation}; see Theorem~\ref{p.i.random-reals}).

\medskip

We now discuss the $\Sigma^1_3$ formulas $\Psi(r,j)$\index[notation]{Psi r j@$\Psi(r,j)$ ($\Sigma^1_3$ definition of $\Gamma^j$)} witnessing that $\Gamma^j$ is $\Sigma^1_3$, for each $j\in\{0,1\}$.
Let $\theta(r,u,\alpha,\beta)$\index[notation]{theta r u alpha beta@$\theta(r,u,\alpha,\beta)$} denote the formula 
\begin{eqpar*}
$L_\beta[r,u]$ is a model of $\ZF^-$ and of
``$\alpha$ is the least Mahlo and $\alpha^{++}$ exists''.
\end{eqpar*}
Also, recall that for an ordinal $\alpha$ and $C \in {}^\alpha 2$, we write $\sigma \is C$ to express that $\sigma$ is an initial segment of $C$, i.e., that for some $\rho < \alpha$, $\sigma = C \res \rho$. An expression such as $\forall \sigma \is C\; \phi(\sigma)$ is, of course, a short-hand for 
$\forall \sigma\; \big( \sigma \is C \Rightarrow \phi(\sigma)\big)$.
Finally, recall that $J$ denotes the set of indices for $\bar T$, i.e., 
$$J={}^{<\kappa}2 \times \omega\times \omega \times 2.$$

For $j \in \{ 0,1\}$, $\Psi(r ,j)$ denotes the formula (cf.\ Definition~\ref{p.n.d.canonical} and Lemma~\ref{p.l.canonical})
\begin{eqpar*}
$\exists u \in 2^\omega \quad \forall \alpha, \beta<\kappa$  if  $\theta(r,u,\alpha,\beta)$ holds, then:\\
$L_\beta[r,u] \models$``$\exists C \in  {}^\alpha 2 \; \forall \sigma \is C \;  \big[ \sigma \in L$ and $\forall n \in \omega$ %
$T^\alpha(\sigma,n,r(n),j)^{L_\beta}$ has a branch.$\big]$''
\end{eqpar*}

To establish our main theorem, it remains to give a proof of the following lemma, which shall take up the remainder of this final chapter.
\begin{lem}\label{coding:survives}
For any $r \in L[G] \cap \omega^\omega$ and $j \in \{0,1\}$,   
\begin{equation}\label{e.coding:survives}
r \in \Gamma^j \iff L[G]\vDash \Psi(r,j).
 \end{equation}
\end{lem}
One direction of \eqref{e.coding:survives} is fairly straightforward.
\begin{proof}[Proof of $\Rightarrow$ in \eqref{e.coding:survives}]
Let $r \in L[G]\cap\omega^\omega$ and $j \in \{0,1\}$ be given.
We show 
\begin{equation}\label{wanted:branches}
r \in \Gamma^j \Rightarrow L[G]\vDash \Psi(r,j).
\end{equation}
To this end, suppose $r \in \Gamma^j$ and show that $\Psi(r,j)$ holds in $L[G]$.
If $j=0$, by definition of $\dot \Gamma^0$ we can find $\eta < \kappa$ which is $0$ or a limit and $\Phi_\xi$, for $\xi< \kappa$ such that $r= (\Phi_\xi(\dot c_\eta))^G$. If $j=1$, we may fix a successor $\eta<\kappa$ such that $r= (\dot s_{\eta-1})^G$.
Let $\dot r_0$ denote $\dot c_\eta$ if $j=0$ and let $\dot r_0$ denote $\dot s_{\eta-1}$ if $j=1$ and let $r_0=(\dot r_0)^G$ .

In either case, at stage $\xi=E^2(\eta)$ we force with Jensen coding, adding a real $s_0$ such that
\begin{eqpar*}
for all $\alpha,\beta <\kappa$, if $\theta(r_0,s_0,\alpha,\beta)$ then $C_\eta\res\alpha, r_0 \in L_\beta[s_0]$ and \\
$L_\beta[s_0]\models$ ``$\forall \sigma$ such that $\sigma \is C_\eta\res\alpha$ and for all $n,i$ such that $r_0(n)=i$,  $T^\alpha(\sigma,n,i,j)$ has a branch''.
\end{eqpar*}
So 
\[ 1_{P_\kappa}\forces \Psi(\dot r_0,j),\]
which completes the proof in case $j=1$. For $j=0$, apply $\Phi_\xi$ to get
\[ 1_{P_\kappa}\forces \Psi(\Phi_\xi(\dot r_0),j), \]
and we are done as $(\Phi_\xi(\dot r_0))^G=r$.

Thus, we have established $L[G]\vDash\Psi(r,j)$ and hence that  in Lemma~\ref{coding:survives}, $\Rightarrow$ of \eqref{e.coding:survives} holds.
\end{proof}

For the other direction of Lemma~\ref{coding:survives}, we make a definition and state a technical lemma.

\pagebreak[3]

\begin{dfn}\label{m.d.F}~\index[notation]{Fxi@$\mathcal{F_\xi}$ (set of partial automorphisms)}
\begin{enumerate}
\item For $\xi\leq \kappa$, let $\mathcal{F_\xi}$ be the smallest set closed under (relational) composition and containing all functions $F=\Phi^\zeta_\rho,(\Phi^\zeta_\rho)^{-1}$ (they where defined on p.~\pageref{def.succ.auto} in the case $k=3$; see also Definition~\ref{d:auto}) such that $\dom F \subseteq P_\xi$. In other words, $\mathcal{F_\xi}$ is the closure under relational composition of
\[
\{ \Phi^\zeta_\rho,(\Phi^\zeta_\rho)^{-1} \setdef E^3(\alpha^\zeta_\rho) < \xi \}.
\]
\item For $\xi \leq \kappa$, working in $L[G_\xi]$ define $I_\xi$ to be the set of triples $(\sigma,n,i,j)$ such that for some $\eta$ with $E^2(\eta)<\xi$ and $\Phi \in \mathcal{F}_\xi$, $\sigma \is \Phi(C_\eta)$ and 
\begin{enumerate}
\item if $\eta$ is limit ordinal, $\Phi(c_\eta)(n)=i$ and $j=0$   
\item if $\eta$ is a successor ordinal, $\Phi(s_{\eta-1})(n)=i$ and $j=1$.
\end{enumerate}
\item For each $\xi \leq \kappa$, let $\dot I_\xi$ be a $P_\xi$-name for $I_\xi$.
\end{enumerate}
\end{dfn}
The heart of the proof ``$\Leftarrow$'' in Lemma~\ref{coding:survives} is the following lemma.
It expresses that with few---less than $\kappa$ many---possible exceptions, never is a branch through one of our trees constructible from a real parameter unless one of the following two cases applies: Firstly, the tree was explicitly made to be so at a coding stage on purpose, or secondly, the first case obtains after applying one of the automorphisms stemming from amalgamation.
\begin{lem}\label{no:unwanted:branches}
Suppose $\theta < \kappa$, $p \in P_\theta$ and $(n,i,j) \in \omega^2\times 2$.
There is $\zeta<\kappa$ and $p' \in P_{\theta}$ stronger than $p$ such that for any $\sigma \in {}^{<\kappa} 2$ with $\lh(\sigma) > \zeta$, 
\begin{equation}\label{e:no:unwanted:branches}
p' \forces_{P_{\theta}}\text{``}(\sigma,n,i,j) \notin \dot I_{\theta} \Rightarrow  \forall u \in \omega^\omega\;  L[u]\vDash T(\sigma,n,i,j) \text{ is Suslin.''}
\end{equation}
\end{lem}
We postpone the proof and first finish the proof of Lemma~\ref{coding:survives}, assuming Lemma~\ref{no:unwanted:branches}.

 \medskip

\begin{proof}[Proof of Lemma~\ref{coding:survives}]
We have alredy shown $\Rightarrow$.
For $\Leftarrow$, fix $r \in \omega^\omega$ and $j \in \{0,1\}$ and, working in $L[G]$, suppose $\Psi(r,j)$ holds; we must show $ r \in \Gamma^j$.
Fix $u$ witnessing that $\Psi(r,j)$ holds.
\begin{claim}\label{c:u:C} 
There is $C \in L[u,r]\cap  {}^\kappa 2$ such that all proper initial segments of $C$ are constructible and 
\begin{equation}\label{u:C}
\forall \sigma \is C \; \forall n \in \omega\; %
L[u,r] \vDash T(\sigma,n,r(n),j)\text{ has a branch.}
\end{equation}
\end{claim}
\begin{proof}[Proof of claim.]
Let $L_\beta[r,u]$ be isomorphic to a countable elementary submodel of $L_{\kappa^{+3}}[r,u]$ which contains $r$ and $u$, and let $\alpha$ be the least Mahlo in $L_\beta[r,u]$. Then as $\theta(r,u,\alpha,\beta)$ holds, by $\Psi(r,j)$,
\begin{equation*}
L_\beta[r,u] \models \exists C \in  {}^\alpha 2 \; \forall \sigma \is C \; ( \sigma \in L \wedge \forall n \in \omega\; %
 T^\alpha(\sigma,n,r(n),j)^{L_\beta}\text{ has a branch.})
\end{equation*}
So by elementarity, the claim holds.
\renewcommand{\qedsymbol}{{\tiny  Claim~\ref{c:u:C}~}$\Box$}
\end{proof}

\medskip

Fix $u$ and $C \in L[u,r]$ as in the claim.
Moreover, by Corollary~\ref{cor:reals:are:caught}, we may fix a limit ordinal $\theta<\kappa$ such that $u,r \in L[G_\theta]$.

\begin{claim}\label{c.single:out:C}
For some $\Phi \in \mathcal{F}_\theta$ and some $\eta<\theta$, we have $C= \Phi(C_\eta)$.
\end{claim}
\begin{proof}[Proof of claim.]
We show that any long enough initial segment $\sigma$ of $C$ satisfies:
\begin{equation}\label{single:out:C}
\exists! (\rho,\eta') \in \theta^2\text{ s.t. } \exists \zeta \text{ with }  \Phi^\zeta_\rho \in \mathcal F_\theta, C_{\eta'} \in \dom( \Phi^\zeta_\rho)  \text{ and } \sigma \is \Phi^\zeta_\rho(C_{\eta'}).
\end{equation}
By Lemma~\ref{index:sequ}, it is clear that any long enough $\sigma$ is an initial segment of $\Phi^\zeta_\rho(C_{\eta'})$ for \emph{at most} one pair $(\rho, \eta')$ as in \eqref{single:out:C}.
It remains to show any long enough $\sigma \in C$ is an initial segment of  $\Phi^\zeta_\rho(C_{\eta'})$ for \emph{at least} one such $(\rho, \eta')$.
But otherwise, for no $n,i,j$ is $(\sigma,n,i,j) \in I_\xi$, which by Lemma~\ref{no:unwanted:branches} contradicts \eqref{u:C}.
\renewcommand{\qedsymbol}{{\tiny  Claim~\ref{c.single:out:C}~}$\Box$}
\end{proof}

\medskip

So we may fix $\Phi$ and $C_\eta$ as in the claim.
Now we distinguish two cases: 
\begin{description}
\item[Case 1] $\eta$ is a limit ordinal. In this case, let $c$ be $c_\eta$ and let $j^* = 0$.
\item[Case 2] $\eta$ is a successor ordinal. In this case, let $c$ be $s_{\eta-1}$ and let $j^* = 1$.
\end{description}
Finally we show
\begin{claim}\label{c.r.j}
It holds that $r = \Phi(c)$ and $j=j^*$.
\end{claim}
\begin{proof}[Proof of claim.]
Let $n \in \omega$ be arbitrary. Towards a contradiction, assume that either $j \neq j^*$ or $r(n) \neq \Phi(c)(n)$.
Fix $P_\theta$-names $\dot r$ and $\dot u$ for $r$ and $u$, and let $\dot c$ denote $\dot c_\eta$ if $j^* = 0$ and $\dot s_{\eta-1}$ if $j^* = 1$.
Let $p_0 \in G \cap P_\theta$ such that
if $j=j^*$,
\begin{equation}\label{e.diff-at-n}
p_0 \forces \dot r(n) \neq \Phi(\dot c)(n).
\end{equation}
For any $p \leq p_0$, we can find $\zeta < \kappa$ and $p' \leq p$ as in Lemma~\ref{no:unwanted:branches}; 
so by a density argument, we may find $p' \in G$ and $\zeta< \kappa$ such that for any  $\sigma \is \Phi(C_\eta)$ of length greater than $\zeta$, \eqref{e:no:unwanted:branches} holds.
By choosing $\sigma$ long enough, we can assume \eqref{single:out:C}  holds, as well. 

As $\sigma \is \Phi(C_\eta)$, by \eqref{single:out:C} and by \eqref{e.diff-at-n},
\[
p' \forces (\sigma,n,\dot r(n),j) \notin \dot I_\theta.
\]
So from \eqref{e:no:unwanted:branches} we infer
\[
p' \forces L[\dot r, \dot u]\vDash T(\sigma,n,\dot r(n),j)\text{ is Suslin,}
\]
which contradicts \eqref{u:C} (remembering $p'\in G$ and $\sigma \is C = \Phi(C_\eta)$).
\renewcommand{\qedsymbol}{{\tiny  Claim~\ref{c.r.j}~}$\Box$}
\end{proof}
\medskip

This finishes the proof of Lemma~\ref{coding:survives} (under the assumption that Lemma~\ref{no:unwanted:branches} holds): 
By the last claim, in Case 1 we have $r \in \Gamma^0$ and in Case 2 we have $r \in \Gamma^1$.
\renewcommand{\qedsymbol}{{\tiny  Lemma~\ref{coding:survives}~}$\Box$}
\end{proof}

\subsection{Cutting Out a Tree}

We now prove Lemma~\ref{no:unwanted:branches} which expresses that no (or few) unwanted branches are constructible from reals.
For this, we first make a definition and show several preliminary lemmas; 
Lemma~\ref{no:unwanted:branches} will be proved as Corollary~\ref{c:no:unwanted:branches}.

\begin{dfn}\index[notation]{Z(p)@$Z(p)$} Let $p \in P_\kappa$.
\begin{enumerate}
\item We denote by $Z(p)$ the smallest set $Z$ such that 
\begin{itemize}
\item $p \in Z$,
\item if $p' \in Z$ and $\xi < \lh(p')$ then $\pi_\xi(p') \in Z$,
\item if $\bar p \in Z$ and $\lh(\bar p)$ is a type-1 amalgamation stage and $k \in \Int$, then
$\bar p(k)^P \in Z$;
\end{itemize}
\item Suppose $\nu \in J$.
We say $p$ \emph{strongly rejects $\nu$}\index{strongly rejects}\index{rejects|see {strongly rejects}} to mean the following:
\[
\forall p' \in Z(p) \; p' \forces \nu \notin \dot I_{\lh(p')}.
\]
\end{enumerate}
\end{dfn}

\pagebreak[3]

\begin{rem}~
\begin{enumerate}
\item This definition is more straightforward than it looks; we shall now give some indication of how it is used.

Suppose $p \forces \nu\notin \dot I_{\lh(p)}$ and $\xi < \lh(p)$. 
Then $\pi_\xi(p)\forces \nu \notin \dot I_\xi$.
Also, if $\xi $ is an amalgamation stage and $\bar p = \pi_{\xi+1}(p)$, $\dot I_{\xi+1}$ was defined so that
for each $i\in\Int$, $\bar p(i) \forces \nu\notin \dot I_{\xi}$. 
But in general, $\bar p(i)^P \not\forces \nu\notin \dot I_{\xi}$!

Now suppose we want to show some set $D$ is dense below 
\[
\bv{\nu\notin \dot I_\xi}^{\ro(P_\xi)}.
\]
So given $p \in P_\xi$ such that $p \forces \nu\notin \dot I_\xi$, we must find $q \in D$, $q \leq p$.
Of course, we construct $q$ by induction on $\xi$ (or on the length of $q$). 

At amalgamation stages, our only hope to find $\bar q \leq \bar p = \pi_{\xi+1}(p)$ is to use direct extension, so let's assume $D$ is even $\leqlol$-dense for some large enough $\lambda$.
For $i\neq0$, 
since it is possible that $\bar p(i)^P \not\forces \nu\notin \dot I_{\xi}$, there is no reason we should be able to find a \emph{direct extension} of $\bar p(i)^P$ in $D$. 
\todo{Maybe find $q \leqlol \bar p(i)^P$ such that $q \leq \bar p(i)$? I suppose this could be arranged quite easily!}
Hence, the natural induction breaks down.
Of the many feasible solutions, one is simply to require that $p$ strongly reject $\nu$.
\item If $p \forces \nu\notin  \dot I_\xi$, we can strengthen $p$ to strongly reject $\nu$, and we shall do so in the proof of Lemma~\ref{no:unwanted:branches}.
\item Equivalently, $p$ strongly rejects $\nu \in J$ if and only if
for any $n \in \omega$, any $\bar \xi \colon n+1 \to \kappa$ 
and any $\bar k\colon n+1 \to \Int$,
it holds that
\[
p_n \forces \nu \notin  \dot I_{\lh(p_n)}
\]
where $p_0, \dots, p_n$ is defined by recursion as follows: let $p_0 = p$ and
\[
p_{l+1} = \begin{cases}
\Phi^{\bar k(l)}(\pi_{\bar \xi (l)}(p_l))^P &\text{if $\xi(l)$ is an amalgamation stage, i.e., $\xi \in E^3$;}\\
\pi_{\bar \xi (l)}(p_l) &\text{otherwise.}
\end{cases}
\]
\end{enumerate}
\end{rem}

\medskip

Next we show that, given $\nu\in J$ and $q \in P_\kappa$ which strongly rejects $\nu$, $P_{\lh(q)}$  can be written as a product $T' \times P'$, where $T'= T(\nu)(\leq t)$ for some $t \in T(\nu)$. 
As $T'$ doesn't add reals, we then finish the argument by showing  that $P'$ doesn't add a branch through $T(\nu)$. 

The definition of the $P'$ would be  straightforward (simply omit $\nu$ everywhere), were it not for the fact that we need certain parameters to match their respective counterparts in $P_\kappa$; the parameters concerned are the reals given by bookkeeping, the auxiliary sets $G^o_\xi$, and $q$ itself.

\begin{lem}\label{l:pnu:IH} Suppose $\nu \in J$ and momentarily write $T$ to mean $T(\nu)$.
For each  $q \in P_\kappa$ which strongly rejects $\nu$ and each $\xi < \kappa$ we define a preorder $\pnu$, $\tm \in T$ and 
\begin{equation}\label{e.dense.embedding}
\emb \colon T(\leq\tm) \times \pnu \to P_\xi(\leq\tm\cdot \pi_\xi(q))
\end{equation}
such that the following hold:
\begin{enumerate}[label=\Roman*., ref=\Roman*]
\item\label{l:IH:order}\label{l:IH:first} For every $t, s \in T(\leq \tm)$ and $r, p \in \pnu$ we have
\[
\emb(t,r) \leq_{P_\xi} \emb(s,p) \iff \big( s \leq_{T} t \wedge r \leq_{\pnu} p)
\]
\item\label{l:IH:surj} For any $p \in P_\xi(\leq \tm\cdot q)$ there is $t \in T(\leq \tm)$ and $p^\nu \in \pnu$ such that
\[
\emb(t,p^\nu)=t \cdot p.
\]
If furthermore $r \in \pnu$ and $t \cdot p \leqlol \emb(t,r)$ (for $\lambda \in [\lambda_\xi,\kappa]$) we can in addition demand that $p^\nu$ is chosen so that $p^\nu \leqlol r$ (with respect to direct extension in $\pnu$).
\item\label{l:IH:proj} If $\xi' < \xi$, there is a strong projection $\pi^{\nu,q}_{\xi'}$ from  $\pnu$ to $\pmao{\xi'}$, and 
for all $(t,p) \in \bar T(\leq\tmao{\xi'})\times \pnu$ 
\begin{equation}\label{commute}
\emao{\xi'}(t, \pi^{\nu,q}_{\xi'}(p)) = \pi_{\xi'}(\emb(t, p)).
\end{equation}
\item\label{l:IH:strat} $\pnu$ is stratified on $[\lambda_\xi,\kappa]$ and $\kappa^{+}$-linked.
\item\label{l:IH:leqlo}\label{l:IH:last} Letting $\leqlo^{\nu,q,\lambda}_\xi$ denote direct extension for $\lambda \in [\lambda_\theta,\kappa]$ in the sense of the stratification of  $\pnu$,
for every $t, s \in T(\leq \tm )$ and $r, p \in \pnu$ we have
\[
\big( s \leq_{T} t \wedge r \leqlo^{\nu,q,\lambda}_\xi p \big)  \iff    \emb(t,r) \leqlo^\lambda_\xi \emb(s,p)
\]
\end{enumerate}
In particular, \eqref{e.dense.embedding} is a dense embedding of preorders and the preorder on the left in \eqref{e.dense.embedding} has the $\kappa^+$-cc.
\end{lem}
In order to prove the last lemma, we need a general observation concerning a product $T' \times P'$ with the $\alpha^+$-chain condition where $T'$ is $\alpha^+$-distributive (in our case, $\alpha=\kappa^+$).
It is clear that in such a case, all  functions from $\alpha$ into $\On$ added by $T'\times P'$ are already added by $P'$---we now show a slightly stronger property.
\begin{lem}\label{l:names}
For preorders $T'$ and $P'$ and $\alpha\in \On$, suppose $T'\times P'$ has the $\alpha^+$-cc and $T'$ is $\alpha^+$-distributive.
Further suppose that $(t,p)\in T'\times P'$ and $\dot x$ is a $T'\times P'$-name such that
\[
(t,p) \forces_{T'\times P'} \dot x\colon \alpha \to V 
\]
where $V$ of course refers to the ground model.
Then there is $t^* \in T'$, $t^*\leq t$ and a $P'$-name $\dot z$ such that
\[
(t^*,p) \forces_{T'\times P'} \dot x = \dot z
\]
where we identify $\dot z$ with $e_1(\dot z)$, for the map $e_1\colon P' \to T'\times P'$ given by $p \mapsto (1_{T'},p)$.
\end{lem}
\begin{proof}
Suppose $(t,p)$,  $\dot x$ are as in the hypothesis of the lemma.
Pick,  for each $\xi <\alpha$ a maximal antichain $\{ (t^\xi_\nu, p^\xi_\nu) \setdef \nu < \alpha\}$ and $x^\xi_\nu$ for each $\nu < \alpha$ such that 
\[
\forall \xi<\alpha \; \forall \nu<\alpha\; (t^\xi_\nu, p^\xi_\nu) \forces_{T'\times P'} \dot x(\xi) = \check x^\xi_\nu.
\]
For each pair $(\xi,\nu) \in \alpha^2$  let $D^\xi_\nu = \{ t \in T' \setdef t \leq t^\xi_\nu \vee t \perp t^\xi_\nu   \}$, a dense subset of $T'$. 
Fix $t^* \in \bigcap_{\xi,\nu<\alpha} D^\xi_\nu$
and for each $\xi < \alpha$, let $I_\xi = \{ \nu \setdef t^* \leq t^\xi_\nu \}$.

It is straightforward to check that for each $\xi < \alpha$ the set
$
\{ p^\xi_\nu \setdef \nu \in I_\xi \}
$
is a maximal antichain in $P'$.
Let $\dot z$ be the $P'$-name such that 
\[
\forall \xi< \alpha \; \forall \nu \in I_\xi \; p^\xi_\nu \forces_{P'} \dot z(\xi) = \check z^\xi_\nu.
\] 
Clearly $(t^*,p) \forces_{T'\times P'} e_1(\dot z)= \dot x$.
\end{proof}

Now we are ready to  deconstruct $P$ as product, below certain conditions.
\begin{proof}[Proof of Lemma~\ref{l:pnu:IH}.]
Fix  $\nu \in J$.	
We define $\emb$, $\pnu$, and $\tm$ by induction on $\xi < \kappa$, simultaneously for all $q \in P_\kappa$ which strongly reject $\nu$.
Of course for fixed $\xi$, these only depends on $\pi_\xi(q)$; yet we allow arbitrary $q \in P_\kappa$ in the upper indices of $\emb$, $\pnu$, and $\tm$ to simplify notation.
We show Items~\ref{l:IH:first}--\ref{l:IH:last} listed in the lemma by induction on $\xi$, simultaneously with our inductive definition.

\medskip

\textbf{Induction start.} Let 
\[
\pmao{0} = \bar T^{\nu} ( \leq \pi_0(q)),
\] 
where
\[
\bar T^{\nu} = \{ t \in \bar T \setdef t(\nu) = \pi_0(q)(\nu) \}.
\]
Moreover, define $\tmao{0} = \pi_0(q)(\nu)$ and, recalling that $\bar T = P_0$, let $\emao{0}$ be the obvious isomorphism
\[
\emao{0} \colon T(\leq \tmao{0}) \times \pmao{0} \to \bar T(\leq \tmao{0})
\]

From now on, the inductive definition of $\pnu$ is precisely that of $P_\xi$ but with two qualifications:
Firstly, we want to make sure we use the ``same'' bookkeeping device and the same sets $G^{\circ}_\xi$, and that we work below $q$; for this, we have to ``line up'' the two iterations step by step.
This makes it necessary to introduce the $\tm$.
Secondly, we forgo all mention of $\nu$.

\medskip

\textbf{Inductive step.}  For the inductive step, suppose we have 
\[
\emb\colon T(\leq \tm) \times \pnu \to P_\xi(\leq \tm\cdot \pi_\xi(q)) 
\] 
We quickly treat stages $\xi \in E^0 \cup E^1$ (the interesting part of the induction will be amalgamation and coding stages).
For such $\xi$, let $\tmao{\xi+1}=\tm$, and let
\[
\pmao{\xi+1} = \begin{cases}
\pnu \times \big( \Coll(\omega,\lambda_{\xi+1}) (\leq q(\xi))\big)  & \text{if $\xi \in E^0$}\\
\pnu \times  \big(\Add(\kappa) (\leq q(\xi))\big)^L & \text{if $\xi \in E^1$}.
\end{cases}
\]
We leave it to the reader to find the obvious maps $\emao{\xi+1}$. In the case $\xi = E^1(\eta)$, write $\dot C^{\nu,q}_\eta$ for the $\pmao{\xi+1}$-name of the $(\Add(\kappa))^L$-generic.

\medskip

\textbf{Inductive step for coding stages.} Next, suppose $\xi = E^2(\eta)$, i.e., $\xi$ is a coding stage.
Since the real we want to code---call it $c$---and $G^o_\xi$ from the definition of $P_{\xi+1}$ and $q(\xi)$ can be viewed as functions from $\kappa^+$ into $\On$, using the induction hypothesis and Lemma~\ref{l:names} we can find 
$\tmao{\xi+1} \leq \tm$ in $T$ and $\pnu$-names $\dot c^{\nu,q}$, $\dot G^{\nu,q}_\xi$, and $q^{*}(\xi)$ such that
\[
\tm \cdot \pi_\xi(q) \forces_{P_\xi} e_2(\dot c^{\nu,q}) = \dot c, e_2(\dot G^{\nu,q}_\xi)= \dot G^o_\xi  \text{ and } e_2(q^{*}(\xi))=q(\xi)
\]
where in the above, $\dot c$ and $\dot G^o_\xi$ are names for $c$ and $G^o_\xi$, and $e_2$ denotes the complete embedding $e_2\colon \pnu\to P_\xi(\leq \tm\cdot \pi_\xi(q))$ given by $p \mapsto (\tmao{\xi},p)$ (which we may use to translate names).

We closely follow the definition of $P_{\xi+1}$ (compare p.~\pageref{def.B.B.minus.etc}) and define $\pnu$-names: 
Define $\dot{\bar B}^{-,\nu,q}$ from $\dot c^{\nu,q}$, $\dot C^{\nu,q}_\eta$ and the $\bar T^{\nu,q}$ -generic added by $\pmao{0}$ just as $\bar B^-$ was defined from $c$, $C_\eta$ and $\bar B$ . 
This is well-defined since
\[
\forces_{T(\leq \tm) \times  \pnu} \;\check \nu \notin {\emb}^{-1}(\dot I_\xi)
\]
and hence $B(\nu)$ is never referenced.
Define $\dot A^{\nu,q}_\xi$ just as we defined $A_\xi$ but with $\bar B^-$ and $G^o_\xi$ replaced by $\dot{\bar B}^{-, \nu,q}$ and $\dot G^{\nu,q}_\xi$, respectively. 
Finally define 
$\pmao{\xi +1} = \pnu * \dot Q^{\nu,q}_\xi$, where 
\[
\forces_{\pnu} \dot Q^{\nu,q}_\xi = J(\dot A^{\nu,q})^{L[\dot G^{\nu,q}_\xi]} (\leq q^{*}(\xi)).
\]
Given $t \in T(\leq \tm)$ and $(p,\dot p) \in \pmao{\xi +1}$, 
define
\[
\emao{\xi+1}(t, (p,\dot p)) = (\emb(t,p), e_2(\dot p)),
\]
again using the compete embedding $e_2\colon \pnu\to P_\xi(\leq \tm\cdot q)$ defined above to translate the $\pnu$-name $\dot p$ into a $P_\xi(\leq \tm\cdot \pi_\xi(q))$-name in the second component. 
\begin{claim}\label{c.IH.E2}
Items~\ref{l:IH:first}--\ref{l:IH:last} of Lemma~\ref{l:pnu:IH} hold when $\xi \in E^2$.
\end{claim}
\begin{proof}[Proof of claim.]
By the definitions, Items~\ref{l:IH:order} and \ref{l:IH:proj} are obvious.
Item~\ref{l:IH:strat} is proved verbatim as for $P_{\xi+1}$ when $\xi \in E^2$ (see Item~\ref{P:strat:ext} of Lemma~\ref{iteration:prop}, p.~\pageref{P:strat:ext}).
Item~\ref{l:IH:leqlo} for $\xi+1$ straightforwardly follows from the inductive hypothesis that it holds for $\xi$. 
We leave further details to the reader.

We show Item~\ref{l:IH:surj} in some detail. %
Let $p \in P_{\xi+1}(\leq \tm\cdot q)$, so that we may write $p = (\pi_\xi(p), p(\xi))$.
By induction hypothesis, forcing with $P_{\xi}(\leq \tm\cdot \pi_\xi(q))$ is the same as forcing with 
$T(\leq \tm) \times  \pnu$, so we may view $p(\xi)$ as a name in the latter forcing and observe
that 
\[
\forces_{T(\leq \tm) \times  \pnu} p(\xi) \in J(\dot A^{\nu,q})^{L[\dot G^{\nu,q}_\xi]}
\]
We may assume by slightly mangling $p(\xi)$ that 
\[
\forces_{T(\leq \tm) \times  \pnu} p(\xi) \in \dot Q^{\nu,q}_\xi
\]
By the induction hypothesis we may find 
$t_0 \in T(\leq \tm)$ and $p^\nu_\xi \in \pnu$ such that $\emb(t_0,p^\nu_\xi) = t_0 \cdot \pi_\xi(p)$.
As $p(\xi)$ can be viewed as a name for a function from $\kappa^+$ into $\On$, by Lemma~\ref{l:names}
we can find $t \in T(\leq t_0)$ and a $\pnu$-name which we denote by $p^\nu(\xi)$ such that
$(t, p^\nu_\xi) \forces p(\xi)= e_2(p^\nu(\xi))$.
We conclude
$\emao{\xi+1}(t, (p^\nu, p^\nu(\xi)) = t \cdot p$.
The additional fact about direct extension is proved similarly and is left to the reader.
\renewcommand{\qedsymbol}{{\tiny  Claim~\ref{c.IH.E2}~}$\Box$}
\end{proof}

\medskip

\textbf{Inductive step for amalgamation stages.} Now suppose $\xi = E^3(\eta)$ where $\eta=\alpha^0_\rho$, i.e., $\xi$ is a type-1 amalgamation stage.
Recall the bookkeeping device $\bar r$ gives us two names $\dot r^0_\rho$ and $\dot r^1_\rho$ which are fully random over $L^{P_{\iota}}$, where $\iota = \bar \iota(\rho)$, 
and $P_{\xi+1} = \am(P_\iota, P_\xi, f)$ where $f$ is a partial automorphism of $\ro(P_\xi)$ coming from the pair $\dot r^0_\rho$, $\dot r^1_\rho$.

Let $\bar q = \pi_{\xi+1} (q)$, so that $\bar q \in \am(P_{\iota}, P_\xi, f, \lambda_\xi)$, and for
brevity write $q_i$ instead of $\bar q(i)^P$ when $i \in \Int$.
Now for all $i\in\Int$, $q_i$ strongly rejects $\nu$ (this is precisely why we use this notion).
Using the induction hypothesis at each step, recursively pick a descending sequence $t_1 \geq t_2 \geq \hdots$ in $T$ such that
$t_0 \leq \tm$ and so that for each $n \in \omega$ and for
\[
i= (-1)^{n \bmod 2} \left\lceil \frac{n}{2} \right\rceil
\]
it holds that
\[
\emao[t_n\cdot q_i]{\xi}\colon T(t_n) \times \pmao[t_{n-1} \cdot q_i]{\xi} \to P_\xi (\leq t_n\cdot q_i).
\]
Let $t_\omega$ be a lower bound of $(t_n)_{n\in\omega}$ in $T$. 

Restrict $\emao[t_n\cdot q_i]{\xi}$ to obtain a dense embedding
\[
e^i\colon T(\leq t_\omega) \times \pmao[t_{n-1} \cdot q_i]{\xi} \to P_\xi (\leq t_\omega \cdot q_i).
\]
We write $P_i$ for $\pmao[t_{n-1} \cdot q_i]{\xi}$ and $e^i_2$ for the complete embedding
\[
e^i_2\colon P^i \to P_\xi (\leq t_\omega\cdot q_i)
\]
given by 
\[
p \mapsto e^i(t_\omega,p).
\]

As we did above for $\dot c$ and $\dot G^o_\xi$, by Lemma~\ref{l:names} we can chose $\tmao{\xi+1} \leq t_\omega$ and for each $i\in \Int$,
$P^i$-names $\dot r^{i,0}$ and $\dot r^{i,1}$ so that for each $j\in\{0,1\}$,
\[
\tmao{\xi+1}\cdot q_i \forces_{P_\xi} e^i_2(\dot r^{i,j}) = \dot r^j_\rho.
\]
Define the complete Boolean algebra $B^i_j$ to be $\ro(P^i)*\langle \dot r^{i,j} \rangle^{\dot P}$ where 
\[
\forces_{\pmao[q_i]{\iota}} \dot P=  P^i \colon \pmao[q_i]{\iota}
\]
Let $\pi_{i,j}$ be the canonical projection from 
\[
\pi_{i,j} \colon \ro(P^i) \to B^i_j
\]
for $i\in \Int$ and $j\in \{0,1\}$. 
Let
\[
f_i \colon B^i_0\to B^i_1
\]
be the isomorphism of the Boolean algebras such that the induced map on names sends $\dot r^{i,0}$ to $\dot r^{i,1}$ for each $i\in \Int$.
Write $\pi_i$ for the projection from $\ro(P^i)$ to $\ro(\pmao[q_i]{\iota})$---noting that its restriction to $P^i$ is the strong projection claimed to exist in Item~\ref{l:IH:proj} of Lemma~\ref{l:pnu:IH}. 

\medskip

Define %
\[
\blowup{P}^{\nu,q_i}_\xi = \{ (p,b^0,b^1)  \in P^i \times B^i_0 \times B^i_1 \setdef \pi_i (p \cdot b^0 \cdot b^1) =\pi_i (p)  \}.
\]
Let  $\Dam^i$ be the set of 
$p \in P^i$ such that %
for all $q \in P^i$, if $q \leqlo^{\lambda_\xi} p$ we have
\begin{equation}\label{lower:part:frozen:nice}
\forall (b_0,b_1) \in B^i_0 \times B^i_1 \quad \big(\pi_i(q)\cdot p \cdot b_0 \cdot b_1 \neq 0\big) \Rightarrow \big(q\cdot b_0 \cdot b_1 \neq 0\big)
\end{equation}
That is, $\Dam^i$ is defined just like $\Dam$ in Definition~\ref{am:gen:D}, but with $P$ replaced by $P^i$, $B_0$ by $B^i_0$ and $B_1$ by $B^i_1$.
\begin{claim}\label{c:D}
Let $\Dam = \Dam(P_\iota, P_\xi, f)$.
For all $i\in\Int$ and $p \in P^i$, the following are equivalent:
\begin{enumerate}[label=\textup{(\alph*)},ref=\alph*]
\item\label{item.forall} $\forall t \in T(\leq\tmao[q_i]{\xi}) \; \emao[q_i]{\xi}(t,p) \in \Dam$.
\item\label{item.exists} $\exists t \in T(\leq\tmao[q_i]{\xi}) \; \emao[q_i]{\xi}(t,p) \in \Dam$;
\item\label{item.p} $p \in \Dam^i$; 
\end{enumerate}
\end{claim} 
\begin{proof}[Proof of claim.] 
Fix $i\in \Int$ and write $\leqlo$ for for the notion of direct extension for $\lambda=\lambda_{\xi}$ from the stratification of $P_\xi$ and write $\leqlo_\nu$ for the analogous notion of direct extension for $\lambda=\lambda_\xi$ from the stratification of $P^i$.
The implication \eqref{item.forall}$\Rightarrow$\eqref{item.exists} is obviously trivial.

\medskip

For \eqref{item.exists}$\Rightarrow$\eqref{item.p}, suppose that $(t,p) \in T(\leq\tmao[q_i]{\xi})\times P^i$ and $e^i(t,p) \in \Dam$, and show $p \in \Dam^i$.
So let $p' \leqlo_{\nu} p$ and $b^{\nu,j} \in B^j_i$, for each $j\in\{0,1\}$  be given, satisfying 
$\pi(p') \cdot b^{\nu,0} \cdot b^{\nu,1} \neq 0$ in $\ro(P^i)$.
We must show $p' \cdot b^{\nu,0} \cdot b^{\nu,1} \neq 0$.
Let $b^j$ be the name which $b^{\nu,j}$ gives rise to via $e^i_2$, and
observe $e^i(t,p') \leqlo e^i(t,p)$.
Since $e^i(t,p) \in \Dam$ we have  $e^i(t,p') \cdot b^0 \cdot b^1 \neq 0$ and so we pick
$(t_r,r)$ such that $e^i(t_r,r) \leq e^i(t,p') \cdot b^0 \cdot b^1 \neq 0$.

For arbitrary $j\in\{0,1\}$, $e^i(t_r,r) \leq b^j$
means
\[
e^i(t_r,r) \forces_{P_\xi} \dot r^j \in b^j
\]
but since $\dot r^j$ and $b^j$ are just $P^i$-names embedded into $P_\xi$ via $e^i_2$, and the latter are names for Borel sets,
\[
e^i_2(r) \forces_{P_\xi} \dot r^j \in b^j
\]
using absoluteness of $\mathbf{\Pi}^1_1$ formulas between the $P^i$-extension and the $(T(\leq\tmao[q_i]{\xi}) \times P^i)$-extension.
So $r \leq  p' \cdot b^{\nu,0} \cdot b^{\nu,1}$, 
showing that $p' \cdot b^{\nu,0} \cdot b^{\nu,1} \neq 0$ and finishing the proof that $p \in \Dam^i$ and hence of the claim.
\renewcommand{\qedsymbol}{{\tiny  Claim~\ref{c:D}~}$\Box$}

\medskip

For \eqref{item.p}$\Rightarrow$\eqref{item.forall}, suppose that $(t,p) \in T(\leq\tmao[q_i]{\xi})\times P^i$ and $p \in \Dam^i$, and writing $p^*$ for $e^i(t,p)$, show $p^* \in \Dam$. 
So suppose we are given $r^* \leqlo p^*$ and $b^j \in B^j$, for each $j\in \{0,1\}$, such that $\pi(r^*) \cdot p^* \cdot b^0 \cdot b^1 \neq 0$; 
we must show that $r^*\cdot b^0 \cdot b^1 \neq 0$.

By Item~\ref{l:IH:surj}, find $(t_r,r)$ such that $e^i(t_r,r) = r^*$.
We may in addition demand that $r \leqlo_\nu p$ by Item~\ref{l:IH:surj}.

To witness the compatibility, find a condition $d^{*} \leq \pi(r^*) \cdot p^*\cdot  b^0 \cdot b^1$ in $P_\xi$, and find
$(t_{d}, d) \in T \times \pmao[q_i]{\xi}$ such that  $e^i(t_{d}, d) = d^{*}$.
Using Lemma~\ref{l:names}, find $t_* \in T$ stronger than $t_d$ and $\pmao[q_i]{\xi}$-names $b^{\nu,j}$ such that
$(t_*,r) \forces b^j = e^i_2(b^{\nu,j})$ for each $j\in\{0,1\}$.
We may chose these names so that $d \leq b^{\nu,0} \cdot b^{\nu,1}$. 

So finally, we have $r \leqlo_{\nu} p$ and $b^{\nu,j}$ for each $j\in\{0,1\}$, 
satisfying $\pi(r) \cdot b^{\nu,0}\cdot b^{\nu,1} \neq 0$ as witnessed by $d$.
Thus, since $p \in \Dam^i$, 
$r \cdot b^{\nu,0}\cdot b^{\nu,1} \neq 0$,
whence also $e^i(t_*,r) \cdot b^0 \cdot b^1 \neq 0$ and so,
as $e^i(t_*,r) \leq r^*$, we have shown $r^* \cdot b^0 \cdot b^1 \neq 0$. This concludes the proof that $p^*\in \Dam$.
\end{proof}

Finally, let $\pmao{\xi+1}$ be the set of $\bar p \colon\Int \to \blowup{P^i}$ such that the following conditions are met.
\begin{enumerate}
\item For all $i \in \Int$, $\pi_i(\bar p(i)^P) = \pi_i(\bar p(0)^P)$.
\item For all but finitely many $i\in\Int$, $\bar p(k)^P \leqlo^{\lambda_0} \pi_i(\bar p(k)^P)$.
\item For all $i \in \Int\setminus\{-1,0\}$, $f_i(\pi_{i,0}(\bar p(i)))=\pi_{i+1,1}(\bar p(i+1))$.
\item $\bar p(0) \in \pmao[q_0]{\xi}$, i.e., $\bar p^0(0)=\bar p^1(0)=1$ and 
\begin{align}
f_{-1}(\pi_{-1,0}(\bar p(-1)))& \geq  \pi_{0,1}(\bar p(0)),\\
f_0(\pi_{0,0}(\bar p(0)) & \leq \pi_{1,1}(\bar p(1)).
\end{align}
\item For $i \in \Int\setminus\{0\}$, $\bar p(i)^P \in \Dam^i$.
\end{enumerate}
Despite the notational complexity, the entire previous construction is parallel to Chapter~\ref{sec:amalgamation} (inserting indices ``$i$'' or ``$\nu, t_\omega\cdot q_i$'' where necessary).

\medskip

For $(t, \bar p^\nu) \in T(\leq \tmao{\xi+1}) \times \pmao{\xi+1}$ define $\emao{\xi+1}(t,\bar p^\nu)$ to be the $\Int$-sequence given by
\[
i \mapsto (e_i(t, \bar p^\nu(i)^P), e^i_2(\bar p^\nu(i)^0), e^i_2(\bar p^\nu(i)^1)). 
\]
It remains to verify the induction hypothesis:
\begin{claim}\label{c.IH.type-1}
Items~\ref{l:IH:first}--\ref{l:IH:last} of Lemma~\ref{l:pnu:IH} hold when $\xi$ is a type-1 amalgamation stage.
\end{claim}
\begin{proof}[Proof of claim.]
We verify Item~\ref{l:IH:surj}. 
Let $\bar p \in \am(P_\iota, P_\xi, f)$ and $\bar p \leq \tmao{\xi+1}\cdot \bar q$ (remember $\bar q =\pi_{\xi-1}(q)$).
Let $t_0 = \pi_0(\bar p(0))$.
As for each $i\in \Int\setminus\{0\}$ and each $j\in \{0,1\}$,
$\bar p(i)^j$ can be viewed as a name for a function from $\kappa^+$ into $V$ (e.g., viewing it as a name for a Borel code), 
As $T$ is $\sigma$-closed we can find $t$ and for each $i \in \Int$, $p^\nu_i \in P^i$ satisfying 
\[
e^i(t, p^\nu_i) = t \cdot \bar p(i)^P.
\]
and for each $i \in \Int\setminus\{0\}$, a pair of $P^i$-names $b^{i,0}$, $b^{i,1}$ such that
\[
\forall j\in \{0,1\} \; e^i(t, p^\nu_i) \forces  e^i_2(b^{i,j}) = \bar p(i)^j
\]
Let $\bar p^\nu$ be the sequence $i \mapsto (p^\nu_i, b^{i,0}, b^{i,j})$.
By construction and by Claim~\ref{c:D}, $\bar p^\nu \in \pmao{\xi+1}$ and $\emao{\xi+1}(t,\bar p^\nu) = t \cdot \bar p$.
The additional requirement in Item~\ref{l:IH:surj} concerning direct extension is proved similarly.

Item~\ref{l:IH:leqlo} for $\xi+1$ straightforwardly follows from the induction hypothesis.
Items~\ref{l:IH:order} and \ref{l:IH:proj} are almost immediate and are left to the reader.
Lastly, Item~\ref{l:IH:strat} (that $\pmao{\xi+1}$ is stratified) is proved verbatim as in Chapter~\ref{sec:amalgamation}.
\renewcommand{\qedsymbol}{{\tiny  Claim~\ref{c.IH.type-1}~}$\Box$}
\end{proof}

\medskip

In case $\xi \in E^3(\eta)$ for $\eta=\alpha^\zeta_\rho$ with $\rho>0$, that is, at type-2 amalgamation stages,
as usual the construction closely resembles the previous. We leave this to the reader.

\medskip

\textbf{Limit step.} Let $\theta<\kappa$ be a limit ordinal. Let $\tmao{\theta}$ be the greatest lower bound in $T$ of $(\tm)_{\xi < \theta}$ ($\theta < \kappa$ and $T$ is even $\kappa^+$-closed).
Let $\pmao{\theta}$ be the $\bar \lambda$-diagonal support limit of the iteration $(\pnu)_{\xi < \theta}$.
The map $\emao{\theta}$ is given its natural definition:
by induction hypothesis and Item~\ref{l:IH:proj}, $(\emb(t, \pi^{\nu,q}_\xi(p)))_{\xi<\theta}$ forms a thread; we define $\emao{\theta}(t,p)$ to be this thread. 
It clearly satisfies the $\lambda$-diagonal support condition, as $p$ did in $\pmao{\theta}$.

Item~\ref{l:IH:strat} (stratification) is proved verbatim as for $P_\xi$.
Items~\ref{l:IH:order},~\ref{l:IH:proj} and~\ref{l:IH:leqlo} follow by induction; we leave details to the reader. 
 
We give some details for Item~\ref{l:IH:surj}.
Suppose $p \in P_\theta(\leq \tmao{\theta})$.
Find by recursion on $\xi < \theta$ a descending sequence $(t_\xi)_{\xi < \theta}$ from $T$ and conditions $p^\nu_\xi \in \pnu$ such that for each $\xi < \theta$
\[
\emb(t_\xi,p^\nu_\xi) = t_\xi \cdot \pi_\xi(p).
\]
Let $t_\theta$ be the greatest lower bound in $T$ of $(t_\xi)_{\xi < \theta}$.
By Item~\ref{l:IH:proj},
$(p^\nu_\xi)_{\xi<\theta}$ represents a thread in the limit $\pmao{\theta}$.
It follows that, letting $p^\nu_\theta$ be this thread,
\[
\emao{\theta}(t_\theta, p^\nu_\theta) = t_\theta \cdot p.
\]
This finishes our definition of $\pnu$, $\tm$ and the dense embeddings $\emb$ and the proof of the lemma.
\renewcommand{\qedsymbol}{{\tiny  Lemma~\ref{l:names}~}$\Box$}
\end{proof}

We need one last lemma.
\begin{lem}\label{l.lastlemma}
Whenever $q \in P_\kappa$, $\nu \in J$ and $q$ strongly rejects $\nu$, the forcing $\pmao{\lh(q)}$ from the previous lemma does not add a branch through $T(\nu)$ (i.e., through $T$).
\end{lem}
\begin{proof}
Let $\theta$ denote $\lh(q)$ and assume to the contrary that $\dot B$ is a $\pmao{\theta}$-name for a branch through $T(\nu)$.
For each $\xi < \kappa^{++}$, let $\dot b_\xi$ be a $\pmao{\theta}$-name for the element of $\dot B$ at level $\xi$ in $T(\nu)$ and fix a $\pmao{\theta}$-name $\dot t_\xi$
for an immediate successor in $T(\nu)$ of $\dot b_\xi$ which does \emph{not} lie on the branch $\dot B$.
We have constructed a name for an antichain of size $\kappa^{++}$ in $\check T(\nu)$:
\[
\forces_{\pmao{\theta}} \forall \xi,\xi' <\kappa^{++} \; \big ( \xi \neq \xi' \Rightarrow \dot t_\xi \perp \dot t_{\xi'} \text{ in $\check T(\nu)$}\big)
\] 
Momentarily, let $\Clink^{\kappa^+}$ denote the linking relation of $\pmao{\theta}$ at $\lambda=\kappa^+$. \todo{Mention this in earlier chapter and refer here!}

We use distributivity of the ``upper part'' w.r.t.\ $\kappa^+$, i.e., the trees:
For each $\xi$, find a condition $p_\xi \in \dom(\Clink^{\kappa^+})$ and $t_\xi \in T(\nu)$ so that
\begin{equation}\label{e:t:tcheck}
p_\xi \forces_{\pmao{\theta}} \dot t_\xi = \check t_\xi.
\end{equation}
There is a set $I$ of size $\kappa^{++}$ such that $\Clink^{\kappa^+}(p_\xi)\cap \Clink^{\kappa^+}(p_{\xi'}) \neq \emptyset$ for all $\xi, \xi' \in I$.
For each $\xi \in I$, define $\bar t_\xi \in \bar T$ as follows: 
\begin{gather*}
\bar t_\xi \res (\kappa\setminus\{\nu\}) = \pi^{\nu,q}_0(p_\xi),\\
\bar t_\xi (\nu) = t_\xi.
\end{gather*}
\begin{claim}\label{c.antichain}\index{independent!Suslin trees}\index{Suslin trees!independent}
The set $\{ \bar t_\xi \setdef \xi \in I \}$ is an antichain of size $\kappa^{++}$ in $\bar T$.
\end{claim}
\begin{proof}[Proof of claim.]
Suppose  $\xi, \xi' \in I$ and $\xi \neq \xi'$; we show $\bar t_\xi$ and $\bar t_{\xi'}$ are incompatible in $\bar T$.
We are done if $\pi^{\nu,q}_0(p_\xi)$ and $\pi^{\nu,q}_0(p_{\xi'})$ are incompatible elements of $\bar T^\nu$.
If on the other hand $\pi^{\nu,q}_0(p_\xi)$ and $\pi^{\nu,q}_0(p_{\xi'})$ are compatible in $\bar T^\nu$, since 
$\Clink^{\kappa^+}(p_\xi)\cap \Clink^{\kappa^+}(p_{\xi'}) \neq \emptyset$, it must be the case that $p_\xi$ and $p_{\xi'}$ are compatible in $\pmao{\theta}$.
But then by \eqref{e:t:tcheck} and the corresponding fact for $\xi'$, $t_\xi$ and $t_{\xi'}$ are incompatible in $T(\nu)$.
Thus, $\bar t_\xi \perp \bar t_{\xi'}$  in $\bar T$. 
\renewcommand{\qedsymbol}{{\tiny  Claim~\ref{c.antichain}~}$\Box$}
\end{proof}
\noindent
As $\bar T$ has the $\kappa^{++}$-cc, we have reached a contradiction, proving the lemma.
\renewcommand{\qedsymbol}{{\tiny  Lemma~\ref{l.lastlemma}~}$\Box$}
\end{proof}

\medskip

Finally, we supply the last step in the proof that $\Gamma^0$ is $\Delta^1_3$.
\begin{cor}\label{c:no:unwanted:branches}
Lemma~\ref{no:unwanted:branches} holds: I.e., supposing $p \in P_\kappa$ and $(n,i,j) \in \omega^2\times 2$,
there is $\zeta<\kappa$ and $p' \in P_{\lh(p)}$ stronger than $p$ such that for any $\sigma \in {}^{<\kappa} 2$ with $\lh(\sigma) > \zeta$, 
\begin{equation}\label{e:no:unwanted:branches:repeat}
p' \forces_{P_{\lh(p)}}\text{``}(\sigma,n,i,j) \notin \dot I_{\lh(p)} \Rightarrow  \forall u \in \omega^\omega\;  L[u,r]\vDash T(\sigma,n,i,j) \text{ is Suslin.''}
\end{equation}
\end{cor}
\begin{proof}
Fix $p \in P_\kappa$ and $(n,i,j) \in\omega^2\times 2$, and let $\theta =\lh(p)$. 
Find $p' \in P_\theta$ stronger than $p$ such that the following holds:
For any $m,m' \in \omega$ and any pair of distinct index sequences  $\xi_0, \dots, \xi_m, k_1, \dots, k_m$ and $\xi'_0, \dots, \xi'_m, k'_1, \dots, k'_{m'}$ (as defined in Section~\ref{sec:reals:areas}),
\[
\Phi^{k_m}_{\xi_m} \circ \dots \Phi^{k_1}_{\xi_1} (p') (E^1(\xi_0)) \perp \Phi^{k'_m}_{\xi'_m} \circ \dots \Phi^{k'_1}_{\xi'_1} (p') (E^1(\xi'_0))
\]
in $\Add(\kappa)^L$ whenever both expressions are defined---i.e., whenever $E^1(\xi_0)$ is in the domain of $\Phi^{k_m}_{\xi_m} \circ \dots \Phi^{k_1}_{\xi_1} (p')$;
analogously for $E^1(\xi'_1)$. 

Define $\zeta$  as follows:
\[
\zeta = \sup \{ \lh(p'(\xi))+1 \setdef \xi \in E^1\cap \lh(p'), p' \in Z(p) \}.
\]

To show \eqref{e:no:unwanted:branches:repeat}, suppose $q_0 \leq p'$, $q_0\in P_\theta$ and we are given $\sigma \in {}^{<\kappa}2$  with $\lh(\sigma) > \zeta$.
Let $\nu = (\sigma,n,i,j)$.
By construction, there is at most one  pair $(\xi,\eta)$ such that $\xi < \theta$, $E^1(\eta) < \theta$, and $\sigma \is \Phi_\xi(q)(E^1(\eta))$.
Fixing such $(\eta,\xi)$, if it exists, we find $q \leq q_0$ in $P_\theta$ and $i' \in \omega$ such that  
$q \forces \Phi_\xi(r_\eta)(n) = \check i'$. If there is no such pair, just let $q = q_0$.

\medskip

Exactly one of the following holds:
\begin{enumerate}
\item For no  $\xi < \theta$ and $\eta$ such that $E^1(\eta) < \theta$ does it hold that $\sigma \is \Phi_\xi(q)(\eta)$;
\item For the unique $\xi < \theta$ and $\eta$ such that $E^1(\eta) < \theta$ and $\sigma \is \Phi_\xi(q)(E^1(\eta))$, $\eta$ is a successor and $j =0$, or otherwise  $\eta$ is a limit and $j=1$.
\item All of the above fail and $i' \neq i$;
\item All of the above fail; thus by definition of $\dot I_\xi$, $q \forces \nu \in \dot I_\xi$.
\end{enumerate}
In all but the last case, not only does $q \forces  \nu \notin \dot I_\xi$ hold, but in fact $q$ strongly rejects $\nu$.

Assume that one of the first three cases holds. 
As we have shown previously in this section, there is $t \in T(\nu)$ such that
forcing with $P_\theta(\leq t\cdot q)$ is the same as forcing with $T(\nu)(\leq t) \times \pnu$.
As $T(\nu)$ does not add reals and $\pnu$ does not add a branch through $T(\nu)$ by Lemma~\ref{l.lastlemma},
clearly
\[
t \cdot q \forces_{P_\theta} \forall u \in \omega^\omega\;  L[u]\vDash T(\nu) \text{ is Suslin}.
\]
Finally, no matter which of the above cases hold, we have found $t \cdot  q \leq q_0$ such that 
\[
t \cdot q \forces_{P_\theta}\text{``}\nu \in \dot I_\xi\text{''} \text{ or }  t\cdot q \forces \forall u \in \omega^\omega\;  L[u,r]\vDash T(\nu) \text{ is Suslin.}
\]
As $q_0$ was arbitrary below $p'$, \eqref{e:no:unwanted:branches:repeat} holds, proving Corollary~\ref{c:no:unwanted:branches} and Lemma~\ref{no:unwanted:branches} and thus proving that $\Gamma^0$ is $\Delta^1_3$. 
This finishes the proof of Theorem~\ref{p.t.main}.
\end{proof}

\backmatter
\bibliographystyle{amsplain}
\bibliography{friedman-schrittesser-projective}

\printindex
\printindex[notation]

\end{document}